\pdfoutput=1
\RequirePackage{ifpdf}
\ifpdf 
\documentclass[pdftex]{sigma}
\else
\documentclass{sigma}
\fi

\usepackage{euscript}
\usepackage{leftidx}
\usepackage[matrix,arrow,curve]{xy}

\newcommand{\End}{\mathop{\mathrm{End}}\nolimits}
\newcommand{\Ker}{\mathop{\mathrm{Ker}}\nolimits}
\newcommand{\Aut}{\mathop{\mathrm{Aut}}\nolimits}
\newcommand{\Hom}{\mathop{\mathrm{Hom}}\nolimits}
\newcommand{\Mat}{\mathop{\mathrm{Mat}}\nolimits}
\newcommand{\tr}{\mathop{\mathrm{tr}}\nolimits}

\newcommand{\id}{\mathop{\mathrm{id}}\nolimits}
\newcommand{\rk}{\mathop{\mathrm{rk}}\nolimits}

\newcommand{\Min}{\mathop{\mathrm{Min}}\nolimits}
\newcommand{\diag}{\mathop{\mathrm{diag}}\nolimits}

\newcommand{\Ob}{\mathop{\mathrm{Ob}}\nolimits}
\newcommand{\whotimes}{\mathbin{\wh\otimes}}
\newcommand{\isoright}{\xrightarrow{\smash{\raisebox{-0.65ex}{\ensuremath{\sim}}}}}

\numberwithin{equation}{section}

\renewcommand{\det}{\mathop{\mathrm{det}}\nolimits}
\newcommand{\perm}{\mathop{\mathrm{perm}}\nolimits}
\newcommand{\qdet}{\mathop{\mathrm{qdet}}\nolimits}
\newcommand{\inv}{\mathop{\mathrm{inv}}\nolimits}
\newcommand{\sgn}{\mathop{\mathrm{sgn}}\nolimits}
\renewcommand{\le}{\leqslant}
\renewcommand{\ge}{\geqslant}

\newcommand{\g}{\mathfrak{g}}
\newcommand{\la}{\langle}
\newcommand{\ra}{\rangle}

\newcommand{\A}{{\mathcal A}}
\newcommand{\B}{{\mathcal B}}
\newcommand{\C}{{\mathcal C}}
\newcommand{\D}{{\mathcal D}}
\newcommand{\G}{{\mathcal G}}
\newcommand{\He}{{\mathcal H}}
\newcommand{\Q}{{\mathcal Q}}
\newcommand{\M}{{\mathcal M}}
\newcommand{\N}{{\mathcal N}}
\newcommand{\K}{{\mathcal K}}
\newcommand{\U}{{\mathcal U}}
\newcommand{\gR}{{\mathfrak R}}
\newcommand{\gS}{{\mathfrak S}}
\newcommand{\gA}{{\mathfrak A}}
\newcommand{\gX}{{\mathfrak X}}
\newcommand{\gU}{{\mathfrak U}}

\newcommand{\wt}{\widetilde}
\newcommand{\wh}{\widehat}

\newtheorem{Th}{Theorem}[section]
\newtheorem{Lem}[Th]{Lemma}
\newtheorem{Prop}[Th]{Proposition}
\newtheorem{Cor}[Th]{Corollary}

{\theoremstyle{definition}

\newtheorem{Def}[Th]{Definition}
\newtheorem{Rem}[Th]{Remark}}

\newcommand{\CC}{{\mathbb{C}}}

\newcommand{\ZZ}{{\mathbb{Z}}}
\newcommand{\NN}{{\mathbb{N}}}
\newcommand{\SSS}{{\mathbf{S}}}

\begin{document}
\allowdisplaybreaks

\newcommand{\arXivNumber}{2009.05993}

\renewcommand{\PaperNumber}{066}

\FirstPageHeading

\ShortArticleName{Manin Matrices for Quadratic Algebras}

\ArticleName{Manin Matrices for Quadratic Algebras}

\Author{Alexey SILANTYEV~$^{\rm ab}$}

\AuthorNameForHeading{A.~Silantyev}

\Address{$^{\rm a)}$~Joint Institute for Nuclear Research, 141980 Dubna, Moscow region, Russia}
\Address{$^{\rm b)}$~State University ``Dubna'', 141980 Dubna, Moscow region, Russia}
\EmailD{\href{mailto:aleksejsilantjev@gmail.com}{aleksejsilantjev@gmail.com}}

\ArticleDates{Received December 11, 2020, in final form June 26, 2021; Published online July 12, 2021}

\Abstract{We give a general definition of Manin matrices for arbitrary quadratic algebras in terms of idempotents. We establish their main properties and give their interpretation in~terms of the category theory. The notion of minors is generalised for a general Manin matrix. We give some examples of Manin matrices, their relations with Lax operators and~obtain the formulae for some minors. In particular, we consider Manin matrices of the types~$B$, $C$ and $D$ introduced by A.~Molev and their relation with Brauer algebras. Infinite-dimensional Manin matrices and their connection with Lax operators are also considered.}

\Keywords{Manin matrices; quantum groups; quadratic algebras}
\Classification{16T99; 15B99}

\section{Introduction}

In the second half of the 80s Yuri Manin proposed to consider non-commutative quadratic algebras as a generalisation of vector spaces and called them ``quantum linear spaces''~\cite{Manin87,Manin88,ManinBook91}. A~usual finite-dimensional vector space is presented by the algebra of polynomials. The~linear maps correspond to the homomorphisms of these algebras in the category of graded algebras. One can consider homomorphisms of a quadratic algebra over a non-commutative ring (algebra). Such ``non-commutative'' homomorphisms of finitely generated quadratic algebras (``finite-dimensional'' quantum linear spaces) are described by matrices with non-commutative entries satisfying some commutation relations. We call them Manin matrices.

The main example proposed by Manin is the ``quantum plane'' defined by the algebra with~$2$~gene\-ra\-tors $x$, $y$ and the commutation relation $yx=qxy$. He established the connection of the quantum plane with some quantum group (Hopf algebra). This quantum group gives a~``non-commutative'' endomorphism of the quantum plane.

This picture was also generalised to the case of $n$ generators with the commutation relations $x_jx_i=qx_ix_j$, $i<j$. The algebra describing all the ``non-commutative'' endomorphisms of the corresponding quantum linear space was called right quantum algebra in~\cite{GLZ}, where the authors proved a $q$-analogue of MacMahon master theorem. The ``non-commutative'' endomorphisms and homomorphisms of these quadratic algebras are described by square and rectangular matrices respectively. We call them $q$-Manin matrices.

`Non-commutative' endomorphisms and homomorphisms of a usual finite-dimensional vector space generalise the matrices with commutative entries to the matrices with non-commutative entries satisfying certain commutation relations (these are the commutation relations of the right quantum algebra for $q=1$). This type of matrices were observed in the Talalaev's formula~\cite{T} and called then Manin matrices~\cite{CF}. These Manin matrices have a lot of applications to integrable systems, Sugawara operators, affine Lie algebras and quantum groups~\cite{CF,CM,RST}. Since they are an immediate generalisation of the usual matrices to the non-commutative case, almost all the formulae of the matrix theory are valid for these Manin matrices~\cite{CFR}. Most of them are also generalised to the $q$-Manin matrices~\cite{qManin}.

The ``non-commutative'' homomorphisms between two quantum linear spaces are described by an algebra that can be constructed as internal hom~\cite{Manin88}. These ``non-commutative'' homo\-mor\-phisms give us a notion of a Manin matrix for a pair of finitely generated quadratic algeb\-ras~$(\mathfrak A,\mathfrak B)$. These are the matrices with non-commutative entries satisfying the commutation relations of the internal hom algebra. In previous works on Manin matrices the attention was concentrated on the case $\mathfrak A=\mathfrak B$ corresponding to the ``non-commutative'' endomorphisms. This general case allows us to consider more examples and answer some questions about $q$-Manin matrices, which was unclear before.

For instance, the column determinant is a natural generalisation of the usual determinant for the $q=1$ Manin matrices, but some properties of this determinant are proved by using a~part of commutation relations for the entries of these matrices. We show here that this part of~commutation relations define Manin matrices for some pair of different quadratic algebras.

In the case $q\ne1$ the column determinant is generalised to (column) $q$-determinant. It is a~natural generalisation for $q$-Manin matrices. Permutation of columns of $q$-Manin matrices gives a matrix, which is not a $q$-Manin matrix. However its natural analogue of the determinant is also $q$-determinant. In contrast, the $q$-determinant is not relevant operation for a $q$-Manin matrix with permuted rows. To consider the permuted $q$-Manin matrices as another type of~Manin matrices and to define natural determinants for them one needs to consider some pairs of different quadratic algebras. It is convenient to do it in a more universal case when the role of quadratic algebras is played by multi-parametric deformations of the polynomial algebras.

A multi-parametric deformation of super-vector spaces was considered by Manin in the article~\cite{Manin89}. The super-versions of Manin matrices and MacMahon master theorem for the non-deformed case were also considered in~\cite{MR}. Here we do not consider the super-case, but we plan to do it in future works.

The notion of Manin matrices for two different quadratic algebras includes Manin matrices of types $B$, $C$ and $D$ introduced by Molev in~\cite{MolevSO}.

In this work we use tensor notations to describe quadratic algebras and Manin matrices by~gene\-ralising the tensor approach given in~\cite{CFR,qManin}. In this frame a quadratic algebra is defined by an idempotent that gives commutation relations for this algebra. E.g., the commutative algebra of polynomials is defined by the anti-symmetrizer of the tensor product of two copies of a vector space. Two different idempotents may define the same quadratic algebra. This gives an equivalence relation between the idempotents.

The relations for Manin matrices can be written in tensor notations with the corresponding idempotents. In the case of two different quadratic algebras these commutation relations are given by a pair of idempotents $A$ and $B$, so we call the matrices, satisfying these relations, $(A,B)$-Manin matrices. In the case $A=B$ we call them $A$-Manin matrices.

The important notion in the matrix theory is the notion of minor, which is usually defined as a determinant of a square submatrix. The minors (defined via column determinant) play a~role of decomposition coefficients in the expression for a ``non-commutative'' homomorphism of the anti-commutative polynomial algebras (Grassmann algebras). The dual notion is permanent of~square submatrices (where some rows and columns may be repeated) that give coefficients in~the case of commutative polynomial algebras.

This is directly generalised to some kinds of Manin matrices including the $q$-Manin matrices and the multi-parametric case. However in the general situation an analogue of minors is not defined by an operation (determinant or permanent) on submatrices. The minors are defined immediately without such operations. To do it we introduce some auxiliary ``dual'' quadratic algebras for a given idempotent and define a non-degenerate pairing between quadratic algebras (if it exists). In the tensor notation this pairing is written by using some higher idempotents, which we call pairing operators. These operators allow to define minors as entries of some operators with non-commutative entries, we call them minor operators. For the main examples the pairing operators is related with the representation theory of the symmetric groups, Hecke algebras and Brauer algebras.

The paper is organised as follows.

In Section~\ref{secQAMM} we consider quadratic algebras and related Manin matrices. In Section~\ref{secQA} we consider the quadratic algebras in terms of idempotents. Section~\ref{secLRE} is devoted to the equivalence relations between idempotents. In Sections~\ref{secAMM} and \ref{secABmanin} we define $A$- and $(A,B)$-Manin matrices in~terms of matrix commutation relations and in terms of quadratic algebras, we give the main properties of these matrices. In Section~\ref{secCommCat} we interpret Manin matrices in terms of comma categories and define a general right quantum algebra via adjoint functors. Section~\ref{secProd} is devoted to multiplication of Manin matrices and to bialgebra structure on the right quantum algebra. A generalisation of Manin matrices to the infinite-dimensional case was done in Section~\ref{secMO}.

In Section~\ref{secPc} we consider the following particular cases: the Manin matrices of~\cite{CF,CFR}, $q$-Ma\-nin matrices~\cite{qManin} and multi-parametric case~\cite{Manin89} (the Sections~\ref{secMm}, \ref{secqM} and \ref{secwhqMM} respectively). We recall basic properties of these Manin matrices and their determinants. In Section~\ref{sec4p} we consider a~generalisation of the $n=3$ case from Section~\ref{secwhqMM}.

Section~\ref{secLO} is devoted to relationship between Manin matrices and Lax operators. In Section~\ref{secUq} we explain the known connection of $q$-Manin matrices with some Lax operators (without spectral parameter) via a decomposition of an $R$-matrix into idempotents. This gives an important example of the idempotent equivalence. In Section~\ref{secLOY} we interpret Lax operators (with a~spectral parameter) for the rational $R$-matrix as infinite-dimensional Manin matrices.

In Section~\ref{secMin} we generalise the notion of minors for Manin matrices. Section~\ref{secPerm} has a motivating role, where we consider minors for $q$-Manin matrices (defined via determinant and permanent) as some decomposition coefficients. In Sections~\ref{secDualQA} and \ref{secPO} we define ``dual'' quadratic algebras and the corresponding parings via pairing operators. In Sections~\ref{secMinO} and \ref{secPMinO} we define minors as entries of minor operators and give some properties of these operators. In Section~\ref{secMinEq} we investigate the relation of pairing and minor operators for equivalent idempotents. Section~\ref{secConstr} devoted to an algebraic approach to construction of pairing operators.

We consider examples of pairing and minor operators in Section~\ref{secExMO}. Section~\ref{secMinqp} is devoted to the multi-parametric case (which includes the case of $q$-Manin matrices). In Section~\ref{secHecke} we again consider the case of $q$-Manin matrices, but we construct the pairing operators by using another idempotent (equivalent to the idempotent used in Section~\ref{secMinqp}), which was found in Section~\ref{secUq}. We show how they give related minor operators. Section~\ref{secPO4p} is devoted to the case defined in Section~\ref{sec4p}.

Section~\ref{secBCD} is devoted to the Manin matrices of types $B$, $C$ and $D$. In Section~\ref{secMolev} we recall the Molev's definition of these matrices and interpret them as Manin matrices for pairs of~idem\-po\-tents. We consider the corresponding quadratic algebras. In Section~\ref{secBrauer} we construct the pairing operators by means of Brauer algebra.

In Appendix~\ref{appRw} we give a formulae for a set of inversions in the case of arbitrary Weyl group. In Appendix~\ref{appLie} the universal enveloping algebras of Lie algebras are represented as quadratic algebras.

\section{Quadratic algebras and Manin matrices}
\label{secQAMM}

As a basic field we choose the field of complex numbers $\CC$, however we note that all the statements are valid for any field of characteristic zero, unless the contrary is explicitly indicated. All the vector spaces are supposed to be defined over $\CC$. By an {\it algebra} we understand an associative unital algebras over $\CC$ (not necessary commutative). By a {\it graded algebra} we mean an $\NN_0$-graded associative unital algebra over $\CC$, where $\NN_0$ is the set of non-negative integers: the grading of~such algebra has the form $\gA=\bigoplus_{k\in\NN_0}\gA_k$ and the condition $\gA_k\gA_l\subset\gA_{k+l}$ is implied for all $k,l\in\NN_0$. By a {\it quadratic algebra}%
\footnote{There are more general quadratic algebras over $\gR$, but we do not consider them in this work.
}
we mean a graded algebra of the form $\gA=\gR\otimes(TV/I)$, where~$\gR$ is an arbitrary algebra, $TV=\bigoplus_{k\in\NN_0}V^{\otimes k}=\CC\oplus V\oplus(V\otimes V)\oplus\cdots$ is the tensor algebra of a~finite-dimensional vector space $V$ and $I\subset TV$ is its (two-sided) ideal generated by a subspace of $V\otimes V$; in particular, $\gA_0=\gR$, $\gA_1=\gR\otimes V$.

Let $\CC^n$ be the space of complex column vectors of the size $n$. Its dual $(\CC^n)^*=\Hom(\CC^n,\CC)$ is the space of complex row vectors. Their standard bases $(e_i)_{i=1}^n$ and $(e^i)_{i=1}^n$ are dual to each other: $e^ie_j=\delta^i_j$. Let $E_{i}{}^{j}$ be the $n\times m$ matrices with entries $\big(E_{i}{}^{j}\big)_{kl}=\delta^k_i\delta^j_l$. It acts on $\CC^m$ from the left and on $(\CC^n)^*$ from the right as $E_{i}{}^{j}e_k=\delta^j_ke_i$, $e^iE_{j}{}^{k}=\delta^i_je^k$. We have $E_{i}{}^{j}=e_ie^j$.

We use the following tensor notations. Let $M\in\gR\otimes\Hom(\CC^m,\CC^n)$ be an $n\times m$ matrix over an algebra $\gR$. It can be considered as an operator from $\CC^m$ to $\CC^n$ with entries $M_{ij}=M^i_j\in\gR$, that is $M=\sum_{i,j}M^i_{j}E_{i}{}^{j}$. Introduce the notation
\begin{align*}
 M^{(a)}=\sum_{i,j}M^i_{j}(1\otimes\cdots\otimes1\otimes E_{i}{}^{j}\otimes1\otimes\cdots\otimes1),
\end{align*}
where $E_{i}{}^{j}$ is placed to the $a$-th site and $1$ are identity matrices (of different sizes in general). This is an operator from $\CC^{m_1}\otimes\cdots\otimes\CC^{m_r}$ to $\CC^{n_1}\otimes\cdots\otimes\CC^{n_r}$, where $m_a=m$, $n_a=n$ and $m_b=n_b$ for all $b\ne a$. Analogously, for a matrix $T=\sum_{i,j,k,l}T^{ij}_{kl}E_{i}{}^{k}\otimes E_{j}{}^{l}\in\gR\otimes\Hom\big(\CC^m\otimes\CC^{m'},\CC^{n}\otimes\CC^{n'}\big)$ with entries $T^{ij}_{kl}\in\gR$ we denote
\begin{align*}
 T^{(ab)}=\sum_{i,j,k,l}T^{ij}_{kl}\big(1\otimes\cdots\otimes1\otimes E_{i}{}^{k}\otimes1\otimes\cdots\otimes1\otimes E_{j}{}^{l}\otimes1\otimes\cdots\otimes1\big),
\end{align*}
with $E_{i}{}^{k}$ in the $a$-th site and $E_{j}{}^{l}$ in the $b$-th site. The numbers $a$ and $b$ should be different but the both cases $a<b$ and $a>b$ are possible. In particular,
\begin{align*}
 T^{(21)}=\sum_{i,j,k,l}T^{ij}_{kl}\big(E_{j}{}^{l}\otimes E_{i}{}^{k}\big).
\end{align*}
In the same way the notations $M^{(a)}$ and $T^{(ab)}$ can be defined for any vector spaces $V$, $W$, $V'$, $W'$ and any operators $M\in\gR\otimes\Hom(V,W)$, $T\in\gR\otimes\Hom(V\otimes V',W\otimes W')$ (in the infinite-dimensional case the tensor product with $\gR$ may be completed in some way).

\subsection{Quadratic algebras}
\label{secQA}

Consider a quadratic algebra with $n$ generators, that is an algebra generated by the ele\-ments $x^1,\dots,x^n$ over $\CC$ with some quadratic commutation relations. Since the number of ele\-ments~$x^ix^j$ is equal to $n^2$ the number of independent quadratic relations is less or equal to $n^2$. It means that these relations can be presented in the form
\begin{gather}
 \sum_{i,j=1}^n A^{kl}_{ij}x^i x^j= 0, \qquad k,l=1,\dots,n, \label{Axx00}
\end{gather}
where $A^{kl}_{ij}\in\CC$. The quadratic algebra with these relations is the quotient $TV/I$, where $V=(\CC^n)^*$ and $I$ is the ideal generated by the elements $\sum_{i,j=1}^n A^{kl}_{ij}e^i\otimes e^j$. An element $x^i\in TV/I$ is the class of $e^i\in(\CC^n)^*\subset T(\CC^n)^*$.

The coefficients $A^{kl}_{ij}$ can be considered as entries of a matrix $A$ acting on $\CC^n\otimes\CC^n$. In~terms of~basis this action looks as follows: $A(e_i\otimes e_j)=\sum_{k,l=1}^n A^{kl}_{ij}(e_k\otimes e_l)$. Note that for any invertible $n^2\times n^2$ matrix $G$, the product $GA$ defines the same quadratic algebra, since the relations $\sum_{ij=1}^n (GA)^{kl}_{ij}x^i x^j=0$ is equivalent to~\eqref{Axx00}.

\begin{Prop} \label{PropAidem}
 For each quadratic algebra the matrix $A$ defining its relations can be chosen to be an idempotent:
\begin{align}
 A^2=A. \label{Aidem}
\end{align}
\end{Prop}

\begin{proof} Let the given quadratic algebra be defined by the relations with a matrix $A$. If one has proved that there exists an invertible $G\in\Aut(\CC^n\otimes\CC^n)$ such that $(GA)^2=GA$ then we can choose $GA$ as an idempotent matrix defining the quadratic relations for this algebra. Since $G$ should be invertible the equation $GAGA=GA$ is equivalent to $AGA=A$. Let us reduce the matrix $A$ to a Jordan form. In the corresponding basis of $\CC^n\otimes\CC^n$ it takes the form
\begin{align*}
A=\begin{pmatrix}
 \alpha_1 & 0 &\cdots & 0 \\
	 0 & \alpha_2 &\cdots & 0 \\
	 \cdots &\cdots &\cdots & \cdots\\
	 0 & 0 &\cdots &\alpha_r
 \end{pmatrix}\!,
\end{align*}
where $\alpha_k$ are Jordan cells. Let us find the solution in the form (in the same basis)
\begin{align*}
G=\begin{pmatrix}
 \beta_1 & 0 &\cdots & 0 \\
	 0 & \beta_2 &\cdots & 0 \\
	 \cdots & \cdots&\cdots &\cdots \\
	 0 & 0 &\cdots &\beta_r
 \end{pmatrix}\!,
\end{align*}
where each $\beta_k$ is an invertible square submatrix, which has the same dimension as the Jordan cell $\alpha_k$. Then, the equation $AGA=A$ is equivalent to system of $r$ equations $\alpha_k\beta_k\alpha_k=\alpha_k$. If~the Jordan cell $\alpha_k$ corresponds to the non-zero eigenvalue of the matrix $A$, then the matrix $\alpha$ is invertible and we can choose $\beta_k=\alpha_k^{-1}$. Otherwise, $\alpha_k$ has the form
\begin{align*}
\alpha_k=\begin{pmatrix}
 0 & 1 & 0 &\cdots & 0 \\
	 0 & 0 & 1 &\cdots & 0 \\
	 \cdots &\cdots &\cdots &\cdots &\cdots \\
	 0 & 0 & 0 &\cdots & 1 \\
	 0 & 0 & 0 &\cdots & 0
 \end{pmatrix}
\end{align*}
and $\beta_k$ can be chosen as
\begin{align*}
\beta_k=\begin{pmatrix}
 0 & 0 & \cdots & 0 & 1 \\
	 1 & 0 &\cdots & 0 & 0 \\
	 0 & 1 &\cdots & 0 & 0 \\
	 \cdots &\cdots &\cdots &\cdots & \cdots\\
	 0 & 0 &\cdots & 1 & 0
 \end{pmatrix}\!.
\end{align*}
For $\alpha_k=0$ we choose $\beta_k=1$. In all the cases such chosen matrices $\beta_k$ are invertible. \end{proof}

\begin{Rem}
 This construction of an idempotent $A$ is appropriate for algebraically closed fields only. In general case one needs to consider decomposition of the vector space $(\CC^n\otimes\CC^n)^*$ into a direct sum $R\oplus R_c$ such that $R$ is the subspace spanned by the quadratic relations. Then the matrix $A$ is the transposed to the idempotent projecting $(\CC^n\otimes\CC^n)^*$ onto $R$ in parallel to the complementary subspace $R_c$.
\end{Rem}

Henceforth we suppose that the matrix $A$ is an idempotent unless otherwise specified. Denote the corresponding quadratic algebra by $\gX_A(\CC)$.

More generally, for an algebra $\gR$ define the quadratic algebra ${\gX}_A(\gR)=\gR\otimes{\gX}_A(\CC)$. This is a graded algebra generated by the elements $x^1,\dots,x^n$ over $\gR$ with the quadratic commutation relations~\eqref{Axx00} and $rx^i=x^ir$, $r\in\gR$, $i=1,\dots,n$. The formulae $\deg x^i=1$, $\deg r=0$ $\;\forall\,r\in\gR$ give the grading on $\gX_A(\gR)$. An idempotent $A\in\End(\CC^n\otimes\CC^n)$ defines a functor ${\gX}_A$ from the category of algebras to the category of graded algebras: ${\gX}_A(f)(r x^{i_1}\cdots x^{i_k})=f(r)x^{i_1}\cdots x^{i_k}$, $\forall f\in\Hom(\gR,\gR')$, $\forall r\in\gR$. Due to Proposition~\ref{PropAidem} any quadratic algebra is isomorphic to~$\gX_A(\gR)$ for some $A$ and $\gR$.

The relations~\eqref{Axx00} can be rewritten in matrix notations as follows. Consider the column vector $X=\sum_{i=1}^n x^i e_i$. Then the relations~\eqref{Axx00} takes the form
\begin{align}
 A(X\otimes X)=0, \label{Axx0}
\end{align}
where $X\otimes X=\sum_{i,j=1}^n x^ix^j (e_i\otimes e_j)$.

The fact that there is no another independent relations for $x^i$ besides~\eqref{Axx00} can be reformulated as follows: if $T_{ij}x^ix^j=0$ for some $T_{ij}\in\gR$ then $T_{ij}=\sum_{k,l=1}^n G_{kl}A^{kl}_{ij}$ for some $G_{kl}\in\gR$.

\begin{Lem} \label{propTXX}
 The following equations for a matrix $T\in\gR\otimes\End(\CC^n\otimes\CC^n)$ are equivalent:
\begin{gather}
 T(X\otimes X)=0, \label{T_XX_0}
 \\
 T(1-A)=0. \label{T_Aorth_0}
\end{gather}
\end{Lem}

\begin{proof}
 The equation~\eqref{T_XX_0} is derived from~\eqref{T_Aorth_0} by multiplying by $(X\otimes X)$ from the right. Conversely, if the equation~\eqref{T_XX_0} holds then we have $n^2$ relations $\sum_{i,j=1}^n T^{ab}_{ij}x^ix^j=0$ enumerated by $a,b=1,\dots,n$, where $T^{ab}_{ij}$ is entries of the matrix $T$. This implies $T^{ab}_{ij}=\sum_{k,l=1}^n G^{ab}_{kl}A^{kl}_{ij}$ for~some matrix $G=(G^{ab}_{kl})$ that is $T=GA$. By multiplying the last equality by $(1-A)$ on the right and by taking into account~\eqref{Aidem} one yields~\eqref{T_Aorth_0}.
\end{proof}

Let us consider the ``change of variables'' $y^i=\sum_{k=1}^mM^i_kx^k$, $i=1,\dots,n$, where $x^1,\dots,x^m$ are the generators of $\gX_A(\CC)$. In general, the transition matrix $M$ is an $n\times m$ matrix with non-commutative entries $M^i_k\in\gR$, so that $y^i\in\gX_A(\gR)$. In terms of $Y=\sum_{i=1}^n y^i e_i$ we have
\begin{align}
 Y=MX. \label{Y_MX_g}
\end{align}

\begin{Lem} \label{propAMMAorth}
The equation
\begin{align}
 A(Y\otimes Y)=0 \label{T_YY_0}
\end{align}
is equivalent to
\begin{align}
 AM^{(1)}M^{(2)}(1-A)=0. \label{AMMAorth0}
\end{align}
\end{Lem}

\begin{proof} Note that the generators $x^i\in\gX_A(\CC)$ commute with $M^k_l\in\gR$ as elements of algebra $\gX_A(\gR)=\gR\otimes\gX_A(\CC)$. Then by substituting~\eqref{Y_MX_g} to~\eqref{T_YY_0} we obtain $AM^{(1)}M^{(2)}(X\otimes X)=0$, which in turn is equivalent to~\eqref{AMMAorth0} by Lemma~\ref{propTXX}. \end{proof}

Consider the quadratic algebra $\Xi_A(\gR)$ generated by the elements $\psi_1,\dots,\psi_n$ over $\gR$ with the commutation relations
\begin{align}
 \psi_i \psi_j=\sum_{k,l=1}^n A^{kl}_{ij} \psi_k \psi_l \label{psi_psi_A_psi_psi}
\end{align}
and $r\psi_i=\psi_ir$, $r\in\gR$.
Each idempotent $A\in\End(\CC^n\otimes\CC^n)$ defines a functor $\Xi_A$ from the category of algebras to the category of graded algebras: $\Xi_A(\gR)=\gR\otimes\Xi_A(\CC)$. Consider the row vector $\Psi=\sum_{i=1}^n\psi_i e^i=(\psi_1,\dots,\psi_n)$. Then the relations~\eqref{psi_psi_A_psi_psi} take the form
\begin{align}
 (\Psi\otimes\Psi)(1-A)=0. \label{PsiPsiA}
\end{align}

\begin{Rem}
 If $\gX_A(\CC)$ is a Koszul algebra, then the algebra $\Xi_A(\CC)$ is its Koszul dual algebra.
\end{Rem}

We will use the following conventions. We write $\gX_A(\gR)=\gX_{A'}(\gR)$ iff there exists an alge\-bra isomorphism $\gX_A(\gR)\isoright\gX_{A'}(\gR)$ identical on generators, that is $x^i\mapsto x^i$ $\forall i$ and $r\mapsto r$ \mbox{$\forall r\in\gR$}. We write $\Xi_A(\gR)=\Xi_{A'}(\gR)$ iff there exists an isomorphism $\Xi_A(\gR)\isoright\Xi_{A'}(\gR)$ identical on~generators $\psi^i$ and $r\in\gR$. We also write $\Xi_A(\gR)=\gX_{A'}(\gR)$ iff there exists an isomorphism $\gX_A(\gR)\isoright\Xi_{A'}(\gR)$ such that $\psi_i\mapsto x^i$ $\forall i$ and $r\mapsto r$ $\forall r\in\gR$.

\begin{Rem}
 We have $\gX_A(\gR)=\gX_{A'}(\gR)$, $\Xi_A(\gR)=\Xi_{A'}(\gR)$ or $\Xi_A(\gR)=\gX_{A'}(\gR)$ iff $\gX_A(\CC)=\gX_{A'}(\CC)$, $\Xi_A(\CC)=\Xi_{A'}(\CC)$ or $\Xi_A(\CC)=\gX_{A'}(\CC)$ respectively. The latter equalities are exactly the isomorphisms of quadratic algebras $TV/I$ as $TV$-modules (left, right of two-sided), where in the case $\Xi_A(\CC)=\gX_{A'}(\CC)$ we identify $\CC^n$ with $(\CC^n)^*$ via $e_i\mapsto e^i$.
\end{Rem}

Note that $S=1-A$ is an idempotent iff $A$ is idempotent. The transformation $S\to SG$ by an invertible matrix $G\in\Aut(\CC^n\otimes\CC^n)$ does not change the relation~\eqref{PsiPsiA}. By trans\-po\-sing the relation~\eqref{PsiPsiA} we see that it is equivalent to the relation~\eqref{Axx0} with $X$ replaced with~$\Psi^\top$ and~$A$ replaced with $S^\top=1-A^\top$, where $(-)^\top$ means the matrix transposition. Hence \mbox{$\Xi_A(\gR)=\gX_{S^\top}(\gR)$} the functor $\Xi_A$ can be identified with the functor ${\gX}_{S^\top}={\gX}_{1-A^\top}$.

\subsection{Left and right equivalence of idempotents}
\label{secLRE}

A fixed quadratic algebra can be defined by different idempotents. Here we give a condition when it happens. To do it we introduce two equivalence relations for idempotent endomorphisms acting on a vector space. We prove some useful properties of these idempotents and their equivalence relations.

Let $V$ be a finite-dimensional vector space (here we use that the field $\CC$ is algebraically closed). Note that the Jordan form of an idempotent $A\in\End(V)$ in an appropriate basis of $V$ is a diagonal matrix $\diag(1,\dots,1,0,\dots,0)$. It can be written in the block form as
\begin{align} \label{AJordan}
 A=\begin{pmatrix}
 1 & 0 \\
 0 & 0
 \end{pmatrix}
\end{align}
with the blocks of the sizes $r\times r$, $r\times d$, $d\times r$, $d\times d$, where $r$ is the rank of $A$ and $d=\dim V-r$. We obtain the first property of the idempotents.

\begin{Prop}
 The rank of an idempotent $A\in\End(V)$ equals to its trace: $\rk A=\tr A$.
\end{Prop}

Let us say that two idempotents $A,A'\in\End(V)$ are {\it left-equivalent} (to each other) if there exists an operator $G\in\Aut(V)$ such that $A'=GA$, and we call them {\it right-equivalent} if there exists $G\in\Aut(V)$ such that $A'=AG$. The both relations are true equivalence relations.

\begin{Prop} \label{PropLRE}
 Two idempotents $A,A'\in\End(V)$ are left-equivalent iff the dual idempotents $S=1-A$ and $S'=1-A'$ are right-equivalent.
\end{Prop}

\begin{proof} Let us fix a basis of $V$ such that the idempotent $A$ has the form~\eqref{AJordan}. Suppose that~$A$ is left-equivalent to $A'$, then there exists an invertible matrix $G$ such that $A'=GA$. The~mat\-rices~$G$ and $A'$ have the form
\begin{gather*}
 G=\begin{pmatrix}
 \alpha & \beta \\
 \gamma & \delta
 \end{pmatrix}\!, \qquad
 A'=GA=
 \begin{pmatrix}
 \alpha & 0 \\
 \gamma & 0
 \end{pmatrix}\!,
\end{gather*}
where $\alpha$, $\beta$, $\gamma$ and $\delta$ are $d\times d$, $d\times\ell$, $\ell\times d$ and $\ell\times\ell$ matrices respectively. The idempotentness of~$A'$ implies $AA'=A$, that is $\alpha=1$. Hence $A'=A_\gamma$, where $A_\gamma:=\left(\begin{smallmatrix} 1 & 0 \\ \gamma & 0 \end{smallmatrix}\right)$. Since $S=1-A=\left(\begin{smallmatrix} 0 & 0 \\ 0 & 1 \end{smallmatrix}\right)$ and $S'=1-A'=\left(\begin{smallmatrix} 0 & 0 \\ -\gamma & 1 \end{smallmatrix}\right)$ we have $S'=SG_{-\gamma}$, where $G_\gamma:=\left(\begin{smallmatrix} 1 & 0 \\ \gamma & 1 \end{smallmatrix}\right)$. Further, note that two idempotents $S_1$ and $S_2$ are right-equivalent iff the idempotents $S_1^\top$ and $S_2^\top$ are left-equivalent. Thus the right equivalence of $S$ and $S'$ implies the left equivalence of $A=1-S$ and \mbox{$A'=1-S'$}. \end{proof}

\begin{Rem} \sloppy
 We see from the proof that any matrix left-equivalent to~\eqref{AJordan} has the form~\mbox{$A_\gamma=G_\gamma A$}. Analogously, any matrix right-equivalent to $S=1-A$ has the form \mbox{$S_\gamma=1-A_\gamma=SG_{-\gamma}$}.
\end{Rem}

\begin{Prop} \label{propSVequiv}
 If two idempotents $S,S'\in\End(V)$ are right-equivalent then $SV=S'V$.
 If~two idempotents $A,A'\in\End(V)$ are left-equivalent then $V^*A=V^*A'$,
 where $V^*=\Hom(V,\CC)$.
\end{Prop}

\begin{proof} If $G\in\Aut(V)$ then $GV=V$, so from $S'=SG$ one yields $S'V=SGV=SV$. The~second statement follows from the first one for $S=A^\top$, $S'=(A')^\top$. \end{proof}

\begin{Lem} \label{lemAS0}
If two idempotents $A,A'\in\End(V)$ satisfy the relations $A'(1-A)=0$ and $A(1-A')=0$ then they are left-equivalent.
\end{Lem}
\begin{proof} By substituting~\eqref{AJordan} and $A'=\left(\begin{smallmatrix} \alpha & \beta \\ \gamma & \delta \end{smallmatrix}\right)$ to these relations we obtain $\beta=0$, $\delta=0$ and~$\alpha=1$, so that $A'=A_\gamma=G_\gamma A$ (see the proof of Proposition~\ref{PropLRE}). \end{proof}

Consider the case $V=\CC^n\otimes\CC^n$. Recall that any idempotent $A\in\End(\CC^n\otimes\CC^n)$ defines quadratic algebras $\gX_A(\CC)$ and $\Xi_A(\CC)$. 

\begin{Prop} \label{PropLEQA}
Let $A,A'\in\End(\CC^n\otimes\CC^n)$ be idempotents. Then the following conditions are equivalent.
\begin{itemize}\itemsep=0pt
 \item[$(a)$] $A$ is left-equivalent to $A'$.
 \item[$(b)$] $S=1-A$ is right-equivalent to $S'=1-A'$.
 \item[$(c)$] $\gX_A(\CC)=\gX_{A'}(\CC)$.
 \item[$(d)$] $\Xi_A(\CC)=\Xi_{A'}(\CC)$.
\end{itemize}
\end{Prop}

\begin{proof} The conditions $(a)$ and $(b)$ are equivalent due to Proposition~\ref{PropLRE}. Two left-equivalent idempotents $A$ and $A'$ define the same commutation relations~\eqref{Axx0}. This means that $(a)$ imp\-lies~$(c)$. Let us prove the converse implication. If $\gX_A(\CC)=\gX_{A'}(\CC)$ then $A'(X\otimes X)=0$ and~$A(X\otimes X)=0$. By Lemma~\ref{propTXX} we obtain the relations $A'(1-A)=0$ and $A(1-A')=0$. Due to Lemma~\ref{lemAS0} these relations imply the left equivalence of $A$ and $A'$. Since $\Xi_A(\CC)=\gX_{S^\top}(\CC)$ the equivalence of $(b)$ and $(d)$ is derived by considering the transposed matrices. \end{proof}

\begin{Rem}
Note that the vanishing of $A'(1-A)$ and $A(1-A')$ is equivalent to \mbox{$V^*A=V^*A'$}. Due to Proposition~\ref{PropLRE} and~Lemma~\ref{lemAS0} can be reformulated in the following way: idempotents~$A$ and~$A'$ are left equivalent iff $V^*A=V^*A'$.
Analogously, idempotents $S$ and $S'$ are right equivalent iff $SV=S'V$.
For the case $V=\CC^n\otimes\CC^n$ we can add two more equivalent conditions to the list of~Proposition~\ref{PropLEQA}:
\begin{itemize}\itemsep=0pt
 \item[$(e)$] $(\CC^n\otimes\CC^n)^*A=(\CC^n\otimes\CC^n)^*A'$ and

 \item[$(f)$] $S(\CC^n\otimes\CC^n)=S'(\CC^n\otimes\CC^n)$.
\end{itemize}
Thus, the choice of an idempotent $A$ from a class of the left equivalence corresponds to the choice of~complementary subspace $R_c$ in the decomposition $V^*=R\oplus R_c$, where $R$ is the fixed subspace defining this class.
Note also that all the statements of this subsection are valid for any field $\CC$.
\end{Rem}

\subsection[A-Manin matrices]{$\boldsymbol A$-Manin matrices}
\label{secAMM}

\begin{Def}
 An $n\times n$ matrix $M$ over some algebra $\gR$ satisfying the equation
\begin{align}
 AM^{(1)}M^{(2)}(1-A)=0 \label{AMMdef}
\end{align}
is called {\it Manin matrix corresponding to the idempotent $A$} or simply {\it $A$-Manin matrix}.
\end{Def}

The definition~\eqref{AMMdef} can be rewritten in the following form
\begin{align}
 AM^{(1)}M^{(2)}&=AM^{(1)}M^{(2)}A. \label{AMMAMMA}
\end{align}
This relation means that the expression $AM^{(1)}M^{(2)}$ is invariant with respect to multiplication by $A$ from the right.

\begin{Prop}\label{propA}
 Let $M$ be an $n\times n$ matrix over $\gR$. Let $x^i$ and $\psi_i$ be the generators of $\gX_A(\CC)$ and $\Xi_A(\CC)$ respectively. Consider the elements $y^i\in\gX_A(\gR)$ and $\phi_i\in\Xi_A(\gR)$ defined as follows:
\begin{gather}
\begin{pmatrix}
y^1 \\ \cdots \\ y^n
\end{pmatrix}
=
\begin{pmatrix}
 M^1_1 & \cdots & M^1_n \\
\cdots &\cdots & \cdots \\
 M^n_1 & \cdots & M^n_n
\end{pmatrix}
\begin{pmatrix}
 x^1 \\ \cdots \\ x^n
\end{pmatrix}\!, \label{PropYMX}
\\[1ex]
(\phi_1,\dots,\phi_n) =
(\psi_1,\dots,\psi_n)
\begin{pmatrix}
 M^1_1 & \cdots & M^1_n \\
\cdots &\cdots & \cdots \\
 M^n_1 & \cdots & M^n_n
\end{pmatrix}\!. \label{PropPhiPsiM}
\end{gather}
Then the following three conditions are equivalent:
\begin{itemize}\itemsep=0pt
\item The matrix $M$ is an $A$-Manin matrix.
\item The elements $y^i$ satisfy the relations~\eqref{Axx00}, i.e., $A(Y\otimes Y)=0$, where $Y=\sum_{i=1}^n y^i e_i$.
\item The elements $\phi_i$ satisfy the relations~\eqref{psi_psi_A_psi_psi}, i.e., $(\Phi\otimes\Phi)(1-A)=0$, where $\Phi=\sum_{i=1}^n \phi_i e^i$.
\end{itemize}
\end{Prop}

\begin{proof} Note that the formulae~\eqref{PropYMX} and \eqref{PropPhiPsiM} have the form $Y=MX$ and $\Phi=\Psi M$. The~second condition is equivalent to the first one due to Lemma~\ref{propAMMAorth}. The third condition can be written as $\big(1-A^\top\big)\big(\Phi^\top\otimes\Phi^\top\big)=0$. Since the operator $1-A^\top$ is also idempotent we can again apply Lemma~\ref{propAMMAorth}, so the third condition is equivalent to the equation
\begin{align*}
 \big(1-A^\top\big)\big(M^\top\big)^{(1)}\big(M^\top\big)^{(2)}A^\top=0. 
\end{align*}
Transposition of this equation and the formula $\big(\big(M^\top\big)^{(1)}\big(M^\top\big)^{(2)}\big)^\top=M^{(1)}M^{(2)}$ gives exa\-c\-tly~\eqref{AMMdef}, which is the first condition. \end{proof}

Note that the relation~\eqref{AMMdef} has the form $AM^{(1)}M^{(2)}S=0$, where $S=1-A$. It is equiva\-lent~to
\begin{align} \label{SMMS}
 M^{(1)}M^{(2)}S=SM^{(1)}M^{(2)}S.
\end{align}

Let $P=1-2A$. Then $P^2=1$. We have $A=\frac{1-P}2$ and $S=\frac{1+P}2$. The idempotents $A$ and $S$ are often given by the matrix $P$ satisfying $P^2=1$. The relation~\eqref{AMMdef} in terms of $P$ takes the form
\begin{align} \label{PMMP}
 M^{(1)}M^{(2)}-PM^{(1)}M^{(2)}+M^{(1)}M^{(2)}P-PM^{(1)}M^{(2)}P=0.
\end{align}

\begin{Rem}
 The notations $A$, $S$, $P$ and the sign of $P$ is chosen according to the basic case: the algebra of commutative polynomials. In this case the roles of $A$, $S$ and $P$ are played by antisymmetrizer, symmetrizer and permutation matrix respectively. See Section~\ref{secMm} for details.
\end{Rem}

Consider some trivial examples. For $A=0$ we have $P=S=1$. The algebra $\gX_0(\CC)=\CC\langle x^1,\dots,x^n\rangle$ is the algebra of all the non-commutative polynomials of the variables $x^1,\dots,x^n$ (without any relations). The algebra $\Xi_0(\CC)$ is defined by the relations $\psi_i\psi_j=0$, $i,j=1,\dots,n$. As a linear space it has the from $\Xi_0(\CC)=\CC\oplus\CC\psi_1\oplus\dots\oplus\CC\psi_n$. For $A=1$ we have $P=-1$, $S=0$, $\gX_1(\CC)=\CC\oplus\CC x^1\oplus\dots\oplus\CC x^n\cong\Xi_0(\CC)$, $\Xi_1(\CC)=\CC\langle\psi_1,\dots,\psi_n\rangle\cong\gX_0(\CC)$. In the both cases the relation defining Manin matrices has the form $0=0$, so that any $n\times n$ matrix $M$ is a~Manin matrix for the idempotent $A=0$ as well as for the idempotent $A=1$.

Due to Propositions~\ref{PropLEQA} and \ref{propA} the notion of $A$-Manin matrix does not change if we substitute $A$ by a left-equivalent idempotent $A'$. This means that $M$ is an $A$-Manin matrix iff it is an $A'$-Manin matrix. This implies that $A$-Manin matrices are associated to the quadratic algebras $\gX_A(\CC)$ and $\Xi_A(\CC)$ (with fixed generators).

More generally, the definition of $A$-Manin matrix can be written in the form
\begin{align*}
 \wt A M^{(1)}M^{(2)}(1-A')=0,
\end{align*}
where $\wt A=\wt G A$, $A'=GA$ and $\wt G,G\in\Aut(\CC^n\otimes\CC^n)$ are such that $A'$ is an idempotent (i.e., $AGA=A$). The most general form is
\begin{align*}
 \wt A M^{(1)}M^{(2)}\wt S=0,
\end{align*}
where $\wt A=\wt G A$ and $\wt S=SG$ for any $\wt G,G\in\Aut(\CC^n\otimes\CC^n)$.

\subsection[(A,B)-Manin matrices]
{$\boldsymbol{(A,B)}$-Manin matrices}\label{secABmanin}

Let $A\in\End(\CC^n\otimes\CC^n)$ and $B\in\End(\CC^m\otimes\CC^m)$ be two idempotents and $\gR$ be an algebra. Let~$x^1,\dots,x^m$ be generators of $\gX_B(\CC)$ and $\psi_1,\dots,\psi_n$ be generators of $\Xi_A(\CC)$. They satisfy the relations $B(X\otimes X)=0$ and $(\Psi\otimes\Psi)(1-A)=0$, where $X=\sum_{j=1}^mx^je_j$ and $\Psi=\sum_{k=1}^n\psi_ke^k$. Let~$M$ be an $n\times m$ matrix with entries $M^i_j\in\gR$. Consider a ``change of variables''
\begin{gather}
y^i=\sum_{j=1}^mM^i_jx^j\in\gX_B(\gR), \qquad i=1,\dots,n, \label{yMx}
\\
\phi_j=\sum_{i=1}^n\psi_iM^i_j\in\Xi_A(\gR), \qquad j=1,\dots,m. \label{phipsiM}
\end{gather}
Consider $n\times1$ and $1\times m$ matrices $Y=\sum_{i=1}^ny^ie_i$ and $\Phi=\sum_{j=1}^m\phi_je^j$. Then one can rewrite~\eqref{yMx}, \eqref{phipsiM} in the matrix form: $Y=MX$ and $\Phi=\Psi M$. Now let us generalise Proposition~\ref{propA} to this case.

\begin{Prop}\label{propAB}
The following three conditions are equivalent:
\begin{gather}
 AM^{(1)}M^{(2)}(1-B)=0,\label{AMMB}
 \\[1ex]
 A(Y\otimes Y)=0,\nonumber 
 \\[1ex]
 (\Phi\otimes\Phi)(1-B)=0.\nonumber
\end{gather}
\end{Prop}

The proof of this proposition repeats the proof of Proposition~\ref{propA} with $B$ instead of $A$ somewhere.

\begin{Def} \label{defABManin}
 An $n\times m$ matrix $M$ over an algebra $\gR$ satisfying the equation~\eqref{AMMB} is called {\it $(A,B)$-Manin matrix} or {\it Manin matrix corresponding to the pair of idempotents $(A,B)$}.
\end{Def}

The defining relation for the $(A,B)$-Manin matrices can be also written in terms of dual idempotents $S$ and matrices $P$. For instance, the relation~\eqref{AMMB} is equivalent to~\eqref{PMMP}, where we should understand the matrices $P$ placing to the left from the matrices $M$ as $1-2A$ and $P$ placing to the right as $1-2B$. In the similar way we can generalise the relation~\eqref{SMMS}.

 {\sloppy\begin{Prop} \label{propABManinxy}
 Let $\gA$ be an algebra and $M$ be an $(A,B)$-Manin matrix with entries $M^i_j\in\gA$. Let $x^1,\dots,x^m\in\gA$ be elements commuting with all $M^i_j$ and satisfying $\sum_{i,j=1}^mB^{kl}_{ij}x^ix^j=0$, \mbox{$k,l=1,\dots,m$}, $($in particular, it is valid for the algebra $\gA=\gX_B(\gR)$, the generators \mbox{$x^i\in\gX_B(\CC)$} and an~$(A,B)$-Manin matrix $M$ over $\gR)$. Then the elements $y^i=\sum_{j=1}^mM^i_jx^j$ satisfy $\sum_{i,j=1}^nA^{kl}_{ij}y^iy^j=0$, $k,l=1,\dots,n$.
\end{Prop}}

\begin{proof} We use the matrix notations introduced below. From $Y=MX$ and~\eqref{AMMB} we obtain
\begin{gather*}
A(Y\otimes Y)=AM^{(1)}M^{(2)}(X\otimes X)=AM^{(1)}M^{(2)}B(X\otimes X).
\end{gather*}
The right hand side vanishes since $B(X\otimes X)=0$. \end{proof}

Note if $M$ is an $(A,B)$-Manin matrix then $M^\top$ is $\big(1-B^\top,1-A^\top\big)$-Manin matrix. By~means of the identification of the functors $\Xi_A$ and $\gX_{1-A^\top}$ we can write Proposition~\ref{propABManinxy} in the follo\-wing~form.

{\sloppy \begin{Prop} \label{propABManinpsiphi}
Let $\gA$ be a algebra and $M$ be an $(A,B)$-Manin matrix with entries \mbox{$M^i_j\in\gA$}. Let $\psi_1,\dots,\psi_n\in\gA$ be elements commuting with all $M^i_j$ and satisfying $\sum_{i,j=1}^mA^{ij}_{kl}\psi_i\psi_j=\psi_k\psi_l$, $k,l=1,\dots,n$, $($in particular, it is valid for the algebra $\gA=\Xi_A(\gR)$, the generators \mbox{$\psi^i\in\Xi_A(\CC)$} and an $(A,B)$-Manin matrix $M$ over $\gR)$. Then the elements $\phi_j=\sum_{i=1}^nM^i_j\psi_i$ satisfy $\sum_{i,j=1}^mB^{ij}_{kl}\phi_i\phi_j=\phi_k\phi_l$, $k,l=1,\dots,m$.
\end{Prop}

}

For arbitrary operators $A'\in\End(\CC^n\otimes\CC^n)$ and $S'\in\End(\CC^m\otimes\CC^m)$ (not necessarily idempotents) the relation $A'M^{(1)}M^{(2)}S'=0$ is equivalent to $(G_1A')M^{(1)}M^{(2)}(S'G_2)=0$, where $G_1\in\Aut(\CC^n\otimes\CC^n)$ and $G_2\in\Aut(\CC^n\otimes\CC^n)$. Hence it is equivalent to~\eqref{AMMB} for some idempotents $A$ and $B$, which have the forms $A=G_1A'$ and $B=1-S'G_2$. In particular, if $A'$ is an~idempotent it is left-equivalent to the idempotent $A$; if $S'$ is an idempotent it is right-equivalent to $1-B$, which means that $1-S$ is left-equivalent to $B$ (see Proposition~\ref{PropLRE}). Thus we obtain the following statement.

\begin{Prop} \label{propMMlequiv}
 Let $M$ be an $n\times m$ matrix over $\gR$. Let $A,A'\in\End(\CC^n\otimes\CC^n)$ and $B,B'\in\End(\CC^m\otimes\CC^m)$ be idempotents. Suppose $A$ is left-equivalent to $A'$ and $B$ is left-equivalent to~$B'$ $($equivalently, $1-A$ is right-equivalent to $1-A'$ and $1-B$ is right-equivalent to $1-B')$. Then~$M$ is an $(A,B)$-Manin matrix iff it is an $(A',B')$-Manin matrix.
\end{Prop}

We see from this proposition that the property of the matrix $M$ to be an $(A,B)$-Manin matrix effectively depends on the algebras $\gX_A(\CC)=\gX_{A'}(\CC)$ and $\gX_B(\CC)=\gX_{B'}(\CC)$ (with fixed generators). So the notion of $(A,B)$-Manin matrix can be associated with the pair of $\gX$-quadratic algebras $\gX_A(\CC)$ and $\gX_B(\CC)$. Alternatively it can be associated with the pair of $\Xi$-quadratic algebras $\Xi_A(\CC)$ and $\Xi_B(\CC)$. We will see this in Section~\ref{secCommCat} more explicitly.

Consider the question of permutation of rows and columns of a Manin matrix. Let $\SSS_n$ be the $n$-th symmetric group. For a permutation $\sigma\in\SSS_n$ let us permute the rows of an $n\times m$ matrix $M$ in the following way: we put the $i$-th row to the place of the $\sigma(i)$-th row. We denote the obtained matrix by $^\sigma M$. By permuting columns of $M$ with a permutation $\sigma\in\SSS_m$ in the same way we obtain a matrix denoted by $\leftidx{_\sigma} M$. More explicitly we have $(^\sigma M)^{\sigma(i)}_j=M^i_j$ and $(_\sigma M)^i_{\sigma(j)}=M^i_j$. We write the permutation-index from the left since $_\tau(_\sigma M)=\leftidx{_{\tau\sigma}}{M}$ and $^\tau(^\sigma M)=\leftidx{^{\tau\sigma}}{M}$ for any~$\tau$,~$\sigma$ from $\SSS_n$ and $\SSS_m$ respectively. Note that the space $\CC^n$ has a structure of $\SSS_n$-module: a~permutation $\sigma\in\SSS_n$ acts by the formula $\sigma e_i=e_{\sigma(i)}$, and we denote the corresponding operator $\CC^n\to\CC^n$ by the same letter $\sigma$. In this notation we have $^\sigma M=\sigma M$ and $_\sigma M=M\sigma^{-1}$.

\begin{Prop} \label{PropPerm}
 Let $M$ be an $n\times m$ matrix over $\gR$. Let $\sigma\in\SSS_n$ and $\tau\in\SSS_m$ or, more generally, $\sigma\in {\rm GL}(n,\CC)$, $\tau\in {\rm GL}(m,\CC)$. Then the following statements are equivalent:
\begin{itemize}
\itemsep=0pt
 \item $M$ is an $(A,B)$-Manin matrix.
 \item $^\sigma M=\sigma M$ is a $\big((\sigma\otimes\sigma)A\big(\sigma^{-1}\otimes\sigma^{-1}\big),B\big)$-Manin matrix.
 \item $_\tau M=M\tau^{-1}$ is a $\big(A,(\tau\otimes\tau)B\big(\tau^{-1}\otimes\tau^{-1}\big)\big)$-Manin matrix.
\end{itemize}
\end{Prop}

\begin{proof} Multiplication of the condition~\eqref{AMMB} by the operator $\sigma\otimes\sigma$ from the left gives the relation $(\sigma\otimes\sigma)A\big(\sigma^{-1}\otimes\sigma^{-1}\big)(\sigma M)^{(1)}(\sigma M)^{(2)}(1-B)=0$. By multiplying~\eqref{AMMB} by $\tau\otimes\tau$ from the right we obtain $A\big(M\tau^{-1}\big)^{(1)}\big(M\tau^{-1}\big)^{(2)}
(\tau\otimes\tau)(1-B)\big(\tau^{-1}\otimes\tau^{-1}\big)=0$.
\end{proof}

Proposition~\ref{PropPerm} can be interpreted in terms of linear change of generators of the corresponding quadratic algebras. Indeed, consider another generators $\tilde\psi_i=\sum_{j=1}^n\alpha^j_i\psi_j$ of the alge\-bra $\Xi_A(\CC)$, where $\alpha^j_i$ are entries of an invertible matrix $\alpha\in {\rm GL}(n,\CC)$. Let $\sigma\in {\rm GL}(n,\CC)$ be the inverse matrix: $\sum_{i=1}^n\alpha^j_i\sigma^i_k=\delta^j_k$. By substituting $\psi_k=\sum_{i=1}^n\sigma^i_k\tilde\psi_i$ to~\eqref{phipsiM} we obtain $\phi_j=\sum_{i=1}^n\tilde\psi_i(\sigma M)^i_j$. The quadratic commutation relation for the generators $\tilde\psi_i$ are $\tilde A^{ij}_{kl}\tilde\psi_i\tilde\psi_j=\tilde\psi_k\tilde\psi_l$, where $\tilde A^{ij}_{kl}$ are entries of the idempotent matrix $\tilde A=(\sigma\otimes\sigma)A\big(\sigma^{-1}\otimes\sigma^{-1}\big)$. The formulae $\tilde\psi_i=\sum_{j=1}^n\alpha^j_i\psi_j$, $\psi_k=\sum_{i=1}^n\sigma^i_k\tilde\psi_i$ describe the isomorphism $\Xi_A(\CC)\cong\Xi_{\tilde A}(\CC)$. It corresponds to the change of basis in $(\CC^n)^*$, namely $\Psi=\sum_{i=1}^n\psi_ie^i=\sum_{k=1}^n\tilde\psi_k\tilde e^k$, where $\tilde e^k=\sum_{i=1}^n\sigma^k_ie^i$. Since $\gX_A(\CC)=\Xi_{1-A^\top}(\CC)$ we have an isomorphism $\gX_A(\CC)\cong\gX_{\tilde A}(\CC)$, where $A$ is an arbitrary idempotent and $\tilde A=(\sigma\otimes\sigma)A\big(\sigma^{-1}\otimes\sigma^{-1}\big)$.

Analogously, the matrix $M\tau^{-1}$ corresponds to the new generators $\tilde x^i=\sum_{j=1}^m\tau^i_jx^j$ of the algebra $\gX_B(\CC)$, where $\tau^j_i$ are entries of $\tau\in {\rm GL}(m,\CC)$. The transition to these generators corresponds to the change of basis in $\CC^n$. The new basis elements are $\tilde e_j =\sum_{i=1}^m\beta^i_je_i$, where~$\beta^i_j$ are such that $\sum_{j=1}^m\beta^i_j\tau^j_k=\delta^i_k$. The formulae $\tilde x^i=\sum_{j=1}^m\tau^i_jx^j$, $x^k=\sum_{i=1}^m\beta^k_i\tilde x^i$ define an~isomorphism $\gX_B(\CC)\cong\gX_{\tilde B}(\CC)$, where $\tilde B=(\tau\otimes\tau)B\big(\tau^{-1}\otimes\tau^{-1}\big)$.

Thus the notion of $(A,B)$-Manin matrix is not associated with a pair of isomorphism classes of quadratic algebras. However, isomorphisms of quadratic algebras gives a transformation of~the Manin matrix in the same way as the change of bases of vector spaces $V$ and $W$ transform a~matrix of a linear operator $V\to W$. The notion associated with isomorphism classes of~quadratic algebras is a notion of Manin operator, which we will introduce in Section~\ref{secMO}.

Finally, consider the case when $A=0$ or $B=1$ in the both cases the defining relation~\eqref{AMMB} is $0=0$. This means that any $n\times m$ matrix is a $(0,B)$-Manin matrix as well as an $(A,1)$-Manin matrix, where $0\in\End(\CC^n\otimes\CC^n)$ and $1\in\End(\CC^m\otimes\CC^m)$. Further, the definition~\eqref{AMMB} for~$B=0\in\End(\CC^m\otimes\CC^m)$ has the from $AM^{(1)}M^{(2)}=0$. By multiplying if by $(1-B)$ from the right with an arbitrary idempotent $B\in\End(\CC^m\otimes\CC^m)$ we obtain~\eqref{AMMB}. Thus any $(A,0)$-Manin matrix is in particular an $(A,B)$-Manin matrix. Analogously, any $(1,B)$-Manin matrix is also an $(A,B)$-Manin matrix, where $1\in\End(\CC^n\otimes\CC^n)$.

\subsection{Comma categories and Manin matrices}
\label{secCommCat}

Consider first the set $\Hom\big(\gX_A(\CC),\gX_B(\gR)\big)$ with $A$ and $B$ as above. It consists of algebra homomorphisms $f\colon\gX_A(\CC)\to\gX_B(\gR)$ preserving the grading. Let $x^1_A,\dots,x^n_A$ and $x^1_B,\dots,x^m_B$ be the generators of $\gX_A(\CC)$ and $\gX_B(\CC)$ respectively, they satisfy $\sum_{i,j=1}^nA^{kl}_{ij}x^i_Ax^j_A=0$ and $\sum_{i,j=1}^m B^{kl}_{ij}x^i_Bx^j_B=0$. To give a homomorphism $f\in\Hom\big(\gX_A(\CC),\gX_B(\gR)\big)$ it is enough to give its value on the generators $x^i_A$; since $f(x^i_A)$ has degree $1$ in $\gX_B(\CC)$ it has the form
\begin{align}
 f(x^i_A)=\sum_{j=1}^m M^i_jx^j_B, \label{whMxMx}
\end{align}
where $M^i_j\in\gR$. This means that each $f$ is given by a matrix $M=\big(M^i_j\big)$ over $\gR$. Proposition~\ref{propA} implies that the formula~\eqref{whMxMx} defines a homomorphism $f$ iff the matrix $M=\big(M^i_j\big)$ is an~$(A,B)$-Manin matrix. Denote this homomorphism $f\colon\gX_A(\CC)\to\gX_B(\gR)$ by $f_M$. Thus we obtain a~bijection $f_M\leftrightarrow M$ between $\Hom\big(\gX_A(\CC),\gX_B(\gR)\big)$ and the set of all $(A,B)$-Manin matrices over the algebra $\gR$. Note that this bijection depends on the choice of generators of the algeb\-ras~$\gX_A(\CC)$ and $\gX_B(\CC)$.

More generally, consider the set $\Hom\big(\gX_A(\gS),\gX_B(\gR)\big)$ for two algebras $\gS$ and $\gR$. It can be identified with a subset of $\Hom(\gS,\gR)\times\Hom\big(\gX_A(\CC),\gX_B(\gR)\big)$, since each graded homomorphism $f\colon\gX_A(\gS)\to\gX_B(\gR)$ is given on the zero-degree elements $s\in\gS$ and the generators $x^i_A\in\gX_A(\CC)$. Let $\alpha\colon\gS\to\gR$ be an algebra homomorphism and $f_M\in\Hom\big(\gX_A(\CC),\gX_B(\gR)\big)$ be a graded algebra homomorphism defined by an $(A,B)$-Manin matrix $M=\big(M^i_j\big)$, then the formulae
\begin{gather} \label{falphaM}
 f(s)=\alpha(s), \qquad s\in\gS, \qquad f(x^i_A)=f_M(x^i_A)=\sum_{j=1}^mM^i_jx^j_B,\qquad i=1,\dots,n,
\end{gather}
define a homomorphism $f\colon\gX_A(\gS)\to\gX_B(\gR)$ iff $\alpha(s)M^i_j=M^i_j\alpha(s)$ for all $s\in\gS$, $i=1,\dots,n$ and $j=1,\dots,m$. We write $f=(\alpha,f_M)$ in this case. We have
\begin{gather}
\Hom\big(\gX_A(\gS),\gX_B(\gR)\big)\nonumber
\\ \qquad
{}= \Big\{(\alpha,f_M)\in\Hom(\gS,\gR)\times\Hom\big(\gX_A(\CC),\gX_B(\gR)\big)
 \mid \big[\alpha(s),M^i_j\big]=0\;\forall\,s,i,j\Big\}.\label{falphaMset}
\end{gather}

Analogously, we obtain that the set $\Hom\big(\Xi_B(\CC),\Xi_A(\gR)\big)$ consists of the algebra homomorphisms $f^M\colon\Xi_B(\CC)\to\Xi_A(\gR)$ defined on generators as $f^M\big(\psi^B_j\big)=\sum_{i=1}^nM^i_j\psi^A_i$, where $M^i_j\in\gR$ are entries of an $(A,B)$-Manin matrix $M=\big(M^i_j\big)$. More generally,
\begin{gather*}
 \Hom\big(\Xi_B(\gS),\Xi_A(\gR)\big)
 \\ \qquad
 {}= \Big\{\big(\alpha,f^M\big)\in\Hom(\gS,\gR)\times\Hom\big(\Xi_B(\CC),\Xi_A(\gR)\big)\mid \big[\alpha(s),M^i_j\big]=0\;\forall\,s,i,j\Big\}.
\end{gather*}

\begin{Def}[{\cite{Mcl}}]\label{defMcl}
Let $\C$ and $\D$ be categories, $c$ be an object of $\C$ and $G\colon\D\to\C$ be a~functor. The {\it comma category} $(c\downarrow G)$ consists of the pairs $(d,f)$, where $d$ is an object of~$\D$ and $f\in\Hom_{\C}(c,Gd)$ is a morphism in $\C$. A morphism in $(c\downarrow G)$ from $(d,f)$ to $(d',f')$ is a~morphism $h\in\Hom_{\D}(d,d')$ such that the diagram
\begin{align*}
 \xymatrix{& c\ar@{->}[dl]_{f}\ar@{->}[dr]^{f'} \\
 Gd\ar@{->}[rr]_{Gh} & & Gd'}
\end{align*}
is commutative.
\end{Def}

Let $\A$ be the category of associative unital algebras over $\CC$ and $\G$ be the category of associative unital $\NN_0$-graded algebras over $\CC$. By setting $\C=\G$, $\D=\A$, $c=\gX_A(\CC)$ and $G=\gX_B$ in Definition~\ref{defMcl} we obtain the comma category $\big(\gX_A(\CC)\downarrow\gX_B\big)$. It consists of the pairs $(\gR,f_M)$, where~$\gR$ is an algebra and $f_M\in\Hom\big(\gX_A(\CC),\gX_B(\gR)\big)$ is a homomorphism corresponding to~some $(A,B)$-Manin matrix $M$ over~$\gR$. Thus we can interpret the $(A,B)$-Manin matrices as~objects of the comma category $\big(\gX_A(\CC)\downarrow\gX_B\big)$.

The morphisms in $\big(\gX_A(\CC)\downarrow\gX_B\big)$ from $(\gR,f_M)$ to $(\gR',f_{M'})$ are all the homomorphisms $h\colon\gR\to\gR'$ such that $h\big(M^i_j\big)=(M')^i_j$, where $M^i_j$ and $(M')^i_j$ are entries of the corresponding $(A,B)$-Manin matrices $M$ and $M'$. Two $(A,B)$-Manin matrices $M\in\gR\otimes\Hom(\CC^m,\CC^n)$ and $M'\in\gR'\otimes\Hom(\CC^m,\CC^n)$ are isomorphic to each other iff there exists an isomorphism $h\colon\gR\isoright\gR'$ such that $h\big(M^i_j\big)=(M')^i_j$ (note that $(A,B)$-Manin matrices which have the same entries but are formally considered over different algebras may be non-isomorphic).

In the same way the $(A,B)$-Manin matrices can be interpreted as objects of the comma category $\big(\Xi_B(\CC)\downarrow\Xi_A\big)$. So that the last one is equivalent to $\big(\gX_A(\CC)\downarrow\gX_B\big)$.

One of the application of the comma categories in the category theory is the universal morphism (universal arrow).

\begin{Def}[{\cite{Mcl}}]
Let ${\C}$ and ${\D}$ be categories, $c$ be an object of ${\C}$ and $G\colon{\D}\to{\C}$ be a~functor. The {\it universal morphism} from $c$ to $G$ is the initial object in the comma category $(c\downarrow G)$, that is a pair $(r,u)$ of an object $r\in{\D}$ and a morphism $u\in\Hom_{\C}(c,Gr)$ such that for any object $d\in{\D}$ and any morphism $f\in\Hom_{\C}(c,Gd)$ there is a unique morphism $h\in\Hom_{\D}(r,d)$ making the diagram
\begin{align*}
 \xymatrix{c\ar@{->}[r]^{u}\ar@{->}[dr]_{f}& Gr\ar@{->}[d]^{Gh} \\
 & Gd}
\end{align*}
commutative.
\end{Def}

As an initial object the universal morphism is unique up to an isomorphism. Let us des\-cribe the universal morphism $(\gU_{A,B},u_{A,B})$ from the object $\gX_A(\CC)$ to the functor $\gX_B$. It consists of the algebra $\gU_{A,B}$ generated by the elements $\M^i_j$, $i=1,\dots,n$, $j=1,\dots,m$, with the commutation relation $A\M^{(1)}\M^{(2)}(1-B)=0$, where $\M$ is the $n\times m$ matrix with ent\-ries~$\M^i_j$. Due~to~Defi\-ni\-tion~\ref{defABManin} it is an $(A,B)$-Manin matrix. It defines a homomorphism $u_{A,B}= f_\M\colon\gX_A(\CC)\to\gX_B(\gU_{A,B})$. On can check that $(\gU_{A,B},u_{A,B})$ is the universal morphism from $\gX_A(\CC)$ to $\gX_B$: for~any~morphism $f=f_M\colon\gX_A(\CC)\to\gX_B(\gR)$ corresponding to the $(A,B)$-Manin matrix $M$ with entries~$M^i_j\in\gR$ there is a unique homomorphism $h\colon\gU_{A,B}\to\gR$ such that $f=\gX_B(h)u_{A,B}$, namely, $h(\mathcal M^i_j)=M^i_j$. Via the equivalence of comma categories $\big(\gX_A(\CC)\downarrow\gX_B\big)$ and~\mbox{$\big(\Xi_B(\CC)\downarrow\Xi_A\big)$} we obtain the universal morphism from the object $\Xi_B(\CC)$ to the functor $\Xi_A$, this is the pair $(\gU_{A,B},\wt u_{A,B})$, where $\wt u_{A,B}=f^\M\colon\Xi_B(\CC)\to\Xi_A(\gU_{A,B})$.

Let us call the matrix $\M$ {\it the universal $(A,B)$-Manin matrix}.
The algebra $\gU_{A,B}$ is a genera\-lisation of the right quantum algebra~\cite{GLZ}, so we call it {\it the right quantum algebra for the pair of~idempotents $(A,B)$}.

Consider more general comma category $\big(\gX_A(\gS)\downarrow\gX_B\big)$. It consists of pairs $(\gR,f)$ of an~alge\-bra $\gR$ and a homomorphism $f\in\Hom\big(\gX_A(\gS),\gX_B(\gR)\big)$. Such a homomorphism $f$ can be identified with a pair of homomorphisms in the sense of the formulae~\eqref{falphaM}, \eqref{falphaMset}. If~$f=(\alpha,f_M)$ we write also $(\gR,f)=(\gR,\alpha,f_M)$. A morphism $h\colon(\gR,\alpha,f_M)\to(\gR',\alpha',f_{M'})$ is a~homomorphism $h\colon\gR\to\gR'$ such that $h\alpha=\alpha'$ and $h\big(M^i_j\big)=(M')^i_j$.

The initial object in $\big(\gX_A(\gS)\downarrow\gX_B\big)$ is the universal morphism $(\gS\otimes\gU_{A,B},\iota,u_{A,B})$, where $\iota\colon\gS\hookrightarrow\gS\otimes\gU_{A,B}$ is the inclusion. A homomorphism $f=(\alpha,f_M)\colon\gX_A(\gS)\to\gX_B(\gR)$ gives the unique homomorphism $h\colon\gS\otimes\gU_{A,B}\to\gR$ defined by the formulae $h(s)=\alpha(s)$ and $h(\mathcal M^i_j)=M^i_j$ and making the diagram
\begin{gather} \label{diagXAS}\begin{split}&
 \xymatrix{\gX_A(\gS)\ar@{->}[rr]^{(\iota,u_{A,B})}\ar@{->}[drr]_{f}&& \gX_B(\gS\otimes\gU_{A,B})\ar@{->}[d]^{\gX_B(h)} \\
 && \gX_B(\gR)}
\end{split}
\end{gather}
commutative.

Now let us calculate a composition
\begin{align*}
 \gX_A(\gS)\xrightarrow{(\alpha,f_M)}\gX_B(\gR)
\xrightarrow{(\beta,f_N)}\gX_C(\mathfrak T),
\end{align*}
where $C\in\End(\CC^k\otimes\CC^k)$ is an idempotent, $\mathfrak T$ is an algebra, $\beta\colon\gR\to\mathfrak T$ is a homomorphism and $f_N\colon\gX_B(\CC)\to\gX_C(\mathfrak T)$ is defined by $(B,C)$-Manin matrix $N=(N^j_l)$. On the ele\-ments~$s\in\gS$ we have $(\beta,f_N)(\alpha,f_M)(s)=\beta(\alpha(s))$. On the generators $x^i_A\in\gX_A(\CC)$ we obtain $(\beta,f_N)(\alpha,f_M)(x^i_A)=(\beta,f_N)(\sum_{j=1}^mM^i_jx^j_B)=\sum_{j=1}^m\sum_{l=1}^k\beta\big(M^i_j\big)N^j_lx^l_C$. Thus,
\begin{align}
(\beta,f_N)(\alpha,f_M)=\big(\beta\alpha,f_K\big), \label{whK}
\end{align}
where $K=\beta(M)N$ is the $n\times k$ matrix with entries $K^i_l=\sum_{j=1}^m\beta\big(M^i_j\big)N^j_l\in\mathfrak T$. Since~\eqref{whK} is a morphism the matrix $K$ is an $(A,C)$-Manin matrix.

Recall also (see~\cite{Mcl}) that the functor $F\colon\C\to\D$ is called {\it left adjoint} to the functor $G\colon\D\to\C$ (while the functor $G\colon\D\to\C$ is called {\it right adjoint} to $F\colon\C\to\D$), iff there is an isomorphism $\Hom_\C(c,Gd)\cong\Hom_\D(Fc,d)$ natural in $c$ and $d$. To construct a left adjoint functor to a functor $G\colon\D\to\C$ it is enough to construct a universal morphism $(r_c,u_c)$ from each $c\in\C$ to $G$. Then the left adjoint functor on objects is defined as $Fc=r_c$. For a morphism $\alpha\colon c\to c'$ the morphism $F\alpha\colon Fc\to Fc'$ is the unique morphism $h\colon Fc\to Fc'$ such that the diagram
\begin{align*}
 \xymatrix{c\ar@{->}[r]^{u_c}\ar@{->}[d]_{\alpha}& GFc\ar@{->}[d]^{Gh} \\
 c'\ar@{->}[r]^{u_{c'}} & GFc'}
\end{align*}
is commutative (as consequence, we obtain a natural transformation $u\colon\id_\C\to GF$ with components $u_c\colon c\to GFc$).

Let $\Q$ be the full subcategory of $\G$ consisting of the graded algebras of the form $\gX_A(\gS)$. This is the category of quadratic algebras. We can consider the functors $\gX_A$ and $\Xi_A$ as functors from~$\A$ to~$\Q$. This means that we can substitute $\G$ by $\Q$ in the previous considerations.

From the diagram~\eqref{diagXAS} we see that the functor $G=\gX_B\colon\A\to\Q$ has a left adjoint functor $F=F_B\colon\Q\to\A$. It is defined on objects as $F_B\big(\gX_A(\gS)\big)=\gS\otimes\gU_{A,B}$. In particular, on the quadratic algebras $\gX_A(\CC)\in\Q$ the functor $F_B$ gives the right quantum algebra $\gU_{A,B}$.

Let us calculate the functor $F_B$ on a morphism $\gX_A(\gS)\xrightarrow{(\alpha,f_N)}\gX_{A'}(\gS')$, where $\alpha\colon\gS\to\gS'$ is a~homomorphism and $f_N$ is defined by an $(A,A')$-Manin matrix $N=(N^i_j)$. We need to construct a~commutative diagram
\begin{align*}
 \xymatrix{\gX_A(\gS)\ar@{->}[rr]^{(\iota,u_{A,B})}\ar@{->}[d]_{(\alpha,f_N)}\ar@{->}[rrd]^{f}&& \gX_B(\gS\otimes\gU_{A,B})\ar@{->}[d]^{\gX_B(h)} \\
 \gX_{A'}(\gS')\ar@{->}[rr]^{(\iota',u_{A',B})} && \gX_B(\gS'\otimes\gU_{A',B}).}
\end{align*}
Remind that $u_{A,B}=f_\M$ and $u_{A',B}=f_{\M'}$, where $\M'$ is the universal $(A',B)$-Manin matrix. By~the formula~\eqref{whK} we have $f=(\iota\alpha,f_K)$, where $K=N\mathcal M'$. Thus we obtain the homomorphism $F_B(\alpha,f_N)=h\colon\gS\otimes\gU_{A,B}\to\gS'\otimes\gU_{A',B}$ defined by the formulae
\begin{gather*}
 h(s)=\alpha(s),\qquad s\in\gS,\qquad h(\mathcal M^i_j)=\sum_l N^i_l(\mathcal M')^l_j.
\end{gather*}

The left adjoint to the functor $\Xi_A\colon\A\to\Q$ is the functor $\wt F_A\colon\Q\to\A$ defined on objects as $\wt F_A\big(\Xi_B(\gS)\big)=\gS\otimes\gU_{A,B}$. The value of $\wt F_A$ on $(\alpha,f^N)\in\Hom\big(\Xi_B(\gS),\Xi_{B'}(\gS')\big)$ is the homomorphism $h\colon\gS\otimes\gU_{A,B}\to\gS'\otimes\gU_{A,B'}$ such that $h(s)=\alpha(s)$ and $h(\M^i_j)=\sum_l(\M')^i_lN^l_j$, where $\M=(\M^i_j)$ and $\M'=\big((\M')^i_l\big)$ are the universal $(A,B)$- and $(A,B')$-Manin matrices.

In particular, for $\gS=\CC$ we have the natural bijections
\begin{align}
 \Hom\big(\gX_A(\CC),\gX_B(\gR)\big)=\Hom\big(\Xi_B(\CC),\Xi_A(\gR)\big)=\Hom(\gU_{A,B},\gR). \label{HomgXHomXi}
\end{align}
This means that the set of $(A,B)$-Manin matrices over $\gR$ is identified with the set of algebra homomorphisms $f\colon\gU_{A,B}\to\gR$. Namely, each $(A,B)$-Manin matrix $M=\big(M^i_j\big)$ has the form $M^i_j=f(\M^i_j)$ for some homomorphism $f\colon\gU_{A,B}\to\gR$ and any homomorphism $f\colon\gU_{A,B}\to\gR$ gives an $(A,B)$-Manin matrix $M$. In terms of non-commutative geometry~\cite{ManinBook91} this means that an $(A,B)$-Manin matrix over $\gR$ is an $\gR$-point of the algebra $\gU_{A,B}$.

\begin{Prop}
 Let $A\in\End(\CC^n\otimes\CC^n)$ and $A'\in\End(\CC^m\otimes\CC^m)$ be idempotents. Then the following three statements are equivalent:
\begin{itemize}
\itemsep=0pt
 \item $\gX_A(\CC)\cong\gX_{A'}(\CC)$ as graded algebras.
 \item $\Xi_A(\CC)\cong\Xi_{A'}(\CC)$ as graded algebras.
 \item $n=m$ and there exists $\sigma\in {\rm GL}(n,\CC)$ such that $A$ is left-equivalent to the idempotent $(\sigma\otimes\sigma)A'\big(\sigma^{-1}\otimes\sigma^{-1}\big)$ $($that is $A'=GA(\sigma\otimes\sigma)$ for some operators $G\in\Aut(\CC^n\otimes\CC^n)$ and~$\sigma\in\Aut(\CC^n))$.
\end{itemize}
\end{Prop}

\begin{proof} If $A$ is left-equivalent to $\wt A= (\sigma\otimes\sigma)A'\big(\sigma^{-1}\otimes\sigma^{-1}\big)$ then $\gX_A(\CC)=\gX_{\wt A}(\CC)\cong\gX_{A'}(\CC)$ and $\Xi_A(\CC)=\Xi_{\wt A}(\CC)\cong\Xi_{A'}(\CC)$ (see Section~\ref{secABmanin}). Conversely let $f\colon\gX_A(\CC)\to\gX_{A'}(\CC)$ be an~isomorphism and $f^{-1}\colon\gX_{A'}(\CC)\to\gX_A(\CC)$ be its inverse. They correspond to $(A,A')$-Manin matrices $\sigma\in\Hom(\CC^m,\CC^n)$ and $\sigma^{-1}\in\Hom(\CC^n,\CC^m)$ respectively. Due to invertibility of these matrices we must have $n=m$. The relations $A(\sigma\otimes\sigma)(1-A')=0$ and $A'\big(\sigma^{-1}\otimes\sigma^{-1}\big)(1-A)=0$ imply $A(1-\wt A)=0$ and $\wt A(1-A)=0$. By virtue of Lemma~\ref{lemAS0} the idempotents $A$ and $\wt A$ are left-equivalent. One can similarly prove that the second statement implies the third one. \end{proof}

\begin{Rem} \label{RemIsom}
 In the same way one can prove that $\gX_A(\gS)$ and $\gX_{A'}(\gR)$ are isomorphic as graded algebras iff $n=m$, $\gS\cong\gR$ and there exist invertible matrices $M\in\mathcal Z(\gR)\otimes\End(\CC^n)$, $G\in\gR\otimes\End(\CC^n\otimes\CC^n)$ such that $A'=GA(M\otimes M)$, where $\mathcal Z(\gR)=\{z\in\gR\mid [z,r]=0\;\forall\,r\in\gR\}$ is the centre of $\gR$. The isomorphism and its inverse have the form $(\alpha,f_M)\colon\gX_A(\gS)\to\gX_{A'}(\gR)$ and $\big(\alpha^{-1},f_N\big)\colon\gX_{A'}(\gR)\to\gX_A(\gS)$, where $\alpha\colon\gS\to\gR$ is an isomorphism in $\A$ and $N=\alpha^{-1}\big(M^{-1}\big)\in\mathcal Z(\gS)\otimes\End(\CC^n)$.
\end{Rem}

\subsection{Products of Manin matrices and co-algebra structure}
\label{secProd}

The main property of Manin matrices is that their product is also a Manin matrix under some condition.

\begin{Prop} \label{propProd}
 Let $\gR$ be an algebra. Let $A\in\End(\CC^n\otimes\CC^n)$, $B\in\End(\CC^m\otimes\CC^m)$ and $C\in\End\big(\CC^k\otimes\CC^k\big)$ be idempotents.
 Let $M$ and $N$ be $n\times m$ and $m\times k$ matrices over $\gR$. Suppose that $M$ and $N$ are $(A,B)$- and $(B,C)$-Manin matrices respectively and that
\begin{gather} \label{MijNjl}
\big[M^i_j,N^{j'}_l\big]=0 \qquad \forall i=1,\dots,n,\quad j,j'=1,\dots,m,\quad l=1,\dots,k.
\end{gather}
Then the $n\times k$ matrix $K=MN$ is an $(A,C)$-Manin matrix over $\gR$.
\end{Prop}

\begin{proof} Since $M$ is an $(A,B)$-Manin matrix we have $AM^{(1)}M^{(2)}=AM^{(1)}M^{(2)}B$. The commutativity condition~\eqref{MijNjl} implies that $N^{(1)}M^{(2)}=M^{(2)}N^{(1)}$, hence
\begin{align*}
 AK^{(1)}K^{(2)}(1-C)&=AM^{(1)}N^{(1)}M^{(2)}N^{(2)}(1-C)=
 AM^{(1)}M^{(2)}N^{(1)}N^{(2)}(1-C)
 \\
 &=AM^{(1)}M^{(2)}BN^{(1)}N^{(2)}(1-C).
\end{align*}
The right hand side vanishes since $N$ is a $(B,C)$-Manin matrix. \end{proof}

\begin{Rem}
 Proposition~\ref{propProd} can be proved by using the functors $\gX_A$ or $\Xi_A$. In the case of $\gX_A$ one needs to consider the elements $y^j=\sum_lN^j_lx^l$ and $z^i=\sum_jM^i_jy^j$ of the algebra $\gA=\gX_C(\gR)$, where $x^l=x^l_C$. By applying Proposition~\ref{propABManinxy} twice we obtain the relations $\sum_{k,l}A^{ij}_{kl}z^kz^l=0$. Since $z^i=\sum_l K^i_lx^l$ the matrix $K$ is an $(A,C)$-Manin matrix by virtue of Proposition~\ref{propAB}. Analogously, one can apply Propositions~\ref{propABManinpsiphi} and \ref{propAB} to the elements $\phi_j=\sum_iM^i_j\psi_i$ and $\chi_l=\sum_jN^j_l\phi_j$ of the algebra $\Xi_A(\gR)$.
\end{Rem}

In some particular case the property claimed in Proposition~\ref{propProd} was deduced in Section~\ref{secCommCat} (see the formula~\eqref{whK} and the text after it). The role of $\gR$ is played by the algebra $\mathfrak T$. The~homo\-morphism $\beta\colon\gR\to\mathfrak T$ gives a $\gR$-algebra structure on $\mathfrak T$ and maps entry-wise the $(A,B)$-Manin matrix $M=\big(M^i_j\big)$ over $\gR$ to a $(A,B)$-Manin matrix $\beta(M)=\big(\beta\big(M^i_j\big)\big)$ over $\mathfrak T$. The entries of this matrix commute with the entries of $N$ since $[\beta(r),N^i_j]=0$ for any $r\in\gR$.

The typical case of Proposition~\ref{propProd} is $\gR=\gS\otimes\gS'$, $M^i_j\in\gS$, $N^j_l\in\gS'$. In this case it can be formulated in terms of comma categories as follows. Let $(\gS,f_M)$ and $(\gS',f_N)$ be objects of the categories $\big(\gX_A(\CC)\downarrow\gX_B\big)$ and $\big(\gX_B(\CC)\downarrow\gX_C\big)$ corresponding to $(A,B)$- and $(B,C)$-Manin matrices $M$ and $N$. Then $(\gS\otimes\gS',f_K)$ is an object of $\big(\gX_A(\CC)\downarrow\gX_C\big)$, where $K=MN$. In~particular, for $A=B=C$ we obtain a structure of tensor category on the comma category $\big(\gX_A(\CC)\downarrow\gX_A\big)$ with the unit object $\big(\CC,\id_{\gX_A(\CC)}\big)$ corresponding to the unit matrix.

Proposition~\ref{propProd} can be formulated in terms of right quantum algebras. This formulation was described in the works~\cite{Manin87,Manin88,ManinBook91}.

\begin{Prop} \label{DeltaABC}
 Let $\M$, $\N$ and $\K$ be universal $(A,B)$-, $(B,C)$- and $(A,C)$-Manin matrices respectively. They have entries $\M^i_j\in\gU_{A,B}$, $\N^i_j\in\gU_{B,C}$ and $\K^i_j\in\gU_{A,C}$. Then the formula
\begin{align*}
 \Delta_{A,B,C}(\K^i_l)=\sum_{j=1}^m\M^i_j\otimes\N^j_l
\end{align*}
defines a homomorphism $\Delta_{A,B,C}\colon\gU_{A,C}\to\gU_{A,B}\otimes\gU_{B,C}$.
\end{Prop}

\begin{proof} Let $\gR=\gU_{A,B}\otimes\gU_{B,C}$. The matrix $\wt\K$ with entries $\wt\K^i_l=\Delta_{A,B,C}\big(\K^i_l\big)\in\gR$ is the product of the matrices $\M$ and $\N$. We need to check that $A\wt\K^{(1)}\wt\K^{(2)}(1-C)=0$. The last equality means that $\wt\K$ is an $(A,C)$-Manin matrix, but this follows from Proposition~\ref{propProd}. \end{proof}

Note that in terms of functors $F_A,\wt F_A\colon\Q\to\A$ constructed in Section~\ref{secCommCat} (as left adjoint functors to the functors $\gX_A,\Xi_A\colon\A\to\Q$) we have $\Delta_{A,B,C}=F_C(u_{A,B})=\wt F_A(\wt u_{B,C})$, where $u_{A,B}=f_\M\colon\gX_A(\CC)\to\gX_B(\gU_{A,B})$ and $\wt u_{B,C}=f^{\N}\colon\Xi_C(\CC)\to\Xi_B(\gU_{B,C})$ are the corresponding universal homomorphisms.

Consider the algebra $\gU_A=\gU_{A,A}$. Proposition~\ref{DeltaABC} implies that the map $\Delta_A=\Delta_{A,A,A}$ is a~homomorphism $\gU_A\to\gU_A\otimes\gU_A$. Moreover, it is easy to check that $\Delta_A$ is a comultiplication, that is $(\id\otimes\Delta_A)\Delta_A=(\Delta_A\otimes\id)\Delta_A$. More generally, we have the following commutative diagram:
\begin{align*}
 \xymatrix{\gU_{A,D}\ar@{->}[rrr]^{\Delta_{A,B,D}}\ar@{->}[d]_{\Delta_{A,C,D}}&&& \gU_{A,B}\otimes\gU_{B,D}\ar@{->}[d]^{\id\otimes\Delta_{B,C,D}} \\
 \gU_{A,C}\otimes\gU_{C,D}\ar@{->}[rrr]^{\Delta_{A,B,C}\otimes\id} &&& \gU_{A,B}\otimes\gU_{B,C}\otimes\gU_{C,D},
}
\end{align*}
which reflects the associativity of the matrix multiplication.

\begin{Prop}
 The algebra $\gU_A$ has a bialgebra structure defined by the following comultiplication $\Delta_A\colon\gU_A\to\gU_A\otimes\gU_A$ and counit $\varepsilon_A\colon\gU_A\to\CC$:
\begin{gather} \label{DeltaA}
 \Delta_A(\M^i_k)=\sum_{j=1}^n \M^i_j\otimes\M^j_k, \qquad
 \varepsilon_A\big(\M^i_j\big)=\delta^i_j.
\end{gather}
\end{Prop}

The formula~\eqref{whMxMx} for the universal $A$-Manin matrix $\M$ gives a coaction of the bialgebra $\gU_A$ on the algebra $\gX_A(\CC)$. This is a homomorphism $\delta=f_\M\colon\gX_A(\CC)\to\gU_A\otimes\gX_A(\CC)$ defined as
\begin{align*}
 \delta(x^i)=\sum_{j=1}^n\M^i_j\otimes x^j.
\end{align*}
It satisfies the coaction axiom $(\id\otimes\delta)\delta=(\Delta\otimes\id)\delta$. In terms of non-commutative geometry the algebra $\gX_A(\CC)$ is interpreted as an algebra of functions on a non-commutative space and~the coaction $\delta$ as an action on this space. Thus the bialgebra $\gU_A$ (or its dual) plays the role of~algebra of endomorphisms of a non-commutative space corresponding to the algebra $\gX_A(\CC)$.

More generally, for arbitrary $A,B,C$ we have $(\id\otimes u_{B,C})u_{A,B}=(\Delta_{A,B,C}\otimes\id)u_{A,C}$ and $(\id\otimes\wt u_{A,B})\wt u_{B,C}=\big(\Delta^{\rm op}_{A,B,C}\otimes\id\big)\wt u_{A,C}$, where $\Delta^{\rm op}_{A,B,C}\colon\gU_{A,C}\to\gU_{B,C}\otimes\gU_{A,B}$ is an algebra homo\-morphism defined as $\Delta^{\rm op}_{A,B,C}\big(\K^i_l\big)=\sum_{j=1}^m\N^j_l\otimes\M^i_j$. In particular, the homomorphism $\wt u_{A,A}$: $\Xi_A(\CC)\to\gU_A\otimes\Xi_A(\CC)$ is a coaction on the algebra $\Xi_A(\CC)$ with respect to the comultiplication $\Delta^{\rm op}_A=\Delta^{\rm op}_{A,A,A}\colon\gU_{A}\to\gU_{A}\otimes\gU_{A}$.

\begin{Rem}
 The bialgebra $\gU_A$ is not a Hopf algebra: an antipode $S$ for the bialgebra structure~\eqref{DeltaA} should have the form $S(\M)=\M^{-1}$, but the matrix $\M$ is not invertible over $\gU_A$. One can extend the algebra $\gA_A$ by adding new generators $\wt\M^i_j$ being elements of the formal inverse matrix $\wt\M=\M^{-1}$. Then the matrix $S(\M)^\top$ should be inverse to $\wt M^\top$, however its invertibility is not guaranteed in this extended algebra. The universal construction of a Hopf algebra extending the bialgebra $\gU_A$ is {\it Hopf envelope}~\cite{ManinBook91}. This is the algebra $\mathcal H_A$ generated by~the entries of the infinite series of matrices $\M_k$, $k\in\NN_0$, with the relations
\begin{gather*}
 \M_k^\top\M_{k+1}=\M_{k+1}\M_k^\top=1, \qquad
 A\M_{2k}^{(1)}\M_{2k}^{(2)}(1-A)=0, \\
 \big(1-A^\top\big)\M_{2k+1}^{(2)}\M_{2k+1}^{(1)}A^\top=0,
\end{gather*}
$k\in\NN_0$. Its Hopf structure is given by the formulae
\begin{gather*} 
 \Delta_A\big((\M_k)^i_l\big)=\sum_{j=1}^n(\M_k)^i_j\otimes(\M_k)^j_l, \qquad
 \varepsilon_A(\M_k)=1, \qquad
 S(\M_k)=\M_{k+1}^\top.
\end{gather*}
The algebra $\gU_A$ is mapped to $\mathcal H_A$ by the formula $\M\mapsto\M_0$. Note also that the matrix $\wt\M=\M_0^{-1}=\M_1$ satisfies the relation
\begin{align*}
A\wt\M^{(2)}\wt\M^{(1)}(1-A)=0. 
\end{align*}
\end{Rem}

 The complex square and rectangular matrices are often interpreted as homomorphisms in a~category with objects $n\in\NN_0$ and homomorphisms $\Hom(m,n)=\Mat_{n\times m}(\CC)$ (it is equi\-va\-lent to the category of finite-dimensional vector spaces). One can interpret the Manin mat\-ri\-ces in a similar way. Let $\A'\subset\A$ be a small full tensor subcategory, this means that $\A'$ is a full subcategory of~$\A$ such that $\CC\in\A'$, $\gR\otimes\gS\in\A'$ $\;\forall\,\gR,\gS\in\A'$ and $\Ob(\A')$ is a~small set (in practice one needs only a small set of algebras, so we can take the full tensor subcategory generated by this set as $\A'$). Define the following category $\M_{\A'}$. The objects of $\M_{\A'}$ are all the idempotents $A\in\End(\CC^n\otimes\CC^n)$, $n\in\NN_0$. The homomorphisms $A\to B$ are all the $(A,B)$-Manin matrices $M$ over algebras $\gR\in\A'$.
 In other words, we define $\Hom(A,B)=\bigsqcup_{\gR\in\A'}\Hom\big(\gX_A(\CC),\gX_B(\gR)\big)$, these sets are small since $\A'$ is small. The composition of homomorphisms $(M,\gR)\in\Hom\big(\gX_A(\CC),\gX_B(\gR)\big)$ and $(N,\gS)\in\Hom\big(\gX_B(\CC),\gX_C(\gR)\big)$ we define as $(NM,\gR\otimes\gS)$. Due to Proposition~\ref{propProd} the product $NM$ is an $(A,C)$-Manin matrix over $\gR\otimes\gS$.

If two idempotents are left-equivalent, then they are isomorphic as objects of $\M_{\A'}$, so it is enough to take the quadratic algebras $\gX_A(\CC)$ as objects of $\M_{\A'}$ instead of the idempotents. Moreover, one can prove that $\gX_A(\CC)$ and $\gX_{A'}(\CC)$ are isomorphic as objects of $\M_{\A'}$ iff they are isomorphic as objects of $\Q$. Indeed, if $(\gR,f_M)\in\Hom\big(\gX_A(\CC),\gX_B(\gR)\big)$ is an isomorphism $\M_{\A'}$ and $(\gS,f_N)\in\Hom\big(\gX_B(\CC),\gX_A(\gS)\big)$ is its inverse then $\gR\otimes\gS=\CC$, hence $\gR=\gS=\CC$ and $f_M\colon\gX_A(\CC)\to\gX_B(\CC)$, $f_N\colon\gX_A(\CC)\to\gX_B(\CC)$ are isomorphisms in $\Q$.

\begin{Rem} \label{remFunctorDual}
 Due to Proposition~\ref{PropLEQA} the formula $\gX_A(\CC)\mapsto\Xi_A(\CC)$ correctly define an ope\-ra\-tion on the quadratic algebras $\gA\in\M_{\A'}$. It was denoted in the works of Manin~\cite{Manin87,Manin88,ManinBook91} by~$\gA\mapsto\gA^!$. The first equality~\eqref{HomgXHomXi} gives a contravariant fully faithful functor $(-)^!\colon\M_{\A'}\to\M_{\A'}$. Since any object of $\M_{\A'}$ isomorphic to $\Xi_A(\CC)$ for some idempotent $A$ this functor is an anti-autoequivalence of the category $\M_{\A'}$. The (quasi-)inverse of $(-)^!$ is $(-)^!$ itself.
\end{Rem}

\begin{Rem}
 One can extend the category $\M_{\A'}$ up to a category $\wh\M_{\A'}$ by taking all the algebras $\gX_A(\gS)\in\Q$ as objects of $\wh\M_{\A'}$. The sets of homomorphisms in $\wh\M_{\A'}$ can be defined as
\begin{align*}
 \Hom(\gX_A(\gS),\gX_B(\gS'))={}&\bigsqcup_{\gR\in\A'}\Big\{(\alpha,M)\mid\alpha\in\Hom(\gS,\gS'),
 \\
 &(M,\gR)\in\Hom\big(\gX_A(\CC),\gX_B(\gS'\otimes\gR)\big),\;
 \big[\alpha(s),M^i_j\big]=0\;\forall\,s\in\gS\Big\}
\end{align*}
with some natural composition rule. In these settings $\gX_A(\gS)$ and $\gX_{A'}(\gS')$ are isomorphic as objects of $\wh\M_{\A'}$ iff they are isomorphic as objects of $\Q$ (see Remark~\ref{RemIsom}).
\end{Rem}

\subsection{Infinite-dimensional case: Manin operators}
\label{secMO}

Let $V$ and $W$ be two vector spaces (may be infinite-dimensional). Let $A\in\End(V\otimes V)$ and $B\in\End(W\otimes W)$ be idempotents and $\gR$ be an algebra. Instead of a matrix over $\gR$ we need to take an element $M$ of the space $\gR\otimes\Hom(W,V)$ or of some completion of this space. We~con\-sider~$M$ as an operator with entries in the algebra $\gR$: this means that $\lambda Mw\in\gR$, where $\lambda\in V^*$ and $w\in W$ (in the case of completion the covector $\lambda$ runs over a subset of the dual space $V^*$ such that $\lambda Mw$ is well defined).

\begin{Def}
 The operator $M$ (from the space $\gR\otimes\Hom(W,V)$ or its completion) is called an~{\it $(A,B)$-Manin operator} if it satisfies the relation $AM^{(1)}M^{(2)}(1-B)=0$, where $M^{(1)}=M\otimes\id_V$ and $M^{(2)}=\id_W\otimes M$. In the case $V=W$, $A=B$ it is called {\it $A$-Manin operator}.
\end{Def}

Let $(v_i)$ be a basis of the space $V$, so that any vector $v\in V$ has the from $v=\sum_i\alpha^i v_i$ for unique coefficients $\alpha^i\in\CC$ and the sum is finite. Then the action of the operator $A$ can be written as $A(v_k\otimes v_l)=\sum_{i,j}A^{ij}_{kl}v_j\otimes v_j$ for unique coefficients $A^{ij}_{kl}\in\CC$. Since the sum~$\sum_{i,j}$ should be finite there are only finitely many non-zero coefficients $A^{ij}_{kl}$ for any fixed $k$, $l$, we call it {\it column finiteness condition} for the matrix $(A^{ij}_{kl})$. This condition allows as to define the quadratic algebra~$\Xi_A(\CC)$ generated by $\psi_i$ with commutation relations $\psi_k\psi_l=\sum_{i,j}A^{ij}_{kl}\psi_i\psi_j$. Formally this is the quotient $\Xi_A(\CC)=\mathfrak T/\mathfrak I$, where $\mathfrak T$ is the algebra of all the non-commutative polynomials of the formal variables $\psi_i$ while $\mathfrak I$ is the ideal of $\mathfrak T$ generated by the elements $\psi_k\psi_l-\sum_{i,j}A^{ij}_{kl}\psi_i\psi_j\in\mathfrak T$. Another choice of the basis $(v_i)$ leads to an isomorphic quadratic algebra, so it essentially depends on the operator $A$ only.

To interpret an $(A,B)$-Manin operator $M$ as an object of a comma category as in Section~\ref{secCommCat} we need to define also an algebra $\Xi_A(\gR)$. Let $(v_i)$ and $(w_i)$ be bases of $V$ and $W$. The formula $Mw_j=\sum_iM^i_jv_i$ defines the entries $M^i_j\in\gR$ of an $(A,B)$-Manin operator $M$ in these bases. If~$M\in\gR\otimes\Hom(W,V)$ (without completion), then there are only finitely many non-zero entries~$M^i_j$ for any fixed $j$ $\big($column finiteness condition for the matrix $\big(M^i_j\big)\big)$ and hence the sum $\sum_iM^i_j\psi_i$ is finite.
 Define $\Xi_A(\gR)= \gR\otimes\Xi_A(\CC)$ (non-completed). Then the set of $(A,B)$-Manin operators $M\in\gR\otimes\Hom(W,V)$ bijectively corresponds to the set $\Hom\big(\Xi_B(\CC),\Xi_A(\gR)\big)$.

To consider the case of arbitrary infinite matrix $\big(M^i_j\big)$ without any finiteness condition we define a completion of the space $\gR\otimes\Hom(W,V)$ as the set of all the infinite formal sums $\sum_{i,j}M^i_jE_{i}{}^{j}$, where $M^i_j\in\gR$ and $E_{i}{}^{j}\in\Hom(W,V)$ are operators acting as $E_{i}{}^{j}w_k=\delta^j_kv_i$. Denote this completion by $\gR\whotimes\Hom(W,V)$. It can be identified with $\Hom\big(W,\gR\whotimes V\big)$, where $\gR\whotimes V$ is the completion of $\gR\otimes V$ consisting of all the infinite formal sums $\sum_{i}r_iv_i$, $r_i\in\gR$. The operator $M=\sum_{i,j}M^i_jE_{i}{}^{j}$ acts by the formula $Mw_j=\sum_i M^i_jv_i$.

Note that the completion $\gR\whotimes V$ \big(and hence $\gR\whotimes\Hom(W,V)=\Hom\big(W,\gR\whotimes V\big)$\big) does not depend on the choice of the basis $(w_i)$, but it does depend on the choice of the basis $(v_i)$. Namely, a basis $(v_i)$ defines the following topology in the $\gR$-module $\gR\otimes V$: neighbourhoods of $0$ are the $\gR$-submodules generated by all $v_i$ except finitely many of them. The module $\gR\whotimes V$ is the completion of $\gR\otimes V$ with respect to this topology. Any two bases $(v_i)$ and $(v'_i)$ are related by~$v_i=\sum_k\alpha^k_iv'_k$, $v'_j=\sum_k\beta^k_jv_k$, they define the same topology of $\gR\otimes V$ iff for any $k$ there are only finitely many non-zero $\alpha^k_i$ and $\beta^k_j$.

Suppose $A$ satisfy {\it row finiteness condition}: there are only finitely many non-zero $A^{ij}_{kl}$ for fixed $i$, $j$. This condition means exactly that the operator $A\colon V\otimes V\to V\otimes V$ is continuous with respect to the topology corresponding to the basis $(v_i\otimes v_j)$. Define $\wh\Xi_A(\gR)=\wh{\mathfrak T}/\wh{\mathfrak I}$, where $\wh{\mathfrak T}=\bigoplus_{k\in\NN_0}\wh{\mathfrak T}_k$ is the graded algebra with grading component $\wh{\mathfrak T}_k$ consisting of the infinite formal sums
\begin{align*}
 \sum\limits_{i_1,\dots,i_k}r_{i_1\dots i_k}\psi_{i_1}\cdots\psi_{i_k},\qquad r_{i_1\dots i_k}\in\gR,
\end{align*}
and $\wh{\mathfrak I}=\Big\{\sum_{k,l}t_{kl}(\psi_k\psi_l-\sum_{i,j}A^{ij}_{kl}\psi_i\psi_j)t'_{kl}\mid t_{kl},t'_{kl}\in\wh{\mathfrak T}\Big\}$ (due to the row finiteness condition the sum over $k$ and $l$ is correctly defined). The algebra $\wh\Xi_A(\gR)$ is a completion of $\Xi_A(\gR)$, it depends on the choice of the basis $(v_i)$ due to the identification $\Xi_A(\CC)_1=V$ via $\psi_i\leftrightarrow v_i$.

The proof of Lemma~\ref{propTXX} holds for the completed algebra $\wh\Xi_A(\gR)$, so that the system of~equations $\sum_{i,j} \psi_i\psi_jT^{ij}_{ab}=0$ for $T^{ij}_{ab}\in\gR$ is equivalent to the system $\sum_{k,l}A^{ij}_{kl}T^{kl}_{ab}=0$. Hence we have a bijection between the set of $(A,B)$-operators $M\in\gR\whotimes\Hom(W,V)$ and the set $\Hom\big(\Xi_B(\CC),\wh\Xi_A(\gR)\big)$.

Let $V\whotimes V$ be the completion of $V\otimes V$ with respect to the topology corresponding to the basis $(v_i\otimes v_j)$. It consists of all the infinite formal sums $\sum_{i,j}\alpha_{ij}v_i\otimes v_j$, $\alpha_{ij}\in\CC$. Any matrix~$\big(A^{ij}_{kl}\big)$ satisfying row finiteness condition define a continuous operator $A\colon V\whotimes V\to V\whotimes V$ by the formula $Av_k\otimes v_l=\sum_{i,j}A^{ij}_{kl}v_i\otimes v_j$. Denote by $\End\big(V\whotimes V\big)$ the space of all the continuous operators $V\whotimes V\to V\whotimes V$. Note that we do not need the column finiteness condition to define the algebra $\wh\Xi_A(\gR)$. Hence the $(A,B)$-Manin operators $M\in\gR\whotimes\Hom(W,V)$ for arbitrary idempotents $A\in\End\big(V\whotimes V\big)$ and $B\in\End(W\otimes W)$ can be identified with the elements of~the set $\Hom\big(\Xi_B(\CC),\wh\Xi_A(\gR)\big)$. Explicitly, the relations $\sum_{k,l,a,b}A^{ij}_{kl}M^k_aM^l_b\big(\delta^a_r\delta^b_s-B^{ab}_{rs}\big)=0$ are~correctly defined in this case since all the sums in these relations are finite.

Thus the $(A,B)$-Manin operators in the non-completed or completed case are objects of the comma category $\big(\Xi_B(\CC)\downarrow\Xi_A\big)$ or $\big(\Xi_B(\CC)\downarrow\wh\Xi_A\big)$ for the functor $\Xi_A\colon\A\to\G$ or $\wh\Xi_A\colon\A\to\G$ respectively.

One can also generalise the graded algebras $\gX_A(\CC)$ to the case of infinite-dimensional mat\-ri\-ces~$(A^{ij}_{kl})$. To do it one needs to require the row finiteness condition, which means that this matrix defines a continuous operator $A\colon V\whotimes V\to V\whotimes V$. For any idempotent $A\in\End\big(V\whotimes V\big)$ define the quadratic algebra $\gX_A(\CC)$ as the algebra generated by $x^i$ with the commutation relations $\sum_{k,l}A^{ij}_{kl}x^kx^l=0$. That is $\gX_A(\CC)=\Xi_{\id-A^\top}(\CC)$, where $A^\top$ is the operator $V\otimes V\to V\otimes V$ defined as $A^\top v_i\otimes v_j=\sum_{k,l}A^{ij}_{kl}v_k\otimes v_l$. Let $W\whotimes W$ be the completion of $W\otimes W$ with respect to the basis $(w_i\otimes w_j)$ and $B\in\End(W\otimes W)$ be an idempotent. The corresponding matrix $(B^{ij}_{kl})$ satisfy the column finiteness condition. Define $\wh\gX_B(\gR)=\wh\Xi_{\id-B^\top}(\gR)$, where $B^\top\in\End\big(W\whotimes W\big)$ is the continuous operator $W\whotimes W\to W\whotimes W$ acting as $B^\top w_i\otimes w_j=\sum_{k,l}B^{ij}_{kl}w_k\otimes w_l$. Then we have the comma category $(\gX_A(\CC)\downarrow\wh\gX_B)$ equivalent to $\big(\Xi_B(\CC)\downarrow\wh\Xi_A\big)$.

In particular, the completions allow us to consider the {\it universal $(A,B)$-Manin operator}. It has the form $\M=\sum_{i,j}\M^i_jE_{i}{}^{j}$, where $\M^i_j$ are generators of the algebra $\gU_{A,B}$ with the commutation relations $\sum_{k,l,a,b}A^{ij}_{kl}\M^k_a\M^l_b\big(\delta^a_r\delta^b_s-B^{ab}_{rs}\big)=0$. These relations are correctly defined iff the matrices $\big(A^{ij}_{kl}\big)$ and $\big(B^{ij}_{kl}\big)$ satisfy the row and column finiteness conditions respectively.

Proposition~\ref{propProd} can be generalised if we additionally suppose that the sums $\sum_j M^i_jN^j_l$ are well defined: if nether $M$ nor $N$ satisfies a needed finiteness condition, then we need to complete~$\gR$ in order to include all these sums. In particular, the map $\Delta_{A,B,C}$ is a homomorphism from~$\gU_{A,C}$ to a completion of $\gU_{A,B}\otimes\gU_{B,C}$ which contains the sums $\sum_j\M^i_j\N^j_l$ (we need to suppose that $A$, $B$ and $C$ satisfy the row, both and column finiteness conditions respectively). In~the case $A=B=C$ the map $\Delta_A$ is a ``completed'' comultiplication for $\gU_A=\gU_{A,B}$. In~terms of representations this means that tensor product of $\gU_A$-modules are not always defined: we need to impose some finiteness condition on the modules to guarantee the existence of their tensor product.

Finally, we consider some examples of completions used below. The simplest example of the infinite-dimensional space is the space of polynomials $V=\CC[u]$. It has the basis \mbox{$\big(v_i=u^{i-1}\big)_{i\ge1}$}. The completion of $\gR[u]=\gR\otimes\CC[u]$ is the algebra of formal series $\gR[[u]]$. It consists of all the formal infinite sums $\sum_{k=0}^\infty r_ku^k$, $r_k\in\gR$. The algebra of finite Laurent polynomials $\gR\big[u,u^{-1}\big]=\gR\otimes\CC\big[u,u^{-1}\big]$ can be completed to the algebra $\gR((u))$ consisting of the series $\sum_{k=N}^\infty r_ku^k$, $N\in\ZZ$, $r_k\in\gR$. It has another completion -- the algebra $\gR((u^{-1}))$ consisting of $\sum_{k=-\infty}^N r_ku^k$, $N\in\ZZ$, $r_k\in\gR$. Note that the space $\gR[[u,u^{-1}]]$ consisting of $\sum_{k=-\infty}^\infty r_ku^k$ is also a completion of~$\gR\big[u,u^{-1}\big]$, but it is not an algebra in the usual sense.

\section{Particular cases}
\label{secPc}

Here we consider the main examples corresponding to the polynomial and Grassmann algebras and their deformations. We also consider a generalisation of a deformed polynomial algebra with three variables. More examples will appear in Sections~\ref{secLO},~\ref{secExMO}, \ref{secBCD} and Appendix~\ref{appLie}.

\subsection{Manin matrices for the polynomial algebras}
\label{secMm}

Let $P_n\in\End(\CC^n\otimes\CC^n)$ be the permutation operator acting as
\begin{align} \label{PnDef}
 P_n(v\otimes w)=w\otimes v,\qquad v,w\in\CC^n.
\end{align}
Substituting basis elements $v=e_k$ and $w=e_l$ we obtain the entries $(P_n)^{ij}_{kl}=\delta^i_l\delta^j_k$. Since $P_n^2=1$ the operators
\begin{align}
 A_n=\frac12\big(1-P_n\big),\qquad S_n=\frac12\big(1+P_n\big)=1-A_n
\end{align}
are idempotents: $A_n^2=A_n$, $S_n^2=S_n$. These are anti-symmetrizer and symmetrizer for two tensor factors respectively. Note that the permutation operator satisfies the braid relation
\begin{align}
 P_n^{(23)} P_n^{(12)}P_n^{(23)}= P_n^{(12)} P_n^{(23)} P_n^{(12)} \label{PBR}
\end{align}
(this is an equality of operators acting on the space $\CC^n\otimes\CC^n\otimes\CC^n$).

The commutation relations~\eqref{Axx00} and \eqref{psi_psi_A_psi_psi} for $A=A_n$ have the form $x^ix^j-x^jx^i=0$ and~$\psi_i\psi_j+\psi_j\psi_i=0$. Hence the algebra $\gX_{A_n}(\CC)$ is the polynomial algebra $\CC[x^1,\dots,x^n]$ and~$\Xi_{A_n}(\CC)$ is the Grassmann algebra with Grassmann variables $\psi_1,\dots,\psi_n$.

The $(A_n,A_m)$-Manin matrices are $n\times m$ matrices $M$ over an algebra $\gR$ satisfying the relation
\begin{align}
 A_nM^{(1)}M^{(2)}S_m=0. \label{AMMS}
\end{align}
These matrices were called {\it Manin matrices} in~\cite{CF}. In this subsection we call them just Manin matrices if there is no confusion with the general notion of Manin matrix.

To write the matrix relation~\eqref{AMMS} in terms of entries one can substitute $P^{ij}_{kl}=\delta^i_l\delta^j_k$ to the formula~\eqref{PMMP} written in entries. We obtain the following system of commutation relations:
\begin{gather}
 \big[M^i_k,M^j_k\big]=0, \qquad i<j, \label{Mikjk}
 \\
 \big[M^i_k,M^j_l\big]+\big[M^i_l,M^j_k\big]=0, \qquad i<j,\quad k<l, \label{Mikjl}
\end{gather}
(the relations~\eqref{Mikjk} is in fact the relations~\eqref{Mikjl} for $k=l$).
The commutation relations~\eqref{Mikjk} mean that any two entries of the same column commute. The formula~\eqref{Mikjl} is so-called {\it cross-relation} for the $2\times2$ submatrix with rows $i$, $j$ and columns $k$, $l$.

For example, consider a $2\times2$ matrix
\begin{align} \label{Mabcd}
 M=\begin{pmatrix} a & b \\ c & d \end{pmatrix}\!.
\end{align}
It is a Manin matrix iff its entries $a,b,c,d\in\gR$ satisfy
\begin{align} \label{abcdManin}
[a,c]=0, \qquad [b,d]=0, \qquad [a,d]+[b,c]=0.
\end{align}
In the case $n\ge2$ an $n\times m$ matrix is a Manin matrix iff any $2\times2$ submatrix of this matrix is a~Manin matrix. In particular, any matrix over a commutative algebra is a Manin matrix.

The notion of $(A_n,A_m)$-Manin matrix is such a non-commutative generalisation of matrix that the most of the properties of the usual matrices are inherited (with some generalisation of~the notion of determinant). These properties are described in the works~\cite{CF} and~\cite{CFR} in details.

It is clear, that a Manin matrix keeps to be a Manin matrix after the following operations: taking a submatrix, permutation of rows or columns, doubling of a row or a column. In other words, if $M=\big(M^i_j\big)$ is an $n\times m$ Manin matrix and $i_1,\dots,i_k\in\{1,\dots,n\}$, $j_1,\dots,j_l\in\{1,\dots,m\}$ then the new $k\times l$ matrix $N$ with the entries $N^s_t=M^{i_s}_{j_t}$ is also a Manin matrix. Note that in the case of a permutation this fact follows from Proposition~\ref{PropPerm} and $(\sigma\otimes\sigma)A_n\big(\sigma^{-1}\otimes\sigma^{-1}\big)=A_n$ $\forall\,\sigma\in \SSS_n$.

Let us recall one more important fact on the Manin matrices~\cite{CFR}.

\begin{Prop} \label{propManinComm}
 A matrix $M\in\gR\otimes\Hom(\CC^m,\CC^n)$ has pairwise commuting entries $($i.e., $[M^i_j,M^k_l]=0$ for any $i$, $j$, $k$, $l)$ iff $M$ and its transposed $M^\top$ are both Manin matrices.
\end{Prop}

\begin{proof} If all entries of $M$ commute with each other then $M$ as well as $M^\top$ is a Manin matrix. In the converse direction it is enough to prove the statement for the case of $2\times2$ matrix, since any two entries are contained in some $2\times2$ submatrix (if $n=1$ or $m=1$ there are no $2\times2$ submatrices -- in this case the commutativity of entries follows from the relations~\eqref{Mikjk} for $M^\top$ or $M$ respectively).

The condition that the matrix $M^\top=\left(\begin{smallmatrix} a & c \\ b & d \end{smallmatrix}\right)$ is a Manin matrix is equivalent to the relations $[a,b]=0$, $[c,d]=0$, $[a,d]-[b,c]=0$. Together with the relations~\eqref{abcdManin} they imply that all the entries $a$, $b$, $c$, $d$ pairwise commute. \end{proof}

For a general Manin matrix the entries from different columns do not commute, so the notion of determinant should be generalised in a special way. It turns out that the natural generalisation is so-called {\it column determinant}. For a $n\times n$ matrix $M$ over an algebra (or a ring)~$\gR$ it is defi\-ned~as
\begin{align*}
 \det M=\det^{\rm col}M=\sum_{\sigma\in\SSS_n}(-1)^\sigma M^{\sigma(1)}_{1} M^{\sigma(2)}_{2} \cdots M^{\sigma(n)}_{n},
\end{align*}
where $(-1)^\sigma$ is the sign of the permutation~$\sigma$. This is the usual expression for the determinant, but with the specified order of entries in each term: they are ordered in accordance with the order of columns. If $M$ is a Manin matrix then the order of columns can be chosen in a different way and this leads to the same result (see~\cite{CF,CFR}):
\begin{align} \label{detPerm}
 \sum_{\sigma\in \SSS_n}(-1)^\sigma M^{\sigma(i_1)}_{i_1}M^{\sigma(i_2)}_{i_2} \cdots M^{\sigma(i_n)}_{i_n}=\det^{\rm col}M, \qquad
\begin{pmatrix} 1& 2&\cdots& n \\ i_1 & i_2 & \cdots& i_n \end{pmatrix} \in \SSS_n.
\end{align}
However, we can not take different order for different terms. For example, the determinant of~the~$2\times2$ matrix~\eqref{Mabcd} is $\det(M)=ad-cb$. If it is a Manin matrix then due to the last relation~\eqref{abcdManin} we have $\det(M)=da-bc$, but in general $\det(M)$ does not equal to $ad-bc$ or to~$da-cb$ even for a Manin matrix.

An important property of the column determinant of Manin matrices is its behaviour \mbox{under} the permutation of rows and columns. The determinant of an $n\times n$ Manin matrix $M$ changes the sign under the transposition of two columns or rows. In the notations of Section~\ref{secABmanin} we~have $\det(^\tau M)=(-1)^\tau\det M$ and $\det(_\tau M)=(-1)^\tau\det M$ for any $\tau\in \SSS_n$. The first formula is~deduced as
\begin{align*}
 \det(_\tau M)=\sum_{\sigma\in \SSS_n}(-1)^\sigma M^{\tau^{-1}\sigma(1)}_{1} \cdots M^{\tau^{-1}\sigma(n)}_{n}=(-1)^\tau\det M,
\end{align*}
where we made the substitution $\sigma\to\tau\sigma$ and used $(-1)^{\tau\sigma}=(-1)^\tau(-1)^\sigma$. The second formula follows from~\eqref{detPerm} in a similar way.

Since any submatrix of a Manin matrix $M$ is also a Manin matrix, it is natural to define $k\times k$ minors of $M$ to be the column determinants of $k\times k$ submatrices. We say that a Manin matrix has {\it rank} $r$ if there is non-zero $r\times r$ minor and all the $k\times k$ minors vanish for all $k>r$. In fact, it is enough to check it for $k=r+1$ (see~\cite{CF,CFR}). Many important properties of the Manin matrices are formulated in terms of column determinants and minors. In particular, one can construct ``spectral invariants'' of square Manin matrices~\cite{CF,CFR}.

The theory of Manin matrices are applies to the Yangians $Y(\mathfrak{gl}_n)$, affine Lie algebras $\wh{\mathfrak{gl}}_n$, Heisenberg $\mathfrak{gl}_n$ $XXX$-chain, Gaudin $\mathfrak{gl}_n$ model~\cite{CF}, to elliptic versions of these models~\cite{RST} etc.

Let us, for example, present the connection of the notion of $A_n$-Manin matrix with the Yangian $Y(\mathfrak{gl}_n)$. Consider the {\it rational $R$-matrix} $R(u)=u-P_n$. The Yangian $Y(\mathfrak{gl}_n)$ is defined as the algebra generated by $t_{ij}^{(r)}$, $i,j=1,\dots,n$, $r\in\ZZ_{\ge1}$, with the commutation relation
\begin{align}
 R(u-v)T^{(1)}(u)T^{(2)}(v)=T^{(2)}(v)T^{(1)}(u)R(u-v), \label{RTT}
\end{align}
where $u$ and $v$ are formal variables and
$T(u)=1+\sum_{r=1}^\infty\!\left(\begin{smallmatrix}t_{11}^{(r)}\cdots t_{1n}^{(r)} \\ \vdots \ddots \vdots \\ t_{n1}^{(r)}\cdots t_{nn}^{(r)}\end{smallmatrix}\right)\!u^{-r}$. Note that for $u=1$ we have $R(u)=2A_n$. Substituting $v=u-1$ to~\eqref{RTT} and multiplying by $S_n=1-A_n$ from the right we obtain $A_nT^{(1)}(u)T^{(2)}(u-1)S_n=0$ (we understand $(u-1)^{-r}$ as a series in~$u^{-1}$). This relation means that the matrix $M=T(u)e^{-\frac{\partial}{\partial u}}$ is a Manin matrix over the algebra $\gR=Y(\mathfrak{gl}_n)\big[\big[u^{-1}\big]\big]\big[e^{-\frac{\partial}{\partial u}}\big]$. See details in~\cite{CF}.

The column determinant of $M=T(u)e^{-\frac{\partial}{\partial u}}$ is related with the notion of quantum determinant important in the theory of Yangians. Namely, we have $\det M=\left(\qdet T(u)\right)e^{-n\frac{\partial}{\partial u}}$, where
\begin{align*}
 \qdet T(u)=\sum_{\sigma\in \SSS_n}(-1)^\sigma T^{\sigma(1)}_{1}(u)T^{\sigma(2)}_{2}(u-1) \cdots T^{\sigma(n)}_{n}(u-n+1) 
\end{align*}
is the {\it quantum determinant}. The coefficients of the series $\qdet T(u)\in Y(\mathfrak{gl}_n)\big[\big[u^{-1}\big]\big]$ generate the centre of the Yangian $Y(\mathfrak{gl}_n)$ (see~\cite{MNO}).

Now let us consider $(A_n,0)$-Manin matrices $M=\big(M^i_j\big)$, where $0\in\End(\CC^m\otimes\CC^m)$. They are defined by the relation $M^i_kM^j_l=M^j_kM^i_l$, where $i,j=1,\dots,n$, $k,l=1\dots,m$. The $2\times2$ matrix of the form~\eqref{Mabcd} is an $(A_2,0)$-Manin matrix iff $ad=cb$, $bc=da$, $ac=ca$, $bd=db$. These are the relations~\eqref{abcdManin} plus the condition $\det M=0$. Again one can see that an~$n\times m$ matrix is an~$(A_n,0)$-Manin matrix iff any $2\times2$ submatrix of this matrix is an~$(A_2,0)$-Manin matrix. Thus the set of an $(A_n,0)$-Manin matrices over $\gR$ coincides with the set of rank $\le1$ Manin matrices over $\gR$ of the size $n\times m$.

Recall that the formula~\eqref{detPerm} is valid for all matrices satisfying the cross-relations~\eqref{Mikjl}. This is more general case than the case of Manin matrices. To describe it we first consider a $2\times2$ matrix of the form~\eqref{Mabcd} and the ``transformation'' \eqref{phipsiM}, that is
\begin{align*}
\phi_1=a\psi_1+c\psi_2, \qquad
\phi_2=b\psi_1+d\psi_2.
\end{align*}
Then we have $\phi_1\phi_2+\phi_2\phi_1=(ab+ba)\psi_1^2+(ad+bc)\psi_1\psi_2+(da+cb)\psi_2\psi_1+(cd+dc)\psi_2^2$. The relation $[a,d]+[b,c]=0$ follows from the relations $\psi_1\psi_2+\psi_2\psi_1=0$ and $\phi_1\phi_2+\phi_2\phi_1=0$ only. Hence by refusing the relations $\phi_i^2=0$ we have only the cross-relations~\eqref{Mikjl} without commutation in columns~\eqref{Mikjk}. If we refuse from $\psi_i^2=0$, then we obtain the anti-commutation in rows:
\begin{align}
 M^i_kM^i_l+M^i_lM^i_k=0, \qquad k<l. \label{Mikil}
\end{align}
Let $\wt P_n\in\End(\CC^n\otimes\CC^n)$ be the matrix with the entries $\big(\wt P_n\big)^{ij}_{kl}=(-1)^{\delta_{ij}}\delta^i_l\delta^j_k=(-1)^{\delta_{kl}}\delta^i_l\delta^j_k$. Then the commutation relations $\psi_i\psi_j+\psi_j\psi_i=0$, $i<j$ (without $\psi_i^2=0$) can be written as $(\Psi\otimes\Psi)\big(1+\wt P_n\big)=0$, so these commutation relations define the quadratic algebra $\Xi_{\wt A_n}(\CC)$, where $\wt A_n=\frac12\big(1-\wt P_n\big)$. Let us conclude.
\begin{itemize}
\itemsep=0pt
 \item The matrix $M$ is an $\big(A_n,\wt A_n\big)$-Manin matrix iff it satisfies~\eqref{Mikjl}.
 \item The matrix $M$ is an $\big(\wt A_n,\wt A_n\big)$-Manin matrix iff it satisfies~\eqref{Mikjl} and~\eqref{Mikil}.
 \item The matrix $M$ is an $\big(\wt A_n,A_n\big)$-Manin matrix iff it satisfies~\eqref{Mikjk},~\eqref{Mikjl} and~\eqref{Mikil}.
\end{itemize}

The notion of Manin matrix can be generalized to the infinite-dimensional case. Denote for any vector space $V$ the operator $P_V\in\End(V\otimes V)$ by
\begin{align*}
 P_V(v\otimes v')=v'\otimes v, \qquad v,v'\in V.
\end{align*}
Let $A_V=\frac{1-P_V}2$ and $S_V=1-A_V=\frac{1+P_V}2$ be corresponding anti-symmetrizer and symmetrizer. Let us call an $(A_V,A_W)$-Manin operator $M\in\gR\otimes\Hom(V,W)$ just Manin operator (over $\gR$, from $V$ to $W$).

Consider the case of the space of polynomials: $V=W=\CC[u]$. Its tensor square can be identified with the space of polynomials of two variables: $V\otimes V=\CC[u_1,u_2]$. In terms of this identification the operator $P_{\CC[u]}$ can be interpreted as the operator permuting the variables $u_1$ and $u_2$, we denote it by $P_{u_1,u_2}$. Let $(v_i)_{i=1}^\infty$ be a basis of the space $\CC[u]$. For example, one can take $v_i=u^{i-1}$. Acting by the operator $M$ to the basis we obtain $Mv_j=\sum_i M^i_jv_i$ (the sum is finite), so we have infinite-dimensional matrix $\big(M^i_j\big)$ with the column finiteness condition. The relation $A_VM^{(1)}M^{(2)}S_V=0$ in the basis $(v_i)$ takes the form~\eqref{Mikjk}, \eqref{Mikjl}, where $i,j,k,l=1,\dots,\infty$. This means that $M$ is a Manin operator iff any $2\times2$ submatrix of $\big(M^i_j\big)$ is a Manin matrix. The set of~Manin operators includes all the $n\times m$ Manin matrices as well as $(A_n,A_{\CC[u]})$- and~$(A_{\CC[u]},A_m)$-Manin operators (all over a fixed $\gR$).

To generalise the consideration to the case of any infinite matrix $\big(M^i_j\big)$ without any finiteness condition we should suppose that the operator $M$ belongs to the completion of $\gR\otimes\End\big(\CC[u]\big)$ consisting of the formal infinite sums $\sum_{i,j=1}^\infty M^i_jE_{i}{}^{j}$, where $M^i_j\in\gR$ and $E_{i}{}^{j}v_k=\delta^j_kv_i$. This completion is the space $\Hom\big(\CC[u],\gR[[u]]\big)$.

\subsection[q-Manin matrices]
{$\boldsymbol q$-Manin matrices}\label{secqM}

Let $q$ be a non-zero complex number. Consider the $q$-commuting variables $x^1,\dots,x^n$, that is $x^jx^i=qx^ix^j$ for $i<j$. By means of the notation
\begin{align*}
 \sgn(k)=\begin{cases}
 +1, &\text{if}\quad k>0, \\
 \phantom{-}0, &\text{if}\quad k=0, \\
 -1, &\text{if}\quad k<0,
 \end{cases}
\end{align*}
one can write these relations as
\begin{align}
 x^ix^j=q^{\sgn(i-j)}x^jx^i, \qquad i,j=1,\dots,n. \label{qxixj}
\end{align}
In the matrix form they have the form $(X\otimes X)=P^q_n(X\otimes X)$, where $X=\sum_{i=1}^nx^ie_i$ and~$P^q_n\in\End(\CC^n\otimes\CC^n)$ is the $q$-permutation operator acting as $P^q_n(e_i\otimes e_j)= q^{-\sgn(i-j)}e_j\otimes e_i$, that~is
\begin{align*}
 P^q_n=\sum_{i,j=1}^n q^{\sgn(i-j)} E_{i}{}^{j} \otimes E_{j}{}^{i}. 
\end{align*}
Its entries are $\big(P^q_n\big)^{ij}_{kl}=q^{\sgn(i-j)}\delta^i_l\delta^j_k=q^{\sgn(l-k)}\delta^i_l\delta^j_k$. It also satisfies the braid relation
\begin{align*}
 \big(P^q_n\big)^{(23)} \big(P^q_n\big)^{(12)}\big(P^q_n\big)^{(23)}= \big(P^q_n\big)^{(12)} \big(P^q_n\big)^{(23)} \big(P^q_n\big)^{(12)}. 
\end{align*}
Since $(P^q_n)^2=1$ the matrices
\begin{align*}
 A_n^q=\frac12\big(1-P^q_n\big),\qquad S_n^q=\frac12\big(1+P^q_n\big)=1-A_n^q
\end{align*}
are idempotents. The corresponding algebra $\gX_{A_n^q}(\CC)$ is generated by $x^i$ with the relations~\eqref{qxixj}. It can be interpreted as an ``algebra of functions'' on the $n$-dimensional quantum space $\CC^n_q$. The algebra $\Xi_{A_n^q}(\CC)$ is the $q$-Grassmann algebra generated by $\psi_1,\dots,\psi_n$ with the relations
\begin{align*}
 \psi_i\psi_j=-q^{-\sgn(i-j)}\psi_j\psi_i, \qquad i,j=1,\dots,n. 
\end{align*}

The matrix $M\in\gR\otimes\Hom(\CC^m,\CC^n)$ is an $\big(A_n^q,A_m^q\big)$-Manin matrix iff the following equivalent relations hold
\begin{gather}
 A_n^qM^{(1)}M^{(2)}S_m^q=0, \nonumber \\
 \big(1-P^q_n\big)M^{(1)}M^{(2)}\big(1+P_m^q\big)=0,\nonumber \\
 A_n^qM^{(1)}M^{(2)}=A_n^qM^{(1)}M^{(2)}A_m^q, \nonumber\\
 M^{(1)}M^{(2)}S_m^q=S_n^qM^{(1)}M^{(2)}S_m^q.\label{AMMSq}
\end{gather}
In terms of entries these relations have the form
\begin{gather}
 M^i_kM^j_k=q^{-1}M^j_kM^i_k, \qquad i<j,\nonumber 
 \\
 \big[M^i_k,M^j_l\big]+qM^i_lM^j_k-q^{-1}M^j_kM^i_l=0, \qquad i<j,\quad k<l. \label{Mqikjl}
\end{gather}
The $(A_n^q,A_m^q)$-Manin matrices are called {\it $q$-Manin matrices}. The Manin matrices considered in~Section~\ref{secMm} are $q$-Manin matrices for $q=1$. The properties of the Manin matrices described in~\cite{CF,CFR} were generalised to the $q$-case in the work~\cite{qManin}.

A natural generalisation of the column determinant to the case of $q$-Manin matrices is {\it $q$-deter\-minant} defined for a matrix $M\in\gR\otimes\End(\CC^n)$ as
\begin{align} \label{detq}
 \det_q(M)=\sum_{\sigma\in \SSS_n}(-q)^{-\inv(\sigma)} M^{\sigma(1)}_{1} M^{\sigma(2)}_{2} \cdots M^{\sigma(n)}_{n},
\end{align}
where $\inv(\sigma)$ is the number of inversions: it is equal to the number of pairs $(i,j)$ such that $1\le i<j\le n$ and $\sigma(i)>\sigma(j)$. It coincides with the length of $\sigma$ defined as the minimal $l$ such that $\sigma$ can be presented as a product of $l$ elementary transpositions $\sigma_i=\sigma_{i,i+1}$ (see Appendix~\ref{appRw} for details).

For the $2\times2$ matrix $M=\left(\begin{smallmatrix} a & b \\ c & d \end{smallmatrix}\right)$ the $q$-determinant has the from $\det_qM=ad-q^{-1}cb$. This matrix $M$ is a $q$-Manin matrix iff
\begin{gather*} 
 ca=qac, \qquad db=qbd, \qquad ad-da+qbc-q^{-1}cb=0.
\end{gather*}
In this case the $q$-determinant can be rewritten as $\det_qM=da-qbc$.

A general $n\times m$ matrix is a $q$-Manin matrix iff any $2\times2$ submatrix of this matrix is a~$q$-Manin matrix ($n\ge2$). For an $n\times n$ $q$-Manin matrix we can change the order of columns in~the expression of $q$-determinant in the following way~\cite{qManin}:
\begin{align}
 \sum_{\sigma\in \SSS_n}(-1)^\sigma M^{\sigma(1)}_{\tau(1)}M^{\sigma(2)}_{\tau(2)} \cdots M^{\sigma(n)}_{\tau(n)}=(-q)^{-\inv(\tau)}\det_q M, \qquad
\tau\in \SSS_n. \label{detqMtau}
\end{align}
By changing $\tau$ by $\tau^{-1}$ in~\eqref{detqMtau} and by taking into account $\inv(\tau^{-1})=\inv(\tau)$ one can write this formula in the form $\det_q(_\tau M)=(-q)^{-\inv(\tau)}\det_q M$. In contrast with the case of Section~\ref{secMm} the $q$-determinant of the matrix $^\tau M$ obtained from an $n\times n$ $q$-Manin matrix by a permutation of rows does not related with $\det_q M$ (the proof done for the $q=1$ case does not work since $q^{\inv(\tau\sigma)}\ne q^{\inv(\tau)}q^{\inv(\sigma)}$). Moreover, neither $^\tau M$ nor $_\tau M$ are $q$-Manin matrices in general. However they are Manin matrices for another idempotents. Namely, by virtue of Proposition~\ref{PropPerm} they are $\big((\tau\otimes\tau)A_n^q\big(\tau^{-1}\otimes\tau^{-1}\big),A_m^q\big)$- and $\big(A_n^q,(\tau\otimes\tau)A_m^q\big(\tau^{-1}\otimes\tau^{-1}\big)\big)$-Manin matrices respectively (see Section~\ref{secwhqMM}). As we will see in Section~\ref{secMinqp} the $q$-determinant is a natural operation for $\big(A_n^q,B\big)$-Manin matrices for any $B$, but not for $\big(B,A_m^q\big)$-Manin matrices (the symmetry of the $q$-determinant of an $\big(A_n^q,B\big)$-Matrix with respect to permutation of columns depends on the choice of the idempotent~$B$).

Analogously to the case $q=1$, the formula~\eqref{detqMtau} is valid for any $M\in\gR\otimes\End(\CC^n)$ satisfying the cross-relations~\eqref{Mqikjl}, that is for any $\big(A_n^q,\wt A_n^q\big)$-Manin matrix $M$, where $\wt A_n^q=\frac{1-\wt P_n^q}2$, $\big(\wt P_n^q\big)^{ij}_{kl}=(-1)^{\delta_{ij}}q^{\sgn(i-j)}\delta^i_l\delta^j_k$, $i,j,k,l=1,\dots,n$.

\subsection[Multi-parametric case: (q, p)-Manin matrices]
{Multi-parametric case: $\boldsymbol{(\wh q,\wh p)}$-Manin matrices}\label{secwhqMM}

In Section~\ref{secqM} we introduced a $q$-deformation of the polynomial algebra $\CC[x^1,\dots,x^n]$. The $q$-commutation of variables was defined by a unique parameter $q$. However we can consider multi-parameter deformation~\cite{Manin89}. One has $n(n-1)/2$ deformation parameters: one for each pair of variables.

We will say that an $n\times n$ matrix $\wh q=(q_{ij})$ is {\it parameter matrix} iff it has entries $q_{ij}\in\CC\backslash\{0\}$ satisfying the conditions
\begin{align*}
q_{ij}=q_{ji}^{-1}, \qquad q_{ii}=1. 
\end{align*}
A parameter matrix $\wh q$ defines the commutation relations
\begin{align*}
 x^jx^i=q_{ij}x^ix^j,
\end{align*}
where $i,j=1,\dots,n$ are arbitrary or subjected to $i<j$. It gives the algebra $\gX_{A_{\wh q}}(\CC)$, where
\begin{align*}
 A_{\wh q}=\frac{1-P_{\wh q}}2, \qquad
 S_{\wh q}=\frac{1+P_{\wh q}}2, \qquad
 (P_{\wh q})^{kl}_{ij}=q_{ij}\delta^k_j\delta^l_i.
\end{align*}
It is immediately checked that $(P_{\wh q})^2=1$ and that $P_{\wh q}$ satisfies the braid relation
\begin{align*}
 P_{\wh q}^{(23)}P_{\wh q}^{(12)}P_{\wh q}^{(23)}=
P_{\wh q}^{(12)}P_{\wh q}^{(23)}P_{\wh q}^{(12)}.
\end{align*}
The corresponding algebra $\Xi_{A_{\wh q}}(\CC)$ is defined by the relations
\begin{align}
 \psi_j\psi_i=-q_{ij}^{-1}\psi_i\psi_j. \label{psipsiwhq}
\end{align}
The independent relations are~\eqref{psipsiwhq} for $i<j$ and $\psi_i^2=0$.

Let $\wh p=(p_{ij})$ be an $m\times m$ parameter matrix. An $\big(A_{\wh q},A_{\wh p}\big)$-Manin matrix $M$ is an $n\times m$ matrix over an algebra $\gR$ satisfying
\begin{align*}
 A_{\wh p}M^{(1)}M^{(2)}S_{\wh p}=0.
\end{align*}
In terms of entries this relation can be written as
\begin{gather}
M^i_kM^j_k=q_{ji}M^j_kM^i_k, \label{Mqpikjk} \\
M^i_kM^j_l-q_{ji}p_{kl}M^j_lM^i_k+p_{kl}M^i_lM^j_k-q_{ji}M^j_kM^i_l=0. \label{Mqpikjl}
\end{gather}
These conditions are empty for $i=j$ and they do not change under $i\leftrightarrow j$ or $k\leftrightarrow l$, hence it is enough to check~\eqref{Mqpikjk} for $i<j$ and~\eqref{Mqpikjl} for $i<j$, $k<l$ (the relation~\eqref{Mqpikjk} is the relation~\eqref{Mqpikjl} for $k=l$).

\begin{Def}
 A matrix $M$ is called a {\it $(\wh q,\wh p)$-Manin matrix} satisfying the relations~\eqref{Mqpikjk}, \eqref{Mqpikjl}. A square matrix $M$ satisfying these relations is called {\it $\wh q$-Manin matrix} if $\wh q=\wh p$.
\end{Def}

Now let us consider the permutation of rows and columns of such matrices.

\begin{Prop} \label{PropPermqp}
 Let $M$ be an $n\times m$ matrix over $\gR$, $\sigma\in \SSS_n$ and $\tau\in \SSS_m$. Then the following statements are equivalent:
\begin{itemize}
\itemsep=0pt
 \item $M$ is an $(\wh q,\wh p)$-Manin matrix.
 \item $^\sigma M=\sigma M$ is a $\big(\sigma\wh q\sigma^{-1},\wh p\big)$-Manin matrix.
 \item $_\tau M=M\tau^{-1}$ is a $\big(\wh q,\tau\wh p\tau^{-1}\big)$-Manin matrix.
\end{itemize}
 The matrix $\sigma\wh q\sigma^{-1}$ has entries
\begin{align} \label{sigmaqsigma}
 \big(\sigma\wh q\sigma^{-1}\big)_{ij}=q_{\sigma^{-1}(i),\sigma^{-1}(j)}.
\end{align}
\end{Prop}

\begin{proof} It follows from Proposition~\ref{PropPerm} and the formula
\begin{align} \label{sigmaPwhq}
 (\sigma\otimes\sigma)P_{\wh q}\big(\sigma^{-1}\otimes\sigma^{-1}\big)=P_{\sigma\wh q\sigma^{-1}},
\end{align}
which in turn is deduced by direct calculation. \end{proof}

\begin{Rem}
 Proposition does not work for any operators $\sigma\in {\rm GL}(n,\CC)$ and $\tau\in {\rm GL}(m,\CC)$ since the formula~\eqref{sigmaPwhq} is not valid for general $\sigma\in {\rm GL}(n,\CC)$.
\end{Rem}

Let us consider more general situation: one can apply the following operations on a mat\-rix~$M$: taking a submatrix, permutation of rows or columns, doubling of a row or a column. The~result of a sequence of such operations is a new matrix $N=M_{IJ}$ considered below.

\begin{Th} \label{ThqpIJ}
 Let $I=(i_1,\dots,i_k)$ and $J=(j_1,\dots,j_l)$, where $1\le i_s\le n$ and $1\le j_t\le m$ for any $s=1,\dots,k$ and $t=1,\dots,l$. Let $M=\big(M^i_j\big)$ be an $n\times m$ matrix over $\gR$ and $M_{IJ}$ be~$k\times l$ matrix with entries $(M_{IJ})^s_t=M^{i_s}_{j_t}$. Let $\wh q$ and $\wh p$ be $n\times n$ and $m\times m$ parameter matrices and let $\wh p_{II}$ and $\wh q_{JJ}$ be $k\times k$ and $l\times l$ matrices with entries $(\wh q_{II})_{su}=q_{i_si_u}$, $s,u=1,\dots,k$, $(\wh p_{JJ})_{tv}=q_{j_tj_v}$, $t,v=1,\dots,l$. They are also parameter matrices. If $M$ is a $(\wh q,\wh p)$-Manin matrix then $M_{IJ}$ is a $(\wh q_{II},\wh p_{JJ})$-Manin matrix.
\end{Th}

\begin{proof} By substituting $i\to i_s$, $j\to i_u$, $k\to j_t$, $l\to j_v$ to~\eqref{Mqpikjl} we obtain the relations~\eqref{Mqpikjl} for the matrix $M_{IJ}$ with coefficients defined by the parameter matrices $\wh q_{II}$ and $\wh p_{JJ}$. \end{proof}

For a non-zero complex number $q$ let us denote by $q^{[n]}$ the $n\times n$ parameter matrix with entries $\big(q^{[n]}\big)_{ij}=q^{\sgn(j-i)}$. Then $P_{q^{[n]}}=P_n^q$ and the $\big(q^{[n]},q^{[m]}\big)$-Manin matrices are exactly the $n\times m$ $q$-Manin matrices. A permutation of rows or columns of a such matrix $M$ gives a~$\big(\sigma q^{[n]}\sigma^{-1},q^{[m]}\big)$- or $\big(q^{[n]},\tau q^{[m]}\tau^{-1}\big)$-Manin matrix respectively. In general these are not $q$-Manin matrices any more (see Section~\ref{secqM}), but they are related with the quadratic algebras isomorphic to $\gX_{A_n^q}(\CC)$ and $\gX_{A_m^q}(\CC)$, so that they have the same properties permuted in some sense. For instance, properties of the $q$-determinant of $_\sigma M$ are similar to ones of the $q$-determinant of a $q$-Manin matrix.

Let $x^1,\dots,x^n$ be the generators of $\gX_{A_n^q}(\CC)$. Then the $n\times n$ diagonal matrix
\begin{align} \label{Mxx}
 M=
\begin{pmatrix}
 x^1 & 0 & \cdots &0 \\
 0 & x^2 & \cdots &0 \\
 \vdots&\vdots&\ddots&\vdots \\
 0&0&\cdots&x^n
\end{pmatrix}
\end{align}
is an $(A_n^q,A_n)$-Manin matrix, i.e., a $\big(q^{[n]},1^{[n]}\big)$-Manin matrix. More generally, for any $p\in\CC\backslash\{0\}$ the matrix~\eqref{Mxx} is a $\big((pq)^{[n]},p^{[n]}\big)$-Manin matrix.

An analogue of the $q$-determinant for a $(\wh q,\wh p)$-Manin matrix depends on $\wh q$, but not on $\wh p$. We~call it {\it $\wh q$-determinant}, it is defined for an $n\times n$ matrix $M$ as follows~\cite{Manin89}:
\begin{align}
 \det_{\wh q}(M)&=\sum_{\sigma\in \SSS_n}(-1)^\sigma\prod_{\substack{i<j\\ \sigma(i)>\sigma(j)}}q_{ij}^{-1}\cdot M^{\sigma^{-1}(1)}_{1}M^{\sigma^{-1}(2)}_{2} \cdots M^{\sigma^{-1}(n)}_{n} \notag
 \\
 &=\sum_{\sigma\in \SSS_n}(-1)^\sigma\prod_{\substack{i<j\\ \sigma^{-1}(i)>\sigma^{-1}(j)}}q_{ij}^{-1}\cdot M^{\sigma(1)}_{1}M^{\sigma(2)}_{2} \cdots M^{\sigma(n)}_{n}. \label{whqdet}
\end{align}
The products in this formula runs over all inversions of the permutations $\sigma$ and $\sigma^{-1}$ respectively. Hence for $\wh q=q^{[n]}$ the formula~\eqref{whqdet} gives the $q$-determinant~\eqref{detq}. The $\wh q$-determinant is a~``gene\-ra\-lised'' determinant for the square $(\wh q,\wh p)$-Manin matrices. In particular, $q$-determinant is a ``generalised'' determinant for the square $\big(q^{[n]},\wh p\big)$-Manin matrices.

\begin{Lem} \label{Lempsisigma}
Let $\psi_1,\dots,\psi_n$ satisfy~\eqref{psipsiwhq}. Then
\begin{gather}
\psi_{i_1}\psi_{i_2}\cdots\psi_{i_n}=0, \qquad\text{if}\quad i_k=i_l\quad \text{for some}\quad k\ne l; \label{Lempsisigma1}
\\
 \psi_{\sigma(1)}\psi_{\sigma(2)}\cdots\psi_{\sigma(n)}=(-1)^\sigma\psi_{1}\psi_{2}
 \cdots\psi_{n}\!\!\!\!\prod_{\substack{i<j\\ \sigma^{-1}(i)>\sigma^{-1}(j)}}\!\!\!\!q_{ij}^{-1}, \qquad
 \text{for any}\quad \sigma\in \SSS_n. \label{Lempsisigma2}
\end{gather}
\end{Lem}

\begin{proof} The relations~\eqref{psipsiwhq} allows us to permute the factors in the left hand side of~\eqref{Lempsisigma1}. In~particular, one can place the factors $\psi_{i_k}$ and $\psi_{i_l}$ to neighbour sites. It $i_k=i_l$ then $\psi_{i_k}\psi_{i_l}=0$.

Let us rewrite the formula~\eqref{Lempsisigma2} in terms of Appendix~\ref{appRw}. Consider the root system for the reflection group $\SSS_n$. Denote $q_\alpha=q_{ij}$ for the root $\alpha=e_i-e_j$. Then due to~\eqref{Rsigma} the formula~\eqref{Lempsisigma2} takes the form
\begin{align}
 \psi_{\sigma(1)}\cdots\psi_{\sigma(n)}&=(-1)^\sigma \psi_{1}\cdots\psi_{n}\prod_{\alpha\in R^+_{\sigma^{-1}}}q_\alpha^{-1}. \label{Lempsisigma2R}
\end{align}
We prove the formula~\eqref{Lempsisigma2R} by using induction on the length $\ell=\ell(\sigma)$. Let $\sigma=\sigma_{i_1}\cdots\sigma_{i_{\ell-1}}\sigma_{i_\ell}$ be a reduced expression and $\tau=\sigma_{i_1}\cdots\sigma_{i_{\ell-1}}$. Then $\psi_{\sigma(1)}\cdots\psi_{\sigma(n)}=\psi_{\tau\sigma_{i_\ell}(1)}\cdots\psi_{\tau\sigma_{i_\ell}(n)}=-q^{-1}_{\tau(\alpha_{i_\ell})}\psi_{\tau(1)}\cdots\psi_{\tau(n)}$. Note that $\ell(\tau)=\ell-1$. From~\eqref{PropRw2} we obtain $R^+_{\sigma^{-1}}=R^+_{\tau^{-1}}\sqcup\{\tau(\alpha_{i_\ell})\}$. Together with the induction assumption this implies the formula~\eqref{Lempsisigma2R}. \end{proof}

Lemma~\ref{Lempsisigma} implies that for any $n\times n$ matrix $M$ such that $[M^i_j,\psi_k]=0$ we have
\begin{align} \label{phipsidet}
 \phi_1\phi_2\cdots\phi_n=\det_{\wh q}(M)\psi_1\psi_2\cdots\psi_n,
\end{align}
where $\phi_j=\sum_{i=1}^n\psi_iM^i_j$. A permutation of rows or columns corresponds to a permutation of $\psi_1,\dots,\psi_n$ or $\phi_1,\dots,\phi_n$ respectively:
\begin{gather} \label{phipsitau}
 \phi_j=\sum_{i=1}^n(^\tau M)^i_j\psi_{\tau^{-1}(i)}, \qquad
 \phi_{\tau^{-1}(j)}=\sum_{i=1}^n(_\tau M)^i_j\psi_i.
\end{gather}

\begin{Th} \label{Thdettau}
 Let $\wh q=(q_{ij})$ and $\wh p=(p_{ij})$ be $n\times n$ parameter matrices and $\tau\in \SSS_n$. Let $M$ be a $(\wh q,\wh p)$-Manin matrix over $\gR$. Then the ``generalised'' determinants of the $(\tau\wh q\tau^{-1},\wh p)$-Manin matrix $^\tau M$ and of the $(\wh q,\tau\wh p\tau^{-1})$-Manin matrix $_\tau M$ have the form
\begin{gather}
 \det_{\tau\wh q\tau^{-1}}(^\tau M)=(-1)^\tau \det_{\wh q}(M) \prod_{\substack{i<j\\ \tau(i)>\tau(j)}}q_{ij}, \label{Thdettaurow} \\
 \det_{\wh q}(_\tau M)=(-1)^\tau \det_{\wh q}(M) \prod_{\substack{i<j\\ \tau(i)>\tau(j)}}p_{ij}^{-1}. \label{Thdettaucol}
\end{gather}
More generally, the formula~\eqref{Thdettaurow} is valid for any $n\times n$ matrix $M$.
\end{Th}

\begin{proof} Let $\psi_i$ be generators of the algebra $\Xi_{A_{\wh q}}(\CC)$. Then Proposition~\ref{propAB} implies that the elements $\phi_j=\sum_{i=1}^nM^i_j\psi_i\in\Xi_{A_{\wh q}}(\gR)$ satisfy the commutation relations~\eqref{psipsiwhq} with the parameter matrix $\wh p$. Due to the formula~\eqref{sigmaqsigma} the elements $\psi'_i=\psi_{\tau^{-1}(i)}$ and $\phi'_j=\phi_{\tau^{-1}(j)}$ satisfy the commutation relations~\eqref{psipsiwhq} with the parameter matrices $\sigma\wh q\sigma^{-1}$ and $\sigma\wh p\sigma^{-1}$ respectively. From the formulae~\eqref{phipsidet} and~\eqref{phipsitau} we obtain
\begin{gather}
 \phi_1\cdots\phi_n=\psi'_1\cdots\psi'_n\det_{\tau\wh q\tau^{-1}}(^\tau M), \qquad
 \phi'_1\cdots\phi'_n=\psi_1\cdots\psi_n\det_{\wh q}(_\tau M). \label{phipsitaudet}
\end{gather}
The formula~\eqref{Lempsisigma2} for $\phi_j$ and $\psi_i$ with $\sigma=\tau^{-1}$ takes the form
\begin{gather*}
 \psi'_1\cdots\psi'_n=(-1)^\tau\psi_1\cdots\psi_n\prod_{\substack{i<j\\ \tau(i)>\tau(j)}}q_{ij}^{-1}, \qquad
 \phi'_1\cdots\phi'_n=(-1)^\tau\phi_1\cdots\phi_n\prod_{\substack{i<j\\ \tau(i)>\tau(j)}}p_{ij}^{-1}.
\end{gather*}
Substitution of these formulae and the formula~\eqref{phipsidet} to~\eqref{phipsitaudet} gives~\eqref{Thdettaurow} and \eqref{Thdettaucol}. We did not use the commutation relations of $\phi_j$ for the proof of the formula~\eqref{Thdettaurow}, so it is valid for any matrix $M$. \end{proof}

Let us consider the case when two rows or two columns coincide. We write some conditions leading to vanishing of the $\wh q$-determinant.

\begin{Cor} \label{Cordetq}
 Let $M$ be a $k\times k$ matrix over $\gR$. Let $\wh q$ and $\wh p$ be $k\times k$ parameter matrices.
\begin{itemize}
\itemsep=0pt
 \item Let $i\ne j$. If $M^i_l=M^j_l$ and $q_{il}=q_{jl}$ $\forall\,l$ $($in particular, $q_{ij}=1)$ then $\det_{\wh q}(M)=0$.
 \item If two columns of a $(\wh q,\wh p)$-Manin matrix $M$ coincide $($that is $M^l_i=M^l_j$ $\;\forall\,l$ for some $i\ne j)$ then $\det_{\wh q}(M)=0$.
\end{itemize}
\end{Cor}

\begin{proof} Let $\sigma_{ij}\in\SSS_k$ be the transposition of $i$ and $j$. Suppose $i<j$ (without loss of generality). The conditions $q_{il}=q_{jl}$ imply $\sigma_{ij}\wh q\sigma_{ij}=\wh q$ and
\[
\prod\limits_{\substack{s<t\\ \sigma_{ij}(s)>\sigma_{ij}(t)}}q_{\rm st}=q_{ij} \prod\limits_{s=i+1}^{j-1}(q_{sj}q_{is})=1,
\]
so the substitution $\tau=\sigma_{ij}$ to~\eqref{Thdettaurow} gives $\det_{\wh q}(M)=-\det_{\wh q}(M)$. For generic $\wh p$ one can similarly prove the second statement by using the formula~\eqref{Thdettaucol} with $\tau=\sigma_{ij}$. For arbitrary $\wh p$ it follows from the relations~\eqref{phipsidet} and $\phi_i\phi_j=\phi_i^2=0$. \end{proof}

\begin{Cor} \label{CordetqIJ}
 Let $M$ be an $n\times m$ matrix over $\gR$. Let $\wh q$ and $\wh p$ be $n\times n$ and $m\times m$ parameter matrices. Let $I=(i_1,\dots,i_k)$ and $J=(j_1,\dots,j_k)$, where $1\le i_s\le n$ and $1\le j_s\le m$ for all $s=1,\dots,k$. Let $\tau\in\SSS_k$, $K=(i_{\tau(1)},\dots,i_{\tau(k)})$ and $L=(j_{\tau(1)},\dots,j_{\tau(k)})$.
\begin{itemize}
\itemsep=0pt
\item We have
\[
\det_{\wh q_{KK}}(M_{KJ})=(-1)^\tau
\det_{\wh q_{II}}(M_{IJ})\prod_{\substack{s<t\\ \tau^{-1}(s)>\tau^{-1}(t)}}q_{i_si_t}.
\]
\item If $i_s=i_t$ for some $s\ne t$ then $\det_{\wh q_{II}}(M_{IJ})=0$.
\item If $M$ is a $(\wh q,\wh p)$-Manin matrix then
\[
\det_{\wh q_{II}}(M_{IL})=(-1)^\tau\det_{\wh q_{II}}(M_{IJ})
\prod_{\substack{s<t\\ \tau^{-1}(s)>\tau^{-1}(t)}}p_{j_sj_t}^{-1}.
\]
\item If $M$ is a $(\wh q,\wh p)$-Manin matrix and $j_s=j_t$ for some $s\ne t$ then $\det_{\wh q_{II}}(M_{IJ})=0$.
\end{itemize}
\end{Cor}

\begin{proof} Note that $\tau^{-1}\wh q_{II}\tau=\wh q_{KK}$, $M_{KJ}=\leftidx{^{\tau^{-1}}}(M_{IJ})$ and $M_{IL}=\leftidx{_{\tau^{-1}}}(M_{IJ})$, so the statements follow from Theorems~\ref{ThqpIJ}, \ref{Thdettau} and Corollary~\ref{Cordetq}. \end{proof}

\begin{Rem}\looseness=-1
 The formula~\eqref{Lempsisigma2} follows from $\psi_j\psi_i=-q_{ij}^{-1}\psi_i\psi_j$, $i<j$. As consequence, we did not need the relations $\phi_i^2=0$ to prove the formula~\eqref{Thdettaucol}. The relations $\psi_j\psi_i=-p_{ij}^{-1}\psi_i\psi_j$, $i<j$, define the algebra $\Xi_{\wt A_{\wh p}}(\CC)$, where $\wt A_{\wh p}=\frac{1-\wt P_{\wh p}}2$ and $\big(\wt P_{\wh p}\big)^{kl}_{ij}=(-1)^{\delta_{ij}}p_{ij}\delta^k_j\delta^l_i$. Hence the formula~\eqref{Thdettaucol} is valid for any $\big(A_{\wh q},\wt A_{\wh p}\big)$-Manin matrix $M$. As consequence, the second statement of Corollary~\ref{Cordetq} is valid for these matrices if $\wh p$ is generic. However they are not valid for some~$\wh p$, so it is necessary to require $M$ to be a $(\wh q,\wh p)$-Manin matrix. Moreover the third and forth statements of Corollary~\ref{CordetqIJ} are not valid for $\big(A_{\wh q},\wt A_{\wh p}\big)$-Manin matrices even if $\wh p$ is generic since we used Theorem~\ref{ThqpIJ}. For example, let $M=\left(\begin{smallmatrix} a & b \\ c & d \end{smallmatrix}\right)$, $I=(1,2)$, $J=(1,1)$. Then the $q$-determinant of the matrix $M_{IJ}=\left(\begin{smallmatrix} a & a \\ c & c \end{smallmatrix}\right)$ is $\det_q(M_{IJ})=ac-q^{-1}ca$. It vanishes if $M$ is a~$\big(q^{[2]},p^{[2]}\big)$-Manin matrix, but the cross relation $ad-q^{-1}pda+pbc-q^{-1}cb=0$ is not enough for $\det_q(M_{IJ})=0$. The cross relation for the matrix $\left(\begin{smallmatrix} a & a \\ c & c \end{smallmatrix}\right)$ itself gives $ac-qp^{-1}ca+pac-q^{-1}ca=0$. This imply $\det_q(M)=0$ unless $p=-1$.
\end{Rem}

\begin{Rem}
 Let $n=k+l$, $q_{ij}=-1$ for $i,j=k+1,\dots,n$, $i\ne j$, and $q_{ij}=1$ for other $i$,~$j$. By factorizing the algebra $\gX_{A_{\wh q}}(\CC)$ over the relations $x_i^2=0$, $i=k+1,\dots,n$, and introducing a $\ZZ_2$-grading we obtain the free super-commutative quadratic algebra with $k$ even and $l$ odd generators. However the approach of Section~\ref{secQAMM} applied to this algebra does not give super-Manin matrices considered in~\cite{Manin89,MR}. The reason is that we suppose commutativity of $x^i$ with entries of $M$, which should be replaced by super-commutativity in the super-case. We will consider Manin matrices for quadratic super-algebras in future works.
\end{Rem}

\subsection{A 4-parametric quadratic algebra}\label{sec4p}

Consider the algebra with generators $x$, $y$, $z$ and relations
\begin{gather}
 axy-a^{-1}yx=\kappa z^2, \qquad
 byz-b^{-1}zy=\kappa x^2, \qquad
 czx-c^{-1}xz=\kappa y^2, \label{abcxyzkappa}
\end{gather}
where $a,b,c\in\CC\backslash\{0\}$, $\kappa\in\CC$.

Let $x^1=x$, $x^2=y$, $x^3=z$, $a_{12}=a$, $a_{23}=b$, $a_{31}=c$, $a_{ij}=a_{ji}^{-1}$, $a_{ii}=1$. Let $\varepsilon_{ijk}$ be the totally antisymmetric tensor such that $\varepsilon_{123}=1$. Then~\eqref{abcxyzkappa} is equivalent to the system
\begin{gather}
 a_{ij}x^ix^j-a_{ji}x^jx^i=\kappa\sum_{k=1}^3 \varepsilon_{ijk} x^kx^k, \qquad
 i,j=1,2,3. \label{aijxx}
\end{gather}
These relations can be written as $x^i x^j=\sum_{k,l=1}^3P^{ij}_{kl}x^kx^l$, where
\begin{align} \label{PaijEntr}
 P^{ij}_{kl}=a_{ji}^2\delta^j_k \delta^i_l+\kappa a_{ji}\delta_{kl} \varepsilon_{ijk}.
\end{align}
The operator $P\in\End\big(\CC^3\otimes\CC^3\big)$ with the entries~\eqref{PaijEntr} satisfies $P^2=1$. Hence the operator $A^{a,b,c}_\kappa:=\frac{1-P}2$ is an idempotent and the relations~\eqref{abcxyzkappa} define the algebra $\gX_{A^{a,b,c}_\kappa}(\CC)$. By setting $\kappa=0$ we obtain the quadratic algebra $\gX_{A_{\wh q}}(\CC)$ with the parameters $q_{ij}=a_{ij}^2$, so the alge\-bra~$\gX_{A^{a,b,c}_\kappa}(\CC)$ is a generalisation of the $3$-dimensional case of the algebra $\gX_{A_{\wh q}}(\CC)$ considered in~Section~\ref{secwhqMM} (this is not a $\kappa$-deformation in general, see Remark~\ref{Rem4p}).

Let us consider some examples of Manin matrices by taking $A=A^{a,b,c}_\kappa$ as one of the idempotents. A $2\times3$ matrix
$ M=\left(\begin{smallmatrix}
 \alpha_1 & \alpha_2 & \alpha_3 \\
 \beta_1 & \beta_2 & \beta_3
\end{smallmatrix}\right)$
is an $\big(A^q_2,A^{a,b,c}_\kappa\big)$-Manin matrix iff
\begin{gather}
 q(a_{ji}\alpha_i\beta_j+a_{ij}\alpha_j\beta_i)=
a_{ji}\beta_i\alpha_j+a_{ij}\beta_j\alpha_i, \label{AA3cross}
\\
 q(a_{ij}\alpha_k\beta_k+\kappa\alpha_i\beta_j)=
a_{ij}\beta_k\alpha_k+\kappa\beta_i\alpha_j \label{AA3col}
\end{gather}
for all cyclic permutation $(i,j,k)$ of $(1,2,3)$. The relation~\eqref{AA3cross} is exactly the cross relation~\eqref{Mqpikjl} for the parameters $q_{ij}=q^{\sgn(j-i)}$ and $p_{ij}=a_{ij}^2$, while the relation~\eqref{AA3col} is a~generalisation of the $q$-commutation~\eqref{Mqpikjk}.

A $3\times2$ matrix
$ M=\left(\begin{smallmatrix}
 \alpha^1 & \beta^1 \\
 \alpha^2 & \beta^2 \\
 \alpha^3 & \beta^3
\end{smallmatrix}\right)$
is an $\big(A^{a,b,c}_\kappa,A^q_2\big)$-Manin matrix iff the relations~\eqref{aijxx} are satisfied by the substitutions $x^i=\alpha^i$ and $x^i=\beta^i$ and
\begin{align*}
 a_{ij}\big(\alpha^i\beta^j+q\beta^i\alpha^j\big)-a_{ji}\big(\alpha^j\beta^i+q\beta^j\alpha^i\big) =\kappa\big(\alpha^k\beta^k+q\beta^k\alpha^k\big)
\end{align*}
for all cyclic permutations $(i,j,k)$ of $(1,2,3)$.

\section{Lax operators}
\label{secLO}

Lax operators are different square matrices and endomorphisms of vector spaces arisen in the theories of integrable systems and quantum groups. We will consider Lax operators satisfying $RLL$-relations with some $R$-matrices (solutions of the Yang--Baxter equation). Different $R$-mat\-rices give different types of Lax operators. Since many quantum groups can be defined by $RLL$-relations the Lax operators of a certain type are related with the representation theory of the corresponding quantum group. Here we consider connections between Manin matrices associated with some quadratic algebras and the Lax operators associated with the quantum groups $U_q(\mathfrak{gl}_n)$, $Y(\mathfrak{gl}_n)$. Notice also that a connection between the $q$-Manin matrices and Lax operators associated the affine quantum group $U_q\big(\wh{\mathfrak{gl}}_n\big)$ was described in~\cite{qManin}.

\subsection[Lax operators of Uq(gln) type and q-Manin matrices]
{Lax operators of $\boldsymbol{U_q(\mathfrak{gl}_n)}$ type and $\boldsymbol q$-Manin matrices}\label{secUq}

A relationship between Lax operator and $q$-Manin matrices was first described by Manin, see~\cite{Manin88}. We investigate this relationship by applying a decomposition of the corresponding $R$-matrix.

Let us first write the relations for a transposed $q$-Manin matrix. Recall that the matrices $P_n$ defined by~\eqref{PnDef} permute the factors $\Hom(\CC^m,\CC^n)\otimes\Hom(\CC^m,\CC^n)$ in the following way:
\begin{align}
 P_nTP_m=T^{(21)}, \qquad P_nM^{(1)}N^{(2)}P_m=M^{(2)}N^{(1)}, \label{P1Def}
\end{align}
where $T\in\gR\otimes\Hom(\CC^m\otimes\CC^m,\CC^n\otimes\CC^n)$, $M,N\in\gR\otimes\Hom(\CC^m,\CC^n)$. Let us note that
\begin{align} \label{PAS21}
 \big(P_n^q\big)^{(21)}=P_n^{q^{-1}}, \qquad
 \big(A_n^q\big)^{(21)}=A_n^{q^{-1}}, \qquad
 \big(S_n^q\big)^{(21)}=S_n^{q^{-1}}.
\end{align}
Note also that transposition gives the same:
\begin{align} \label{PAStop}
 \big(P_n^q\big)^\top=P_n^{q^{-1}}, \qquad
 \big(A_n^q\big)^\top=A_n^{q^{-1}}, \qquad
 \big(S_n^q\big)^\top=S_n^{q^{-1}}.
\end{align}

\begin{Lem} \label{PropMtr}
 Let $M\in\gR\otimes\Hom(\CC^m,\CC^n)$. The transposed matrix $M^\top$ is a $q$-Manin matrix iff the matrix $M$ satisfies one of the following equivalent relations:
\begin{gather}
 S^{q^{-1}}_n M^{(1)}M^{(2)} A_m^{q^{-1}}=0, \label{MtrR1} \\
 S_n^q M^{(2)}M^{(1)} A_m^q=0. \label{MtrR2}
\end{gather}
\end{Lem}

\begin{proof} The relation~\eqref{AMMSq} for the $m\times n$ matrix $M^\top$ has the form $A_m^q\big(M^\top\big)^{(1)}\big(M^\top\big)^{(2)}S_n^q=0$. If we transpose the both hand sides and take into account $\big(M^{(1)}M^{(2)}\big)^\top=\big(M^\top\big)^{(1)}\big(M^\top\big)^{(2)}$ and~\eqref{PAStop} we obtain~\eqref{MtrR1}. Due to~\eqref{PAS21} the permutation of tensor factors yields~\eqref{MtrR2}. \end{proof}

Suppose that $q^2\ne-1$. Consider the $R$-matrix
\begin{align*}
R^q=R^q_n= q^{-1} \sum_{i=1}^n E_{i}{}^{i}\otimes E_{i}{}^{i} +\sum_{i\ne j} E_{i}{}^{i}\otimes E_{j}{}^{j} +\big(q^{-1}-q\big) \sum_{i>j} E_{i}{}^{j}\otimes E_{j}{}^{i}. 
\end{align*}
It satisfies the Yang--Baxter equation
\begin{align} \label{RqYBR}
 (R^q)^{(12)} (R^q)^{(13)}(R^q)^{(23)}= (R^q)^{(23)} (R^q)^{(13)} (R^q)^{(12)}.
\end{align}
A Lax operator of $U_q(\mathfrak{gl}_n)$ type is an $n\times n$ matrix $L\in\gR\otimes\End(\CC^n)$ satisfying the $RLL$-relation
\begin{align*}
 R^q L^{(1)}L^{(2)} =L^{(2)}L^{(1)} R^q.
\end{align*}
More generally, consider an $n\times m$ matrix $L\in\gR\otimes\Hom(\CC^m,\CC^n)$ satisfying
\begin{align}
 R^q_n L^{(1)}L^{(2)} =L^{(2)}L^{(1)} R^q_m. \label{RLL}
\end{align}

\begin{Rem}
 The commutation relations for the quantum group $U_q(\mathfrak{gl}_n)$ can be written as three matrix relations $R^q_nL_\pm^{(1)}L_\pm^{(2)}=L_\pm^{(2)}L_\pm^{(1)} R^q_n$, $R^q_nL_+^{(1)}L_-^{(2)}=L_-^{(2)}L_+^{(1)} R^q_n$ for some matrices $L_+,L_-\in U_q(\mathfrak{gl}_n)\otimes\End(\CC^n)$~\cite{FRT88,FRT89}.
\end{Rem}

By multiplying the relation~\eqref{RLL} by $P_n$ from the left and by taking into account~\eqref{P1Def} we obtain the equivalent relation
\begin{align}
 \wh R^q_n L^{(1)}L^{(2)} =L^{(1)}L^{(2)} \wh R^q_n, \label{whRLL}
\end{align}
where
\begin{align}
 \wh R^q=\wh R^q_n:=P_nR^q_n= q^{-1} \sum_{i=1}^n E_{i}{}^{i}\otimes E_{i}{}^{i} +\sum_{i\ne j} E_{i}{}^{j}\otimes E_{j}{}^{i} +\big(q^{-1}-q\big) \sum_{i<j} E_{i}{}^{i}\otimes E_{j}{}^{j}. \label{whRq}
\end{align}
The matrix~\eqref{whRq} satisfies the braid relation
\begin{align}
 \big(\wh R^q\big)^{(23)} \big(\wh R^q\big)^{(12)}\big(\wh R^q\big)^{(23)}= \big(\wh R^q\big)^{(12)} \big(\wh R^q\big)^{(23)} \big(\wh R^q\big)^{(12)}, \label{BRq}
\end{align}
which is obtained by multiplying left and right hand sides of \eqref{PBR} and \eqref{RqYBR}.

\begin{Lem}[{\cite{FRT89}}] \label{LemRdec}
The matrix~\eqref{whRq} can be decomposed as
\begin{align}
 \wh R^q=q^{-1}\wh R_+^q-q\wh R_-^q, \label{whRqpDec}
\end{align}
where
\begin{align}
\wh R_+^q&=\wh R_{n+}^q:=\frac{q+\wh R^q_n}{q+q^{-1}}\notag
\\
&=\sum_{i=1}^n\! E_{i}{}^{i}\otimes E_{i}{}^{i}\!+\!\frac1{q\!+\!q^{-1}}\sum_{i\ne j}\! E_{i}{}^{j}\otimes E_{j}{}^{i}\!+\!\frac1{q\!+\!q^{-1}}\sum_{i<j}\!\big(q^{-1}E_{i}{}^{i}\otimes E_{j}{}^{j}\!+\!q E_{j}{}^{j}\otimes E_{i}{}^{i}\big), \label{whRqp}
\\
\wh R_-^q&=\wh R^q_{n-}:=\frac{q^{-1}-\wh R^q_n}{q+q^{-1}}\notag
\\
&=-\frac1{q+q^{-1}}\sum_{i\ne j} E_{i}{}^{j}\otimes E_{j}{}^{i}
+\frac1{q+q^{-1}}\sum_{i<j}\big(q E_{i}{}^{i}\otimes E_{j}{}^{j}+q^{-1} E_{j}{}^{j}\otimes E_{i}{}^{i}\big) \label{whRqm}
\end{align}
are orthogonal idempotents:
\begin{gather}
\wh R_+^q+\wh R_-^q=1, \qquad \big(\wh R_+^q\big)^2=\wh R_+^q, \qquad \big(\wh R_-^q\big)^2=\wh R_-^q, \qquad
\wh R_+^q\wh R_-^q=\wh R_-^q\wh R_+^q=0. \label{whRqpmOI}
\end{gather}
The matrix~\eqref{whRq} satisfies the Hecke relation:
\begin{align}
 \big(\wh R^q-q^{-1}\big)\big(\wh R^q+q\big)=0. \label{HeckeR}
\end{align}
\end{Lem}

\begin{proof} The formulae~\eqref{whRqpDec}, \eqref{whRqp}, \eqref{whRqm} and the first of~\eqref{whRqpmOI} are obtained directly from the definitions of $\wh R_+^q=\frac{q+\wh R^q_n}{q+q^{-1}}$ and $\wh R_-^q=\frac{q^{-1}-\wh R^q_n}{q+q^{-1}}$. Further
\begin{gather*}
\big(q+q^{-1}\big)^2\big(\wh R_-^q\big)^2
=\bigg({-}\sum_{i\ne j} E_{i}{}^{j}\otimes E_{j}{}^{i}+\sum_{i<j}\big(q E_{i}{}^{i}\otimes E_{j}{}^{j}+q^{-1} E_{j}{}^{j}\otimes E_{i}{}^{i}\big)\bigg)^2
\\ \hphantom{\big(q+q^{-1}\big)^2\big(\wh R_-^q\big)^2}
{}= \sum_{i\ne j} E_{i}{}^{i}\otimes E_{j}{}^{j}-\sum_{i<j}\big(q E_{i}{}^{j}\otimes E_{j}{}^{j}+q^{-1} E_{j}{}^{i}\otimes E_{i}{}^{i}\big)
\\ \hphantom{(q+q^{-1})^2\big(\wh R_-^q\big)^2=}
{}-\!\sum_{i<j}\big(q^{-1} E_{i}{}^{j}\otimes E_{j}{}^{j}+q E_{j}{}^{i}\otimes E_{i}{}^{i}\big)
+\!\sum_{i<j}\big(q^2 E_{i}{}^{i}\otimes E_{j}{}^{j}+q^{-2} E_{j}{}^{j}\otimes E_{i}{}^{i}\big)
\\ \hphantom{\big(q+q^{-1}\big)^2\big(\wh R_-^q\big)^2}
{}= -(q\!+\!q^{-1})\sum_{i\ne j}\! E_{i}{}^{j}\otimes E_{j}{}^{i}+\!\sum_{i<j}\!\big(\big(q^2+1\big) E_{i}{}^{i}\otimes E_{j}{}^{j}\!+\!\big(1\!+\!q^{-2}\big) E_{j}{}^{j}\otimes E_{i}{}^{i}\big)
\\ \hphantom{\big(q+q^{-1}\big)^2\big(\wh R_-^q\big)^2}
{} = \big(q+q^{-1}\big)\bigg({-}\sum_{i\ne j} E_{i}{}^{j}\otimes E_{j}{}^{i}+\sum_{i<j}\big(q E_{i}{}^{i}\otimes E_{j}{}^{j}+q^{-1} E_{j}{}^{j}\otimes E_{i}{}^{i}\big)\bigg)
\\ \hphantom{\big(q+q^{-1}\big)^2\big(\wh R_-^q\big)^2}
{}=\big(q+q^{-1}\big)^2\wh R_-^q.
\end{gather*}
Thus we obtain $\big(\wh R_-^q\big)^2=\wh R_-^q$. Other relations~\eqref{whRqpmOI} follows from $\wh R_+^q=1-\wh R_-^q$. Then the Hecke relation~\eqref{HeckeR} is a consequence of $\wh R_-^q\wh R_+^q=0$. \end{proof}

We see that
\begin{align}
 \wh R_{n-}^q=\frac{P_nP^q_n-P_n}{q+q^{-1}}=-\frac2{q+q^{-1}}P_nA_n^q. \label{RAn}
\end{align}
This means that the idempotents $A_n^q$ and $\wh R_{n-}^q$ are left-equivalent. Thus an $n\times m$ $q$-Manin matrix is exactly the Manin matrix for the pair $\big(\wh R_{n-}^q,\wh R_{m-}^q\big)$. Due to~\eqref{PAS21} we obtain
\begin{align}
 \wh R_{n-}^q=-\frac2{q+q^{-1}}A_n^{q^{-1}}P_n, \label{RAnRE}
\end{align}
so idempotent $\wh R_{n-}^q$ is right-equivalent to $A_n^{q^{-1}}$.

Note that the relation $\big(\wh R_-^q\big)^2=\wh R_-^q$ implies
\begin{align*}
 \bigg({-}\frac2{q+q^{-1}}\bigg)^2P_nA_n^qP_nA_n^q=-\frac2{q+q^{-1}}P_nA_n^q,
\end{align*}
so that
\begin{gather}
 A_n^{q^{-1}}A_n^q=-\frac{q+q^{-1}}2P_nA_n^q, \qquad
 A_n^qA_n^{q^{-1}}=-\frac{q+q^{-1}}2A_n^qP_n. \label{AAPA}
\end{gather}
\begin{Th} \label{ThqLax}
 Let $L\in\gR\otimes\Hom(\CC^m,\CC^n)$, then the following statements are equivalent:
\begin{itemize}
\itemsep=0pt
 \item $L$ satisfies~\eqref{RLL}, that is $R^q_n L^{(1)}L^{(2)} =L^{(2)}L^{(1)} R^q_m$.
 \item $L$ satisfies~\eqref{whRLL}, that is $\wh R^q_n L^{(1)}L^{(2)} =L^{(1)}L^{(2)} \wh R^q_m$.
 \item $L$ satisfies the relation
\begin{align}
\wh R_{n+}^q L^{(1)}L^{(2)} =L^{(1)}L^{(2)} \wh R_{m+}^q. \label{whRpLL}
\end{align}
 \item $L$ satisfies the relation
\begin{align}
\wh R_{n-}^q L^{(1)}L^{(2)} =L^{(1)}L^{(2)} \wh R_{m-}^q. \label{whRmLL}
\end{align}
 \item $L$ satisfies the relation
\begin{align}
 A_n^q L^{(1)}L^{(2)} =L^{(2)}L^{(1)}A_m^q. \label{AqLL}
\end{align}
 \item $L$ satisfies the relations
\begin{align}
 A_n^q L^{(1)}L^{(2)}S_m^q =0, \qquad S_n^qL^{(2)}L^{(1)}A_m^q=0. \label{AqLLSq}
\end{align}
 \item The matrices $L$ and $L^\top$ are both $q$-Manin matrices.
\end{itemize}
\end{Th}

\begin{proof} By adding $qL^{(1)}L^{(2)}$ to the both hand sides of~\eqref{whRLL} and by dividing by $q+q^{-1}$ we obtain the equivalent relation~\eqref{whRpLL}. The equivalence of the relations~\eqref{whRLL} and~\eqref{whRmLL} is proved similarly. Further, by using~\eqref{RAn} one establishes the equivalence of~\eqref{whRmLL} and \eqref{AqLL}. The relations~\eqref{AqLLSq} are obtained from~\eqref{AqLL} by multiplying by $S_m^q$ from the right and by multiplying by $S_n^q$ from the left respectively. Conversely, suppose that $L$ satisfies the relations~\eqref{AqLLSq}. By virtue of the formulae~\eqref{PAS21} the second of the relations~\eqref{AqLLSq} can be written in the form
$S^{q^{-1}}_nL^{(1)}L^{(2)}A_m^{q^{-1}}=0$. Thus we have
\begin{align*}
 A_n^q L^{(1)}L^{(2)}A_m^q=A_n^q L^{(1)}L^{(2)}, \qquad L^{(1)}L^{(2)}A_m^{q^{-1}}=A_n^{q^{-1}}L^{(1)}L^{(2)}A_m^{q^{-1}}.
\end{align*}
These relations implies
\begin{align*}
 A_n^q L^{(1)}L^{(2)}A_m^qA_m^{q^{-1}}=
 A_n^qL^{(1)}L^{(2)}A_m^{q^{-1}}=A_n^qA_n^{q^{-1}}L^{(1)}L^{(2)}A_m^{q^{-1}}.
\end{align*}
By using~\eqref{AAPA} one yields
\begin{align*}
 A_n^q L^{(1)}L^{(2)}A_m^qP_m=A_n^qP_nL^{(1)}L^{(2)}A_m^{q^{-1}}.
\end{align*}
Multiplication by $P_m$ from the right gives
\begin{align*}
 A_n^q L^{(1)}L^{(2)}A_m^q=A_n^qL^{(2)}L^{(1)}A_m^q.
\end{align*}
By taking into account~\eqref{AqLLSq} we obtain~\eqref{AqLL} in the following way:
\begin{align*}
 A_n^q L^{(1)}L^{(2)}&=A_n^q L^{(1)}L^{(2)}\big(A_m^q+S_m^q\big)
=A_n^q L^{(1)}L^{(2)}A_m^q=A_n^qL^{(2)}L^{(1)}A_m^q
\\
&=\big(A_n^q+S_n^q\big)L^{(2)}L^{(1)}A_m^q=L^{(2)}L^{(1)}A_m^q.
\end{align*}
Finally, by virtue of Lemma~\ref{PropMtr} the relations~\eqref{AqLLSq} mean exactly that $L$ and $L^\top$ are $q$-Manin matrices. \end{proof}

Theorem~\ref{ThqLax} implies that a Lax operator of $U_q(\mathfrak{gl}_n)$ type is a particular case of $q$-Manin matrix. Some properties of these Lax operators can be generalised to the case of $q$-Manin matrices. The $q$-determinant~\eqref{detq} arose as a natural generalisation of the determinant for the Lax operators of $U_q(\mathfrak{gl}_n)$ type, its properties were generalised for the case of $q$-Manin matrices in~\cite{qManin}.

Note that the fact that the $RLL$-relation~\eqref{RLL} is equivalent to the claim that $L$ and $L^\top$ are both $q$-Manin matrices can be proved in the same way as Proposition~\ref{propManinComm} (see~\cite{CFR, Manin88}). The approach considered here explains this fact in terms of left equivalence of idempotents, which will be applied in Section~\ref{secHecke}. This allows to explain why the Newton identities for the $q$-Manin matrices proved in~\cite{qManin} differs from the Newton identities for $L$-operators and $\wh R^q_{n-}$-Manin matrices deduced in~\cite{IO,IOP98,IOP99,PS}.

\begin{Rem}
 Replacement of $\wh R^q_n$ by the idempotent $A^q_n$ in the equation~\eqref{whRLL} leads to another relations for the entries of the matrix $L$. Though $L$ is also a $q$-Manin matrix in this case, the matrix $L^\top$ is not. For $n=2$ it is described in details in~\cite{GS}.
\end{Rem}

Decomposition of the operator $\wh R^q$ into dual idempotents described in Lemma~\ref{LemRdec} gives a~gene\-ral idea how to connect Lax operators with Manin matrices. It can be applied to some general class of $R$-matrices.

\subsection{Lax operators of Yangian type as Manin operators}
\label{secLOY}

The decomposition method described in Section~\ref{secUq} is generalised here to the case of the rational $R$-matrix. This gives an interpretation of the corresponding Lax matrices as a class of Manin operators.

Let $h\colon Y(\mathfrak{gl}_n)\to\gR$ be a homomorphism from the Yangian to some algebra $\gR$. It is defined by the image of the matrix $T(u)$. It has the form $L(u)=1+\sum_{r=1}^\infty\sum_{i,j=1}^n \ell^{(r)}_{ij}E_{i}{}^{j}u^{-r}$, where $\ell^{(r)}_{ij}=h(t^{(r)}_{ij})$, and satisfies the $RLL$-relation
\begin{align}
 R(u_1-u_2) L^{(1)}(u_1)L^{(2)}(u_2)=L^{(2)}(u_2)L^{(1)}(u_1) R(u_1-u_2), \label{RLLun}
\end{align}
where $R(u)=R_n(u)=u-P_n$. Conversely, any $n\times n$ matrix over $\gR$ which has this form and satisfies~\eqref{RLLun} defines a homomorphism $Y(\mathfrak{gl}_n)\to\gR$. These are Lax operators of Yangian type.

We consider a more general matrix $L(u)\in\gR((u^{-1}))\otimes\Hom(\CC^m,\CC^n)$ satisfying the $RLL$-relation
\begin{align}
 R_n(u_1-u_2) L^{(1)}(u_1)L^{(2)}(u_2)=L^{(2)}(u_2)L^{(1)}(u_1) R_m(u_1-u_2). \label{RLLu}
\end{align}
It could be the Lax operator of the $\mathfrak{gl}_n$ $XXX$-model (which is an image of $T(u)$ under some representation of $Y(\mathfrak{gl}_n)$ multiplied by a polynomial), the Lax operator of the $n$ particle Toda chain etc.

Note that the matrix~$R(u)$ satisfies the Yang--Baxter equation
\begin{align} \label{YBE}
 R^{(12)}(u_{12}) R^{(13)}(u_{13})R^{(23)}(u_{23})= R^{(23)}(u_{23}) R^{(13)}(u_{13}) R^{(12)}(u_{12})
\end{align}
and
\begin{align}
 R^{(21)}(u_{21})R^{(12)}(u_{12})=1-u_{12}^2, \label{RhbarUnitaryCnn}
\end{align}
where $u_{ij}=u_i-u_j$.

Define an operator on $\CC^n\big(\big(u^{-1}\big)\big)^{\otimes2}=\CC^n\otimes\CC^n\big(\big(u_1^{-1},u_2^{-1}\big)\big)$ by the formula
\begin{align}
 \wh R=\wh R_n:=P_{u_1,u_2}P_nR_n(u_{12})= P_{u_1,u_2}(u_{12}P_n-1), \label{whRhbar}
\end{align}
where $P_{u_1,u_2}$ is the operator permuting $u_1$ and $u_2$, that is $(P_{u_1,u_2}f)(u_1,u_2)=f(u_2,u_1)$. The relation~\eqref{RLLu} is equivalent to
\begin{align}
 \wh R_n L^{(1)}(u_1)L^{(2)}(u_2) =L^{(1)}(u_1)L^{(2)}(u_2) \wh R_m. \label{whRLLu}
\end{align}
The relation~\eqref{RhbarUnitaryCnn} takes the from
\begin{align}
 \wh R_n\wh R_n=1-u_{12}^2. \label{whR2}
\end{align}

\begin{Rem}
The operator~\eqref{whRhbar} satisfies the braid relation
\begin{align} \label{whRBR}
 \wh R^{(23)}\wh R^{(12)}\wh R^{(23)}=\wh R^{(12)}\wh R^{(23)}\wh R^{(12)}.
\end{align}
It is obtained from the Yang--Baxter equation~\eqref{YBE} via multiplication by $P_{u_2,u_3}P_{u_1,u_2}P_{u_2,u_3}=P_{u_1,u_2}P_{u_2,u_3}P_{u_1,u_2}$ and~\eqref{PBR} from the left. Due to the formula~\eqref{whR2} this operator gives a~representation of the group $\SSS_k$ for an arbitrary $k$ after some renormalization. Namely, the normalised rational $R$-matrix $\overline R(u)=(1-u)^{-1}R(u)$ satisfies $\overline R^{(21)}(u_{21})\overline R^{(12)}(u_{12})=1$ and the same Yang Baxter equation~\eqref{YBE}. Hence the operator $\wt R=P_{u_1,u_2}P_n\wt R(u_{12})$ satisfies $\wt R^2=1$ and the braid relation~\eqref{whRBR}. This implies that the map $\sigma_a\mapsto\wt R^{(a,a+1)}$ gives a representation of $\SSS_k$ on the space $(\CC^n)^{\otimes k}(u_1,\dots,u_k)$.
\end{Rem}

\begin{Prop}
The operators
\begin{align*}
\wh R^+=\wh R_{n+}:=\frac{1-u_{12}+\wh R_n}2,\qquad
\wh R^-=\wh R_{n-}:=\frac{1+u_{12}-\wh R_n}2 
\end{align*}
are orthogonal idempotents dual to each other:
\begin{align}
\wh R^++\wh R^-=1,\qquad
\big(\wh R^+\big)^2=\wh R^+, \qquad
\big(\wh R^-\big)^2=\wh R^-, \qquad
\wh R^+\wh R^-=\wh R^-\wh R^+=0. \label{whRhbarpmOI}
\end{align}
\end{Prop}

\begin{proof} The first of~\eqref{whRhbarpmOI} is obvious. By using~\eqref{whR2} and $\wh R f(u_1,u_2)=f(u_2,u_1)\wh R$ we obtain $\big(\wh R+1-u_{12}\big)\big(\wh R-1-u_{12}\big)=0$. This implies the rest of~\eqref{whRhbarpmOI}. \end{proof}

Now consider
\begin{align} \label{whA}
 \wh A_n:=\wh R_{n-}\cdot\frac{1}{u_{12}}=\frac{1+u_{12}^{-1}-P_{u_1,u_2}P_nR(u_{12})u_{12}^{-1}}{2}=\frac{1+u_{12}^{-1}+P_{u_1,u_2}\big(u_{12}^{-1}-P_n\big)}{2}.
\end{align}

\begin{Lem} 
The operator $\wh A_n$ is an idempotent acting on the space $\CC^n\big[u_1,u_1^{-1}\big]\otimes\CC^n\big[u_2,u_2^{-1}\big]$. It preserves the subspaces $\CC^n[u_1]\otimes\CC^n[u_2]$ and $\CC^n\big[u_1^{-1}\big]\otimes\CC^n\big[u_2^{-1}\big]$.
\end{Lem}

\begin{proof} Let us rewrite~\eqref{whA} in the form
\begin{align*}
 \wh A_n=\frac{1-P_{u_1,u_2}P_n}{2}+\frac{1}{2(u_1-u_2)}(1-P_{u_1,u_2}).
\end{align*}
The first term is obviously preserves all tree spaces. Let us prove that so does the second term. By acting by $(1-P_{u_1,u_2})$ on a Laurent polynomial $p(u_1,u_2)\in\CC\big[u_1,u_1^{-1},u_2,u_2^{-1}\big]$ we obtain the Laurent polynomial $q(u_1,u_2)=p(u_1,u_2)-p(u_2,u_1)$. Since $q(u_1,u_1)=0$ we have the formula $q(u_1,u_2)=(u_1-u_2)r(u_1,u_2)$ for some $r(u_1,u_2)\in\CC\big[u_1,u_1^{-1},u_2,u_2^{-1}\big]$. Hence $\frac{1}{2(u_1-u_2)}(1-P_{u_1,u_2})p(u_1,u_2)=\frac12r(u_1,u_2)$ is also a Laurent polynomial. If $p(u_1,u_2)\in\CC[u_1,u_2]$ then $r(u_1,u_2)\in\CC[u_1,u_2]$. Analogously, for a Laurent polynomial $p(u_1,u_2)\in\CC\big[u_1^{-1},u_2^{-1}\big]$ we obtain $q(u_1,u_2)=p(u_1,u_2)-p(u_2,u_1)=\big(u_1^{-1}-u_2^{-1}\big)r(u_1,u_2)$ for some $r(u_1,u_2)\in\CC\big[u_1^{-1},u_2^{-1}\big]$ and hence $\frac{1}{2(u_1-u_2)}(1-P_{u_1,u_2})p(u_1,u_2)
=-\frac{u_1^{-1}u_2^{-1}}{2}r(u_1,u_2)\in\CC\big[u_1^{-1},u_2^{-1}\big]$.

By multiplying $2\wh A_n=1+u_{12}^{-1}+u_{12}^{-1}\wh R$ with itself we obtain
\begin{align*}
4\wh A_n^2&=\big(1+u_{12}^{-1}\big)^2 +\big(1+u_{12}^{-1}\big)u_{12}^{-1}\wh R+
u_{12}^{-1}\wh R\big(1+u_{12}^{-1}\big)+u_{12}^{-1}\wh R u_{12}^{-1}\wh R
\\
&= 1+2u_{12}^{-1}+u_{12}^{-2} +\big(1+u_{12}^{-1}+1-u_{12}^{-1}\big)u_{12}^{-1}\wh R
-u_{12}^{-2}\wh R^2.
\end{align*}
By taking into account~\eqref{whR2} we obtain $4\wh A_n^2=2+2u_{12}^{-1} +2u_{12}^{-1}\wh R=4\wh A_n$. \end{proof}

The basis of $\CC^n\otimes\CC^n\big[u_1,u_1^{-1},u_2,u_2^{-1}\big]$ is $\big(e_i\otimes e_j u_1^ku_2^l\big)$, where $k,l\in\ZZ$, $i,j=1,\dots,n$, and the completion with respect to this basis is the space $\CC^n\otimes\CC^n\big[\big[u_1,u_1^{-1},u_2,u_2^{-1}\big]\big]$. Since $\wh A_n$ preserves the space $\CC^n\otimes\CC^n\big[u_1,u_1^{-1},u_2,u_2^{-1}\big]$ its matrix in this basis satisfies the column finiteness condition introduced in Section~\ref{secMO}. However it does not satisfies the row finiteness condition since it can not be extended to the completed tensor product space $\CC^n\otimes\CC^n\big[\big[u_1,u_1^{-1},u_2,u_2^{-1}\big]\big]$. Explicitly these can be seen from the formulae
\begin{align}
 &2\wh A_n\big(e_i\otimes e_j u_1^ku_2^l\big)=e_i\otimes e_j u_1^ku_2^l-e_j\otimes e_i u_1^ku_2^l+e_i\otimes e_j\frac{u_1^ku_2^l-u_1^lu_2^k}{u_1-u_2}, \label{Anbasis1} \\
 &\frac{u_1^ku_2^l-u_1^lu_2^k}{u_1-u_2}=
\begin{cases}
-\sum_{m=k}^{l-1}u_1^mu_2^{k+l-1-m}, &k<l, \\
0, &k=l, \\
\sum_{m=l}^{k-1}u_1^mu_2^{k+l-1-m}, &k>l.
\end{cases} \label{Anbasis2}
\end{align}
For instance, for any $k\ge1$ the vector $\wh A_n\big(e_i\otimes e_j u_1^ku_2^{1-k}\big)$ has a non-zero coefficient at the term $e_i\otimes e_j u_1^0u_2^0$.

Consider the topology of the space $V=\CC^n\big[u,u^{-1}\big]$ defined by the neighbourhoods of $0$ of the form $V_r=\big\{\sum_{k=-r}^{N}t_ku^k\mid N\ge-r,t_k\in\CC^n\big\}$. The completion of $V$ with respect to this topology is the space $\CC^n\big(\big(u^{-1}\big)\big)$. The corresponding completion of the space $\gR\otimes V$ is $\gR\whotimes V=\gR\big(\big(u^{-1}\big)\big)\otimes\CC^n$. The neighbourhoods $V_r\otimes V_s$ defines the topology of $V\otimes V$ which gives the completion $\CC^n\otimes\CC^n\big(\big(u_1^{-1},u_2^{-1}\big)\big)$. The operator $\wh A_n\in\End(V\otimes V)$ is continuous with respect to this topology. In particular, it means that $\wh A_n$ is extended to the completion $\CC^n\otimes\CC^n\big(\big(u_1^{-1},u_2^{-1}\big)\big)$.

Thus an $\big(\wh A_n,\wh A_m\big)$-Manin operator is an element $M$ of the space $\Hom\big(\CC^m\big[\big[u,u^{-1}\big]\big],\gR\whotimes V\big)=\Hom\big(\CC^m\big[\big[u,u^{-1}\big]\big],\gR\big(\big(u^{-1}\big)\big)\otimes\CC^n\big)$, satisfying
\begin{align*}
 \wh A_n M^{(1)}M^{(2)}=\wh A_n M^{(1)}M^{(2)}\wh A_m.
\end{align*}

 Note also that the operators $e^{a(\partial_{u_1}+\partial_{u_2})}=e^{a\partial_u}\otimes e^{a\partial_u}$ commute with $\wh A_n$. Hence if $M$ is an $\big(\wh A_n,\wh A_m\big)$-Manin operator then $e^{a\partial_u}Me^{b\partial_u}$ is also an $\big(\wh A_n,\wh A_m\big)$-Manin operator for any $a,b\in\CC$ (this follows from Proposition~\ref{PropPerm} generalised to the infinite-dimensional case).

\begin{Th} \label{ThwhA}
Let $L(u)\in\gR\big(\big(u^{-1}\big)\big)\otimes\Hom(\CC^m,\CC^n)$. The matrix $L(u)$ is an $\big(\wh A_n,\wh A_m\big)$-Manin operator iff it satisfies $RLL$-relation~\eqref{RLLu}.
\end{Th}

\begin{proof} Remind first that the relation~\eqref{RLLu} is equivalent to~\eqref{whRLLu}.
Let us multiply
\begin{align*}
\wh A_n L^{(1)}(u_1)L^{(2)}(u_2)=\wh A_n L^{(1)}(u_1)L^{(2)}(u_2)\wh A_m 
\end{align*}
by $4u_{12}$ from the right and substitute $2u_{12}\wh A_n=1+u_{12}+\wh R_n$, $2\wh A_m=1+u_{12}^{-1}+u_{12}^{-1}\wh R_m$. This gives the equivalent relation
\begin{gather}
 \big(u_{12}-u_{12}^{-1}\big)L^{(1)}(u_1)L^{(2)}(u_2)+u_{12}^{-1}\wh R_n L^{(1)}(u_1)L^{(2)}(u_2)\wh R_m\nonumber
 \\ \qquad
 {}=
 \big(1+u_{12}^{-1}\big)L^{(1)}(u_1)L^{(2)}(u_2)\wh R_m-\big(1+u_{12}^{-1}\big)\wh R_n L^{(1)}(u_1)L^{(2)}(u_2).
 \label{RLLR}
 \end{gather}
The left hand side of~\eqref{RLLR} does not change at the conjugation $T\mapsto u_{12}Tu_{12}^{-1}$ while the right hand side changes the sign. This means that~\eqref{RLLR} is valid iff the both hand sides vanish.
Due to~\eqref{whR2} the vanishing of each hand side of~\eqref{RLLR} is equivalent to~\eqref{whRLLu}. \end{proof}

Let $L(u)$ be an $\big(\wh A_n,\wh A_m\big)$-Manin operator. Then $M=L(u+a)e^{b \partial_u}$ is also an $\big(\wh A_n,\wh A_m\big)$-Manin operator for any $a,b\in\CC$. In particular, the $A_n$-Manin matrix $M=L(u)e^{-\partial_u}$ considered in Section~\ref{secMm} is an $\wh A_n$-Manin operator.

\begin{Rem} 
 One can renormalise the matrix $L(u)\in\gR\big(\big(u^{-1}\big)\big)\otimes\Hom(\CC^m,\CC^n)$ by multiplying it by a function in~$u$. Such renormalization does not violate the $RLL$-relation~\eqref{RLLu}. Hence one can suppose that $L(u)\in\gR[[u^{-1}]]\otimes\Hom(\CC^m,\CC^n)$ by multiplying $L(u)$ by $u^k$ for big enough~$k$ if needed. A matrix $L(u)\in\gR\big[\big[u^{-1}\big]\big]\otimes\Hom(\CC^m,\CC^n)$ satisfying $RLL$-relation is exactly an $\big(\big(\wh A_n\big)^-_{\rm res},\big(\wh A_m\big)^-_{\rm res}\big)$-Manin operator, where $\big(\wh A_n\big)^-_{\rm res}$ is the restriction of $\wh A_n$ to $\CC^n\otimes\CC^n\big[u_1^{-1},u_2^{-1}\big]$. We see from the formulae~\eqref{Anbasis1}, \eqref{Anbasis2} that $\big(\wh A_n\big)^-_{\rm res}$ satisfies the row finiteness condition. In~particular, we have the right quantum algebra $\gU_{(\wh A_n)^-_{\rm res}}$ and Theorem~\ref{ThwhA} implies that the Yan\-gian~$Y(\mathfrak{gl}_n)$ is a factor algebra of $\gU_{(\wh A_n)^-_{\rm res}}$.
\end{Rem}

\begin{Rem} 
 The facts described in this Section works also for a matrix $L(u)\in\gR((u))\otimes\Hom(\CC^m,\CC^n)$ since we can consider another completion of $\gR\big[u,u^{-1}\big]$. Again, one can suppose that $L(u)\in\gR[[u]]\otimes\Hom(\CC^m,\CC^n)$ making a renormalization if needed. A matrix $L(u)\in\gR[[u]]\otimes\Hom(\CC^m,\CC^n)$ satisfying $RLL$-relation is an $\big(\big(\wh A_n\big)^+_{\rm res},\big(\wh A_m\big)^+_{\rm res}\big)$-Manin operator, where $\big(\wh A_n\big)^+_{\rm res}$ is the restriction of $\wh A_n$ to $\CC^n\otimes\CC^n[u_1,u_2]$.
\end{Rem}

\section{Minors of Manin matrices}
\label{secMin}

The notion of minor (determinant of a submatrix) is an important tool in the classical matrix theory. It can be interpreted in terms of the Grassmann algebra as some coefficients. This gives minors of $(A_n,A_m)$-Manin matrices which we defined in Section~\ref{secMm} as column determinants of submatrices. The same can be done for $q$- and $(\wh q,\wh p)$-Manin matrices by considering the quadratic algebras $\Xi_A(\CC)$ for $A=A_n^q$ and $A_{\wh q}$ respectively. There is a dual notion of minors corresponding to the quadratic algebras $\gX_A(\CC)$. In the case of the polynomial algebras these dual minors are written via permanent.

For a general $(A,B)$-Manin matrix $M$ we define two types of minors corresponding to the homomorphisms $f_M\colon\gX_A(\CC)\to\gX_B(\gR)$ and $f^M\colon\Xi_B(\CC)\to\Xi_A(\gR)$. In fact these minors give the graded components of these homomorphisms. As usual minors they have a good behaviour at the multiplication of Manin matrices and under permutations of rows and columns. In future works we hope to find more properties of these minors by generalising the properties of the usual minors.

\subsection[The q-minors and permanents]
{The $\boldsymbol q$-minors and permanents}\label{secPerm}

First we remind that the $q$-determinant~\eqref{detq} is the coefficient of proportionality for the product of the $q$-Grassmann variables~\cite{qManin} (see also the formula~\eqref{phipsidet} for the $\wh q$-version). Namely, let $\psi_1,\dots,\psi_n$ be the generators of $\Xi_{A_n^q}(\CC)$, $M$ be an $n\times n$ matrix over an algebra $\gR$ and $\phi_j=\sum_{i=1}^n\psi_iM^i_j$, where $j=1,\dots,n$, then
\begin{align*} 
 \phi_1\phi_2\cdots\phi_n=\det_q(M)\psi_1\psi_2\cdots\psi_n.
\end{align*}

More generally, let $M$ be $n\times m$ matrix over $\gR$ and $\phi_j=\sum_{i=1}^n\psi_iM^i_j$, where $j=1,\dots,m$. For two $k$-tuples of indices $I=(i_1,\dots,i_k)$ and $J=(j_1,\dots,j_k)$ we denote by $M_{IJ}$ the $k\times k$ matrix over $\gR$ with entries
\begin{align*}
(M_{IJ})_{ab}=M^{i_a}_{j_b},\qquad a,b=1,\dots,k,
\end{align*}
where we suppose $1\le i_a\le n$ and $1\le j_b\le m$. If $i_1<\dots<i_k$ and $j_1<\dots<j_k$ then $M_{IJ}$ is a $k\times k$ submatrix of $M$ and $\det_q(M_{IJ})$ is a $q$-analogue of minor defined in Section~\ref{secMm}. The $q$-determinants $\det_q(M_{IJ})$ are the coefficients in the decomposition
\begin{align} \label{phipsiMin}
 \phi_{j_1}\phi_{j_2}\cdots\phi_{j_k}=\sum_{I=(i_1<\dots<i_k)}\det_q(M_{IJ}) \psi_{i_1}\psi_{i_2}\cdots\psi_{i_k},
\end{align}
where the sum is taken over $k$-tuples $I=(i_1,\dots,i_k)$ such that $1\le i_1<\dots<i_k\le n$.

If $M$ is a $q$-Manin matrix then $\phi_1,\dots,\phi_m$ are also $q$-anticommuting and hence the elements $\phi_{j_1}\phi_{j_2}\cdots\phi_{j_k}$ with $1\le j_1<\dots<j_k\le m$ span the subalgebra of $\Xi_{A_n^q}(\gR)$ generated by $\phi_1,\dots,\phi_m$. This means that the $q$-determinants of submatrices of $M$ are enough to describe the decompositions~\eqref{phipsiMin} in this case.

A dual notion to column determinant is row permanent. Recall that the {\it permanent} of an~$n\times n$ matrix $M$ is
\begin{align} \label{rowperm}
 \perm(M)=\sum_{\sigma\in \SSS_n}M^{1}_{\sigma(1)}\cdots M^{n}_{\sigma(n)}.
\end{align}
This is the same expression as for the determinant but without the factors $(-1)^\sigma$. If $M$ is over non-commutative algebra the factors in this expression should be placed in a certain way. The~formula~\eqref{rowperm} defines the {\it row permanent} of $M$ (see~\cite{CF,CFR}). It is invariant under a permutation of columns: $\perm(_\tau M)=\perm(M)$. Contrary to the determinant it does not change sign, so, in~par\-ticular, the permanent of a matrix with coinciding rows can be non-zero. If $M$ is an~$A_n$-Manin matrix, then the permanent is invariant under a permutation of rows: $\perm(^\tau M)=\perm(M)$.

Now consider the analogues decompositions of products of $y^i=\sum_{j=1}^mM^i_jx^j$, where $x^1,\dots,x^m$ are generators of $\gX_{A_m}(\CC)$ (for simplicity we consider the case $q=1$). Define an action of~the group $\SSS_k$ on $k$-tuples by the formula $\sigma(j_1,\dots,j_k)=\big(j_{\sigma^{-1}(1)},\dots,j_{\sigma^{-1}(k)}\big)$. For any tuple $I=(i_1,\dots,i_k)$ we have (cf.~\cite{CFR})
\begin{align*}
 y^{i_1}\cdots y^{i_k}&=\sum_{l_1,\dots,l_k=1}^mM^{i_1}_{l_1}\cdots M^{i_k}_{l_k}x^{l_1}\cdots x^{l_k}
=\sum_{J=(j_1\le\cdots\le j_k)}x^{j_1}\cdots x^{j_k}
 \sum_{(l_1,\dots,l_k)\in\SSS_kJ}M^{i_1}_{l_1}\cdots M^{i_k}_{l_k}
 \\
 &=\sum_{J=(j_1\le\cdots\le j_k)}x^{j_1}\cdots x^{j_k}\frac1{|(\SSS_k)_J|}\perm(M_{IJ}),
\end{align*}
where $|(\SSS_k)_J|$ is the order of the stabiliser subgroup $(\SSS_k)_J=\{\sigma\in \SSS_k\mid\sigma J=J\}$. For a fixed $J=(j_1,\dots,j_k)$ let $\nu_j=|\{a\mid j_a=j\}|$, then $|(\SSS_k)_J|=\nu_1!\nu_2!\cdots \nu_m!$. The normalised permanents $\frac1{|(\SSS_k)_J|}\perm(M_{IJ})=\frac1{\nu_1!\nu_2!\cdots \nu_m!}\perm(M_{IJ})$ are analogues of the minors for the polynomial algebra $\gX_{A_m}(\CC)=\CC\big[x^1,\dots,x^m\big]$.

For a general idempotent $A$ we will define minors of two types and they will be coefficients of decomposition of $y^{i_1}\cdots y^{i_k}$ and $\phi_{i_1}\cdots\phi_{i_k}$ into sums of $x^{j_1}\cdots x^{j_k}$ and $\psi_{j_1}\cdots\psi_{j_k}$ respectively (up to a factor).

\subsection{Dual quadratic algebras and their pairings}
\label{secDualQA}

Consider the question of decomposition coefficients in general settings. Let $V$ and $W$ be vector spaces with a non-degenerate pairing $\la\cdot,\cdot\ra\colon V\times W\to\CC$. Let them have dual bases $(v^i)$ and~$(w_i)$, $\la v^i,w_j\ra=\delta^i_j$. Let $\gR$ be an algebra, then $\gR\otimes V$ is a free $\gR$-module with the basis~$(v^i)$. Consider a~decomposition $\alpha=\sum_i \alpha_iv^i$ of an element $\alpha\in\gR\otimes V$ in the basis $(v^i)$. The coefficients $\alpha_i\in\gR$ can be calculated via the pairing: $\alpha_i=\la\alpha,w_i\ra$, where $\la r\otimes v,w \ra:=r\cdot \la v,w\ra$ $\forall r\in\gR,\,v\in V$, $w\in W$.

More generally, let $A\in\End(V)$ and $A^*\in\End(W)$ be adjoint operators: $\la A v, w\ra=\la v,A^* w\ra$ $\forall v\in V,\, w\in W$. They act on basis vectors as $Av^i=\sum_jA^i_jv^j$, $A^* w_j=\sum_i A_j^iw_i$, $A^i_j\in\CC$ (the both sums are finite). Suppose they are idempotents: $\sum_jA^i_jA^j_k=A^i_k$. Consider the subspaces $\overline V=AV$ and $\overline W=A^*W$. These subspaces are spanned by the vectors $\overline v^i=Av^i$ and $\overline w_i=A^*w_i$ respectively. We have $\la\overline v^i,\overline w_j\ra=A^i_j$. A decomposition $\alpha=\sum_i\alpha_i\overline v^i$ of an element $\alpha\in\gR\otimes \overline V$ is not unique since $\overline v^i$ are not linearly independent in general. However we can fix the coefficients~$\alpha_i$ by imposing some ``symmetry'' conditions.

\begin{Prop} \label{propVAW}
\quad
\begin{enumerate}\itemsep=0pt
 \item[$1.$] The restriction of the non-degenerate pairing $\la\cdot,\cdot\ra\colon V\times W\to\CC$ to $\overline V\times\overline W$ is also non-degenerate.

 \item[$2.$] For any $\alpha\in\gR\otimes\overline V$ there is a unique $($finite$)$ sequence $(\alpha_i)$ such that $\alpha=\sum_i\alpha_i\overline v^i$ and $\sum_{i}A_j^i\alpha_i=\alpha_j$. The coefficients $\alpha_i\in\gR$ fixed by these conditions can be found by the formula $\alpha_i=\la\alpha,\overline w_i\ra$.
 \end{enumerate}
\end{Prop}

\begin{proof} (1) If $w\in\overline W$ and $\la v,w\ra=0$ for all $v\in\overline V$, then for any $v\in V$ we have $\la v,w\ra=\la v,A^*w\ra=\la Av,w\ra=0$ and hence $w=0$. (2) Since $\alpha\in\gR\otimes\overline V$ we have $A\alpha=\alpha$, where the action of $A$ is extended on $\gR\otimes V$ by the formula $A(r\otimes v)=r\otimes Av$. Let $\alpha_i=\la\alpha,\overline w_i\ra$, then $\alpha_i=\la\alpha,A^*w_i\ra=\la A\alpha,w_i\ra=\la\alpha,w_i\ra$. Hence $\sum_{i}A_j^i\alpha_i=\sum_{i}A_j^i\la\alpha,w_i\ra=\la\alpha,\overline w_j\ra=\alpha_j$. Moreover, $\la\alpha-\sum_i\alpha_i\overline v^i,w_j\ra=\alpha_j-\sum_{i}A_j^i\alpha_i=0$, so that $\alpha-\sum_i\alpha_i\overline v^i=0$. To show the uniqueness suppose that $\alpha=\sum_i\alpha_i\overline v^i$ for some $\alpha_i\in\gR$ such that $\sum_{i}A_j^i\alpha_i=\alpha_j$, then we have $\la\alpha,\overline w_j\ra=\sum_i\alpha_i A^i_j=\alpha_j$. \end{proof}

Let $A\in\End(\CC^n\otimes\CC^n)$ be an idempotent. Consider the corresponding quadratic alge\-bra~$\gX_A(\CC)$. We need to define a vector space ${\gX}^*_A(\CC)$ and a non-degenerate pairing with $\gX_A(\CC)$. This implies that ${\gX}^*_A(\CC)$ is a graded vector space and the pairing respects the grading.

Define the algebra ${\gX}^*_A(\CC)$ as the graded algebra generated by the elements $x_1,\dots,x_n$ with the ``dual'' quadratic commutation relations: $\sum_{i,j=1}^n x_i x_j A^{ij}_{kl}=0$, that is
\begin{align}
 (X^*\otimes X^*)A=0, \label{Axx0_dual}
\end{align}
where $X^*=\sum_{i=1}^nx_i e^i$. We have the algebra isomorphisms ${\gX}^*_A(\CC)=\Xi_{1-A}(\CC)=\gX_{A^\top}(\CC)$. Let us introduce a pairing $\la\cdot,\cdot\ra\colon{\gX}_A(\CC)\times{\gX}^*_A(\CC)\to\CC$ respecting the grading, i.e., the product $\la x^{i_1}\cdots x^{i_l},x_{j_1}\cdots x_{j_k}\ra$ vanishes if $k\ne l$. The pairing for the elements of the degree $0$ and $1$ are defined as follows: $\la1,1\ra=1$ and $\la x^i,x_j\ra=\delta^i_j$. For the higher degree elements the pairing has the form
\begin{align}
 \la x^{i_1}\cdots x^{i_k},x_{j_1}\cdots x_{j_k}\ra=S^{i_1\dots i_k}_{j_1 \dots j_k}, \label{XXXX_Sij}
\end{align}
for some $S^{i_1\dots i_k}_{j_1\dots j_k}\in\CC$. The commutation relations~\eqref{Axx0} and \eqref{Axx0_dual} implies that
\begin{gather}
 \sum_{i_a,i_{a+1}=1}^nA^{lm}_{i_ai_{a+1}}S^{i_1\dots i_k}_{j_1\dots j_k}=0, \qquad
 \sum_{j_s,j_{s+1}=1}^nS^{i_1\dots i_k}_{j_1\dots j_k}A^{j_aj_{a+1}}_{lm}=0, \label{ASSA0}
\end{gather}
for all $a=1,\dots,k-1$ and $l,m=1,\dots,n$. Under these conditions the formula~\eqref{XXXX_Sij} correctly defines a pairing $\la\cdot,\cdot\ra\colon\gX_A(\CC)\times\gX^*_A(\CC)\to\CC$ that respects the grading.

\begin{Rem} 
 In contrast to the situation described in Remark~\ref{remFunctorDual} one can not correctly define an operation on quadratic algebras by mapping an algebra $\gX_A(\CC)$ to $\gX^*_A(\CC)$. It happens since the equality $\gX_A(\CC)=\gX_{A'}(\CC)$ does not imply an isomorphism of $\gX^*_A(\CC)$ and $\gX^*_{A'}(\CC)$ (see Section~\ref{secMinEq} for details).
\end{Rem}

To write the formulae~\eqref{XXXX_Sij}, \eqref{ASSA0} in a matrix form we introduce some conventions. Let $V$ and $W$ be vector spaces and $W^*=\Hom(W,\CC)$. The product $\pi\xi$ of elements $\pi\in V$ and $\xi\in W^*$ is usually identified with the linear operator $W\to V$ acting as $(\pi\xi)(w)=\xi(w)\cdot\pi$, $w\in W$. For example, $e_ie^j=E_{i}{}^{j}\in\End(\CC^n)$. More generally, $e_{i_1\dots i_k}e^{j_1\dots j_k}=E_{i_1}{}^{j_1}\otimes\cdots\otimes E_{i_k}{}^{j_k}$, where $e^{j_1\dots j_k}=e^{j_1}\otimes\cdots\otimes e^{j_k}$ and $e_{i_1\dots i_k}=e_{i_1}\otimes\cdots\otimes e_{i_k}$.

Let us introduce the notation $\la\alpha,\beta\ra$ for $\alpha\in\gX_A(\CC)\otimes V$ and $\beta\in\gX^*_A(\CC)\otimes W^*$. The pairing acts on the first tensor factors while in the second factor the elements are multiplied as above: $\la u\pi,v\xi\ra=\la u,v\ra\pi\xi\in\Hom(W,V)$, where $u\in\gX_A(\CC)$, $v\in\gX^*(\CC)$, $\pi\in V$ and $\xi\in W^*$. In~particular,
\[
\bigg\la\sum\limits_{i_1,\dots,i_k} u_{i_1\dots i_k}e_{i_1\dots i_k},\sum\limits_{j_1,\dots,j_k} v_{j_1 \dots j_k} e^{j_1 \dots j_k}\bigg\ra=\sum\limits_{i_1,\dots,i_k,j_1\dots,j_k}\la u_{i_1 \dots i_k}, v_{j_1 \dots j_k}\ra E_{i_1}{}^{j_1}\otimes\cdots\otimes E_{i_k}{}^{j_k}.
\]
By using this notation we can write the pairing of degree $1$ elements in matrix form:
\begin{align*}
 \la X,X^*\ra=\la x^ie_i,x_je^j\ra=\delta^i_jE_{i}{}^{j}=1. 
\end{align*}

Consider the operators $S_{(k)}=\sum_{i_1,\dots,i_k,j_1,\dots,j_k}S^{i_1\dots i_k}_{j_1\dots j_k}E_{i_1}{}^{j_1}\otimes\cdots\otimes E_{i_k}{}^{j_k}\in\End\big((\CC^n)^{\otimes k}\big)$. For $k=1$ the operator $S_{(1)}$ is the $n\times n$ unit matrix. The formula~\eqref{XXXX_Sij} in matrix form reads
\begin{align}
 \la X\otimes\cdots\otimes X,X^*\otimes\cdots\otimes X^*\ra=S_{(k)} \label{XXXX_Sm}
\end{align}
(here and below dots mean that we have $k$ tensor factors).
The conditions~\eqref{ASSA0} is written as
\begin{align}
 A^{(a,a+1)}S_{(k)}=S_{(k)}A^{(a,a+1)}=0, \label{ASSAk0}
\end{align}
$a=1,\dots,k-1$. In terms of $P=1-2A$ we can rewrite~\eqref{ASSAk0} in the form
\begin{align*}
 P^{(a,a+1)}S_{(k)}=S_{(k)}P^{(a,a+1)}=S_{(k)}.
\end{align*}

Analogously, one can define the algebra $\Xi^*_A(\CC)$ generated by the elements $\psi^1,\dots,\psi^n$ over $\CC$ with the commutation relations $\psi^i\psi^j=\sum_{k,l}A^{ij}_{kl}\psi^k\psi^l$, i.e.,
\begin{align*}
 (1-A)(\Psi^*\otimes\Psi^*)=0,
\end{align*}
where $\Psi^*=\sum\limits_{i=1}^n\psi^i e_i$. The pairing $\la\cdot,\cdot\ra\colon{\Xi}_A(\CC)\times{\Xi}^*_A(\CC)\to\CC$ is defined by the formula
\begin{align}
 \la\psi^{i_1}\cdots\psi^{i_l},\psi_{j_1}\cdots \psi_{j_k}\ra=\delta_{kl}A^{i_1\dots i_k}_{j_1\dots j_k}, \label{XXXXAij}
\end{align}
where $A^{i_1\dots i_k}_{j_1\dots j_k}\in\CC$, $A^i_j=\delta^i_j$, $A^\varnothing_\varnothing=1$. The formula~\eqref{XXXXAij} for $k=l$ can be written as
\begin{align}
 \la\Psi^*\otimes\cdots\otimes\Psi^*,\Psi\otimes\cdots\otimes\Psi\ra=A_{(k)}, \label{XXXXAk}
\end{align}
where $A_{(k)}=\sum\limits_{i_1,\dots,i_k,j_1,\dots,j_k}A^{i_1\dots i_k}_{j_1\dots j_k}E_{i_1}{}^{j_1}\otimes\cdots\otimes E_{i_k}{}^{j_k}$, $A_{(1)}=1$. Again, we should require
\begin{align*}
 S^{(a,a+1)}A_{(k)}=A_{(k)}S^{(a,a+1)}=0, 
\end{align*}
where $a=1,\dots,k-1$, $S=1-A$, or, equivalently,
\begin{align*}
 P^{(a,a+1)}A_{(k)}=A_{(k)}P^{(a,a+1)}=-A_{(k)}.
\end{align*}

\begin{Lem} \label{lemXTX}
Let $V$ and $W$ be vector spaces. For any operators $T\in\Hom\big((\CC^n)^{\otimes k},V\big)$ and $\wt T\in\Hom\big(W,(\CC^n)^{\otimes k}\big)$ we have
\begin{gather}
 \big\la T(X\otimes\cdots\otimes X),(X^*\otimes\cdots\otimes X^*)\wt T\big\ra=T S_{(k)}\wt T, \label{XTTX}
 \\
 \big\la T(\Psi^*\otimes\cdots\otimes\Psi^*),(\Psi\otimes\cdots\otimes\Psi)\wt T\big\ra=TA_{(k)}\wt T \label{PsiTTPsi}
\end{gather}
$($we understand $T(X\otimes\cdots\otimes X)$, $(X^*\otimes\cdots\otimes X^*)\wt T$, $T(\Psi^*\otimes\cdots\otimes\Psi^*)$ and $(\Psi\otimes\cdots\otimes\Psi)\wt T$ as elements of $\gX_A(\CC)\otimes V$, $\gX^*_A(\CC)\otimes W^*$, $\Xi^*_A(\CC)\otimes V$ and $\Xi_A(\CC)\otimes W^*$ respectively$)$.
\end{Lem}

\begin{proof} By substituting $T(X\otimes\cdots\otimes X)=\sum\limits_{i_1,\dots,i_k}x^{i_1}\cdots x^{i_k}Te_{i_1 \dots i_k}$ and $(X^*\otimes\cdots\otimes X^*)\wt T=\sum\limits_{j_1,\dots,j_k}x_{j_1}\cdots x_{j_k}e^{j_1 \dots j_k}\wt T$ we derive
\begin{gather*}
 \big\la T(X\otimes\cdots\otimes X),(X^*\otimes\cdots\otimes X^*)\wt T\big\ra
 \\ \qquad
{} =\sum_{i_1,\dots,i_k,j_1,\dots,j_k}\la x^{i_1}\otimes\cdots\otimes x^{i_k},x_{j_1}\otimes\cdots\otimes x_{j_k}\ra
 Te_{i_1 \dots i_k}e^{j_1 \dots j_k}\wt T
 \\ \qquad
 {}=\sum_{i_1,\dots,i_k,j_1,\dots,j_k}S^{i_1\dots i_k}_{j_1\dots j_k}T(E_{i_1}{}^{j_1}\otimes\cdots\otimes E_{i_k}{}^{j_k})\wt T=T S_{(k)}\wt T.
\end{gather*}
The formula~\eqref{PsiTTPsi} are proved in a similar way (it can be considered as the formula~\eqref{XTTX} for~the algebras $\Xi_A(\CC)={\gX}^*_{1-A}(\CC)$ and ${\Xi}^*_A(\CC)=\gX_{1-A}(\CC)$).
\end{proof}

\begin{Cor} \label{corTXXk}
 Let $T\in\Hom\big((\CC^n)^{\otimes k},V\big)$ and $\wt T\in\Hom\big(W,(\CC^n)^{\otimes k}\big)$.
\begin{itemize}\itemsep=0pt
 \item[$(a)$] If $T(X\otimes\cdots\otimes X)=0$ then $TS_{(k)}=0$.
 \item[$(b)$] If $(X^*\otimes\cdots\otimes X^*)\wt T=0$ then $S_{(k)}\wt T=0$.
 \item[$(c)$] If $(\Psi\otimes\cdots\otimes\Psi)\wt T=0$ then $A_{(k)}\wt T=0$.
 \item[$(d)$] If $T(\Psi^*\otimes\cdots\otimes\Psi^*)=0$ then $TA_{(k)}=0$.
\end{itemize}
In particular, these statements are valid for $T,\wt T\in\End\big((\CC^n)^{\otimes k}\big)$ and for $T\in\big((\CC^n)^{\otimes k}\big)^*$, $\wt T\in(\CC^n)^{\otimes k}$.
\end{Cor}

\begin{proof} The statement $(a)$ are obtained from~\eqref{XTTX} in the case $W=(\CC^n)^{\otimes k}$, $\wt T=1\in\End\big((\CC^n)^{\otimes k}\big)$. For $V=(\CC^n)^{\otimes k}$, $T=1\in\End\big((\CC^n)^{\otimes k}\big)$ we obtain $(b)$. The statements~$(c)$ and $(d)$ are obtained from~\eqref{PsiTTPsi} in the same way. In the cases $V=W=(\CC^n)^{\otimes k}$ and $V=W=\CC$ we obtain $\Hom\big((\CC^n)^{\otimes k},V\big)=\Hom\big(W,(\CC^n)^{\otimes k}\big)=\End\big((\CC^n)^{\otimes k})\big)$ and $\Hom\big((\CC^n)^{\otimes k},V\big)=\big((\CC^n)^{\otimes k}\big)^*$, $\Hom\big(W,(\CC^n)^{\otimes k}\big)=(\CC^n)^{\otimes k}$ respectively. \end{proof}

\subsection{Pairing operators}
\label{secPO}

Let $A\in\End(\CC^n\otimes\CC^n)$ be an idempotent and $S=1-A$. We formulate some conditions on the operators~\eqref{XXXX_Sm} and \eqref{XXXXAk} that guarantee non-degeneracy of the corresponding pairings.

\begin{Def}
 Operators $S_{(k)},\,A_{(k)}\in\End\big((\CC^n)^{\otimes k}\big)$ are called {\it pairing operators} (for the idempotent $A$) if they satisfy the following conditions:
\begin{gather}
 A^{(a,a+1)}S_{(k)}=S_{(k)}A^{(a,a+1)}=0, \label{PSSPS}
 \\
 (X^*\otimes\cdots\otimes X^*)S_{(k)}=(X^*\otimes\cdots\otimes X^*), \label{SkXXXXXX1}
 \\
 S_{(k)}(X\otimes\cdots\otimes X)=(X\otimes\cdots\otimes X), \label{SkXXXXXX2}
 \\
 S^{(a,a+1)}A_{(k)}=A_{(k)}S^{(a,a+1)}=0, \label{PAAPA}
 \\
 (\Psi\otimes\cdots\otimes\Psi)A_{(k)}=(\Psi\otimes\cdots\otimes\Psi), \label{AkXXXXXX1}
 \\
 A_{(k)}(\Psi^*\otimes\cdots\otimes\Psi^*)=(\Psi^*\otimes\cdots\otimes\Psi^*). \label{AkXXXXXX2}
\end{gather}
We call them the {\it $k$-th $S$-operator} and the {\it $k$-th $A$-operator}.
\end{Def}

The conditions~\eqref{SkXXXXXX2} and \eqref{AkXXXXXX1} for $k=1$ implies that $S_{(1)}=A_{(1)}=1$. For $k=2$ the equations~\eqref{PSSPS}--\eqref{SkXXXXXX2} and \eqref{PAAPA}--\eqref{AkXXXXXX2} have the solutions $S_{(k)}=S$ and $S_{(k)}=A$ respectively. Let us prove the uniqueness of solutions of these equations for any $k$ (we do not prove their existence for $k>2$ in a general case).

\begin{Prop} \label{propUnique}
 The pairing operators are unique. In particular, $S_{(2)}=S$ and $A_{(2)}=A$.
\end{Prop}

\begin{proof} Let $S'_{(k)}$ and $\bar S_{(k)}$ be $k$-th $S$-operators for the same idempotent $A$. By applying the part~$(a)$ of Corollary~\ref{corTXXk} for $S_{(k)}=S'_{(k)}$, $T=1-\bar S_{(k)}$ and the part~$(b)$ of Corollary~\ref{corTXXk} for $S_{(k)}=\bar S_{(k)}$, $T=1-S'_{(k)}$ we obtain $S'_{(k)}=\bar S_{(k)}S'_{(k)}$ and $\bar S_{(k)}=\bar S_{(k)}S'_{(k)}$. Hence we obtain $S'_{(k)}=\bar S_{(k)}$. The uniqueness of the $A$-operators follows similarly from the parts $(c)$ and $(d)$ of~Corollary~\ref{corTXXk}. \end{proof}

\begin{Prop}
 Pairing operators $S_{(k)}$ and $A_{(k)}$ are idempotents; they are orthogonal for $k\ge2$:
\begin{gather}
 S_{(k)}^2=S_{(k)}, \qquad A_{(k)}^2=A_{(k)}, \qquad S_{(k)}A_{(k)}=A_{(k)}S_{(k)}=0. \label{SkAk0}
 \\
 (X^*\otimes\cdots\otimes X^*)A_{(k)}=0, \qquad A_{(k)}(X\otimes\cdots\otimes X)=0, \label{AkX0}
 \\
 (\Psi\otimes\cdots\otimes\Psi)S_{(k)}=0, \qquad S_{(k)}(\Psi^*\otimes\cdots\otimes\Psi^*)=0. \label{SkX0}
\end{gather}
\end{Prop}

\begin{proof} Due to Corollary~\ref{corTXXk} the formula~\eqref{SkXXXXXX2} gives $S_{(k)}^2=S_{(k)}$, while the formula~\eqref{AkXXXXXX1} implies $A_{(k)}^2=A_{(k)}$. Further, from~\eqref{PSSPS} and \eqref{PAAPA} we derive $S_{(k)}A_{(k)}=S_{(k)}P^{(a,a+1)}A_{(k)}=-S_{(k)}A_{(k)}$ and $A_{(k)}S_{(k)}=A_{(k)}P^{(a,a+1)}S_{(k)}=-A_{(k)}S_{(k)}$. The formulae~\eqref{AkX0} and \eqref{SkX0} follows from the orthogonality and~\eqref{SkXXXXXX1},~\eqref{SkXXXXXX2},~\eqref{AkXXXXXX1}, \eqref{AkXXXXXX2}. For example, the relation~\eqref{SkXXXXXX2} implies $A_{(k)}(X\otimes\cdots\otimes X)=A_{(k)}S_{(k)}(X\otimes\cdots\otimes X)=0$. \end{proof}

Below we suppose that $S_{(k)}$ and $A_{(k)}$ are pairing operators for an idempotent $A$.
Note that in general the sum $S_{(k)}+A_{(k)}$ does not coincide with the identity operator.

The following formulae generalise the equalities~\eqref{SkAk0}. They are proved in the same way.

\begin{Prop} \label{PropASkl}
If $\ell\le k$ then
\begin{gather*}
 S_{(\ell)}^{(a,\dots,a+\ell-1)}S_{(k)}=S_{(k)}S_{(\ell)}^{(a,\dots,a+\ell-1)}=S_{(k)},\quad
 A_{(\ell)}^{(a,\dots,a+\ell-1)}A_{(k)}=A_{(k)}A_{(\ell)}^{(a,\dots,a+\ell-1)}=A_{(k)}, \\
 A_{(\ell)}^{(a,\dots,a+\ell-1)}S_{(k)}=S_{(k)}A_{(\ell)}^{(a,\dots,a+\ell-1)}=0, \qquad
 S_{(\ell)}^{(a,\dots,a+\ell-1)}A_{(k)}=A_{(k)}S_{(\ell)}^{(a,\dots,a+\ell-1)}=0,
\end{gather*}
where $a=1,\dots,k-\ell+1$ and we used the notation $T^{(a,\dots,a+\ell-1)}=\id^{\otimes(a-1)}\otimes T\otimes\id^{\otimes(k-\ell-a+1)}$ for $T\in\End\big((\CC^n)^{\otimes\ell}\big)$.
\end{Prop}

Since $\Xi_A(\CC)={\gX}^*_S(\CC)$ and ${\Xi}^*_A(\CC)=\gX_S(\CC)$ the conditions~\eqref{PSSPS}, \eqref{SkXXXXXX1}, \eqref{SkXXXXXX2} interchange with the conditions~\eqref{PAAPA}, \eqref{AkXXXXXX1}, \eqref{AkXXXXXX2} under the substitutions
\begin{gather}
A\leftrightarrow S, \qquad
P\leftrightarrow -P,\qquad
A_{(k)}\leftrightarrow S_{(k)}, \qquad
X\leftrightarrow\Psi^*, \qquad
X^*\leftrightarrow\Psi. \label{ASinter}
\end{gather}
Thus if $S_{(k)}$ and $A_{(k)}$ are pairing operators for $A$ then $A_{(k)}$ and $S_{(k)}$ are pairing operators for $S=1-A$.

The following property of the pairing operators shows their role for the quadratic algebras. Recall that we have identifications $\gX_A(\CC)_1=(\CC^n)^*$ and $\Xi_A(\CC)_1=\CC^n$. Let us identify the higher graded components $\gX_A(\CC)_k$ and $\Xi_A(\CC)_k$ with subspaces of $\big((\CC^n)^{\otimes k}\big)^*$ and $\big((\CC^n)^{\otimes k}\big)$ by using the idempotents $S_{(k)}$ and $A_{(k)}$ respectively.

\begin{Prop} \label{propIsom}
We have the following isomorphisms of vector spaces:
\begin{gather}
 \big((\CC^n)^{\otimes k}\big)^*S_{(k)}\cong\gX_A(\CC)_k, \qquad
 \xi\mapsto\xi(X\otimes\cdots\otimes X), \nonumber
 \\
 S_{(k)}\big((\CC^n)^{\otimes k}\big)\cong\gX^*_A(\CC)_k, \qquad\
 \pi\mapsto(X^*\otimes\cdots\otimes X^*)\pi,\nonumber
 \\
 A_{(k)}\big((\CC^n)^{\otimes k}\big)\cong\Xi_A(\CC)_k, \qquad\
 \pi\mapsto(\Psi\otimes\cdots\otimes\Psi)\pi, \nonumber
 \\
 \big((\CC^n)^{\otimes k}\big)^*A_{(k)}\cong\Xi^*_A(\CC)_k, \qquad \xi\mapsto\xi(\Psi^*\otimes\cdots\otimes\Psi^*).\label{gXk}
\end{gather}
The pairings $\la\cdot,\cdot\ra\colon\gX_A(\CC)_k\times\gX^*_A(\CC)_k\to\CC$ and $\la\cdot,\cdot\ra\colon\Xi^*_A(\CC)_k\times\Xi_A(\CC)_k\to\CC$ correspond to the multiplication $(\xi,\pi)\mapsto\xi\pi\in\CC$.
\end{Prop}

\begin{proof} All these maps are linear. Consider the map~\eqref{gXk}. Due to~\eqref{SkXXXXXX2} the image of the covector $\xi=e^{i_1 \dots i_k}S_{(k)}$ is $e^{i_1 \dots i_k}S_{(k)}(X\otimes\cdots\otimes X)=e^{i_1 \dots i_k}(X\otimes\cdots\otimes X)=x^{i_1}\cdots x^{i_k}$. Hence this map is surjective. Let us check injectivity. If $\xi(X\otimes\cdots\otimes X)=0$ then by using Corollary~\ref{corTXXk} for $T=\xi$ we obtain $\xi S_{(k)}=0$. Since $\xi\in\big((\CC^n)^{\otimes k}\big)^*S_{(k)}$ we have $\xi=\xi S_{(k)}=0$. The bijectivity of other maps are proved in the same way. The last statement follows from Lemma~\ref{lemXTX}, for example, $\big\la\xi(X\otimes\cdots\otimes X),(X^*\otimes\cdots\otimes X^*)\pi\big\ra=\xi S_{(k)}\pi=\xi\pi$. \end{proof}

Let $d_k=\rk S_{(k)}$ and $r_k=\rk A_{(k)}$. Then Proposition~\ref{propIsom} implies
\begin{align} \label{dkrk}
 \dim\gX_A(\CC)_k=\dim\gX^*_A(\CC)_k=d_k, \qquad \dim\Xi_A(\CC)_k=\dim\Xi^*_A(\CC)_k=r_k.
\end{align}

\begin{Prop}
 Let $S_{(k)}$ and $A_{(k)}$ be the pairing operators for the idempotent $A$. Then the pairings $\la\cdot,\cdot\ra\colon\gX_A(\CC)\times\gX^*_A(\CC)\to\CC$ and $\la\cdot,\cdot\ra\colon\Xi^*_A(\CC)\times\Xi_A(\CC)\to\CC$ defined by the formulae~\eqref{XXXX_Sm} and \eqref{XXXXAk} are non-degenerate.
\end{Prop}

\begin{proof} Since the pairing $\la\cdot,\cdot\ra\colon\gX_A(\CC)\times\gX^*_A(\CC)\to\CC$ respects the grading its non-degeneracy is equivalent to non-degeneracy of the restricted pairings $\la\cdot,\cdot\ra\colon\gX_A(\CC)_k\times\gX^*_A(\CC)_k\to\CC$. Due to Proposition~\ref{propIsom} it is enough to prove that the pairings $\big((\CC^n)^{\otimes k}\big)^*S_{(k)}\times S_{(k)}\big((\CC^n)^{\otimes k}\big)\to\CC$ given by $(\xi,\pi)\mapsto\xi\pi$ are non-degenerate, but this fact follows from the part (1) of Proposition~\ref{propVAW}. The non-degeneracy of the pairing $\la\cdot,\cdot\ra\colon\Xi^*_A(\CC)\times\Xi_A(\CC)\to\CC$ is obtained in the same way. \end{proof}

The space $(\CC^n)^{\otimes k}$ is decomposed into the direct sum of $S_{(k)}(\CC^n)^{\otimes k}$ and $\big(1-S_{(k)}\big)(\CC^n)^{\otimes k}$. Let $v_\alpha=v_\alpha^{S_{(k)}}\in(\CC^n)^{\otimes k}$ be eigenvectors: $S_{(k)}v_\alpha=v_\alpha$ for $\alpha=1,\dots,d_k$ and $S_{(k)}v_\alpha=0$ for $\alpha=d_k+1,\dots,n^k$.
 Then $(v_\alpha)$ is a basis of $(\CC^n)^{\otimes k}$ such that $(v_\alpha)_{\alpha\le d_k}$ and $(v_\alpha)_{\alpha>d_k}$ are bases of the subspaces $S_{(k)}(\CC^n)^{\otimes k}$ and $\big(1-S_{(k)}\big)(\CC^n)^{\otimes k}$ respectively. Let $(v^\alpha=v^\alpha_{S_{(k)}})$ be dual basis of $\big((\CC^n)^{\otimes k}\big)^*$, i.e., $v^\alpha v_\beta=\delta^\alpha_\beta$.
 Since $v^\alpha S_{(k)}v_\beta=\delta^\alpha_\beta$ for $\beta\ge d_k$ and $v^\alpha S_{(k)}v_\beta=0$ for $\beta>d_k$ we obtain $v^\alpha S_{(k)}=v_\alpha$, $\alpha=1,\dots,d_k$, $v^\alpha S_{(k)}=0$, $\alpha=d_k+1,\dots,n^k$. In coordinates we have $v_\alpha=\sum_{i_1,\dots,i_k}v_\alpha^{i_1 \dots i_k}e_{i_1 \dots i_k}$, $v^\alpha=\sum_{i_1,\dots,i_k}v^\alpha_{i_1 \dots i_k}e^{i_1 \dots i_k}$ for some $v_\alpha^{i_1 \dots i_k},v^\alpha_{i_1 \dots i_k}\in\CC$ such that $\sum_{i_1,\dots,i_k}v_\alpha^{i_1 \dots i_k}v^\beta_{i_1 \dots i_k}=\delta^\beta_\alpha$ and $\sum_{\alpha=1}^{n^k}v_\alpha^{i_1 \dots i_k}v^\alpha_{j_1 \dots j_k}=\delta^{i_1}_{j_1}\cdots\delta^{i_k}_{j_k}$.

Analogously, let $w_\alpha=v_\alpha^{A_{(k)}}=\sum_{i_1,\dots,i_k}w_\alpha^{i_1 \dots i_k}e_{i_1 \dots i_k}$ and $w^\alpha=v^\alpha_{A_{(k)}}=\sum_{i_1,\dots,i_k}w^\alpha_{i_1 \dots i_k}e^{i_1 \dots i_k}$ form dual bases of $(\CC^n)^{\otimes k}$ and $\big((\CC^n)^{\otimes k}\big)^*$ such that $A_{(k)}w_\alpha=w_\alpha$ for $\alpha=1,\dots,r_k$ and \mbox{$A_{(k)}w_\alpha=0$} for $\alpha=r_k+1,\dots,n^k$. Then $w^\alpha A_{(k)}=w^\alpha$, $\alpha=1,\dots,r_k$ and $w^\alpha A_{(k)}=0$, $\alpha=r_k+1,\dots,n^k$.

\begin{Prop} 
 The elements
\begin{align}
 &x_{(k)}^\alpha=v^\alpha (X\otimes\cdots\otimes X)=\sum_{i_1,\dots,i_k}v^\alpha_{i_1 \dots i_k}x^{i_1}\cdots x^{i_k}, \qquad\alpha=1,\dots,d_k, \label{xalphak} \\
 &x^{(k)}_\alpha=(X^*\otimes\cdots\otimes X^*)v_\alpha=\sum_{i_1,\dots,i_k}v_\alpha^{i_1 \dots i_k}x_{i_1}\cdots x_{i_k}, \qquad\alpha=1,\dots,d_k, \label{xklapha} \\
 &\psi^{(k)}_\alpha=(\Psi\otimes\cdots\otimes\Psi)w_\alpha=\sum_{i_1,\dots,i_k}w_\alpha^{i_1 \dots i_k}\psi_{i_1}\cdots \psi_{i_k}, \qquad\alpha=1,\dots,r_k, \label{psikalpha} \\
 &\psi_{(k)}^\alpha=w^\alpha (\Psi^*\otimes\cdots\otimes\Psi^*)=\sum_{i_1,\dots,i_k}w^\alpha_{i_1 \dots i_k}\psi^{i_1}\cdots \psi^{i_k}, \qquad\alpha=1,\dots,r_k, \label{psialphak}
\end{align}
form the dual bases of the $k$-th graded components $\gX_A(\CC)_k$, $\gX^*_A(\CC)_k$, $\Xi_A(\CC)_k$, $\Xi^*_A(\CC)_k$.
\end{Prop}

\begin{proof} This is consequence of Proposition~\ref{propIsom} and the fact that $(v^\alpha)_{\alpha=1}^{d_k}$, $(v_\alpha)_{\alpha=1}^{d_k}$, $(w_\alpha)_{\alpha=1}^{r_k}$ and $(w^\alpha)_{\alpha=1}^{r_k}$ are dual bases of
 $\big((\CC^n)^{\otimes k}\big)^*S_{(k)}$, $S_{(k)}\big((\CC^n)^{\otimes k}\big)$, $A_{(k)}\big((\CC^n)^{\otimes k}\big)$ and $\big((\CC^n)^{\otimes k}\big)^*A_{(k)}$ respectively. \end{proof}

In contrast to the uniqueness the existence of the pairing operators is not guaranteed for an arbitrary idempotent $A$. In many interesting cases the pairing operators can be found explicitly. In general situation we can claim the existence of $S_{(1)}$, $A_{(1)}$, $S_{(2)}$ and $A_{(2)}$ only. In Section~\ref{secPO4p} we consider cases when the third pairing operators do not exist. Now we give necessary and sufficient conditions for existence of a pairing operator.

\begin{Th} \label{ThPOviaDB}
 Let $A\in\End(\CC^n\otimes\CC^n)$ be an idempotent and $k\ge2$. Consider the subspaces
\begin{align*}
 &V_k=\big\{\pi\in(\CC^n)^{\otimes k}\mid A^{(a,a+1)}\pi=0\big\}, \qquad
 \overline V_k=\big\{\xi\in\big((\CC^n)^{\otimes k}\big)^*\mid \xi A^{(a,a+1)}=0\big\},
 \\
 &W_k=\big\{\pi\in(\CC^n)^{\otimes k}\mid S^{(a,a+1)}\pi=0\big\}, \qquad
 \overline W_k=\big\{\xi\in\big((\CC^n)^{\otimes k}\big)^*\mid \xi S^{(a,a+1)}=0\big\}.
\end{align*}
We have isomorphisms $\gX_A(\CC)_k\cong V_k^*$, $\gX^*_A(\CC)_k\cong \overline V_k^*$, $\Xi_A(\CC)_k\cong\overline W_k^*$, $\Xi^*_A(\CC)_k\cong W_k^*$ given by the pairings
\begin{align}
 &\big\la\xi(X\otimes\cdots\otimes X),\pi\big\ra=\xi\pi, \qquad
 \xi\in\big((\CC^n)^{\otimes k} \big)^*,\quad\pi\in V_k, \label{pairVk} \\
 &\big\la(X^*\otimes\cdots\otimes X^*)\pi,\xi\big\ra=\xi\pi, \qquad
 \pi\in(\CC^n)^{\otimes k},\quad\xi\in\overline V_k, \label{pairVkD} \\
 &\big\la(\Psi\otimes\cdots\otimes\Psi)\pi,\xi\big\ra=\xi\pi, \qquad
 \pi\in(\CC^n)^{\otimes k},\quad\xi\in\overline W_k, \label{pairWkD} \\
 &\big\la\xi(\Psi^*\otimes\cdots\otimes\Psi^*),\pi\big\ra=\xi\pi, \qquad
 \xi\in\big((\CC^n)^{\otimes k} \big)^*,\quad\pi\in W_k.\label{pairWk}
\end{align}
\begin{itemize}
\itemsep=0pt
 \item The $k$-th $S$-operator $S_{(k)}$ for the idempotent $A$ exists iff the spaces $V_k$ and $\overline V_k$ are dual via the natural pairing, that is $\dim V_k=\dim\overline V_k$ and there are bases $(v_\alpha)$ and $(v^\alpha)$ of $V_k$ and~$\overline V_k$ such that $v^\alpha v_\beta=\delta^\alpha_\beta$. This pairing operator has the form $S_{(k)}=\sum_{\alpha}v_\alpha v^\alpha$.
 \item The $k$-th $A$-operator $A_{(k)}$ for the idempotent $A$ exists iff the spaces $W_k$ and $\overline W_k$ are dual via the natural pairing, that is $\dim W_k=\dim\overline W_k$ and there are bases $(w_\alpha)$ and $(w^\alpha)$ of~$W_k$ and $\overline W_k$ such that $w^\alpha w_\beta=\delta^\alpha_\beta$. This pairing operator is $A_{(k)}=\sum_{\alpha}w_\alpha w^\alpha$.
\end{itemize}
\end{Th}

\begin{proof} Due to the symmetry~\eqref{ASinter} it is enough to prove the statements concerning $V_k$ and $\overline V_k$.

 By definition the algebra $\gX^*_A(\CC)$ is the quotient of the tensor algebra $T\CC^n=\bigoplus_{k\in\NN_0}(\CC^n)^{\otimes k}$ by its two sided ideal $I$ generated by the elements $\sum_{s,t=1}^ne_{\rm st}A^{\rm st}_{ij}$. In the $k$-th graded component we have $\gX^*_A(\CC)_k=(\CC^n)^{\otimes k}/I_k$, where $I_k$ is a subspace of $(\CC^n)^{\otimes k}$ spanned by the vectors $\sum_{s,t=1}^n e_{i_1\dots i_{a-1}} \otimes\big(e_{\rm st}A^{\rm st}_{i_ai_{a+1}}\big)\otimes e_{i_{a+2}\dots i_k}=A^{(a,a+1)}e_{i_1\dots i_k}$, where $a=1,\dots,k-1$, $i_1,\dots,i_k=1,\dots,n$. That is $I_k=\sum_{a=1}^{k-1}A^{(a,a+1)}(\CC^n)^{\otimes k}$.
 This definition implies that the element $x_i$ is the class $[e_i]$, hence $[e_{i_1\dots i_k}]=[e_{i_1}\otimes\cdots\otimes e_{i_k}]=x_{i_1}\cdots x_{i_k}=(X^*\otimes\cdots\otimes X^*)e_{i_1\dots i_k}$, so the canonical projection $p_k\colon(\CC^n)^{\otimes k}\to\gX^*_A(\CC)_k$ has the form $p_k\colon\pi\mapsto(X^*\otimes\cdots\otimes X^*)\pi$. The subspace orthogonal to $\Ker p_k=I_k$ is $I_k^\bot=\big\{\xi\in\big((\CC^n)^{\otimes k}\big)^*\mid\xi I_k=0\big\}=\overline V_k$.

Note that for any subspace $U_0$ of a vector space $U$ we have an isomorphism $U_0^\bot\cong(U/U_0)^*$ via the pairing $\big\la[u],\lambda\big\ra=\lambda(u)$, where $\lambda\in U_0^\bot$ and $[u]\in U/U_0$ is the class of a vector $u\in U$. Hence we obtain the isomorphism $\overline V_k=I_k^\bot\cong\big((\CC^n)^{\otimes k}/I_k\big)^*=\big(\gX^*_A(\CC)_k\big)^*$ given by the pairing~\eqref{pairVkD}. In the same way we obtain the isomorphism $V_k\cong\big(\gX_A(\CC)_k\big)^*$.

If $S_{(k)}$ exists then there are dual bases $(v_\alpha)_{\alpha=1}^{d_k}$ and $(v^\alpha)_{\alpha=1}^{d_k}$ of the spaces $S_{(k)}(\CC^n)^{\otimes k}$ and $\big((\CC^n)^{\otimes k}\big)^*S_{(k)}$. Since $S_{(k)}A^{(a,a+1)}=0$ we have $\big((\CC^n)^{\otimes k}\big)^*S_{(k)}\subset\overline V_k$. The equality~\eqref{dkrk} and the isomorphism $\overline V_k^*\cong\gX^*_A(\CC)_k$ imply $\dim\overline V_k=\dim\gX^*_A(\CC)_k=d_k=\dim\big((\CC^n)^{\otimes k}\big)^*S_{(k)}$, hence $\big((\CC^n)^{\otimes k}\big)^*S_{(k)}=\overline V_k$. Similarly we obtain $\dim V_k=d_k$ and $S_{(k)}(\CC^n)^{\otimes k}=V_k$. Thus $(v_\alpha)_{\alpha=1}^{d_k}$ and $(v^\alpha)_{\alpha=1}^{d_k}$ are dual bases of $V_k$ and $\overline V_k$ respectively.

Conversely, let $d=\dim V_k=\dim\overline V_k$ and let $(v_\alpha)_{\alpha=1}^{d}$ and $(v^\alpha)_{\alpha=1}^{d}$ be dual bases of $V_k$ and~$\overline V_k$. Since the vectors $v_\alpha$ are not orthogonal to $\overline V_k$ they do not belong to $I_k$, that is $V_k\cap I_k=0$. Moreover, $\dim I_k=n^k-\dim\overline V_k=n^k-\dim V_k$ and hence $(\CC^n)^{\otimes k}=V_k\oplus I_k$. This implies that the restriction of the projection $p_k$ to the subspace $V_k$ is an isomorphism and hence the elements $x_\alpha^{(k)}:=(X^*\otimes\cdots\otimes X^*)v_\alpha$ form a basis of $\gX^*_A(\CC)_k$. In particular, $x_{i_1}\cdots x_{i_k}=\sum_{\alpha=1}^{d}c_{i_1\dots i_k}^\alpha x^{(k)}_\alpha$ for some $c_{i_1\dots i_k}^\alpha\in\CC$, so that $(X^*\otimes\cdots\otimes X^*)=\sum_{\alpha=1}^{d}x^{(k)}_\alpha\xi^\alpha$, where $\xi^\alpha=\sum_{i_1,\dots,i_k=1}^n c_{i_1\dots i_k}^\alpha e^{i_1\dots i_k}\in\big((\CC^n)^{\otimes k}\big)^*$. By multiplying this by $A^{(a,a+1)}$ from the right and taking into account~\eqref{Axx0_dual} we obtain $\sum_{\alpha=1}^{d}x^{(k)}_\alpha\xi^\alpha A^{(a,a+1)}=0$. Since the elements $x^{(k)}_\alpha$ are linearly independent, we have $\xi^\alpha A^{(a,a+1)}=0$ for all $a=1,\dots,k-1$, so that $\xi^\alpha\in\overline V_k$. Multiplication of the same relation by $v_\beta$ from the right gives $\xi^\alpha v_\beta=\delta^\alpha_\beta$, hence $\xi^\alpha=v^\alpha$. Let $S_{(k)}=\sum_{\alpha=1}^dv_\alpha v^\alpha$, then we have $(X^*\otimes\cdots\otimes X^*)S_{(k)}=\sum_{\alpha,\beta=1}^{d}x^{(k)}_\alpha v^\alpha v_\beta v^\beta=\sum_{\alpha=1}^{d}x^{(k)}_\alpha v^\alpha=(X^*\otimes\cdots\otimes X^*)$. Analogously we obtain $S_{(k)}(X\otimes\cdots\otimes X)=(X\otimes\cdots\otimes X)$. Since $S_{(k)}$ satisfies also $A^{(a,a+1)}S_{(k)}=S_{(k)}A^{(a,a+1)}=0$ it is the $k$-th $S$-operator for the idempotent $A$. \end{proof}

\begin{Rem}
 Theorem~\ref{ThPOviaDB} means in fact that the non-degeneracies of the pairings $\overline V_k\times V_k\to\CC$ and $\overline W_k\times W_k\to\CC$ given by $\la\xi,\pi\ra=\xi\pi$ imply that they induce non-degenerate pairings $\gX^*_A(\CC)_k\times\gX_A(\CC)_k\to\CC$ and $\Xi_A(\CC)_k\times\Xi^*_A(\CC)_k\to\CC$ respectively. Conversely, if there exists a non-degenerate pairing $\gX^*_A(\CC)_k\times\gX_A(\CC)_k\to\CC$ or $\Xi_A(\CC)_k\times\Xi^*_A(\CC)_k\to\CC$ then the induced pairing $\overline V_k\times V_k\to\CC$ or $\overline W_k\times W_k\to\CC$ is non-degenerate but it may differ from the corresponding pairing defined by $\la\xi,\pi\ra=\xi\pi$. In other words, the conditions $\dim\gX^*_A(\CC)_k=\dim\gX_A(\CC)_k$ and $\dim\Xi^*_A(\CC)_k=\dim\Xi_A(\CC)_k$ do not guarantee the existence of $S_{(k)}$ and $A_{(k)}$ respectively (e.g., see~Section~\ref{secPO4p}).
\end{Rem}

\begin{Rem}
 If some pairing $S$- or $A$-operators do not exist then one can consider the dual space $\big(\gX_A(\CC)\big)^*=\bigoplus_{k=0}^\infty V_k$ or $\big(\Xi_A(\CC)\big)^*=\bigoplus_{k=0}^\infty\overline W_k$ (without a structure of algebra) instead of the algebra $\gX^*_A(\CC)$ or $\Xi^*_A(\CC)$ respectively. The algebra structures on the spaces $\gX^*_A(\CC)$ and~$\Xi^*_A(\CC)$ are auxiliary. They are not in agreement with the algebra structures of $\gX_A(\CC)$ and $\Xi_A(\CC)$, but they are used to define these spaces in a more convenient way.
\end{Rem}

\subsection{Minor operators}
\label{secMinO}

 Let $S_{(k)},A_{(k)}\in\End\big((\CC^n)^{\otimes k}\big)$ and $\wt S_{(k)},\wt A_{(k)}\in\End\big((\CC^m)^{\otimes k}\big)$ be pairing operators for idempotents $A\in\End(\CC^n\otimes\CC^n)$ and $\wt A\in\End(\CC^n\otimes\CC^n)$ respectively.
Let $X$, $X^*$, $\Psi$, $\Psi^*$ denote the same as previous subsection. The corresponding column- and row-vectors for $\wt A$ we denote by~$\wt X$, $\wt X^*$,~$\wt\Psi^*$, $\wt\Psi$.

By virtue of Proposition~\ref{propIsom} any graded linear operator $\gX_A(\CC)\to\gX_{\wt A}(\gR)$ is given by the formula
\begin{align}
 \xi(X\otimes\cdots\otimes X)\mapsto\xi T_k\big(\wt X\otimes\cdots\otimes\wt X\big), \qquad
 \xi\in\big((\CC^n)^{\otimes k}\big)^* \label{TkX}
\end{align}
for some operators $T_k\in\gR\otimes\Hom\big((\CC^m)^{\otimes k},(\CC^n)^{\otimes k}\big)$ such that $S_{(k)}T_k\wt S_{(k)}=T_k\wt S_{(k)}$. In the same time a graded linear operator $\gX^*_{\wt A}(\CC)\to\gX^*_A(\gR)$ has the form
\begin{align}
 (\wt X^*\otimes\cdots\otimes\wt X^*)\pi\mapsto(X^*\otimes\cdots\otimes X^*) T_k\pi, \qquad \pi\in(\CC^m)^{\otimes k} \label{XTk}
\end{align}
for some $T_k\in\gR\otimes\Hom\big((\CC^m)^{\otimes k},(\CC^n)^{\otimes k}\big)$ such that $S_{(k)}T_k\wt S_{(k)}=S_{(k)}T_k$.

Analogously, a graded linear operator $\Xi_{\wt A}(\CC)\to\Xi_A(\gR)$ can be written as
\begin{align}
 (\wt\Psi\otimes\cdots\otimes\wt\Psi)\pi\mapsto(\Psi\otimes\cdots\otimes\Psi) R_k\pi,\qquad \pi\in(\CC^m)^{\otimes k} \label{PsiRk}
\end{align}
for $R_k\in\gR\otimes\Hom\big((\CC^m)^{\otimes k},(\CC^n)^{\otimes k}\big)$ such that $A_{(k)}R_k\wt A_{(k)}=A_{(k)}R_k$, while a graded linear operator $\Xi^*_A(\CC)\to\Xi^*_{\wt A}(\gR)$ has the form
\begin{align}
 \xi(\Psi^*\otimes\cdots\otimes\Psi^*)\mapsto\xi R_k\big(\wt\Psi^*\otimes\cdots\otimes\wt\Psi^*\big), \qquad \xi\in\big((\CC^n)^{\otimes k}\big)^* \label{RkPsi}
\end{align}
for some operators $R_k\in\gR\otimes\Hom\big((\CC^m)^{\otimes k},(\CC^n)^{\otimes k}\big)$ such that $A_{(k)}R_k\wt A_{(k)}=R_k\wt A_{(k)}$.

Note that $T_k$ and $R_k$ can be replaced by $S_{(k)}T_k\wt S_{(k)}$ and $A_{(k)}R_k\wt A_{(k)}$ respectively and this does not change the maps~\eqref{TkX}, \eqref{XTk}, \eqref{PsiRk}, \eqref{RkPsi}. Hence we can always suppose $S_{(k)}T_k\wt S_{(k)}=T_k$ and $A_{(k)}R_k\wt A_{(k)}=R_k$.

Denote $\gX^*_A(\gR)=\gR\otimes\gX^*_A(\CC)$ and $\Xi^*_A(\gR)=\gR\otimes\Xi^*_A(\CC)$. The pairings are invariant in the following sense.

\begin{Prop} 
 Let $t\colon\gX_A(\CC)\to\gX_{\wt A}(\gR)$ and $t^*\colon\gX^*_{\wt A}(\CC)\to\gX^*_A(\gR)$ be the operators~\eqref{TkX} and \eqref{XTk} defined by the same $T_k\in\gR\otimes\Hom\big((\CC^m)^{\otimes k},(\CC^n)^{\otimes k}\big)$ satisfying $S_{(k)}T_k\wt S_{(k)}=T_k$. Let $r\colon\Xi_{\wt A}(\CC)\to\Xi_A(\gR)$ and $r^*\colon\Xi^*_A(\CC)\to\Xi^*_{\wt A}(\gR)$ be the operators~\eqref{PsiRk} and \eqref{RkPsi} defined by the same $R_k\in\gR\otimes\Hom\big((\CC^m)^{\otimes k},(\CC^n)^{\otimes k}\big)$ satisfying $A_{(k)}R_k\wt A_{(k)}=R_k$. Then for any $u\in\gX_A(\CC)$, $v\in\gX^*_{\wt A}(\CC)$, $\nu\in\Xi^*_A(\CC)$, $\mu\in\Xi_{\wt A}(\CC)$ we have
\begin{align}
 \la u,t^*(v)\ra=\la t(u),v\ra, \qquad
 \la\nu,r(\mu)\ra=\la r^*(\nu),\mu\ra, \label{InvPair}
\end{align}
where the pairings $\gX_A(\CC)\times\gX^*_A(\CC)\to\CC$, $\Xi^*_A(\CC)\times\Xi_A(\CC)\to\CC$ and $\gX_{\wt A}(\CC)\times\gX^*_{\wt A}(\CC)\to\CC$, $\Xi^*_{\wt A}(\CC)\times\Xi_{\wt A}(\CC)\to\CC$ are defined by the pairing operators $S_{(k)}$, $A_{(k)}$ and $\wt S_{(k)}$, $\wt A_{(k)}$ respectively.
\end{Prop}

\begin{proof} Let $u=\xi(X\otimes\cdots\otimes X)$ and $v=\big(\wt X^*\otimes\cdots\otimes\wt X^*\big)\pi$. Then by using Lemma~\ref{lemXTX} we have
\begin{align*}
 \la u,t^*(v)\ra&=\big\la\xi(X\otimes\cdots\otimes X),(X^*\otimes\cdots\otimes X^*)T_k\pi\big\ra=\xi S_{(k)}T_k\pi=\xi T_k\wt S_{(k)}\pi
 \\
 &=\big\la\xi T_k\big(\wt X\otimes\cdots\otimes\wt X\big),\big(\wt X^*\otimes\cdots\otimes\wt X^*\big)\pi\big\ra=\la t(u),v\ra.
\end{align*}
The second formula~\eqref{InvPair} is proved similarly. \end{proof}

Recall that in Section~\ref{secCommCat} we introduced the homomorphisms $f_M\colon\gX_A(\CC)\to\gX_{\wt A}(\CC)$ and $f^M\colon\Xi_{\wt A}(\CC)\to\Xi_A(\CC)$ for a $\big(A,\wt A\big)$-Manin matrix $M\in\gR\otimes\Hom(\CC^m,\CC^n)$. Their definition on~generators can be written in the matrix form as
\begin{align*}
 f_M(X)=M\wt X,\qquad
 f^M(\wt\Psi)=\Psi M.
\end{align*}
Since homomorphisms preserve multiplications we obtain
\begin{gather}
 f_M(X\otimes\cdots\otimes X)=M^{(1)}\cdots M^{(k)}\big(\wt X\otimes\cdots\otimes\wt X\big), \label{fMXk} \\
 f^M(\wt\Psi\otimes\cdots\otimes\wt\Psi)=(\Psi\otimes\cdots\otimes\Psi) M^{(1)}\cdots M^{(k)}. \label{fMPsik}
\end{gather}
For general elements of $\gX_A(\CC)$ and $\Xi_{\wt A}(\CC)$ the values of the maps $f_M$ and $f^M$ are obtained via multiplication of these formulae by $\xi$ and $\pi$ respectively, so these maps have the form~\eqref{TkX} and~\eqref{PsiRk} for $T_k=R_k=M^{(1)}\cdots M^{(k)}$. In this way we obtain the following generalisation of~\eqref{AMMAMMA} and \eqref{SMMS}.

\begin{Prop} \label{propAMMMAk}
 Any $\big(A,\wt A\big)$-Manin matrix $M\in\gR\otimes\Hom(\CC^m,\CC^n)$ satisfy the relations
\begin{gather}
 M^{(1)}\cdots M^{(k)}\wt S_{(k)} =S_{(k)} M^{(1)}\cdots M^{(k)}\wt S_{(k)}, \label{SMMMSk} \\
 A_{(k)} M^{(1)}\cdots M^{(k)}=A_{(k)} M^{(1)}\cdots M^{(k)}\wt A_{(k)}. \label{AMMMAk}
\end{gather}
\end{Prop}

\begin{proof} Note that the left hand sides of~\eqref{fMXk} and \eqref{fMPsik} are invariant under multiplication by~$S_{(k)}$ from the left and by $\wt A_{(k)}$ from the right respectively. As consequence, we obtain
\begin{gather*}
S_{(k)}M^{(1)}\cdots M^{(k)}\big(\wt X\otimes\cdots\otimes\wt X\big)=
M^{(1)}\cdots M^{(k)}\big(\wt X\otimes\cdots\otimes\wt X\big),
\\
(\Psi\otimes\cdots\otimes\Psi)M^{(1)}\cdots M^{(k)}\wt A_{(k)}= (\Psi\otimes\cdots\otimes\Psi)M^{(1)}\cdots M^{(k)}.
\end{gather*}
Then, \eqref{SMMMSk} and \eqref{AMMMAk} are derived by application of Corollary~\ref{corTXXk} for $V=\gR\otimes(\CC^n)^{\otimes k}$, $T=S_{(k)}M^{(1)}\cdots M^{(k)}-M^{(1)}\cdots M^{(k)}\in\Hom\big((\CC^m)^{\otimes k},V\big)$ and for $W=\gR^*\otimes(\CC^m)^{\otimes k}$, $\wt T=M^{(1)}\cdots M^{(k)}\wt A_{(k)}-M^{(1)}\cdots M^{(k)}\in\Hom\big(W,(\CC^n)^{\otimes k}\big)$ respectively. \end{proof}

For an arbitrary $M\in\gR\otimes\Hom(\CC^m,\CC^n)$ define the linear operators $t_M\colon\gX_A(\CC)\to\gX_{\wt A}(\gR)$ and $t^*_M\colon\gX^*_{\wt A}(\CC)\to\gX^*_A(\gR)$ by the operators $T_k=S_{(k)}M^{(1)}\cdots M^{(k)}\wt S_{(k)}$, that is
\begin{gather}
 t_M(X\otimes\cdots\otimes X)=S_{(k)}M^{(1)}\cdots M^{(k)}\big(\wt X\otimes\cdots\otimes\wt X\big), \label{tMX}
 \\
 t^*_M\big(\wt X^*\otimes\cdots\otimes\wt X^*\big)=(X^*\otimes\cdots\otimes X^*) M^{(1)}\cdots M^{(k)}\wt S_{(k)}. \label{tMXdual}
\end{gather}

Analogously, for an arbitrary matrix $M\in\gR\otimes\Hom(\CC^m,\CC^n)$ define the linear operators $r_M\colon\Xi_{\wt A}(\CC)\to\Xi_A(\gR)$ and $r^*_M\colon\Xi^*_A(\CC)\to\Xi^*_{\wt A}(\gR)$ by the operators $R_k=A_{(k)}M^{(1)}\cdots M^{(k)}\wt A_{(k)}$, that is
\begin{gather}
 r_M\big(\wt\Psi\otimes\cdots\otimes\wt\Psi\big)=(\Psi\otimes\cdots\otimes\Psi) M^{(1)}\cdots M^{(k)}\wt A_{(k)}, \label{rMPsi}
 \\
 r_M^*(\Psi^*\otimes\cdots\otimes\Psi^*)=A_{(k)}M^{(1)}\cdots M^{(k)}\big(\wt\Psi^*\otimes\cdots\otimes\wt\Psi^*\big). \label{rMPsidual}
\end{gather}

\begin{Rem} 
 If $M$ is an $\big(A,\wt A\big)$-Manin matrix then the maps~\eqref{tMX} and \eqref{rMPsi} are homomorphisms: $t_M=f_M$ and $r_M=f_M$ (however, the maps $t^*_M$ and $r_M^*$ are not homomorphisms). Conversely, if $t_M$ or $r_M$ is a homomorphism then $M$ is an $\big(A,\wt A\big)$-Manin matrix. The maps~$t^*_M$ and $r_M^*$ are homomorphisms iff $M$ is a $\big(S,\wt S\big)$-Manin matrix, where $S=1-A$, $\wt S=1-\wt A$.
\end{Rem}

For a matrix $M\in\gR\otimes\Hom(\CC^m,\CC^n)$ we introduce {\it minor operator} corresponding to a pair of operators $T\in\End\big((\CC^n)^{\otimes k}\big)$ and $\wt T\in\End\big((\CC^m)^{\otimes k}\big)$ by the formula
\begin{align}
 \Min^T_{\wt T} M&:=T M^{(1)}\cdots M^{(k)}\wt T. \label{MinTTM}
\end{align}
From~\eqref{tMX}--\eqref{rMPsidual} we obtain\vspace{-1ex}
\begin{align*}
 \Min^{S_{(k)}}_{\wt S_{(k)}} M&=\big\la t_M(X\otimes\cdots\otimes X),\wt X^*\otimes\cdots\otimes\wt X^*\big\ra=\big\la X\otimes\cdots\otimes X,t^*_M\big(\wt X^*\otimes\cdots\otimes\wt X^*\big)\big\ra, \\
 \Min^{A_{(k)}}_{\wt A_{(k)}}M&=\big\la\Psi^*\otimes\cdots\otimes \Psi^*,r_M(\wt\Psi\otimes\cdots\otimes\wt\Psi)\big\ra=\big\la r^*_M(\Psi^*\otimes\cdots\otimes \Psi^*),\wt\Psi\otimes\cdots\otimes\wt\Psi\big\ra.
\end{align*}
If $M$ is an $\big(A,\wt A\big)$-Manin matrix then due to Proposition~\ref{propAMMMAk} these operators take the form\vspace{-1ex}
\begin{align}
 \Min^{S_{(k)}}_{\wt S_{(k)}} M&=S_{(k)}M^{(1)}\cdots M^{(k)}\wt S_{(k)}=M^{(1)}\cdots M^{(k)}\wt S_{(k)}, \label{MinSS} \\
 \Min^{A_{(k)}}_{\wt A_{(k)}} M&=A_{(k)}M^{(1)}\cdots M^{(k)}\wt A_{(k)}=A_{(k)}M^{(1)}\cdots M^{(k)}. \label{MinAA}
\end{align}
In these case these minor operators are defined by one operator only: $\wt S_{(k)}$ and $A_{(k)}$ respectively, so we can denote them as
\begin{align}
 \Min_{\wt S_{(k)}} M&:=M^{(1)}\cdots M^{(k)}\wt S_{(k)}=\big\la f_M(X\otimes\cdots\otimes X),\wt X^*\otimes\cdots\otimes\wt X^*\big\ra, \label{MinS} \\
 \Min^{A_{(k)}} M&:=A_{(k)}M^{(1)}\cdots M^{(k)}=\big\la\Psi^*\otimes\cdots\otimes\Psi^*,f^M\big(\wt\Psi\otimes\cdots\otimes\wt\Psi\big)\big\ra. \label{MinA}
\end{align}

\begin{Def} 
Let $M\in\gR\otimes\Hom(\CC^m,\CC^n)$ be an $\big(A,\wt A\big)$-Manin matrix. Then the minor operators~\eqref{MinS} and \eqref{MinA} are called {\it $S$-minor} and {\it $A$-minor operators} respectively. We also call them {\it minor operators for $\big(A,\wt A\big)$-Manin matrix}. Their entries
\begin{align}
 \big(\Min_{\wt S_{(k)}} M\big)^{i_1 \dots i_k}_{j_1 \dots j_k}&=e^{i_1 \dots i_k}\big(\Min_{\wt S_{(k)}} M\big)e_{j_1 \dots j_k}=\big\la f_M\big(x^{i_1}\cdots x^{i_k}\big),\wt x_{j_1}\cdots\wt x_{j_k}\big\ra, \label{DefSMin} \\
 \big(\Min^{A_{(k)}} M\big)^{i_1 \dots i_k}_{j_1 \dots j_k}&=e^{i_1 \dots i_k}\big(\Min^{A_{(k)}} M\big)e_{j_1 \dots j_k}=\big\la \psi^{i_1}\cdots\psi^{i_k},f^M\big(\wt\psi_{j_1}\cdots\wt\psi_{j_k}\big)\big\ra \label{DefAMin}
\end{align}
are called {\it $S$-minors} and {\it $A$-minors of the order $k$} or simply {\it minors for $\big(A,\wt A\big)$-Manin matrix} (here $\wt x_i$ and $\wt\psi_i$ are the generators of $\gX^*_{\wt A}(\CC)$ and $\Xi_{\wt A}(\CC)$ respectively).
\end{Def}

In terms of $y^i=f_M(x^i)=\sum_{j=1}^mM^i_j\wt x^j$ and $\phi_j=f^M\big(\wt\psi_j\big)=\sum_{i=1}^n\psi_iM^i_j$ the minors can be written as
\begin{gather}
 \big(\Min_{\wt S_{(k)}} M\big)^{i_1 \dots i_k}_{j_1 \dots j_k}
 =\big\la y^{i_1}\cdots y^{i_k},\wt x_{j_1}\cdots\wt x_{j_k}\big\ra, \label{SMin}
 \\
 \big(\Min^{A_{(k)}} M\big)^{i_1 \dots i_k}_{j_1 \dots j_k} =\big\la\psi^{i_1}\cdots\psi^{i_k},\phi_{j_1}\cdots\phi_{j_k}\big\ra. \label{AMin}
\end{gather}
They are coefficients in the decompositions
\begin{gather}
 y^{i_1}\cdots y^{i_k}=\sum_{j_1,\dots,j_k=1}^m\big(\Min_{\wt S_{(k)}} M\big)^{i_1 \dots i_k}_{j_1 \dots j_k}\;\wt x^{j_1}\cdots\wt x^{j_k}, \label{yydecMin}
 \\
 \phi_{j_1}\cdots\phi_{j_k}=\sum_{i_1 \dots i_k=1}^n\big(\Min^{A_{(k)}} M\big)^{i_1 \dots i_k}_{j_1 \dots j_k}\;\psi_{i_1}\cdots\psi_{i_k}. \label{phiphidecMin}
\end{gather}
In the matrix form these formulae are written as
\begin{gather*}
 (Y\otimes\cdots\otimes Y)=\big(\Min_{\wt S_{(k)}} M\big)\big(\wt X\otimes\cdots\otimes\wt X\big),
 \\
 (\Phi\otimes\cdots\otimes\Phi)=(\Psi\otimes\cdots\otimes\Psi)\big(\Min^{A_{(k)}} M\big),
\end{gather*}
where $Y=\sum_{i=1}^ny^ie_i$ and $\Phi=\sum_{j=1}^m\phi_je^j$.

The formulae~\eqref{yydecMin} and \eqref{phiphidecMin} are not decompositions by bases. However, due to Proposition~\ref{propIsom} they have the form of the decompositions considered in the part (2) of Proposition~\ref{propVAW}. For example, for the formula~\eqref{yydecMin} we need to set $V=\big((\CC^m)^{\otimes k}\big)^*$ and $W=(\CC^m)^{\otimes k}$, while the operator $S_{(k)}$ acting on $\xi\in\big((\CC^m)^{\otimes k}\big)^*$ from the right plays the role of the idempotent $A$. Thus the minors of an $\big(A,\wt A\big)$-Manin matrix $M$ are the coefficients of the decompositions~\eqref{yydecMin} and~\eqref{phiphidecMin} satisfying the conditions
\begin{align*}
 \sum_{l_1,\dots,l_k=1}^m\big(\Min_{\wt S_{(k)}} M\big)^{i_1 \dots i_k}_{l_1 \dots l_k}\wt S^{l_1 \dots l_k}_{j_1 \dots j_k}=\big(\Min_{\wt S_{(k)}} M\big)^{i_1 \dots i_k}_{j_1 \dots j_k}, 
 \\
 \sum_{l_1,\dots,l_k=1}^nA^{i_1 \dots i_k}_{l_1 \dots l_k}\big(\Min^{A_{(k)}} M\big)^{l_1 \dots l_k}_{j_1 \dots j_k}=\big(\Min_{\wt S_{(k)}} M\big)^{i_1 \dots i_k}_{j_1 \dots j_k}. 
\end{align*}
In operator form these symmetries can be written as
\begin{align*}
\big(\Min_{\wt S_{(k)}}M\big)\wt S_{(k)}=\Min_{\wt S_{(k)}}M, \qquad
A_{(k)}\big(\Min^{A_{(k)}}M\big)=\Min^{A_{(k)}}M.
\end{align*}

The expression for the $S$- and $A$-minors of an $\big(A,\wt A\big)$-Manin matrix $M$ depends on the pairing operators $\wt S_{(k)}$ and $A_{(k)}$ only, hence they are defined if these pairing operators exist even if $S_{(k)}$ and $\wt A_{(k)}$ do not exist. The condition that $M$ is a $\big(A,\wt A\big)$-Manin matrix implies the symmetry of the minor $S$- and $A$-operators with respect to the upper and lower indices:
\begin{align*}
S^{(a,a+1)}\big(\Min_{\wt S_{(k)}}M\big)=\Min_{\wt S_{(k)}}M,\qquad
\big(\Min^{A_{(k)}}M\big)\wt A^{(a,a+1)}=\Min^{A_{(k)}}M.
\end{align*}
If $S_{(k)}$ and $\wt A_{(k)}$ do exist these symmetries can be written in the form
\begin{align*}
S_{(k)}\big(\Min_{\wt S_{(k)}}M\big)=\Min_{\wt S_{(k)}}M, \qquad
\big(\Min^{A_{(k)}}M\big)\wt A_{(k)}=\Min^{A_{(k)}}M.
\end{align*}

\subsection{Properties of the minor operators}
\label{secPMinO}

The determinant of usual complex matrices is a homomorphism: $\det(MN)=\det(M)\det(N)$. The generalisation of this property to the case of $k\times k$ minors is (a generalisation of the) Cauchy--Binet formula:
\begin{align*}
 \det\big((MN)_{IJ}\big)=\sum_{L=(l_1<\dots<l_k)}\det(M_{IL})\det(N_{LJ}).
\end{align*}
The right hand side corresponds to the product of the $A$-minor operators. This property is generalised to Manin matrices.

\begin{Prop}
 Let $S_{(k)}$, $A_{(k)}$, $\wt S_{(k)}$, $\wt A_{(k)}$ and $\wh S_{(k)}$, $\wh A_{(k)}$ be the pairing operators for idempotents $A\in\End(\CC^n\otimes\CC^n)$, $\wt A\in\End(\CC^m\otimes\CC^m)$ and $\wh A\in\End\big(\CC^l\otimes\CC^l\big)$ respectively. Let $M$ and~$N$ be $n\times m$ and $m\times l$ matrices over an algebra $\gR$. Suppose that the entries of the first one commute with the entries of the second one: $\big[M^i_j,N^k_l\big]=0$.
\begin{itemize}
\itemsep=-1pt
 \item If $N$ is an $(\wt A,\wh A)$-Manin matrix then
\begin{align}
 \Min^{S_{(k)}}_{\wh S_{(k)}}(MN)=\big(\Min^{S_{(k)}}_{\wt S_{(k)}} M\big)\big(\Min_{\wh S_{(k)}} N\big). \label{MinSMN}
\end{align}
 \item If $M$ is an $\big(A,\wt A\big)$-Manin matrix then
\begin{align}
 \Min^{A_{(k)}}_{\wh A_{(k)}}(MN)=\big(\Min^{A_{(k)}} M\big)\big(\Min^{\wt A_{(k)}}_{\wh A_{(k)}} N\big). \label{MinAMN}
\end{align}
 \item If $M$ and $N$ are Manin matrices for $\big(A,\wt A\big)$ and $\big(\wt A,\wh A\big)$ respectively then $MN$ is a Manin matrix for $\big(A,\wh A\big)$ and the formulae~\eqref{MinSMN},~\eqref{MinAMN} take the form
\begin{gather}
 \Min_{\wh S_{(k)}}(MN)=\big(\Min_{\wt S_{(k)}} M\big)\big(\Min_{\wh S_{(k)}} N\big), \label{SMinMN}
 \\
 \Min^{A_{(k)}}(MN)=\big(\Min^{A_{(k)}} M\big)\big(\Min^{\wt A_{(k)}} N\big). \label{AMinMN}
\end{gather}
\end{itemize}
\end{Prop}

\begin{proof} The first and second statements follow from Proposition~\ref{propAMMMAk}. For instance, the formula~\eqref{MinAMN} is derived in the following way:
\begin{align*}
 A_{(k)}(MN)^{(1)}\cdots (MN)^{(k)}\wh A_{(k)}&=
 A_{(k)}M^{(1)}\cdots M^{(k)}N^{(1)}\cdots N^{(k)}\wh A_{(k)}
 \\
 &= A_{(k)}M^{(1)}\cdots M^{(k)}\wt A_{(k)}N^{(1)}\cdots N^{(k)}\wh A_{(k)},
\end{align*}
where we used the commutativity in the form $N^{(i)}M^{(j)}=M^{(j)}N^{(i)}$, $i<j$. The last statement is implied by Proposition~\ref{propProd} and the formulae~\eqref{MinSS}, \eqref{MinAA}. \end{proof}

Now we give formulae for permutations of rows and columns. For any $n$ and $\sigma\in {\rm GL}(n,\CC)$ denote the conjugation by the element $\sigma^{\otimes k}=\sigma\otimes\cdots\otimes\sigma$ as
\begin{align}
 \iota_\sigma T:=\sigma^{\otimes k} T\big(\sigma^{\otimes k}\big)^{-1},
\end{align}
where $T\in\End\big((\CC^n)^{\otimes k}\big)$ and $\big(\sigma^{\otimes k}\big)^{-1}=\sigma^{-1}\otimes\cdots\otimes\sigma^{-1}$.
Note that $\iota_\sigma S_{(k)}$ and $\iota_\sigma A_{(k)}$ are pairing operators for the idempotent $\iota_\sigma A=(\sigma\otimes\sigma)A\big(\sigma^{-1}\otimes\sigma^{-1}\big)$.

\begin{Prop} \label{propMinPerm}
 For any matrix $M\in\gR\otimes\Hom(\CC^m,\CC^n)$ and operators $\sigma\in {\rm GL}(n,\CC)$, $\tau\in {\rm GL}(m,\CC)$ we have
\begin{align}
 &\Min^{\iota_\sigma S_{(k)}}_{\iota_\tau \wt S_{(k)}}\big(\sigma M\tau^{-1}\big)=\sigma^{\otimes k}\big(\Min^{S_{(k)}}_{\wt S_{(k)}}M\big)\big(\tau^{\otimes k}\big)^{-1}, \label{SSMinPerm} \\
 &\Min^{\iota_\sigma A_{(k)}}_{\iota_\tau \wt A_{(k)}}\big(\sigma M\tau^{-1}\big)=\sigma^{\otimes k}\big(\Min^{A_{(k)}}_{\wt A_{(k)}}M\big)\big(\tau^{\otimes k}\big)^{-1}. \label{AAMinPerm}
\end{align}
If $M$ is an $\big(A,\wt A\big)$-Manin matrix then
\begin{align}
 &\Min_{\iota_\tau \wt S_{(k)}}\big(\sigma M\tau^{-1}\big)=\sigma^{\otimes k}\big(\Min_{\wt S_{(k)}}M\big)\big(\tau^{\otimes k}\big)^{-1}, \label{SMinPerm} \\
 &\Min^{\iota_\sigma A_{(k)}}\big(\sigma M\tau^{-1}\big)=\sigma^{\otimes k}\big(\Min^{A_{(k)}}M\big)\big(\tau^{\otimes k}\big)^{-1}. \label{AMinPerm}
\end{align}
\end{Prop}

\begin{proof} The formulae~\eqref{SSMinPerm}, \eqref{AAMinPerm} follow directly from the definition~\eqref{MinTTM}. If $M$ is an $\big(A,\wt A\big)$-Manin matrix, then Proposition~\ref{PropPerm} implies that $\sigma M\tau^{-1}$ is a $\big(\iota_\sigma A,\iota_\tau\wt A\big)$-Manin matrix. Thus we obtain~\eqref{SMinPerm}, \eqref{AMinPerm}. \end{proof}

In particular, Proposition~\ref{propMinPerm} gives the minors of $^\sigma M=\sigma M$ and $_\tau M=M\tau^{-1}$ for a matrix $M\in\gR\otimes\Hom(\CC^m,\CC^n)$ and permutations $\sigma\in \SSS_n$, $\tau\in\SSS_m$.

Let us consider the minor operators $\Min_{\wt S_{(k)}}M$ and $\Min^{A_{(k)}}M$ for an $\big(A,\wt A\big)$-Manin matrix~$M$ as $n^k\times m^k$ matrices over $\gR$ with the entries~\eqref{SMin} and~\eqref{AMin}. They are also Manin matrices for some pairs of idempotents.

\begin{Prop} \label{propManin2k}
Let $M$ be an $\big(A,\wt A\big)$-Manin matrix. Let $S_{(k)}$, $A_{(k)}$ and $\wt S_{(k)}$, $\wt A_{(k)}$ are the pairing operators for $A$ and $\wt A$. For any $k,\ell\ge1$ we have
\begin{gather}
 \big(\Min_{\wt S_{(k)}}M\big)\otimes\big(\Min_{\wt S_{(\ell)}}M\big)\wt S_{(k+\ell)}=
 M^{(1)}\cdots M^{(k+\ell)}\wt S_{(k+\ell)}=\Min_{\wt S_{(k+\ell)}}M, \label{SMinkl}
 \\
 A_{(k+\ell)}\big(\Min^{A_{(k)}}M\big)\otimes\big(\Min^{A_{(\ell)}}M\big)=
 A_{(k+\ell)}M^{(1)}\cdots M^{(k+\ell)}=\Min^{A_{(k+\ell)}}M. \label{AMinkl}
\end{gather}
In particular, $\Min_{\wt S_{(k)}}M$ and $\Min^{A_{(k)}}M$ are $\big(1-S_{(2k)},1-\wt S_{(2k)}\big)$- and $\big(A_{(2k)},\wt A_{(2k)}\big)$-Manin matrices respectively.
\end{Prop}

\begin{proof} The formulae~\eqref{SMinkl} and \eqref{AMinkl} follow from Proposition~\ref{PropASkl}. To prove the second statement one needs to put $\ell=k$ in these formulae and to apply Proposition~\ref{propAMMMAk} for $2k$. \end{proof}

Let us finally write the minor operators in terms of bases. Denote $y^\alpha_{(k)}=v^\alpha(Y\otimes\cdots\otimes Y)$, $\wt x^\alpha_{(k)}=\wt v^\alpha\big(\wt X\otimes\cdots\otimes\wt X\big)$, $\phi_\alpha^{(k)}=(\Phi\otimes\cdots\otimes\Phi)\wt w_\alpha$, where $\wt v^\alpha=v^\alpha_{\wt S_{(k)}}$ and $\wt w_\alpha=v_\alpha^{\wt A_{(k)}}$ are eigenvectors of the idempotents $\wt S_{(k)}$ and $(\wt A_{(k)})^\top$ respectively. Let $\wt d_k$ and $\wt r_k$ be the ranks of $\wt S_{(k)}$ and $\wt A_{(k)}$. Consider matrix entries of the minor operators:
\begin{align}
&\big(\Min_{\wt S_{(k)}} M\big)^\alpha_\beta:=v^\alpha\big(\Min_{\wt S_{(k)}} M\big)\wt v_\beta=v^\alpha M^{(1)}\cdots M^{(k)}\wt v_\beta, \label{MinSab} \\
 &\big(\Min^{A_{(k)}} M\big)_\delta^\gamma:=w^\gamma\big(\Min^{A_{(k)}} M\big)\wt w_\delta=w^\gamma M^{(1)}\cdots M^{(k)}\wt w_\delta, \label{MinAab}
\end{align}
where $\alpha\le d_k$, $\beta\le\wt d_k$, $\gamma\le r_k$ and $\delta\le\wt r_k$. Then the formulae~\eqref{yydecMin}, \eqref{phiphidecMin} are rewritten as
\begin{align}
 y^\alpha_{(k)}=\sum_{\gamma=1}^{\wt d_k}\big(\Min_{\wt S_{(k)}} M\big)^\alpha_\gamma\;\wt x^\gamma_{(k)}, \qquad
 \phi_\beta^{(k)}=\sum_{\gamma=1}^{r_k}\big(\Min^{A_{(k)}} M\big)_\beta^\gamma\;\psi_\gamma^{(k)}, \label{yalphaMin}
\end{align}
where $\alpha=1,\dots,d_k$, $\beta=1,\dots,\wt r_k$. These are decompositions by bases. They generalise the formulae given in Section~\ref{secPerm}.

Note that $y^\alpha_{(k)}=f_M\big(x^\alpha_{(k)}\big)$ and $\phi^{(k)}_\alpha=f^M\big(\wt\psi^{(k)}_\alpha\big)$. Thus the formulae~\eqref{yalphaMin} describe the homomorphisms $f_M\colon\gX_A(\CC)\to\gX_{\wt A}(\gR)$ and $f^M\colon\Xi_{\wt A}(\CC)\to\Xi_A(\gR)$ in terms of bases~\eqref{xalphak}--\eqref{psialphak}.

The formulae~\eqref{SMinMN} and \eqref{AMinMN} are written in terms of bases in the form
\begin{gather}
 \big(\Min_{\wh S_{(k)}}(MN)\big)^\alpha_\gamma=\sum_{\beta=1}^{\wt d_k}\big(\Min_{\wt S_{(k)}} M\big)^\alpha_\beta\big(\Min_{\wh S_{(k)}} N\big)^\beta_\gamma, \label{SMinMNb}
 \\
 \big(\Min^{A_{(k)}}(MN)\big)^\alpha_\gamma=\sum_{\beta=1}^{\wt r_k}\big(\Min^{A_{(k)}} M\big)^\alpha_\beta\big(\Min^{\wt A_{(k)}} N\big)^\beta_\gamma. \label{AMinMNb}
\end{gather}

\subsection{Minors for left-equivalent idempotents}
\label{secMinEq}

In contrast with the quadratic algebras $\gX_A(\CC)$ and $\Xi_A(\CC)$ their dual algebras $\gX^*_A(\CC)$ and $\Xi^*_A(\CC)$ do not coincide with the corresponding algebras for a left-equivalent idempotent $A'$ in general, but they do coincide for a right-equivalent idempotent $A'$. Indeed, by applying Proposition~\ref{PropLEQA} to for $A^\top$ and $(A')^\top$ we see that the following 4 conditions are equivalent: $\gX^*_A(\CC)=\gX^*_{A'}(\CC)$, $\Xi^*_A(\CC)=\Xi^*_{A'}(\CC)$, $A$ is right-equivalent to $A'$, $S=1-A$ is left-equivalent to $S'=1-A'$.

\begin{Prop} \label{propSAequiv}
 Let $S_{(k)}$, $A_{(k)}$ and $S'_{(k)}$, $A'_{(k)}$ be the pairing operators for idempotents $A$ and~$A'$. If $A$ is left-equi\-valent to $A'$ then $A_{(k)}$ is left-equivalent to $A'_{(k)}$, while $S_{(k)}$ is right-equi\-va\-lent to $S'_{(k)}$ for each $k\ge0$. If $A$ is right-equivalent to $A'$ then $A_{(k)}$ is right-equivalent to~$A'_{(k)}$, while $S_{(k)}$ is left-equivalent to $S'_{(k)}$ for each $k\ge0$.
\end{Prop}

\begin{proof} The left equivalence of $A$ and $A'$ implies $\Xi_A(\CC)=\Xi_{A'}(\CC)$ and hence we obtain the equality $(\Psi\otimes\cdots\otimes\Psi)A_{(k)}=(\Psi\otimes\cdots\otimes\Psi)=(\Psi\otimes\cdots\otimes\Psi)A'_{(k)}$. By applying Corollary~\ref{corTXXk} we obtain $A'_{(k)}\big(1-A_{(k)}\big)=0$ and $A_{(k)}\big(1-A'_{(k)}\big)=0$. Then, due to Lemma~\ref{lemAS0} the idempotents~$A_{(k)}$ and $A'_{(k)}$ are left-equivalent. Analogously, we obtain $\big(1-S_{(k)}\big)S'_{(k)}=0$, $\big(1-S'_{(k)}\big)S_{(k)}=0$. By~Lemma~\ref{lemAS0} this implies the left equivalence of $1-S_{(k)}$ and $1-S'_{(k)}$, which means in turn the right equivalence of $S_{(k)}$ and $S'_{(k)}$ by Proposition~\ref{PropLRE}. \end{proof}

Let $A,A'\in\End(\CC^n\otimes\CC^n)$ be left-equivalent idempotents and let $S_{(k)}$, $A_{(k)}$ and $S'_{(k)}$, $A'_{(k)}$ be the corresponding pairing operators. From Proposition~\ref{propSAequiv} we obtain
\begin{align*}
 A'_{(k)}=G_{[k]}A_{(k)}, \qquad
 S'_{(k)}=S_{(k)}G_{(k)}
\end{align*}
for some $G_{[k]},G_{(k)}\in\Aut\big((\CC^n)^{\otimes k}\big)$. Then, by means of Proposition~\ref{propIsom} the identification $\gX_A(\CC)_k=\gX_{A'}(\CC)_k$ induces an isomorphism $\big((\CC^n)^{\otimes k}\big)^*S_{(k)}\cong\big((\CC^n)^{\otimes k}\big)^*S'_{(k)}$. Explicitly it has the form
\begin{align*}
\xi\mapsto\xi'=\xi G_{(k)}=\xi S'_{(k)}, \qquad\xi\in\big((\CC^n)^{\otimes k}\big)^*S_{(k)}.
\end{align*}
Indeed, we have $\xi'\in\big((\CC^n)^{\otimes k}\big)^*S'_{(k)}$ and $\xi'(X\otimes\cdots\otimes X)=\xi S'_{(k)}(X\otimes\cdots\otimes X)=\xi(X\otimes\cdots\otimes X)$. The inverse map is $\xi'\mapsto\xi' G_{(k)}^{-1}=\xi S_{(k)}$.
Analogously, the identification $\Xi_A(\CC)_k=\Xi_{A'}(\CC)_k$ gives $A_{(k)}\big((\CC^n)^{\otimes k}\big)\cong A'_{(k)}\big((\CC^n)^{\otimes k}\big)$,
\begin{align*}
 \pi\mapsto\pi'=G_{[k]}\pi=A'_{(k)}\pi, \qquad
 \pi\in A_{(k)}\big((\CC^n)^{\otimes k}\big).
\end{align*}

Proposition~\ref{propSVequiv} implies the equalities of subspaces: $S_{(k)}\big((\CC^n)^{\otimes k}\big)=S'_{(k)}\big((\CC^n)^{\otimes k}\big)$ and $\big((\CC^n)^{\otimes k}\big)^*A_{(k)}=\big((\CC^n)^{\otimes k}\big)^*A'_{(k)}$. Proposition~\ref{propIsom} in turn gives the isomorphisms of vector spaces $\gX^*_A(\CC)_k\cong\gX^*_{A'}(\CC)_k$ and $\Xi^*_A(\CC)_k\cong\Xi^*_{A'}(\CC)_k$ (these are not homomorphisms of algebras). They are given by the formulae $x^A_{i_1}\cdots x^A_{i_k}\mapsto x^{A'}_{i_1}\cdots x^{A'}_{i_k}$ and $\psi_A^{i_1}\cdots\psi_A^{i_k}\mapsto\psi_{A'}^{i_1}\cdots\psi_{A'}^{i_k}$, where $x^A_i$, $x^{A'}_i$, $\psi^i_A$ and $\psi^i_{A'}$ are the generators of the algebras $\gX^*_A(\CC)$, $\gX^*_{A'}(\CC)$, $\Xi^*_A(\CC)$ and $\Xi^*_{A'}(\CC)$ respectively.

Let $(v_\alpha)_{\alpha=1}^{d_k}$, $(v^\alpha)_{\alpha=1}^{d_k}$ and $(w_\alpha)_{\alpha=1}^{r_k}$, $(w^\alpha)_{\alpha=1}^{r_k}$ be dual bases of $S_{(k)}\big((\CC^m)^{\otimes k}\big)$, $\big((\CC^m)^{\otimes k}\big)^*S_{(k)}$, $A_{(k)}\big((\CC^m)^{\otimes k}\big)$ and $\big((\CC^m)^{\otimes k}\big)^*A_{(k)}$ respectively. Then
\begin{gather}
v'_\alpha=v_\alpha, \qquad
(v')^\alpha=v^\alpha G_{(k)}=v^\alpha S'_{(k)}, \qquad
\alpha=1,\dots,d_k, \nonumber
\\
(w')^\alpha=w^\alpha, \qquad
w'_\alpha=G_{[k]}w_\alpha=A'_{(k)}w_\alpha, \qquad
\alpha=1,\dots,r_k, \label{walphaG}
\end{gather}
are dual bases of $S'_{(k)}\big((\CC^m)^{\otimes k}\big)$, $\big((\CC^m)^{\otimes k}\big)^*S'_{(k)}$, $\big((\CC^m)^{\otimes k}\big)^*A'_{(k)}$ and $A'_{(k)}\big((\CC^m)^{\otimes k}\big)$.

By substituting $(v')^\alpha$ and $w'_\alpha$ to the formulae~\eqref{xalphak} and \eqref{psikalpha} we obtain the same bases of~$\gX_{A'}(\CC)_k=\gX_A(\CC)_k$ and $\Xi_{A'}(\CC)_k=\Xi_A(\CC)_k$:
\begin{gather}
 (v')^\alpha(X^*\otimes\cdots\otimes X^*)=v^\alpha S'_{(k)}(X^*\otimes\cdots X^*)=v^\alpha(X^*\otimes\cdots X^*)=x_{(k)}^{\alpha}, \nonumber
 \\
 (\Psi\otimes\cdots\otimes\Psi)w'_\alpha= (\Psi\otimes\cdots\otimes\Psi)A'_{(k)}w_\alpha=(\Psi\otimes\cdots\otimes\Psi)w_\alpha=\psi_\alpha^{(k)}. \label{psialphaLE}
\end{gather}
Note that the bases of $\gX^*_{A'}(\CC)_k$ and $\gX_A(\CC)_k$ defined by the formula~\eqref{xklapha} for $v'_\alpha=v_\alpha$ are not identified since they are elements of different algebras. The same is valid for the bases~\eqref{psialphak} of~$\Xi_{A'}(\CC)_k$ and $\Xi_A(\CC)_k$.

Let $\wt S_{(k)}$, $\wt A_{(k)}$ and $\wt S'_{(k)}$, $\wt A'_{(k)}$ be the pairing operators for two left-equivalent idempotents $\wt A,\wt A'\in\End(\CC^m\otimes\CC^m)$. We have $\wt S'_{(k)}=\wt S_{(k)}\wt G_{(k)}$ and $\wt A'_{(k)}=\wt G_{[k]}\wt A_{[k]}$ for some matrices $\wt G_{(k)},\wt G_{[k]}\in\Aut\big((\CC^m)^{\otimes k}\big)$. Let $M$ be an $\big(A,\wt A\big)$-Manin matrix. Due to Proposition~\ref{propMMlequiv} this means that $M$ is an $\big(A',\wt A'\big)$-Manin matrix. We can consider different minor operators for~$M$, but they are related by a multiplication of a complex matrix in the following way.

\begin{Prop} \label{propMinlequiv}
 For any $\big(A,\wt A\big)$-Manin matrix we have
\begin{gather}
\Min_{\wt S'_{(k)}}M=\big(\Min_{\wt S_{(k)}}M\big)\wt G_{(k)}
=\big(\Min_{\wt S_{(k)}}M\big)\wt S'_{(k)}, \nonumber
\\
\Min^{A'_{(k)}}M=G_{[k]}\big(\Min^{A_{(k)}}M\big)
=A'_{(k)}\big(\Min^{A_{(k)}}M\big). \label{MinAGA}
\end{gather}
\end{Prop}

\begin{proof} The formulae are obtained by the definitions of the minors:
\begin{gather*}
\Min_{\wt S'_{(k)}}M=M^{(1)}\cdots M^{(k)}\wt S'_{(k)} =M^{(1)}\cdots M^{(k)}\wt S_{(k)}\wt G_{(k)}=\big(\Min_{\wt S_{(k)}}M\big)\wt G_{(k)},
\\
\Min_{\wt S'_{(k)}}M=M^{(1)}\cdots M^{(k)}\wt S'_{(k)}\wt S'_{(k)}= \big(\Min_{\wt S_{(k)}}M\big)\wt S'_{(k)},
\\
\Min^{A'_{(k)}}M=A'_{(k)}M^{(1)}\cdots M^{(k)}=G_{[k]}A_{(k)}M^{(1)}\cdots M^{(k)}=G_{[k]}\big(\Min^{A_{(k)}}M\big),
\\
\Min^{A'_{(k)}}M=A'_{(k)}A'_{(k)}M^{(1)}\cdots M^{(k)}= A'_{(k)}\big(\Min^{A_{(k)}}M\big).\tag*{\qed}
\end{gather*}
\renewcommand{\qed}{}
\end{proof}

Note that an $n^k\times m^k$ matrix is a $\big(A_{(2k)},\wt A_{(2k)}\big)$-Manin matrix iff it is an $\big(A'_{(2k)},\wt A'_{(2k)}\big)$-Manin matrix. It is a $\big(1-S_{(k)},1-\wt S_{(k)}\big)$-Manin matrix iff it is a $\big(1-S'_{(2k)},1-\wt S'_{(2k)}\big)$-Manin matrix. Due to Proposition~\ref{propManin2k} the both matrices $\Min_{\wt S_{(k)}}M$ and $\Min_{\wt S'_{(k)}}M$ are $\big(1-S_{(k)},1-\wt S_{(k)}\big)$-Manin matrices as well as $\big(1-S'_{(2k)},1-\wt S'_{(2k)}\big)$-Manin matrices. They are related by the change of basis in the space $(\CC^m)^{\otimes k}$ corresponding to the matrix $\wt G_{(k)}^{-1}$ (see Section~\ref{secABmanin}). In~the same way the matrices $\Min^{A_{(k)}}M$ and $\Min^{A'_{(k)}}M$ are $\big(A_{(k)},\wt A_{(k)}\big)$-Manin matrices as well as $\big(A'_{(2k)},\wt A'_{(2k)}\big)$-Manin matrices, and they are related by the change of basis in the space $(\CC^n)^{\otimes k}$ corresponding to the matrix $G_{[k]}$.

Let $(\wt v_\alpha)_{\alpha=1}^{\wt d_k}$, $(\wt v^\alpha)_{\alpha=1}^{\wt d_k}$ and $(\wt w_\alpha)_{\alpha=1}^{\wt r_k}$, $(\wt w^\alpha)_{\alpha=1}^{\wt r_k}$ be dual bases of $\wt S_{(k)}\big((\CC^m)^{\otimes k}\big)$, $\big((\CC^m)^{\otimes k}\big)^*\wt S_{(k)}$, $\wt A_{(k)}\big((\CC^m)^{\otimes k}\big)$ and $\big((\CC^m)^{\otimes k}\big)^*\wt A_{(k)}$. Let
\begin{gather*}
 \wt v'_\alpha=\wt v_\alpha, \qquad (\wt v')^\alpha=\wt v^\alpha\wt G_{(k)}=\wt v^\alpha\wt S'_{(k)}, \qquad \alpha=1,\dots,\wt d_k,
 \\
 (\wt w')^\alpha=\wt w^\alpha, \qquad \wt w'_\alpha=\wt G_{[k]}\wt w_\alpha=\wt A'_{(k)}\wt w^\alpha, \qquad \alpha=1,\dots,\wt r_k.
\end{gather*}

\begin{Prop} \label{propMinabLE}
{\samepage Consider the entries of minor operators~\eqref{MinSab} and \eqref{MinAab} in the bases defined above:
\begin{gather*}
\big(\Min_{\wt S_{(k)}} M\big)^\alpha_\beta=v^\alpha\big(\Min_{\wt S_{(k)}} M\big)\wt v_\beta, \qquad \big(\Min_{\wt S'_{(k)}} M\big)^\alpha_\beta=(v')^\alpha\big(\Min_{\wt S'_{(k)}} M\big)\wt v'_\beta,
\\
\big(\Min^{A_{(k)}} M\big)^\gamma_\delta=w^\gamma\big(\Min^{A_{(k)}} M\big)\wt w_\delta, \qquad \big(\Min^{A'_{(k)}} M\big)^\gamma_\delta=(w')^\gamma\big(\Min^{A'_{(k)}} M\big)\wt w'_\delta.
\end{gather*}}
These entries coincide:
\begin{align*}
 \big(\Min_{\wt S_{(k)}}M\big)^\alpha_\beta= \big(\Min_{\wt S'_{(k)}}M\big)^\alpha_\beta, \qquad
 \big(\Min^{A_{(k)}}M\big)^\gamma_\delta= \big(\Min^{A'_{(k)}}M\big)^\gamma_\delta.
\end{align*}
\end{Prop}

\begin{proof} By using $\big(\Min^{A'_{(k)}}M\big)\wt A'_{(k)}=\big(\Min^{A'_{(k)}}M\big)$, Proposition~\ref{propMinlequiv} and $(w')^\gamma A'_{(k)}=(w')^\gamma$ we obtain
\begin{align*}
\big(\Min^{A'_{(k)}}M\big)_\delta^\gamma
&=(w')^\gamma\big(\Min^{A'_{(k)}}M\big)\wt w'_\delta
=(w')^\gamma\big(\Min^{A'_{(k)}}M\big)\wt A'_{(k)}\wt w_\delta
\\[.5ex]
&= (w')^\gamma\big(\Min^{A'_{(k)}}M\big)\wt w_\delta
=(w')^\gamma A'_{(k)}\big(\Min^{A_{(k)}}M\big)\wt w_\delta
\\[.5ex]
&=(w')^\gamma\big(\Min^{A_{(k)}}M\big)\wt w_\delta
=\big(\Min^{A_{(k)}}M\big)_\delta^\gamma.
\end{align*}
The equality of the entries of the $S$-minors is proved similarly. \end{proof}

\subsection{Construction of pairing operators}
\label{secConstr}

Theorem~\ref{ThPOviaDB} gives a formula for the pairing operators via dual bases. However there is basisless method of construction. It uses the representation theory of groups and algebras, for which some appropriate idempotents are already constructed.

The operators $A^{(1,2)},A^{(2,3)},\dots,A^{(k-1,k)}\in\End\big((\CC^n)^{\otimes k}\big)$ generate a subalgebra $\U_k$ of the algebra $\End\big((\CC^n)^{\otimes k}\big)$. Equivalently, the algebra $\U_k$ can be defined as a subalgebra of $\End\big((\CC^n)^{\otimes k}\big)$ generates by $P^{(a,a+1)}$ or $S^{(a,a+1)}$.
Let $\U_k^+$ and $\U_k^-$ be ideals of $\U_k$ generated by $A^{(a,a+1)}$ and~$S^{(a,a+1)}$ respectively.
The commutation relations for the algebras $\gX^*_A(\CC)$, $\gX_A(\CC)$, $\Xi_A(\CC)$ and~$\Xi_A^*(\CC)$ imply
\begin{gather}
 (X^*\otimes\cdots\otimes X^*)T=0, \qquad
 T(X\otimes\cdots\otimes X)=0\qquad \forall\,T\in\U_k^+, \nonumber
 \\[.5ex]
 (\Psi\otimes\cdots\otimes\Psi)T=0, \qquad
 T(\Psi^*\otimes\cdots\otimes\Psi^*)=0\qquad \forall\,T\in\U_k^-.\label{XXXT} 
\end{gather}
Thus, if $S_{(k)}\in\U_k$ satisfies~\eqref{PSSPS} and $1-S_{(k)}\in \U_k^+$ then $S_{(k)}$ is the $k$-th $S$-operator. Analogously, an operator $A_{(k)}\in\U_k$ satisfying~\eqref{PAAPA} and $1-A_{(k)}\in\U_k^-$ is the $k$-th $A$-operator. If~the alge\-bra~$\U_k$ admits an involution $\omega\colon P^{(a,a+1)}\mapsto-P^{(a,a+1)}$ then $A_{(k)}=\omega\big(S_{(k)}\big)$, so by using this involution one can obtain the $k$-th $A$-operator from the $k$-th $S$-operator and vice versa.

Let us consider a case when the algebra $\U_k$ is a group algebra of a finite group.

\begin{Prop} \label{propG}
Let $G^+_k$ and $G^-_k$ be subgroups of $\End\big((\CC^n)^{\otimes k}\big)$ generated by the operators $P^{(1,2)},P^{(2,3)},\dots,P^{(k-1,k)}$ and $-P^{(1,2)},-P^{(2,3)},\dots,-P^{(k-1,k)}$ respectively. The group $G^+_k$ is finite iff the group $G^-_k$ is finite. In this case the pairing operators exist and have the form
\begin{align}
 S_{(k)}=\frac1{|G^+_k|}\sum_{g\in G^+_k}g, \qquad
 A_{(k)}=\frac1{|G^-_k|}\sum_{g\in G^-_k}g. \label{SkG}
\end{align}
\end{Prop}

\begin{proof} Suppose $G_k^+$ is finite. If $-1\in G_k^-$ then $G_k^-=G_k^+\cup\big({-}G_k^+\big)$, so that $G_k^-$ is also finite (more precisely we have $G_k^-=G_k^+$ if $-1\in G_k^+$ or $G_k^-=G_k^+\sqcup\big({-}G_k^+\big)$ if $-1\notin G_k^+$). Hence we can suppose $-1\notin G_k^-$. For brevity we denote $g_a=P^{(a,a+1)}$ and $g^*_a=-g_a=-P^{(a,a+1)}$. The finiteness of the group $G_k^+$ means that there exists $N\in\NN_0$ such that any element $g\in G_k^+$ can be written as $g_{a_1}g_{a_2}\cdots g_{a_m}$ with $m\le N$. Then any product $g^*_{a_1}\cdots g^*_{a_{N+1}}=(-1)^{N+1}g_{a_1}\cdots g_{a_{N+1}}$ can be written as $(-1)^{N+1}g_{a'_1}\cdots g_{a'_m}=(-1)^{N+1+m}g^*_{a'_1}\cdots g^*_{a'_m}$ for some $m\le N$. By using $(g^*_s)^2=1$ we obtain $(-1)^{N+1+m}=g^*_{a_1}\cdots g^*_{a_{N+1}}g^*_{a'_m}\cdots g^*_{a'_1}\in G_k^-$. Since $-1\notin G_k^-$ we have $(-1)^{N+1+m}=1$, so for any $a_1,\dots,a_N,a_{N+1}$ there exist $m\le N$ and $a'_1,\dots,a'_m$ such that $g^*_{a_1}\cdots g^*_{a_{N+1}}=g^*_{a'_1}\cdots g^*_{a'_m}$. This means that by induction we can write any element of $G_k^-$ as a product $g^*_{a_1}\cdots g^*_{a_m}$ for some $m\le N$, which implies the finiteness of the group $G_k^-$. The converse implication is obtained by changing the sign of $P$.

We have $\U_k=\CC[G_k^+]=\CC[G_k^-]$. Moreover, $1-g\in\U_k^\pm$ for any $g\in G_k^\pm$. Indeed, for $g=\pm P^{(a,a+1)}$ it follows from the definition of the ideals $\U_k^\pm$; further, if $1-g$ and $1-g'$ belongs to $\U_k^\pm$, then $1-gg'=(1-g)+(1-g')-(1-g)(1-g')\in\U_k^\pm$. Note that the operators~\eqref{SkG} satisfy $gS_{(k)}=S_{(k)}g=S_{(k)}$ $\;\forall\,g\in G_k^+$ and $AS_{(k)}=A_{(k)}g=A_{(k)}$ $\forall g\in G_k^-$, hence conditions~\eqref{PSSPS} and~\eqref{PAAPA} are valid. Since $1-S_{(k)}=\frac1{|G_k^+|}\sum_{g\in G_k^+}(1-g)\in\U_k^+$ and $1-A_{(k)}=\frac1{|G_k^-|}\sum_{g\in G_k^-}(1-g)\in\U_k^-$, the operators $S_{(k)}$ and $A_{(k)}$ are the pairing operators for the idempotent $A$. \end{proof}

\looseness=1
At $k=2$ the groups consist of exactly two elements: $G_2^+=\{1,P\}$, $G^-_2=\{1,-P\}$. Hence we obtain $S_{(2)}=\frac12(1+P)=S$ and $A_{(2)}=\frac12(1-P)=A$, which is in accordance with Proposition~\ref{propUnique}.

Let $\gU_k$ be an abstract algebra with an augmentation $\varepsilon\colon\gU_k\to\CC$. We call an element $s_{(k)}$ {\it left invariant} or {\it right invariant} (with respect to $\varepsilon$) if $us_{(k)}=\varepsilon(u)s_{(k)}\;\forall\,u\in\gU_k$ or $s_{(k)}u=\varepsilon(u)s_{(k)}\;\forall\,u\in\gU_k$ respectively. If an element $s_{(k)}\in\gU_k$ is left or right invariant and normalised as $\varepsilon(s_{(k)})=1$, then it is an idempotent.

Let $\rho\colon\gU_k\to\End\big((\CC^n)^{\otimes k}\big)$ be an algebra homomorphism (representation) satisfying the condition $\U_k\subset\rho(\gU_k)$. By applying it to the conditions of left and right invariance of $s_{(k)}$ we see that $S_{(k)}=\rho(s_{(k)})$ satisfy~\eqref{PSSPS}, if we additionally suppose that $\varepsilon(u_1)=\dots=\varepsilon(u_{k-1})=0$ for some $u_1,\dots,u_{k-1}\in\gU_k$ such that $\rho(u_a)=A^{(a,a+1)}$. In the case $\rho(\gU_k)=\U_k$ we have $S_{(k)}\in\U_k$, so due to the formulae~\eqref{XXXT} the operator $S_{(k)}$ is the $k$-th $S$-operator. In more general case one need to check the conditions~\eqref{SkXXXXXX1}, \eqref{SkXXXXXX2}; due to the normalisation $\varepsilon(s_{(k)})=1$ it is enough to~show that the operators $\rho(u)-\varepsilon(u)$ annihilate $(X^*\otimes\cdots\otimes X^*)$ and $(X\otimes\cdots\otimes X)$ by acting from the right and left respectively, where $u$ runs over the algebra $\gU_k$ or at least over a set of its generators.

The pairing operators $A_{(k)}$ are obtained in the same way. Usually one needs to consider the same $\gU_k$, $\varepsilon$, $s_{(k)}$ with different representation $\rho$ or the same representation with different $s_{(k)}$ and~$\varepsilon$. Let us formulate a general statement.

\begin{Th} \label{ThUkrhosk}
 Let $\rho\colon\gU_k\to\End\big((\CC^n)^{\otimes k}\big)$ and $\varepsilon\colon\gU_k\to\CC$ be algebra homomorphisms. Let $s_{(k)}\in\gU_k$ be a normalised left and right invariant element:
\begin{align}
 us_{(k)}=s_{(k)}u=\varepsilon(u)s_{(k)}\qquad\forall\,u\in\gU_k, \qquad\varepsilon(s_{(k)})=1. \label{usk}
\end{align}
\begin{itemize}
\itemsep=0pt
 \item Suppose there exist elements $u_1,\dots,u_k\in\gU_k$ such that $\rho(u_a)=A^{(a,a+1)}$ and $\varepsilon(u_a)=0$. If~$\rho$ and $\varepsilon$ satisfy
\begin{gather}
 (X^*\otimes\cdots\otimes X^*)\rho(u)=\varepsilon(u)(X^*\otimes\cdots\otimes X^*), \label{Xpi}
 \\
 \rho(u)(X\otimes\cdots\otimes X)=\varepsilon(u)(X\otimes\cdots\otimes X) \label{piX}
\end{gather}
for all $u\in\gU_k$, then $S_{(k)}=\rho(s_{(k)})\in\End\big((\CC^n)^{\otimes k}\big)$ is the $k$-th $S$-operator.
\item Suppose there exist elements $u_1,\dots,u_k\in\gU_k$ such that $\rho(u_a)=S^{(a,a+1)}$ and $\varepsilon(u_a)=0$. If~$\rho$ and $\varepsilon$ satisfy
{\samepage\begin{gather}
 (\Psi\otimes\cdots\otimes\Psi)\rho(u) =\varepsilon(u)(\Psi\otimes\cdots\otimes\Psi), \label{Psipi}
 \\
 \rho(u)(\Psi^*\otimes\cdots\otimes\Psi^*) =\varepsilon(u)(\Psi^*\otimes\cdots \otimes\Psi^*) \label{piPsi}
\end{gather}
for all $u\in\gU_k$, then $A_{(k)}=\rho(s_{(k)})\in\End\big((\CC^n)^{\otimes k}\big)$ is the $k$-th $A$-operator.}
\end{itemize}
\end{Th}

\begin{Rem}
 If $\gX^*_A(\CC)_k\ne0$ or $\gX_A(\CC)_k\ne0$, then $X^*\otimes\cdots\otimes X^*\ne0$ or $X\otimes\cdots\otimes X\ne0$. Analogously, $\Xi_A(\CC)_k\ne0$ and $\Xi^*_A(\CC)$ imply $\Psi\otimes\cdots\otimes\Psi\ne0$ and $\Psi^*\otimes\cdots\otimes\Psi^*\ne0$ respectively. In these cases the condition $\varepsilon(u_a)=0$ can be derived from~\eqref{Xpi}, \eqref{piX}, \eqref{Psipi} or \eqref{piPsi} by the substitution $u=u_a$. However, to prove the non-vanishing of a component of a quadratic algebra (and, more generally, to find its dimension) one needs sometimes to construct the corresponding pairing operator by using Theorem~\ref{ThUkrhosk}.
\end{Rem}

\begin{Rem}
 The augmentation $\varepsilon\colon\gU_k\to\CC$ defines the ideal $\mathfrak I_k:=\Ker(\varepsilon)$ of $\gU_k$. It is a~maximal ideal consisting of the elements $u-\varepsilon(u)$, where $u\in\gU_k$. In terms of this ideal the conditions~\eqref{usk} take the form $us_{(k)}=s_{(k)}u=0$ $\;\forall\,u\in\mathfrak I_k$ and $1-s_{(k)}\in\mathfrak I_k$.

Conversely, for any maximal ideal $\mathfrak I_k$ of a finite-dimensional algebra $\gU_k$ there is a unique algebra isomorphism $\gU_k/\mathfrak I_k\cong\CC$, so the canonical projection $\gU_k\twoheadrightarrow\gU_k/\mathfrak I_k$ defines an augmentation $\varepsilon\colon\gU_k\to\CC$. In particular, if the ideal $\U_k^\pm\subset\U_k$ is proper, it gives an augmentation $\U_k\twoheadrightarrow\U_k/\U_k^\pm\cong\CC$.
\end{Rem}

If the algebra $\gU_k$ admits an anti-automorphism which does not change generators then it is enough to check only left invariance (or only right invariance) due to the following fact.

\begin{Prop} \label{propInv}
Let $\varepsilon\colon\gU_k\to\CC$ be a homomorphism.
\begin{itemize}
\itemsep=0pt
 \item Let $s_{(k)}$ be left invariant and $\bar s_{(k)}$ be right invariant elements of the algebra $\gU_k$. If $\varepsilon(s_{(k)})=\varepsilon(\bar s_{(k)})=1$, then $s_{(k)}=\bar s_{(k)}$.
 \item Let $s_{(k)}\in\gU_k$ be left-invariant and $\varepsilon(s_{(k)})=1$. If there exists an anti-automorphism $\alpha\colon\gU_k\to\gU_k$ such that $\varepsilon\alpha=\varepsilon$ then $\bar s_{(k)}:=\alpha(s_{(k)})$ is right invariant and $\varepsilon(\bar s_{(k)})=1$. Hence $s_{(k)}=\bar s_{(k)}$, and $s_{(k)}$ is right invariant as well.
\item If a solution of~\eqref{usk} exists then it is unique.
\end{itemize}
\end{Prop}

\begin{proof} The first part is proved as Proposition~\ref{propUnique}, that is $\bar s_{(k)}s_{(k)}=\varepsilon(\bar s_{(k)})s_{(k)}=s_{(k)}$, $\bar s_{(k)}s_{(k)}=\varepsilon(s_{(k)})\bar s_{(k)}=\bar s_{(k)}$. The second part follows from the fact that $\bar s_{(k)}=\alpha(s_{(k)})$ is right invariant with respect to the augmentation $\varepsilon\alpha^{-1}=\varepsilon$. \end{proof}

Consider the case when $P$ satisfies the braid relation $P^{(12)}P^{(23)}P^{(12)}=P^{(23)}P^{(12)}P^{(23)}$. Since $P^2=1$ we have the homomorphisms $\rho^\pm\colon\CC[\SSS_k]\to\End\big((\CC^n)^{\otimes k}\big)$ defined by the formulae $\rho^\pm(\sigma_a)=\pm P^{(a,a+1)}$. The role of $\gU_k$ is played by $\CC[\SSS_k]$. Since $G_k^\pm=\rho^\pm(\SSS_k)$ the groups $G_k^+$ and~$G_k^-$ are finite. The pairing operator can be obtained by Proposition~\ref{propG} or by Theorem~\ref{ThUkrhosk}. The augmentation $\varepsilon\colon\CC[\SSS_k]\to\CC$ is the counit $\varepsilon(\sigma)=1$ $\;\forall\,\sigma\in\SSS_k$. The elements $u_a$ have the form $\frac{1-\sigma_a}2$. The operators~\eqref{SkG} coincide with the image of $s_{(k)}=\frac1{k!}\sum_{\sigma\in\SSS_k}\sigma$ under the homomorphisms $\rho^+$ and $\rho^-$, these are
\begin{align}
S_{(k)}=\frac1{k!}\sum_{\sigma\in\SSS_k}\rho^+(\sigma), \qquad
A_{(k)}=\frac1{k!}\sum_{\sigma\in\SSS_k}(-1)^\sigma\rho^+(\sigma). \label{POBR}
\end{align}

Note that $A_{(k)}$ can be obtained as the image of $a_{(k)}=\frac1{k!}\sum_{\sigma\in\SSS_k}(-1)^\sigma\sigma$ under $\rho^+$. In this case one need to consider the augmentation $\varepsilon(\sigma_a)=-1$ and $u_a=\frac{1+\sigma_a}2$.

\section{Examples of minor operators}
\label{secExMO}

Here we construct pairing operators for the examples given in Section~\ref{secPc} and consider the corresponding minors. Since the Manin matrices described in Sections~\ref{secMm} and \ref{secqM} are particular cases of the $(\wh q,\wh p)$-Manin matrices it is sufficient to consider minors for the case of Section~\ref{secwhqMM}. The formulae for $S$- and $A$-minors of the $(\wh q,\wh p)$-Manin matrices are valid for more general case: for $(B,A_{\wh p})$- and $(A_{\wh q},B)$-Manin matrices respectively. By starting with the idempotents $\wh R^q_{-}$ int\-roduced in Section~\ref{secUq} we write another pairing operators, which gives related minor operators for $q$-Manin matrices. Finally we investigate the case of Section~\ref{sec4p}.

\subsection[Minors for the (q,p)-Manin matrices]
{Minors for the $\boldsymbol{(\wh q,\wh p)}$-Manin matrices}\label{secMinqp}

Consider the idempotent $A=A_{\wh q}$. The pairing operators for $A$ are given by the formulae~\eqref{POBR}, where $\rho^+=\rho_{\wh q} \colon\CC[\SSS_k]\to\End\big((\CC^n)^{\otimes k}\big)$, $\rho_{\wh q}(\sigma_a)= P_{\wh q}^{(a,a+1)}$.

Let $I=(i_1,\dots,i_k)$ and $\wh q_{II}$ be the corresponding $k\times k$ matrix with entries $(\wh q_{II})_{\rm st}=q_{i_si_t}$. Then we have homomorphism $\Xi_{A_{\wh q_{II}}}(\CC)\to\Xi_{A_{\wh q}}(\CC)$ given by the formula $\psi_s\mapsto\psi_{i_s}$. By applying it to~\eqref{Lempsisigma2} one yields\vspace{-1ex}
\begin{align}
 \psi_{i_{\sigma(1)}}\cdots\psi_{i_{\sigma(k)}}=\frac{(-1)^\sigma}{\mu(\wh q_{II},\sigma)}\psi_{i_1}\cdots\psi_{i_k}, \label{Lempsisigma2k}
\end{align}\vspace{-1ex}
where\vspace{-1ex}
\begin{align*}
\mu(\wh q,\sigma):=\!\!\!\!\prod_{\substack{s<t\\ \sigma^{-1}(s)>\sigma^{-1}(t)}}\!\!\!\!q_{\rm st}, \qquad
\mu(\wh q_{II},\sigma)=\!\!\!\!\prod\limits_{\substack{s<t\\ \sigma^{-1}(s)>\sigma^{-1}(t)}}\!\!\!\! q_{i_si_t}.\vspace{-1ex}
\end{align*}
Since $\Xi_{A_{\wh q}}^*(\CC)=\Xi_{A_{\wh{q'}}}(\CC)$, where $q'_{ij}=q_{ij}^{-1}$ we obtain\vspace{-1ex}
\begin{align}
 \psi^{i_{\sigma(1)}}\cdots\psi^{i_{\sigma(k)}}=(-1)^\sigma\mu(\wh q_{II},\sigma)\psi^{i_1}\cdots\psi^{i_k}. \label{Lempsisigma2kDual}\vspace{-1ex}
\end{align}

Denote $e_I:=e_{i_1 \dots i_k}=e_{i_1}\otimes\cdots\otimes e_{i_k}$ and $e^J:=e^{j_1 \dots j_k}=e^{j_1}\otimes\cdots\otimes e^{j_k}$ for arbitrary $I=(i_1,\dots,i_k)$ and $J=(j_1,\dots,j_k)$. Let us write $I=(i_1<\dots<i_k\le n)$ if $I=(i_1,i_2,\dots,i_k)$ such that $1\le i_1<i_2<\dots<i_k\le n$.

Since $\rho_{\wh q}(\sigma)e_I$ is proportional to $e_{\sigma I}$ we have\vspace{-1ex}
\begin{align} \label{psipsiA}
 \la\psi^{j_1}\cdots\psi^{j_k},\psi_{i_1}\cdots\psi_{i_k}\ra=A^{j_1 \dots j_k}_{i_1 \dots i_k}=\frac1{k!}\sum_{\sigma\in\SSS_k}(-1)^\sigma e^J\rho_{\wh q}(\sigma)e_I=\frac1{k!}e^Je_I=\frac1{k!}\delta^J_I\vspace{-1ex}
\end{align}
for any $I=(i_1<\dots<i_k\le n)$ and $J=(j_1<\dots<j_k\le n)$.

The formulae~\eqref{Lempsisigma2k}, \eqref{Lempsisigma2kDual} and \eqref{psipsiA} give the formula for the entries of the $A$-operator:
\begin{align} \label{Awhq}
 A^{i_{\sigma(1)}\dots i_{\sigma(k)}}_{j_{\tau(1)} \dots j_{\tau(k)}}=
 \big\la\psi^{i_{\sigma(1)}}\cdots\psi^{i_{\sigma(k)}},\psi_{j_{\tau(1)}}\cdots\psi_{j_{\tau(k)}}\big\ra =\frac1{k!}(-1)^\tau(-1)^\sigma\frac{\mu(\wh q_{II},\sigma)}{\mu(\wh q_{II},\tau)}\delta_J^I,
\end{align}
where $I=(i_1<\dots<i_k\le n)$, $J=(j_1<\dots<j_k\le n)$ and $\sigma,\tau\in\SSS_k$. Due to the formula~\eqref{Lempsisigma1} the other entries vanish.

Note that the vectors $\psi_{i_1}\cdots\psi_{i_k}$, where $1\le i_1<\dots<i_k\le n$, form a basis of $\Xi_{A_{\wh q}}(\CC)_k$, so its dimension is $r_k=\dim\Xi_{A_{\wh q}}(\CC)_k=\dim\Xi_{A_{\wh q}}^*(\CC)_k={n\choose k}$. By using the formula~\eqref{Awhq} one can directly check that $\tr A_{(k)}=r_k$:
\begin{align} \label{trAwtq}
 \tr A_{(k)}=\sum_{l_1,\dots,l_k=1}^nA^{l_1\dots l_k}_{l_1\dots l_k}=\sum_{\substack{1\le i_1<\dots<i_k\le n\\ \sigma\in\SSS_k}} A^{i_{\sigma(1)} \dots i_{\sigma(k)}}_{i_{\sigma(1)} \dots i_{\sigma(k)}}=\sum_{1\le i_1<\dots<i_k\le n}1={n\choose k}.
\end{align}

Let us calculate the $A$-minors of an $\big(A_{\wh q},\wt A\big)$-Manin matrix $M\in\gR\otimes\Hom(\CC^m,\CC^n)$, where $\wt A\in\End(\CC^m\otimes\CC^m)$ is an arbitrary idempotent. For any $I=(i_1<\dots<i_k\le n)$ and $J=(j_1,\dots,j_k)$, where $j_1,\dots,j_k=1,\dots,m$, we have
\begin{align}
 \big(\Min^{A_{(k)}}M\big)^{i_{\sigma(1)}\dots i_{\sigma(k)}}_{j_1\dots j_k}&=\sum_{l_1,\dots,l_k=1}^n A^{i_{\sigma(1)}\dots i_{\sigma(k)}}_{l_1 \dots l_k}M^{l_1}_{j_1}\cdots M^{l_k}_{j_k}\nonumber
 \\
 &=
 \frac1{k!}\sum_{\tau\in\SSS_k}(-1)^\tau(-1)^\sigma\frac{\mu(\wh q_{II},\sigma)}{\mu(\wh q_{II},\tau)}M^{i_{\tau(1)}}_{j_1}\cdots M^{i_{\tau(k)}}_{j_k}\nonumber
 \\
 &=\frac1{k!} (-1)^\sigma\mu(\wh q_{II},\sigma)\det_{\wh q_{II}}(M_{IJ})\label{AMinqp}
\end{align}
(the other entries vanish). Note that this is in agreement with the formulae~\eqref{DefAMin} and \eqref{Lempsisigma2kDual}. Together with the properties the $\wh q$-determinant with respect to a change of rows formulated in~Corollary~\ref{CordetqIJ} the relation~\eqref{AMinqp} implies
\begin{align} \label{Mindetq}
 \big(\Min^{A_{(k)}}M\big)^{i_1\dots i_k}_{j_1\dots j_k}=
 \frac1{k!}\det_{\wh q_{II}}(M_{IJ})
\end{align}
for any $I=(i_1,\dots,i_k)$ and $J=(j_1,\dots,j_k)$.

Let $\wt A=A_{\wh p}$, then $M$ is a $(\wh q,\wh p)$-Manin matrix. The formulae~\eqref{DefAMin} and~\eqref{Lempsisigma2k} gives us the symmetry with respect to the lower indices in the form
\begin{align*}
\big(\Min^{A_{(k)}}M\big)^{i_1\dots i_k}_{j_{\tau(1)} \dots j_{\tau(k)}}=\frac{(-1)^\tau}{\mu(\wh p_{JJ},\tau)}\big(\Min^{A_{(k)}}M\big)^{i_1\dots i_k}_{j_1\dots j_k},
\end{align*}
where $J=(j_1<\dots<j_k\le m)$ and $i_1,\dots,i_k$ are arbitrary (the entries for other lower indices vanish). This also follows from the formula~\eqref{Mindetq} and Corollary~\ref{CordetqIJ}.

Since $A_{(k)}e_{\sigma I}$ is proportional to $A_{(k)}\rho_{\wh q}(\sigma)e_I=(-1)^\sigma A_{(k)}e_I$ the vectors
\begin{align} \label{wI}
 w_I=A_{(k)}e_I=\frac1{k!}\sum_{\tau\in\SSS_k}(-1)^\tau\mu(\wh q_{II},\tau)e_{i_{\tau(1)} \dots i_{\tau(k)}}, \qquad I=(i_1<\dots<i_k\le n),
\end{align}
form the basis of $A_{(k)}(\CC^n)^{\otimes k}$. The corresponding elements of the dual basis are
\begin{align*}
 w^I=k!e^IA_{(k)}=\sum_{\sigma\in\SSS_k}(-1)^\sigma\frac1{\mu(\wh q_{II},\sigma)}e^{i_{\sigma(1)} \dots i_{\sigma(k)}}.
\end{align*}
The $A$-minors in these bases are exactly $\wh q$-minors of the matrix $M$:
\begin{align*}
 \big(\Min^{A_{(k)}}M\big)^I_J&=w^I\big(\Min^{A_{(k)}}M\big)\wt w_J=k!e^IA_{(k)}M^{(1)}\cdots M^{(k)}\wt A_{(k)}e_J
 =w^IM^{(1)}\cdots M^{(k)}e_J
 \\
 &=\sum_{\sigma\in\SSS_k}(-1)^\sigma\frac1{\mu(\wh q_{II},\sigma)}M^{i_{\sigma(1)}}_{j_1}\cdots M^{i_{\sigma(k)}}_{j_k}=
\det_{\wh q_{II}}(M_{IJ}),
\end{align*}
where $I=(i_1<\dots<i_k\le n)$ and $J=(j_1<\dots<j_k\le m)$, $\wt w_J=\wt A_{(k)}e_J$. Thus we have $\phi^{(k)}_J=\sum_{I=(i_1<\dots<i_k\le n)}\det_{\wh q_{II}}(M_{IJ}) \psi^{(k)}_I$, where
\begin{gather*}
 \phi^{(k)}_J=\frac1{k!}\sum_{\sigma\in\SSS_k}(-1)^\sigma\mu(\wh p_{JJ},\sigma)\phi_{j_{\sigma(1)}}\cdots\phi_{j_{\sigma(k)}}=\phi_{j_1}\cdots\phi_{j_k}, \qquad
\phi_j=\sum_{i=1}^nM^i_j\psi_i,
 \\
\psi^{(k)}_I=\frac1{k!}\sum_{\sigma\in\SSS_k}(-1)^\sigma\mu(\wh q_{II},\sigma) \psi_{i_{\sigma(1)}}\cdots\psi_{i_{\sigma(k)}}=\psi_{i_1}\cdots\psi_{i_k}.
\end{gather*}

To consider $S$-operators for $A=A_{\wh q}$ and corresponding $S$-minors we first note the formula
\begin{align}
 x^{\sigma(1)}\cdots x^{\sigma(n)}=x^1\cdots x^n\!\!\!\!\prod_{\substack{i<j\\ \sigma^{-1}(i)>\sigma^{-1}(j)}}\!\!\!\!q_{ij}, \qquad \sigma\in \SSS_n, \label{xsigma}
\end{align}
It is proved in the same way as~\eqref{Lempsisigma2}. Let $I=(i_1\le\cdots\le i_k\le n)$, i.e., $I=(i_1,\dots,i_k)$ such that $1\le i_1\le\cdots\le i_k\le n$. By applying the homomorphism
 $\gX_{A_{\wh q_{II}}}(\CC)\to\gX_{A_{\wh q}}(\CC)$, $x^s\mapsto x^{i_s}$, to~\eqref{xsigma} we obtain
\begin{align}
 x^{i_{\sigma(1)}}\cdots x^{i_{\sigma(k)}}&=\mu(\wh q_{II},\sigma)x^{i_1}\cdots x^{i_k}. \label{xsigmak}
\end{align}
By changing $x^i\to x_i$, $q_{ij}\to q_{ij}^{-1}$ we derive
\begin{align}
 x_{i_{\sigma(1)}}\cdots x_{i_{\sigma(k)}}&=\frac1{\mu(\wh q_{II},\sigma)}x_{i_1}\cdots x_{i_k}. \label{xsigmakDual}
\end{align}

Recall that for a tuple $I=(i_1,\dots,i_k)$ we defined $\nu_i=|\{s\mid i_s=i\}|$ (see~Section~\ref{secPerm}). As~we mentioned the stabiliser $(\SSS_k)_I$ has order $\nu_1!\cdots \nu_n!$. This means that the group $(\SSS_k)_I$ is generated by $\sigma_s\in(\SSS_k)_I$, since the subgroup of $\SSS_k$ generated by these elements has exactly the order $\nu_1!\cdots \nu_n!$. In other words, we have $(\SSS_k)_I\cong\SSS_{\nu_1}\times\cdots\times\SSS_{\nu_n}$. Denote
\begin{align}
 \nu_I:=\big|(\SSS_k)_I\big|=\nu_1!\cdots \nu_n!. \label{nuI}
\end{align}
By taking into account $P_{\wh q}(e_i\otimes e_i)=e_i\otimes e_i$ we obtain $\rho_{\wh q}(\sigma)e_I=e_I$ for any $\sigma\in(\SSS_k)_I$, so that
\begin{align*}
 \big\la x^{i_1}\cdots x^{i_k},x_{j_1}\cdots x_{j_k}\big\ra=S^{i_1 \dots i_k}_{j_1 \dots j_k}=\frac{\nu_I}{k!}\delta^I_J,
\end{align*}
where $I=(i_1\le\cdots \le i_k\le n)$, $J=(j_1\le\cdots \le j_k\le n)$. Any $k$-tuple is a permutation of some $I=(i_1\le\cdots \le i_k\le n)$, so an arbitrary entry of $S_{(k)}$ has the form
\begin{align} \label{Swhq}
 S_{j_{\tau(1)} \dots j_{\tau(k)}}^{i_{\sigma(1)} \dots i_{\sigma(k)}}=\big\la x^{i_{\sigma(1)}}\cdots x^{i_{\sigma(k)}},x_{j_{\tau(1)}}\cdots x_{j_{\tau(k)}}\big\ra=\frac{\nu_I}{k!}\cdot\frac{\mu(\wh q_{II},\sigma)}{\mu(\wh q_{II},\tau)}\delta^I_J.
\end{align}

Note that $\sigma^{-1}I\equiv(i_{\sigma(1)},\dots,i_{\sigma(k)})=(i_{\tau(1)},\dots,i_{\tau(k)})\equiv\tau^{-1}I$ iff $(\SSS_k)_I\cdot\sigma=(\SSS_k)_I\cdot\tau$. Hence by permuting $I=(i_1\le\cdots\le i_k\le n)$ by all the permutations $\sigma\in\SSS_k$ we obtain $\frac{k!}{\nu_I}$ groups of $\nu_I$ identical tuples. This implies the formula
\begin{align}
\sum_{l_1,\dots,l_k=1}^n C_{l_1 \dots l_k}=\sum_{I=(i_1\le\cdots\le i_k\le n)}\frac1{\nu_I}\sum_{\sigma\in\SSS_k} C_{i_{\sigma(1)} \dots i_{\sigma(k)}}. \label{sumll}
\end{align}

Since the vectors $x^{i_1}\cdots x^{i_k}$, $1\le i_1\le\cdots\le i_k\le n$ form a basis of the vector space $\gX_{A_{\wh q}}(\CC)_k$, its dimension is $d_k=\dim\gX_{A_{\wh q}}(\CC)_k=\dim\gX_{A_{\wh q}}^*(\CC)_k={k+n-1\choose k}$. The equality $\tr S_{(k)}=d_k$ for~\eqref{Swhq} can be checked as
\begin{gather*}
 \tr S_{(k)}=\sum_{l_1,\dots,l_k=1}^nS^{l_1\dots l_k}_{l_1\dots l_k}=\sum_{\substack{I=(i_1\le\cdots\le i_k\le n)\\ \sigma\in\SSS_k}}\frac1{\nu_I} S^{i_{\sigma(1)} \dots i_{\sigma(k)}}_{i_{\sigma(1)} \dots i_{\sigma(k)}}=\sum_{I=(i_1\le\cdots\le i_k\le n)}1={k+n-1\choose k}.
\end{gather*}

Let $M\in\gR\otimes\Hom(\CC^m,\CC^n)$ be a $(B,A_{\wh p})$-Manin matrix and $I=(i_1,\dots,i_k)$, $1\le i_s\le n$, $J=(j_1\le\cdots\le j_k\le m)$. Let $\wt S_{(k)}$ be the $k$-th $S$-operator for $\wt A=A_{\wh p}$. Then by using the formula~\eqref{sumll} we obtain
{\samepage\begin{align*}
 \big(\Min_{\wt S_{(k)}}M\big)^{i_1 \dots i_k}_{j_1 \dots j_k}=\sum_{l_1,\dots,l_k=1}^m M^{i_1}_{l_1}\cdots M^{i_k}_{l_k}\wt S^{l_1 \dots l_k}_{j_1 \dots j_k}=
 \frac{1}{k!}\sum_{\sigma\in\SSS_k}\mu(\wh p_{JJ},\sigma)M^{i_1}_{j_{\sigma(1)}}\cdots M^{i_k}_{j_{\sigma(k)}},
\end{align*}
where $\wt S^{l_1 \dots l_k}_{j_1 \dots j_k}$ are entries of $\wt S_{(k)}$.

}

Let us introduce the $\wh q$-permanent for a $n\times n$ matrix $M$ by the formula
\begin{align*}
 \perm_{\wh q}(M)=\sum_{\sigma\in\SSS_n}M^{1}_{\sigma(1)}\cdots M^{n}_{\sigma(n)}\prod_{\substack{i<j\\ \sigma^{-1}(i)<\sigma^{-1}(j)}}q_{ij},
\end{align*}
Then we can write
\begin{align*}
\big \la f_M(x^{i_{1}}\cdots x^{i_{k}}),\wt x_{j_{1}}\cdots\wt x_{j_{k}}\big\ra= \big(\Min_{\wt S_{(k)}}M\big)^{i_1 \dots i_k}_{j_1 \dots j_k}=\frac{1}{k!} \perm_{\wh p_{JJ}}(M_{IJ}),
\end{align*}
where $\wt x_j$ are the generators of $\gX_{A_{\wh p}}^*(\CC)$. The other entries of the $S$-minors are calculated by means of the formulae~\eqref{xsigmakDual}:\vspace{-1ex}
\begin{align}
\big(\Min_{\wt S_{(k)}}M\big)^{i_1\dots i_k}_{j_{\tau(1)} \dots j_{\tau(k)}}&=\big\la f_M\big(x^{i_1}\cdots x^{i_k}\big),\wt x_{j_{\tau(1)}}\cdots\wt x_{j_{\tau(k)}}\big\ra\nonumber
\\
&= \frac{1}{\mu(\wh p_{JJ},\tau)} \big\la f_M\big(x^{i_{1}}\cdots x^{i_{k}}\big),\wt x_{j_{1}}
\cdots\wt x_{j_{k}}\big\ra\nonumber
\\
&=\frac{1}{k!}\,\frac{1}{\mu(\wh p_{JJ},\tau)}\perm_{\wh p_{JJ}}(M_{IJ}). \label{MinSMperm}
\end{align}

Note that $\wh q$-determinant becomes $\wh q$-permanent of the transposed matrix after the substitution $q_{ij}\to -q_{ij}^{-1}$ ($i\ne j$), so the formula~\eqref{Thdettaurow} gives the corresponding properties of $\wh q$-permanent. In~this way we obtain
\[
\perm_{\wh p_{LL}}(M_{IL})=
\det_{\wh p_{JJ}}(M_{IJ})\prod_{\substack{s<t\\ \tau^{-1}(s)>\tau^{-1}(t)}}p_{i_si_t}^{-1}
\]
for arbitrary matrix $M$ and $k$-tuples $I=(i_1,\dots,i_k)$, $J=(j_1,\dots,j_k)$, $L=(j_{\tau(1)},\dots,j_{\tau(k)})$ (cf.~Corollary~\ref{CordetqIJ}). By using this formula and the relation~\eqref{MinSMperm} we derive
\begin{gather*}
\big(\Min_{\wt S_{(k)}}M\big)^{i_1\dots i_k}_{j_1\dots j_k}=
\frac{1}{k!}\,
\perm_{\wh p_{JJ}}(M_{IJ}).\end{gather*}
for any $I=(i_1,\dots,i_k)$, $J=(j_1,\dots,j_k)$.

Let $B=A_{\wh q}$, so $M$ is a $(\wh q,\wh p)$-Manin matrix. The symmetry of the $S$-minors with respect to the upper indices follows from the formulae~\eqref{xsigmak}:
\begin{gather*}
\big(\Min_{\wt S_{(k)}}M\big)^{i_{\sigma(1)} \dots i_{\sigma(k)}}_{j_1\dots j_k}=\big\la f_M\big(x^{i_{\sigma(1)}}\cdots x^{i_{\sigma(k)}}\big),\wt x_{j_1}\cdots\wt x_{j_k}\big\ra=
 \mu(\wh p_{II},\sigma) \big(\Min_{\wt S_{(k)}}M\big)^{i_1\dots i_k}_{j_1\dots j_k},
\end{gather*}
where $I=(i_1,\dots,i_k)$ and $J=(j_1,\dots,j_k)$ are arbitrary.

The basis of $S_{(k)}(\CC^n)^{\otimes k}$ and its dual are formed by the vectors
\begin{gather*}
 v_I=\frac{k!}{\nu_I}S_{(k)}e_I=\frac1{\nu_I}\sum_{\sigma\in\SSS_k}\mu(\wh q_{II},\sigma)e_{i_{\sigma(1)} \dots i_{\sigma(k)}},
 \\
 v^I=e^IS_{(k)}=\frac1{k!}\sum_{\sigma\in\SSS_k}\frac1{\mu(\wh q_{II},\sigma)}e^{i_{\sigma(1)} \dots i_{\sigma(k)}},
\end{gather*}
where $I=(i_1\le\cdots\le i_k\le n)$. The $S$-minors in these bases are
\begin{align*}
 \big(\Min_{\wt S_{(k)}}M\big)^I_J&=v^I\big(\Min_{\wt S_{(k)}}M\big)\wt v_J=\frac{k!}{\nu_J}e^I S_{(k)}M^{(1)}\cdots M^{(k)}\wt S_{(k)}e_J= e^IM^{(1)}\cdots M^{(k)}\wt v_J
 \\
 &=\frac1{\nu_J}\sum_{\sigma\in\SSS_k}\mu(\wh p_{JJ},\sigma)M^{i_1}_{j_{\sigma(1)}}\cdots M^{i_k}_{j_{\sigma(k)}}=\frac1{\nu_J}
\perm_{\wh p_{JJ}}(M_{IJ}),
\end{align*}
where $I=(i_1\le\cdots\le i_k\le n)$, $J=(j_1\le\cdots\le j_k\le m)$ and $\wt v_J=\frac{k!}{\nu_J}\wt S_{(k)}e_J$. We derive $y_{(k)}^I=\sum_{J=(i_1\le\cdots\le i_k\le m)}\frac1{\nu_J}\perm_{\wh q_{JJ}}(M_{IJ})\; \wt x_{(k)}^J$, where
\begin{gather*}
 y_{(k)}^I=\frac1{k!}\sum_{\sigma\in\SSS_k}\frac1{\mu(\wh q_{II},\sigma)} y^{i_{\sigma(1)}}\cdots y^{i_{\sigma(k)}}=y^{i_1}\cdots y^{i_k}, \qquad
 y^i=\sum_{j=1}^mM^i_j\wt x^j,
 \\
 \wt x_{(k)}^J=\frac1{k!}\sum_{\sigma\in\SSS_k}\frac1{\mu(\wh p_{JJ},\sigma)} \wt x^{j_{\sigma(1)}}\cdots\wt x^{j_{\sigma(k)}}=\wt x^{j_1}\cdots\wt x^{j_k}.
\end{gather*}

Let $M$ be an $n\times m$ $(\wh q,\wh p)$-Manin matrix and $N$ be an $m\times\ell$ $(\wh p,\wh r)$-Manin matrix such that $[M^i_j,N^k_l]=0$. From~\eqref{SMinMNb} and \eqref{AMinMNb} we obtain the formulae
\begin{gather}
 \perm_{\wh r_{LL}}\big((MN)_{IL}\big)=\sum_{J=(j_1\le\cdots\le j_k\le m)} \frac1{\nu_J}\perm_{\wh p_{JJ}}(M_{IJ})\perm_{\wh r_{LL}}(N_{JL}), \label{permMN}
 \\
 \det_{\wh q_{II}}\big((MN)_{IL}\big)=\sum_{J=(j_1<\dots<j_k\le m)}\det_{\wh q_{II}}(M_{IJ})\det_{\wh p_{JJ}}(N_{JL}), \label{detMN}
\end{gather}
where $I=(i_1\le\cdots\le i_k\le n)$ and $L=(l_1\le\cdots\le l_k\le\ell)$ in~\eqref{permMN} and $I=(i_1<\dots<i_k\le n)$ and $L=(l_1<\dots<l_k\le\ell)$ in~\eqref{detMN}.

Now let us demonstrate how the general formulae~\eqref{SMinPerm}, \eqref{AMinPerm} can be applied to $\wh q$-permanent and $\wh q$-determinant. Let $M$ be an $n\times n$ $(\wh q,\wh p)$-Manin matrix. According to Proposition~\ref{PropPermqp} the matrix $^\sigma_\tau M=\sigma M\tau^{-1}$ is a $\big(\sigma\wh q\sigma^{-1},\tau\wh p\tau^{-1}\big)$-Manin matrix. Let us use the formulae~\eqref{SMinPerm}, \eqref{AMinPerm} for $M$ and $\sigma,\tau\in\SSS_n$ (this is the case $m=n=k$). The formula~\eqref{sigmaPwhq} implies $\iota_\tau\big(\rho_{\wh q}(\sigma)\big)=\tau^{\otimes k}\rho_{\wh q}(\sigma)\big(\tau^{\otimes k}\big)^{-1}=\rho_{\tau\wh q\tau^{-1}}(\sigma)$. Hence $\iota_\tau\wt S_{(n)}$ and $\iota_\sigma A_{(n)}$ are the $n$-th $S$- and $A$-operators for~$A_{\tau\wh p\tau^{-1}}$ and $A_{\sigma\wh q\sigma^{-1}}$ respectively.
 Note that $\big(\sigma^{\otimes n}T(\tau^{\otimes n})^{-1}\big)^{i_1 \dots i_n}_{j_1 \dots j_n}=T^{\sigma^{-1}(i_1)\dots\sigma^{-1}(i_n)}_{\tau^{-1}(j_1)\dots\tau^{-1}(j_n)}$. By taking $i_s=s$ and $j_s=s$ we obtain
\begin{gather}
\perm_{\tau\wh p\tau^{-1}}(\sigma M\tau^{-1})= \prod_{\substack{s<t\\ \sigma(s)>\sigma(t)}}\!\!\!\!q_{\rm st}
\prod_{\substack{s<t\\ \tau(s)>\tau(t)}}\!\!\!\!p_{\rm st}^{-1}\perm_{\wh p}(M), \label{permPermqp}
\\
\det_{\sigma\wh q\sigma^{-1}}(\sigma M\tau^{-1})= \prod_{\substack{s<t\\ \sigma(s)>\sigma(t)}}\!\!\!\!q_{\rm st}
\prod_{\substack{s<t\\ \tau(s)>\tau(t)}}\!\!\!\!p_{\rm st}^{-1}\det_{\wh q}(M). \label{detPermqp}
\end{gather}

Consider the $2\times2$ example of a $(\wh q,\wh p)$-Manin matrix:
\begin{gather*}
M=\begin{pmatrix} a & b \\ c & d \end{pmatrix}\!, \qquad
\wh q=\begin{pmatrix} 1 & q \\ q^{-1} & 1 \end{pmatrix}\!, \qquad
\wh p=\begin{pmatrix} 1 & p \\ p^{-1} & 1 \end{pmatrix}\!.
\end{gather*}
The formulae~\eqref{Mqpikjk}, \eqref{Mqpikjl} have the form
\begin{gather*}
ac=q^{-1}ca,\qquad
bd=q^{-1}db,\qquad
ad-q^{-1}pda+pbc-q^{-1}cb=0.
\end{gather*}
The $\wh q$-determinant and $\wh p$-permanent have the from
\begin{gather*}
\det_{\wh q}M=ad-q^{-1}cb=\det_qM, \qquad
\perm_p M:=\perm_{\wh p}M=ad+pbc.
\end{gather*}
Let $\sigma=\tau=\sigma_{12}\in\SSS_2$. Since $\sigma M=\left(\begin{smallmatrix} c & d \\ a & b \end{smallmatrix}\right)$, $M\tau^{-1}=\left(\begin{smallmatrix} b & a \\ d & c \end{smallmatrix}\right)$ we obtain
\begin{gather*}
 \det_q(M\tau^{-1})=bc-q^{-1}da=-p^{-1}\det_qM, \\
 \perm_{p^{-1}}(M\tau^{-1})=bc+p^{-1}ad=p^{-1}\perm_pM,
 \\
 \det_{q^{-1}}(\sigma M)=cb-qad=-q\det_qM, \qquad
 \perm_p(\sigma M)=cb+pda=q\perm_pM.
\end{gather*}
These formulae are particular cases of~\eqref{permPermqp}, \eqref{detPermqp}.

\subsection[Pairing operators for q-Manin matrices from Hecke algebras]
{Pairing operators for $\boldsymbol q$-Manin matrices from Hecke algebras}\label{secHecke}

The formulae of the Section~\ref{secMinqp} at $\wh q=q^{[n]}$ and $\wh p=q^{[m]}$ give the corresponding formulae for the $q$-Manin case. Here we construct the minors for the $q$-Manin matrices by using the idempotent~\eqref{whRqm} by supposing that $q\in\CC\backslash\{0\}$ is not a root of unity. In particular, this condition implies that the $q$-numbers
{\samepage\begin{gather*}
 k_q:=\frac{q^k-q^{-k}}{q-q^{-1}}=q^{k-1}+q^{k-3}+q^{k-5}+\dots+q^{1-k}
\end{gather*}
do not vanish for any $k\in\ZZ_{\ge1}$.}

Remind that due to the formula~\eqref{RAn} the idempotents $A=\wh R^q_{n-}=2_q^{-1}\big(q^{-1}-\wh R^q_n\big)$ and $A'=A^q_n$ are left-equivalent, so the quadratic algebras for these idempotents coincide with each other: $\gX_{\wh R^q_{n-}}(\CC)=\gX_{A^q_n}(\CC)$ and $\Xi_{\wh R^q_{n-}}(\CC)=\Xi_{A^q_n}(\CC)$. The ``dual'' quadratic algebras do not coincide. Namely, since the idempotent $\wh R^q_{n-}$ is right-equivalent to $A^{q^{-1}}_n$ (the formula~\eqref{RAnRE}) we have $\gX^*_{\wh R^q_{n-}}(\CC)=\gX^*_{A^{q^{-1}}_n}(\CC)=\gX_{A^q_n}(\CC)$ and $\Xi^*_{\wh R^q_{n-}}(\CC)=\Xi^*_{A^{q^{-1}}_n}(\CC)=\Xi_{A^q_n}(\CC)$.

Let us construct pairing operators for the idempotent $A=R^q_{n-}$. Due to Lemma~\ref{LemRdec} the dual idempotent $S=1-A=R^q_{n+}$ has the form~\eqref{whRqp}. The algebra $\U_k$ in this case is a subalgebra of $\End\big((\CC^n)^{\otimes k}\big)$ generated by the matrices $\big(\wh R^q_n\big)^{(a,a+1)}$, while its maximal ideals $\U_k^\pm$ are generated by $\big(\wh R^q_{n\pm}\big)^{(a,a+1)}$.
 Since $\wh R^q_n$ satisfies the braid relation~\eqref{BRq} and Hecke relation~\eqref{HeckeR} the role of this algebra is $\gU_k$ is played by the Hecke algebra.

Recall that the Hecke algebra $\He_k^q$ is the algebra generated by $h_1,\dots,h_{k-1}$ with the relations
\begin{gather}
 \big(h_a-q^{-1}\big)(h_a+q)=0 \qquad a=1,\dots,k-1, \label{DefHe1} \\
 h_ah_b=h_bh_a, \qquad |a-b|>1, \label{DefHe2} \\
 h_ah_{a+1}h_a=h_{a+1}h_ah_{a+1}, \qquad a=1,\dots,k-2. \label{DefHe3}
\end{gather}
For $b\le k$ we will identify the subalgebra generated by $h_1,\dots,h_b$ with $\He^q_b$. Note that the relations~\eqref{DefHe2} implies that the elements of $\He^q_b$ commute with $h_{b+1},\dots,h_{k-1}$.

The algebra $\U_k$ is the image of the representation
\begin{align}
 \rho\colon\quad \He_k^q\to\End\big((\CC^n)^{\otimes k}\big), \qquad \rho(h_a)=\big(\wh R^q_n \big)^{(a,a+1)}. \label{piHecke}
\end{align}
The relations~\eqref{DefHe1}, \eqref{DefHe3} and the conditions on the parameter $q$ imply that there are exactly two augmentations
\begin{align}
 \varepsilon^+\colon\quad \He_k^q\to\CC, \qquad\varepsilon^+(h_a)=q^{-1}, \qquad \varepsilon^-\colon\quad \He_k^q\to\CC, \qquad\varepsilon^-(h_a)=-q. \label{gammapm}
\end{align}
Since $\wh R^q_n-q^{-1}=-2_q\wh R^q_{n-}$ and $\wh R^q_n+q=2_q\wh R^q_{n+}$, we have $\rho(u_a^\pm)=\big(\wh R^q_{n\mp}\big)^{(a,a+1)}$ and $\varepsilon(u_a^\pm)=0$ for $u_a^\pm=\frac{q^{\mp1}\mp h_a}{2_q}$. Hence to apply Theorem~\ref{ThUkrhosk} we only need idempotents $s^+_{(k)},s^-_{(k)}\in\He^q_k$ invariant with respect to the augmentations $\varepsilon^+$ and $\varepsilon^-$ respectively. Such idempotents were constructed by D. Gurevich on the level of representation~\eqref{piHecke}. We define the idempotents $s^\pm_{(k)}$ (as some elements of the abstract Hecke algebra) by following his work~\cite{Gur}.

Consider the elements
\begin{align}
 &t^+_k=q^{k-1}+q^{k-2}h_{k-1}+q^{k-3}h_{k-2}h_{k-1}+\dots+h_1\cdots h_{k-1}, \label{tpk} \\
 &t^-_k=q^{1-k}-q^{2-k}h_{k-1}+q^{3-k}h_{k-2}h_{k-1}-\dots+(-1)^{k-1}h_1\cdots h_{k-1}. \label{tmk}
\end{align}
By applying the augmentations~\eqref{gammapm} we obtain
\begin{align}
 \varepsilon^\pm(t^\pm_k)=k_q. \label{gammat}
\end{align}
The elements~\eqref{tpk}, \eqref{tmk} can be defined iteratively by the formulae $t^\pm_1=1$ and
\begin{align*}
 t^\pm_k=q^{\pm(k-1)}\pm t^\pm_{k-1}h_{k-1}. 
\end{align*}

Define the elements $s^\pm_{(k)}$ iteratively as
\begin{align}
s^\pm_{(k)}=k_q^{-1}t^\pm_ks^\pm_{(k-1)}, \qquad s^\pm_{(1)}=1. \label{spmHecke}
\end{align}
Note that $\rho\big(s^\pm_{(2)}\big)=\wh R^q_{n\pm}$.

It was proved in~\cite{Gur} that $\rho\big(s^\pm_{(k)}\big)$ are left and right invariant with respect to corresponding augmentations of $\U_k=\rho(\He^q_k)$. This proof is suitable to establish the invariance of $s^\pm_{(k)}$ with respect to $\varepsilon^\pm$, but we give an alternative proof.

\begin{Prop}
 The formulae~\eqref{spmHecke} define idempotents $s^\pm_{(k)}$ satisfying
\begin{align}
\varepsilon^\pm\big(s^\pm_{(k)}\big)=1, \qquad us^\pm_{(k)}=s^\pm_{(k)}u=\varepsilon^\pm(u)s^\pm_{(k)} \qquad \forall\,u\in\He^q_k. \label{ussu}
\end{align}
Hence $S_{(k)}=\rho\big(s^+_{(k)}\big)$ and $A_{(k)}=\rho\big(s^-_{(k)}\big)$ are pairing operators for $A=\wh R^q_{n-}$.
\end{Prop}

\begin{proof} The first formula follows from~\eqref{gammat} and the equality $\varepsilon^\pm\big(s^\pm_{(k-1)}\big)=1$ which can be supposed by induction. Further we prove the left invariance. Suppose $us^\pm_{(k-1)}=\varepsilon^\pm(u)s^\pm_{(k-1)}$ $\forall\,u\in\He^q_{k-1}$.
We need to prove $h_as^\pm_{(k)}=\pm q^{\mp1}s^\pm_{(k)}$ for $a=1,\dots,k-1$. Denote by $I_{b,k}^\pm$ the left ideal of $\He^q_k$ generated by the elements $u-\varepsilon^{\pm}(u)$, $u\in\He^q_b$. Due to the induction assumption we~have $I^\pm_{b,k}s^\pm_{(k-1)}=0$ for $b=1,\dots,k-1$. Then the relations we need to prove follows from the formulae $h_{k-l}t^\pm_k\in\pm q^{\mp1}t^\pm_k+I^\pm_{k-l,k}$ which we prove by induction in $l\ge1$. For $l=1$ we have
\[
h_{k-1}t^\pm_k=q^{\pm(k-1)}h_{k-1}\pm q^{\pm(k-2)}h_{k-1}^2+h_{k-1}\sum_{a=2}^{k-1}(\pm1)^a q^{\mp a}h_{k-a}\cdots h_{k-3}h_{k-2}h_{k-1}.
\]
 By using the relations~\eqref{DefHe1}--\eqref{DefHe3} we obtain
\begin{gather*}
h_{k-1} t^\pm_k=q^{\pm(k-1)}h_{k-1}\pm q^{\pm(k-2)}\pm q^{\pm(k-2)}\big(q^{-1}-q\big)h_{k-1}
\\ \hphantom{h_{k-1} t^\pm_k=}
{}+\sum_{a=2}^{k-1}(\pm1)^a q^{\mp a}h_{k-a}\cdots h_{k-3}h_{k-2}h_{k-1}h_{k-2}.
\end{gather*}
By taking into account $q^{k-1}+q^{k-2}\big(q^{-1}-q\big)=q^{-1}q^{k-2}$ and $q^{1-k}-q^{2-k}\big(q^{-1}-q\big)=qq^{2-k}$ we derive $h_{k-1} t^\pm_k\in\pm q^{\mp1} t^\pm_k+I^\pm_{k-1,k}$. Suppose that $h_{k-l}t^\pm_k\in\pm q^{\mp1}t^\pm_k+I^\pm_{k-l,k}$ for some $l\ge1$ and all~$k>l$ then for any $k>l+1$ we obtain
\begin{gather*}
h_{k-l-1}t^\pm_k=q^{\pm(k-1)}h_{k-l-1}\pm h_{k-l-1}t^\pm_{k-1}h_{k-1}
\\ \hphantom{h_{k-l-1}t^\pm_k}
\in q^{\pm(k-1)}h_{k-l-1}\pm\big({\pm} q^{\mp1}\big) t^\pm_{k-1}h_{k-1}+I^\pm_{k-l-1,k-1}h_{k-1}.
\end{gather*}
 Since $uh_{k-1}=h_{k-1}u$ for any $u\in \He^q_{k-l-1}$ we have the inclusion $I^\pm_{k-l-1,k-1}h_{k-1}\subset I^\pm_{k-l-1,k}$, so that $h_{k-l-1}t^\pm_k\in\pm q^{\mp1}t^\pm_k+I^\pm_{k-l-1,k}$. The right invariance of $s^\pm_{(k)}$ now follows from Proposition~\ref{propInv}, since the formula $\rho(h_a)=h_a$ defines an anti-automorphism of $\He^q_k$. \end{proof}

{\samepage\begin{Cor}
We have the following symmetric recurrent formulae:
\begin{align}
 &s^\pm_{(k)}=\frac{q^{\pm(k-1)}}{k_q}s^\pm_{(k-1)}\pm \frac{(k-1)_q}{k_q}s^\pm_{(k-1)}h_{k-1}s^\pm_{(k-1)}, \label{shs} \\
 &S_{(k)}=\frac{q^{k-1}}{k_q}S_{(k-1)}+\frac{(k-1)_q}{k_q}S_{(k-1)}\big(\wh R^q_n \big)^{(k-1,k)}S_{(k-1)}, \label{SRS} \\
 &A_{(k)}=\frac{q^{1-k}}{k_q}A_{(k-1)}-\frac{(k-1)_q}{k_q}A_{(k-1)}\big(\wh R^q_n \big)^{(k-1,k)}A_{(k-1)}. \label{ARA}
\end{align}
\end{Cor}}

\begin{proof} By multiplying $t^\pm_k=q^{\pm(k-1)}\pm\sum_{a=0}^{k-2}(\pm1)^a q^{\pm(k-2-a)}h_{k-1-a}\cdots h_{k-1}$ by $s^\pm_{(k-1)}$ from the right and taking into account~\eqref{ussu} for $k-1$ we obtain
\begin{align*}
 s^\pm_{(k-1)}t^\pm_k&=q^{\pm(k-1)}s^\pm_{(k-1)}\pm\sum_{a=0}^{k-2}(\pm1)^a q^{\pm(k-2-a)}\varepsilon^\pm(h_{k-1-a}\cdots h_{k-2})s^\pm_{(k-1)} h_{k-1}
 \\
 &= q^{\pm(k-1)}s^\pm_{(k-1)}\pm\sum_{a=0}^{k-2}q^{\pm(k-2-a)}q^{\mp a}s^\pm_{(k-1)} h_{k-1}
 \\
 &= q^{\pm(k-1)}s^\pm_{(k-1)}\pm(k-1)_qs^\pm_{(k-1)} h_{k-1}.
\end{align*}
Substituting it to $s^\pm_{(k)}=s^\pm_{(k-1)}s^\pm_{(k)}=\frac1{k_q}s^\pm_{(k-1)}t^\pm_ks^\pm_{(k-1)}$ we obtain~\eqref{shs}. Application of the homomorphism $\rho$ to~\eqref{shs} gives the formulae~\eqref{SRS} and~\eqref{ARA}. \end{proof}

Let us calculate the entries of the $A$-operators for the idempotent $A=R^q_{n-}$. Since we have $\gX_A(\CC)=\gX_{A_{\wh q}}(\CC)$ and $\gX^*_A(\CC)=\gX^*_{A_{\wh p}}(\CC)$ for $q_{ij}=q^{\sgn(j-i)}$ and $p_{ij}=q^{\sgn(i-j)}$ the formulae~\eqref{Lempsisigma2k}, \eqref{Lempsisigma2kDual} take the form
\begin{gather*}
 \psi_{i_{\sigma(1)}}\cdots\psi_{i_{\sigma(k)}}
 =(-1)^\sigma q^{-\inv(\sigma)}\psi_{i_1}\cdots\psi_{i_k},\\
 \psi^{i_{\sigma(1)}}\cdots\psi^{i_{\sigma(k)}}=(-1)^\sigma q^{-\inv(\sigma)}\psi^{i_1}\cdots\psi^{i_k},
\end{gather*}
where $i_1<\dots<i_k$. Hence
\begin{align}
 A^{i_{\sigma(1)} \dots i_{\sigma(k)}}_{j_{\tau(1)}\dots j_{\tau(k)}}
=(-1)^\sigma(-1)^\tau q^{-\inv(\sigma)-\inv(\tau)}A^{i_1 \dots i_k}_{j_1 \dots j_k}, \label{Asitau}
\end{align}
for $i_1<\dots<i_k$, $j_1<\dots<j_k$ and the other entries of $A_{(k)}$ vanish.

To find $A^{i_1 \dots i_k}_{j_1 \dots j_k}$ we first write the formula~\eqref{Asitau} for the longest permutation $\sigma=\tau$:
\begin{align}
A^{i_k \dots i_1}_{j_k \dots j_1}= q^{-k(k-1)}A^{i_1 \dots i_k}_{j_1 \dots j_k}. \label{AsitauLong}
\end{align}
Let us prove that there are numbers $\lambda_k\in\CC$ such that
\begin{align}
 A^{i_k \dots i_1}_{j_k \dots j_1}= \lambda_k\delta^{i_k}_{j_k}\cdots\delta^{i_1}_{j_1} \label{Alambdak}
\end{align}
for any $i_k>i_{k-1}>\dots>i_1$ and $j_k>j_{k-1}>\dots>j_1$. This holds for $k=1$ with $\lambda_1=1$. Suppose it holds for $k-1$, then by taking the corresponding entries in~\eqref{ARA} we obtain
\begin{align*}
A^{i_k \dots i_1}_{j_k \dots j_1}=\frac{q^{1-k}}{k_q}A^{i_k \dots i_2}_{j_k \dots j_2}\delta^{i_1}_{j_1}-\frac{(k-1)_q}{k_q}\sum_{l_k \dots l_2 l'_2}A^{i_k \dots i_2}_{l_k \dots l_2}\big(\wh R^q_n \big)^{l_2i_1}_{l'_2j_1}A^{l_k \dots l_3 l'_2}_{j_k \dots j_3 j_2}. 
\end{align*}
If the term in the sum does not vanish then $l_2=i_s$ for some $s=k,\dots,2$ and $l'_2=j_t$ for some $t=k,\dots,2$, so we have $l_2>i_1$ and $l'_2>j_1$. The expression~\eqref{whRq} implies that $\big(\wh R^q_n\big)^{ii'}_{jj'}=0$ for~$i>i'$ and $j>j'$. Hence the sum vanishes and we obtain~\eqref{Alambdak} with $\lambda_k=\frac{q^{1-k}}{k_q}\lambda_{k-1}$. Iteratively we derive $\lambda_k=\frac{q^{-k(k-1)/2}}{k_q!}$, where $k_q!=k_q(k-1)_q\cdots2_q1_q$. Thus due to~\eqref{Asitau}, \eqref{AsitauLong} and \eqref{Alambdak} we obtain
\begin{align*}
 A^{i_{\sigma(1)} \dots i_{\sigma(k)}}_{j_{\tau(1)} \dots j_{\tau(k)}}
=\frac{q^{k(k-1)/2}}{k_q!}(-1)^\sigma(-1)^\tau q^{-\inv(\sigma)-\inv(\tau)}\delta^{i_1}_{j_1}\cdots\delta^{i_k}_{j_k} 
\end{align*}
for $i_1<\dots<i_k$ and $j_1<\dots<j_k$.

Since $r_k=\dim\Xi_{\wh R^q_{n-}}(\CC)_k=\dim\Xi_{A^q_n}(\CC)_k={k\choose n}$ we have $\tr A_{(k)}={k\choose n}$. Explicit calculation of the trace (cf.~\eqref{trAwtq}) gives us the formula
\begin{align} \label{invGF}
 \sum_{\sigma\in\SSS_k}q^{-2\inv(\sigma)}=q^{-k(k-1)/2}\,k_q!
\end{align}
(it can be also checked by induction).

Denote by $A'_{(k)}$ the $k$-th $A$-operator for $A'=A^q_n$. Its non-zero entries are given by the formula~\eqref{Awhq} for $\wh q=q^{[n]}$, that is
\begin{gather*}
 (A')^{i_{\sigma(1)} \dots i_{\sigma(k)}}_{j_{\tau(1)} \dots j_{\tau(k)}}=\frac1{k!}(-1)^\tau(-1)^\sigma q^{\inv(\sigma)-\inv(\tau)}\delta^I_J,
\end{gather*}
where $I=(i_1<\dots<i_k\le n)$, $J=(j_1<\dots<j_k\le n)$. Due Proposition~\ref{propSAequiv} the idempotents~$A'_{(k)}$ and $A_{(k)}$ are left-equivalent, so that $A'_{(k)}=G_{[k]}A_{(k)}$ for some invertible matrix $G_{[k]}$. The latter can be chosen as a diagonal matrix with entries
\begin{gather}
 \big(G_{[k]}\big)^{i_{\sigma(1)} \dots i_{\sigma(k)}}_{i_{\sigma(1)} \dots i_{\sigma(k)}}=\frac{k_q!}{k!}q^{2\inv(\sigma)-k(k-1)/2}, \qquad\text{for}\quad i_1<\dots<i_k,\quad \sigma\in\SSS_k, \nonumber
 \\
\big(G_{[k]}\big)^{l_1 \dots l_k}_{l_1 \dots l_k}=1,
\qquad\text{if}\quad l_s=l_t\qquad \text{for some}\quad s\ne t.\label{Gksigma}
\end{gather}

Let $M$ be an $\big(\wh R^q_{n-},A_{\wh p}\big)$-Manin matrix, i.e., a~$\big(q^{[n]},\wh p\big)$-Manin matrix. Its $A$-minors are related by the formula~\eqref{MinAGA}. Substitution $\wh q=q^{[n]}$ to~\eqref{AMinqp} yields
\[
 \big(\Min^{A'_{(k)}}M\big)^{i_{\sigma(1)} \dots i_{\sigma(k)}}_{j_{\tau(1)} \dots j_{\tau(k)}}=\frac1{k!}(-1)^\tau(-1)^\sigma \frac{q^{\inv(\sigma)}}{\mu(\wh p_{JJ},\tau)}\det_q(M_{IJ}).
\]
By dividing this by~\eqref{Gksigma} we obtain the $A$-minors
\[
 \big(\Min^{A_{(k)}}M\big)^{i_{\sigma(1)} \dots i_{\sigma(k)}}_{j_{\tau(1)} \dots j_{\tau(k)}}=\frac{q^{k(k-1)/2}}{k_q!}(-1)^\tau(-1)^\sigma \frac{q^{-\inv(\sigma)}}{\mu(\wh p_{JJ},\tau)}\det_q(M_{IJ}).
\]

Let $w'_I=A'_{(k)}e_I$ be the basis~\eqref{wI} for $\wh q=q^{[n]}$. Then by the formula~\eqref{walphaG} we obtain the corresponding basis of $A_{(k)}\big((\CC^n)^{\otimes k}\big)$:
\begin{align}
 w_I=G_{[n]}^{-1}w'_I=G_{[n]}^{-1}A'_{(k)}e_I=A_{(k)}e_I=\frac{q^{k(k-1)/2}}{k_q!} \sum_{\sigma\in\SSS_k}(-1)^\sigma q^{-\inv(\sigma)}e_{i_{\sigma(1)} \dots i_{\sigma(k)}}. \label{wIHecke}
\end{align}
One can directly check that the basis $\psi^{(k)}_I$ defined by the bases $w_I$ and $w'_I$ coincide (see~\eqref{psialphaLE}): by using the formulae~\eqref{wIHecke}, \eqref{Lempsisigma2k} and~\eqref{invGF} we derive
\begin{align*}
 \psi^{(k)}_I=\frac{q^{k(k-1)/2}}{k_q!} \sum_{\sigma\in\SSS_k}(-1)^\sigma q^{-\inv(\sigma)}\psi_{i_{\sigma(1)}}\cdots\psi_{i_{\sigma(k)}}=\psi_{i_1}\cdots\psi_{i_k}.
\end{align*}
Due to Proposition~\ref{propMinabLE} the entries of the corresponding $A$-minor operators coincide for the corresponding bases:
\[
 w^I\big(\Min^{A'_{(k)}}\big)\wt w_J=w^I\big(\Min^{A_{(k)}}\big)\wt w_J,
\]
where $\wt w_J$ is the same as in Section~\ref{secMinqp}.

\subsection{Pairing operators for the 4-parametric case}\label{secPO4p}

Here we consider the case described in Section~\ref{sec4p}. The idempotent $A=A^{a,b,c}_\kappa$ is parametrised by $a,b,c\in\CC\backslash\{0\}$ and $\kappa\in\CC$. The commutation relations for the quadratic algebra $\Xi_A(\CC)$ are $\psi_i\psi_j+\sum_{k,l=1}^3\psi_k\psi_lP^{kl}_{ij}=0$. They can be written in the form $\psi_i\psi_j+a_{ij}^2\psi_j\psi_i=0$, $i\ne j$, and $2\psi_i^2=-\kappa\sum_{k,l=1}^3\varepsilon_{ikl}a_{lk}\psi_k\psi_l$ or
\begin{align}
\psi_i^2=\kappa a_{jk}\psi_k\psi_j=-\kappa a_{kj}\psi_j\psi_k,\qquad
(i,j,k)\in C_3, \label{XiAabc}
\end{align}
where $C_3$ is the set of cyclic permutations $(1,2,3),(2,3,1)$ and $(3,1,2)$.

The algebra $\Xi^*_A(\CC)$ is defined by the commutation relations $\psi^i\psi^j+\sum_{k,l=1}^3P^{ij}_{kl}\psi^k\psi^l=0$, these are $\psi^i\psi^j+a_{ji}^2\psi^j\psi^i+\kappa a_{ji}\sum_{k=1}^3\varepsilon_{ijk}\psi^k\psi^k=0$. For $i=j$ we obtain $\psi^i\psi^i=0$, so we have $\Xi^*_{A^{a,b,c}_\kappa}(\CC)=\Xi^*_{A_{\wh q}}(\CC)$, where $q_{ij}=a_{ij}^2$. This implies that the idempotents $A^{a,b,c}_\kappa$ and $A_{\wh q}$ are right-equivalent. In particular, $\Xi^*_{A^{a,b,c}_\kappa}(\CC)=\bigoplus_{k=0}^3\Xi^*_{A^{a,b,c}_\kappa}(\CC)_k$ and the dimensions of the components are $\dim\Xi^*_{A^{a,b,c}_\kappa}(\CC)_0=1$, $\dim\Xi^*_{A^{a,b,c}_\kappa}(\CC)_1=3$, $\dim\Xi^*_{A^{a,b,c}_\kappa}(\CC)_2=3$, \mbox{$\dim\Xi^*_{A^{a,b,c}_\kappa}(\CC)_3=1$}, $\dim\Xi^*_{A^{a,b,c}_\kappa}(\CC)_k=0$ for $k\ge4$.

From the relations~\eqref{XiAabc} we see that $\dim\Xi_{A^{a,b,c}_\kappa}(\CC)_2=3$. However, the dimension of $\Xi_{A^{a,b,c}_\kappa}(\CC)_3$ is not always equal to $1$, it depends on the parameters. The following theorem describes these dependence and gives the necessary and sufficient condition for existence of the corresponding pairing operator. Note that it is enough to consider the case $\kappa\ne0$ since the case of the idempotent $A^{a,b,c}_0=A_{\wh q}$ was considered in Section~\ref{secMinqp}.

\begin{Th}
Assume $\kappa\ne0$. Consider the conditions
\begin{gather}
(i)\phantom{ii}\quad a^2=b^2=c^2; \notag \\
(ii)\phantom{i} \quad a^4b^2=b^4c^2=c^4a^2=-1, \quad\kappa^3=-a^3b^{-1}c; \notag\\
(iii)\quad a^4c^2=b^4a^2=c^4b^2=-1, \quad\kappa^3=a^{-3}b^{-1}c. \label{condThabckappa}
\end{gather}
Any two of these conditions implies the third one. We have $\Xi_{A^{a,b,c}_\kappa}(\CC)=\bigoplus_{k=0}^3\Xi_{A^{a,b,c}_\kappa}(\CC)_k$ with the dimensions of the components
\begin{gather*}
\dim\Xi_{A^{a,b,c}_\kappa}(\CC)_0=1, \qquad
 \dim\Xi_{A^{a,b,c}_\kappa}(\CC)_1=3, \qquad
 \dim\Xi_{A^{a,b,c}_\kappa}(\CC)_2=3,
 \\
\dim\Xi_{A^{a,b,c}_\kappa}(\CC)_3=
\begin{cases}
3 & \text{iff all three conditions~\eqref{condThabckappa} hold,} \\
1 & \text{iff one and only one of three conditions~\eqref{condThabckappa} holds,} \\
0 & \text{iff no one of three conditions~\eqref{condThabckappa} holds,}
\end{cases}
\end{gather*}
$\dim\Xi_{A^{a,b,c}_\kappa}(\CC)_k=0$ for $k\ge4$.

The third $A$-operator exists iff the condition $(i)$ holds and $(ii)$, $(iii)$ do not. It equals
\begin{gather}
A_{(3)}=w_1w^1,\qquad \text{where} \quad w_1=\frac16(e_{123}+e_{231}+e_{312})-\frac{a^2}6(e_{132}+e_{213}+e_{321}), \label{Aw1}
 \\[1ex]
w^1=e^{123}+e^{231}+e^{312}-a^{-2}\big(e^{132}+e^{213}+e^{321}\big)-\kappa\big(b^{-1}e^{111}+c^{-1}e^{222}+a^{-1}e^{333}\big). \label{w1}
\end{gather}
In this case the elements $\psi_i\psi_j\psi_k\in\Xi_{A^{a,b,c}_\kappa}(\CC)_3$ have the form
\begin{gather}
\psi_i\psi_j\psi_k=\psi_1^{(3)}, \qquad
\psi_i\psi_k\psi_j=-a^{-2}\psi_1^{(3)}, \qquad(i,j,k)\in C_3, \label{psi3a1}
\\
\psi_i^2\psi_j=\psi_i\psi_j\psi_i=\psi_j\psi_i^2=0, \qquad i\ne j,
\\
\psi_1^3=-\kappa b^{-1}\psi_1^{(3)}, \qquad
\psi_2^3=-\kappa c^{-1}\psi_1^{(3)}, \qquad
\psi_3^3=-\kappa a^{-1}\psi_1^{(3)}, \label{psi3a3}
\end{gather}
where $\psi_1^{(3)}=(\Psi\otimes\Psi\otimes\Psi)w_1=\psi_1\psi_2\psi_3$.
\end{Th}

\begin{proof} Under the condition $(i)$ both conditions $(ii)$ and $(iii)$ gives $a^6=-1$, which in turn implies $-a^3b^{-1}c=a^{-3}b^{-1}c$. Hence $(i)$ implies equivalence of $(ii)$ and $(iii)$. Further, by comparing the conditions $(ii)$ and $(iii)$ we see that they imply $(i)$.

Now let us use Theorem~\ref{ThPOviaDB}. Since the idempotents $S^{a,b,c}_\kappa=1-A^{a,b,c}_\kappa$ and $S_{\wh q}=1-A_{\wh q}$ are left-equivalent the space $W_k$ for the case $A=A^{a,b,c}_\kappa$ coincide with the space $W_k$ for $A=A_{\wh q}$. In~particular, $W_3=\CC w_1$. The space $\overline W_3$ consists of the covectors $\xi=\sum_{i,j,k=1}^3\xi_{ijk}e^{ijk}$ satisfying $\xi P^{(12)}=\xi P^{(23)}=-\xi$. This gives us the system of equations
\begin{gather}
 \xi_{ijl}=-a_{ij}^2\xi_{jil},\qquad \xi_{lij}=-a_{ij}^2\xi_{lji}, \qquad i\ne j \label{relxi1},
 \\
 \xi_{iil}=\kappa a_{jk}\xi_{kjl}, \qquad \xi_{lii}=\kappa a_{jk}\xi_{lkj}, \qquad (i,j,k)\in C_3. \label{relxi2}
\end{gather}
The coefficients $\xi_{ijk}$ can be divided into three sets:
\begin{gather}
 \{\xi_{iii}\mid i=1,2,3\}\cup\{\xi_{ijk}\mid i\ne j\ne k\ne i\}, \label{xi1} \\
 \{\xi_{iij},\xi_{iji},\xi_{jii}\mid(i,j,k)\in C_3\}, \label{xi2} \\
 \{\xi_{ijj},\xi_{jij},\xi_{jji}\mid(i,j,k)\in C_3\}. \label{xi3}
\end{gather}
The relations~\eqref{relxi1}, \eqref{relxi2} imply that any two coefficients from the same set are proportional to each other (with a non-zero coefficient of proportionality), so that $\dim\overline W_3\le3$. The isomorphism $\Xi_A(\CC)_3\cong\overline W_3^*$ implies that $\dim\Xi_A(\CC)_3=\dim\overline W_3$. Let us prove that non-vanishing of the coefficients from the sets~\eqref{xi1}, \eqref{xi2}, \eqref{xi3} corresponds exactly to the conditions~$(i)$, $(ii)$, $(iii)$ respectively.

Note that there are two types of symmetries of the system of equations~\eqref{relxi1}, \eqref{relxi2}. First, it is invariant under the cyclic permutations of indices $1\mapsto2\mapsto3\mapsto1$ and $1\mapsto3\mapsto2\mapsto1$. Second, it is invariant under other permutations $i\leftrightarrow j$, $i\ne j$, with the sign change $\kappa\mapsto-\kappa$. This system implies that $\xi_{111}=\kappa a_{23}\xi_{132}=\kappa a_{23}a_{13}^2a_{12}^2\xi_{321}=a_{13}^2a_{12}^2\xi_{111}$. A cyclic permutation gives $\xi_{222}=a_{21}^2a_{23}^2\xi_{222}$. If the coefficients from the set~\eqref{xi1} do not vanish then $a_{31}^2=a_{12}^2=a_{23}^2$. This is the condition $(i)$. Conversely, if the condition $(i)$ holds then there exists a solution with non-vanishing coefficients~\eqref{xi1}, namely $\xi=w^1$ (the coefficients from~\eqref{xi2} and \eqref{xi3} vanish in this solution).

Write down some relations between the coefficients~\eqref{xi2} that follows from the system~\eqref{relxi1}, \eqref{relxi2}:
\begin{gather}
 \xi_{112}=-a_{12}^2\xi_{121}=a_{12}^4\xi_{211}=\kappa a_{23}a_{12}^4\xi_{232}, \label{xi2rel1} \\
 \xi_{112}=\kappa a_{23}\xi_{322}=-\kappa a_{23}a_{32}^2\xi_{232}, \label{xi2rel2} \\
 \xi_{322}=-\kappa a_{13}\xi_{331}=\kappa^2 a_{13}a_{21}\xi_{121}. \label{xi2rel3}
\end{gather}
All other relations is obtained by the symmetry of the first type. By equating the same coef\-ficients we obtain the condition of existence of non-vanishing solution of~\eqref{xi2rel1}--\eqref{xi2rel3}, they are $a_{12}^4a_{23}^2=-1$, $\kappa^3=-a_{12}^3a_{23}^{-1}a_{31}$. The symmetry gives the whole condition~$(ii)$. Thus there is a~non-zero solution $\xi=w^2$ with vanishing coefficients from~\eqref{xi1} and \eqref{xi3}. By applying the symmetry of the second type we obtain the relations between the coefficients~\eqref{xi3} and the condition~$(iii)$. This means that $(iii)$ is necessary and sufficient condition for existence of a~non-zero solution $\xi=w^3$ with vanishing coefficients~\eqref{xi1} and \eqref{xi2}.

The isomorphism $\Xi_A(\CC)_3\cong\overline W_3^*$ identifies the elements $\psi_i\psi_j\psi_k$ with linear functions on $\overline W_3$. By substituting $\pi=e_{ijk}$ to the formula~\eqref{pairWkD} we obtain $\psi_i\psi_j\psi_k(\xi)=\xi_{ijk}$, where $i,j,k=1,2,3$ and $\xi=\sum_{i,j,k=1}^3\xi_{ijk}e^{ijk}\in\overline W_k$. Thus the elements $\psi_i\psi_j\psi_k$ for $i\ne j\ne k\ne i$ and $\psi_i^3$ do not vanish iff $(i)$ holds. Similarly, $(ii)$ and $(iii)$ are conditions of non-vanishing of the elements $\psi_i^2\psi_j,\psi_i\psi_j\psi_i,\psi_j\psi_i^2$, where $(i,j)=(1,2),(2,3),(3,1)$ and $(i,j)=(2,1),(3,2),(1,3)$ respectively.

To prove that $\dim\Xi_A(\CC)_k=0$ for $k\ge0$ it is enough to check it for $k=4$. Note that the relations~\eqref{XiAabc} implies that all the elements $\psi_{i_1}\psi_{i_2}\psi_{i_3}\psi_{i_4}$ are proportional to each other, that is $\dim\Xi_A(\CC)_4\le1$. If $\dim\Xi_A(\CC)_3<3$ then $\psi_{i_1}\psi_{i_2}\psi_{i_3}$ vanishes for some $i_1$, $i_2$, $i_3$ and hence $\Xi_A(\CC)_4=0$. In~the case $\dim\Xi_A(\CC)_3=3$ all the conditions~\eqref{condThabckappa} hold. In particular, $a^6=-1$. Suppose that $\dim\Xi_A(\CC)_4\ne0$ and hence all the elements $\psi_{i_1}\psi_{i_2}\psi_{i_3}\psi_{i_4}$ do not vanish. From the chain of relations $\psi_3\psi_1^2\psi_2=\kappa a_{23}\psi_3^2\psi_2^2=\kappa^3 a_{23}a_{12}a_{31}\psi_2\psi_1^2\psi_3=\kappa^3 a_{23}a_{12}a_{31}\psi_2\psi_1^2\psi_3=-\kappa^3 a_{23}a_{12}a_{31}a_{21}^4a_{23}^2a_{13}^4\psi_3\psi_1^2\psi_2$ we obtain $\kappa^3 a_{21}^3a_{23}^3a_{13}^3=-1$. By substituting $(i)$ and $(ii)$ we obtain $-1=-a^3b^{-1}ca^{-1}b^3c^{-3}=-a^2b^2c^{-2}=-a^2$, which implies $a^6=1$. This contradicts with $a^6=-1$, hence $\Xi_A(\CC)_4=0$.

The existence of $A_{(3)}$ implies $\dim\overline W_3=\dim W_3=1$, so that one and only one of the conditions~\eqref{condThabckappa} holds. Namely $\overline W_3$ has the form $\CC w^1$, $\CC w^2$ or $\CC w^3$ under the condition $(i)$, $(ii)$ or $(iii)$ respectively. But $w^\alpha w_1=0$ for $\alpha=2,3$, so that only the condition $(i)$ is relevant. Since $w^1w_1=1$ we have $A_{(3)}=w_1w^1$. Note that $\psi_i\psi_j\psi_k=(X\otimes X\otimes X)e_{ijk}=(X\otimes X\otimes X)A_{(3)}e_{ijk}=\psi_1^{(3)}w^1e_{ijk}$, $i,j,k=1,2,3$. By substituting the explicit expression for~$w^1$ we derive the formulae~\eqref{psi3a1}--\eqref{psi3a3}. \end{proof}

\begin{Rem} \label{Rem4pCond}
 The three cases when $\dim\Xi_{A^{a,b,c}_\kappa}(\CC)_3=1$ correspond to the conditions
\begin{gather*}
(I)\phantom{II}\quad a^2=b^2=c^2\qquad \text{and}\qquad \big(a^6\ne-1\text{ or }\kappa^3\ne-abc\big);
 \\
(II)\phantom{I}\quad b^2=a^{14},\qquad c^2=a^{26},\qquad a^{18}=-1,\qquad a^6\ne-1,\qquad \kappa^3=a^7bc;
 \\
(III)\quad b^2=a^{26},\qquad c^2=a^{14},\qquad a^{18}=-1,\qquad a^6\ne-1,\qquad \kappa^3=a^7bc
\end{gather*}
(cf.~\cite[conditions~(1.4.1) and~(1.4.5)]{IS}). The condition~$(I)$ is exactly the case when $A_{(3)}$ exists.
\end{Rem}

\begin{Rem} \label{Rem4p}
 The dimension of the space $\gX_A(\CC)_3$ also depends on the values of the parame\-ters~$a$, $b$, $c$, $\kappa$. By using the relations~\eqref{abcxyzkappa} one can relate $xyz$ with $zyx$ in two different ways. As~a~re\-sult we obtain two different expressions for $xyz-a^{-2}b^{-2}c^2zyx$, namely,
\begin{align*}
 \kappa b^{-1}x^3-\kappa cb^{-2}y^3+\kappa a^{-1}b^{-2}c^2z^3
=\kappa a^{-1}z^3-\kappa c a^{-2}y^3+\kappa b^{-1}a^{-2}c^2x^3.
\end{align*}
Assume $\kappa\ne0$. Then we see that the elements $x^3$, $y^3$, $z^3$ are linearly independent iff the condition~$(i)$ holds. One can also obtain the relation
\begin{gather*}
 \big(\kappa a^{-1}bc^{-1} +a^2\big)x^2y+\big(\kappa a^{-1}b^3c^{-1} -a^{-2}\big)yx^2\\
 \qquad{} =\kappa\big(a+a^{-1}c^{-4}\big)xz^2+\kappa a^{-1}\big(c^{-3}+b^4c^{-1}\big)y^2z.
\end{gather*}
One can deduce that the dimension of the subspace spanned by the elements $x_i^2x_j$, $x_ix_jx_i$, $x_jx_i^2$, $(i,j,k)\in C_3$, equals $4$ if the condition $(ii)$ holds and it equals $3$ if this condition is false. The dimension of the subspace spanned by the elements $x_j^2x_i$, $x_jx_ix_j$, $x_ix_j^2$, $(i,j,k)\in C_3$, depends on the condition $(iii)$ in the same way (see~\cite[Theorem~1.4]{IS}). Thus the difference $\dim\gX_A(\CC)_3-\dim\Xi_A(\CC)_3$ does not depend on the values of the parameters and equals $9$. Moreover, the third $S$-operator $S_{(3)}$ for the idempotent $A^{a,b,c}_\kappa$ exists iff the $A$-operator $A_{(3)}$ exists.
\end{Rem}

Finally we write the $A$-Minors for an $\big(A^{a,b,c}_\kappa,B\big)$-Manin matrix $M$. By substituting the expression $A^{ij}_{kl}=\frac12(\delta^i_k\delta^j_l-a_{ji}^2\delta^i_l\delta^j_k-\kappa a_{ji}\delta_{kl}\varepsilon_{ijk})$ to $\big(\Min^{A_{(2)}}M\big)^{ij}_{j_1j_2}=\sum_{k,l=1}^3A^{ij}_{kl}M^{k}_{j_1}M^{l}_{j_2}$ we obtain
\begin{gather*}
 \big(\Min^{A_{(2)}}M\big)^{ij}_{j_1j_2}=\frac12\big(M^{i}_{j_1}M^j_{j_2}-a_{ji}^2M^j_{j_1}M^i_{j_2}-\kappa a_{ji}M^k_{j_1}M^k_{j_2}\big),
 \\
 \big(\Min^{A_{(2)}}M\big)^{ji}_{j_1j_2}=-a_{ij}^2\big(\Min^{A_{(2)}}M\big)^{ij}_{j_1j_2}, \qquad \big(\Min^{A_{(2)}}M\big)^{ii}_{j_1j_2}=0,
\end{gather*}
where $(i,j,k)\in C_3$.

Let $a^2=b^2=c^2$ and the condition~$(ii)$ do not hold ($a^6\ne-1$ or $\kappa\ne-abc$). The components of the pairing operator $A_{(3)}=w_1w^1$ have the form $A^{i_1i_2i_3}_{k_1k_2k_3}=w_1^{i_1i_2i_3} w^1_{k_1k_2k_3}$. Hence
 $\big(\Min^{A_{(3)}}M\big)^{i_1i_2i_3}_{j_1j_2j_3}=w_1^{i_1i_2i_3}\sum_{k_1,k_2,k_3=1}^3 w^1_{k_1k_2k_3}M^{k_1}_{j_1}M^{k_2}_{j_2}M^{k_3}_{j_3}$.
By using~\eqref{Aw1}, \eqref{w1} we obtain
\begin{gather*}
 \big(\Min^{A_{(3)}}M\big)^{123}_{j_1j_2j_3}=\frac16\sum_{(i,j,k)\in C_3}\big(M^{i}_{j_1}M^{j}_{j_2}M^{k}_{j_3}-a^{-2}M^{j}_{j_1}M^{i}_{j_2}M^{k}_{j_3}\big)
 \\ \hphantom{ \big(\Min^{A_{(3)}}M\big)^{123}_{j_1j_2j_3}=}
{}-\frac\kappa6\big(b^{-1}M^{1}_{j_1}M^{1}_{j_2}M^{1}_{j_3}+c^{-1}M^{2}_{j_1}M^{2}_{j_2}M^{2}_{j_3}+a^{-1}M^{3}_{j_1}M^{3}_{j_2}M^{3}_{j_3}\big), \\
\big(\Min^{A_{(3)}}M\big)^{i_1i_2i_3}_{j_1j_2j_3}=
\begin{cases}
 \big(\Min^{A_{(3)}}M\big)^{123}_{j_1j_2j_3} & \text{if }(i_1,i_2,i_3)\in C_3,
 \\[1ex]
 -a^2\big(\Min^{A_{(3)}}M\big)^{123}_{j_1j_2j_3} & \text{if }(i_2,i_1,i_3)\in C_3,
 \\
 0 & \text{if }i_k=i_l\text{ for some }k\ne l.
\end{cases}
\end{gather*}

\section[Manin matrices of types B, C and D]
{Manin matrices of types $\boldsymbol B$, $\boldsymbol C$ and $\boldsymbol D$}\label{secBCD}

{\sloppy
Remind that the $A_n$-Manin matrices and $A_n^q$-Manin matrices are related with the Yan\-gi\-ans~$Y(\mathfrak{gl}_n)$ and quantum affine algebras $U_q\big(\wh{\mathfrak{gl}}_n\big)$ respectively. The Lie algebra $\mathfrak{gl}_n=\mathfrak{gl}(n,\CC)$ is usually considered as the case ``type $A$'', since $\mathfrak{gl}_n\cong\CC\oplus\mathfrak{sl}_n$, where $\mathfrak{sl}_n=\mathfrak{sl}(n,\CC)$ is the simple algebra of the type $A_{n-1}$. Hence these Manin matrices can be referred to the type~$A$. Moreover, the minor operators for more general $(\wh q,\wh p)$-Manin matrices are described by using the symmetric groups~$\SSS_k$, which are the Weyl groups of the type $A_{k-1}$ and participate in Schur--Weyl duality with the Lie algebras $\mathfrak{gl}_n$.

}

A generalisation of the $A_n$-Manin matrices to the types $B$, $C$ and $D$ was introduced by A.~Molev. Remind that $\mathfrak{so}_{2r+1}=\mathfrak{so}(2r+1,\CC)$, $\mathfrak{sp}_{2r}=\mathfrak{sp}(2r,\CC)$ and $\mathfrak{so}_{2r}=\mathfrak{so}(2r,\CC)$ are simple Lie algebras of types $B_r$, $C_r$ and $D_r$ respectively (where $r\ge2$ for $D_r$). In this section we assume that $n=2r+1$ for the $B_r$ case and $n=2r$ for the $C_r$ and $D_r$ cases.

\subsection{Molev's definition and corresponding quadratic algebras}
\label{secMolev}

By starting with the definition of Manin matrices of types $B$, $C$ and $D$ we interpret them as Manin matrices for quadratic algebras. To do it we use the notations $i'=n-i+1$ and
\begin{align*}
 \varepsilon_i=
\begin{cases}
\phantom{-}1, &i=1,\dots,r, \\
-1, &i=r+1,\dots,n;
\end{cases}
\end{align*}
(the notation $\varepsilon_i$ is used for the case $C_r$ only). Introduce the operators $Q_n\in\End(\CC^n\otimes\CC^n)$ for the $B$ and $D$ cases and $\wt Q_n\in\End(\CC^n\otimes\CC^n)$ for the $C$ case:
\begin{align*}
 Q_n=\sum_{i,j=1}^nE_{i}{}^{j}\otimes E_{i'}{}^{j'},\qquad
 \wt Q_n=\sum_{i,j=1}^n\varepsilon_i\varepsilon_jE_{i}{}^{j}\otimes E_{i'}{}^{j'}.
\end{align*}
One can check that these operators satisfies the following relations:
\begin{gather}
(Q_n)^2=nQ_n, \qquad P_nQ_n=Q_nP_n=Q_n, \label{Qproper}
\\
(\wt Q_n)^2=n\wt Q_n, \qquad P_n\wt Q_n=\wt Q_nP_n=-\wt Q_n, \label{wtQproper}
\end{gather}
where $P_n$ is the permutation operator defined in Section~\ref{secMm}.


\begin{Def}[{\cite{MolevSO}}]
 A matrix $M\in\gR\otimes\End(\CC^n)$ is a Manin matrix of the type $B$ (for odd~$n$) or $D$ (for even $n$) if it satisfies
\begin{align}
 (1-P_n)M^{(1)}M^{(2)}\bigg(\frac{1+P_n}2-\frac{Q_n}n\bigg)=0. \label{MMson}
\end{align}
A matrix $M\in\gR\otimes\End(\CC^n)$ for even $n$ is a Manin matrix of the type $C$ if it satisfies
\begin{align}
 \bigg(\frac{1-P_n}2-\frac{\wt Q_n}n\bigg)M^{(1)}M^{(2)}(1+P_n)=0. \label{MMsp}
\end{align}
\end{Def}

Introduce operators $B_n\in\End(\CC^n\otimes\CC^n)$ for the $B$ and $D$ cases and $\wt B_n\in\End(\CC^n\otimes\CC^n)$ for the $C$ case as
\[
 B_n=\frac{1-P_n}2+\frac{Q_n}n=A_n+\frac{Q_n}n, \qquad
 \wt B_n=\frac{1-P_n}2-\frac{\wt Q_n}n=A_n-\frac{\wt Q_n}n.
\]
The formulae~\eqref{Qproper}, \eqref{wtQproper} implies that these operators are idempotents. We see that the relations~\eqref{MMson} and \eqref{MMsp} can be written as the definition~\eqref{AMMB} by means of the idempotents $A_n=\frac{1-P_n}2$, $B_n$ and $\wt B_n$.

\begin{Prop} \label{PropMMBCD}
A matrix $M\in\gR\otimes\End(\CC^n)$ is a Manin matrix of type $B$ or $D$ iff it is an~$(A_n,B_n)$-Manin matrix. A matrix $M\in\gR\otimes\End(\CC^n)$ is a Manin matrix of type $C$ iff it is an~$\big(\wt B_n,A_n\big)$-Manin matrix $($for even $n)$.
\end{Prop}

Now let us consider the quadratic algebras related with the idempotents $B_n$ and $\wt B_n$.

\begin{Prop} \label{PropBnXX}
 The commutation relations for the algebras $\gX_{B_n}(\CC)$ and $\Xi_{\wt B_n}(\CC)$ are
\begin{gather}\label{BnXX}
 P_n(X\otimes X)=X\otimes X,\qquad Q_n(X\otimes X)=0
 \end{gather}
 and
 \begin{gather}
 (\Psi\otimes\Psi)P_n=-\Psi\otimes\Psi, \qquad (\Psi\otimes \Psi)\wt Q_n=0 \label{PsiwtBn}
\end{gather}
respectively.
\end{Prop}

\begin{proof} The relation $B_n(X\otimes X)=0$ has the form
\[
 \bigg(\frac{1-P_n}2+\frac{Q_n}n\bigg)(X\otimes X)=0. 
\]
Multiplication by $Q_n$ from the left gives $Q_n(X\otimes X)=0$, so we derive~\eqref{BnXX}. Analogously, \eqref{PsiwtBn} can obtained from $(\Psi\otimes\Psi)\big(1-\wt B_n\big)=0$ by multiplication by $\wt Q_n$ from the right. The converse implications are obvious. \end{proof}

The algebra $\gX_{B_n}(\CC)$ is the quotient of the polynomial algebra $\CC[x^1,\dots,x^n]=\gX_{A_n}(\CC)$ by the relation
\[
 \sum_{i=1}^n x^ix^{i'}=0.
\]
The group of matrices $G\in {\rm GL}(n,\CC)$ preserving the symmetric bilinear form $g_n(x,y)=\!\sum_{i=1}^n x^iy^{i'}$ is isomorphic to $O(n,\CC)$.

The algebra $\Xi_{B_n}(\CC)$ is generated by $\lambda,\psi_1,\dots,\psi_n$ with the relations
\[
 \psi_i\psi_j+\psi_j\psi_i=\lambda\delta_{i,j'}.
\]
Note that $\lambda=\psi_1\psi_n+\psi_n\psi_1=\frac2n\sum_{i=1}^n\psi_i\psi_{i'}$ is a central element and the Grassmann algebra $\Xi_{A_n}(\CC)$ is the quotient of $\Xi_{B_n}(\CC)$ by the relation $\lambda=0$ (by fixing a non-zero value of $\lambda$ we obtain the Clifford algebra $\mathrm{Cl}_n(\CC)$).

The algebra $\gX_{\wt B_n}(\CC)$ is generated by $\wt\lambda,x^1,\dots,x^n$ with the relations
\[
 x^ix^j-x^jx^i=\wt\lambda\varepsilon_i\delta_{i,j'},
\]
where $n=2r$. If $n\ge4$ then the element $\wt\lambda=x^1x^n-x^nx^1=\frac1r\sum_{i=1}^r(x^ix^{i'}-x^{i'}x^i)$ is central and~$\gX_{\wt B_n}(\CC)$ is the universal enveloping of the $(n+1)$-dimensional Heisenberg Lie algebra (by fixing a non-zero value of $\wt\lambda$ we obtain the Weyl algebra $A_r(\CC)$).

The algebra $\Xi_{\wt B_n}(\CC)$ the quotient of the Grassmann algebra $\Xi_{A_n}(\CC)$ with the Grassmann variables $\psi_1,\dots,\psi_n$ by the relation
\[
 \sum_{i=1}^r\psi_i\psi_{i'}=0.
\]
Note that $\sum_{i=1}^r\psi_i\psi_{i'}=\frac12\sum_{i=1}^r(\psi_i\psi_{i'}-\psi_{i'}\psi_i)$.
The group of matrices $G\in {\rm GL}(n,\CC)$ preserving the antisymmetric bilinear form $\omega_r(x,y)=\sum_{i=1}^r (x_iy_{i'}-y_{i'}x_i)$ is isomorphic to ${\rm Sp}(2r,\CC)$.

These forms of the quadratic algebras $\gX_{B_n}(\CC)$ and $\Xi_{\wt B_n}(\CC)$ give us the relation between the Manin matrices of the types $B$, $C$, $D$ and the Manin matrices (of type $A$) considered in~Section~\ref{secMm}.

\begin{Prop} \label{PropMAAAB}
Let $M\in\gR\otimes\Hom(\CC^m,\CC^n)$.
\begin{itemize}
\itemsep=0pt
 \item {\sloppy If $M$ is an $(A_n,A_m)$-Manin matrix, then $M$ is an $(A_n,B_m)$-Manin matrix and an $\big(\wt B_n,A_m\big)$-Manin matrix in the same time.

 }
 \item If $M$ is a $(B_n,B_m)$-Manin matrix, then $M$ is an $(A_n,B_m)$-Manin matrix.
 \item If $M$ is a $\big(\wt B_n,\wt B_m\big)$-Manin matrix, then $M$ is a $\big(\wt B_n,A_m\big)$-Manin matrix.
\end{itemize}
\end{Prop}

This proposition follows from Propositions~\ref{PropBnXX} and~\ref{propAB}. Alternatively one can use the algebras $\gX_{\wt B_n}(\CC)$ and $\Xi_{B_n}(\CC)$. One can also prove Proposition~\ref{PropMAAAB} directly by multiplying the definition~\eqref{AMMB} from the left or from the right by an appropriate matrix: for example, multiplication of $\frac{1-P_n}2M^{(1)}M^{(2)}\frac{1+P_m}2=0$ by the operator $1-\frac{Q_m}m$ from the right gives $\frac{1-P_n}2M^{(1)}M^{(2)}\big(\frac{1+P_m}2-\frac{Q_m}m\big)=0$.

Consider a $2\times2$ matrix $M=\left(\begin{smallmatrix} a & b \\ c & d \end{smallmatrix}\right)$. It is an $(A_2,B_2)$-Manin matrix (a Manin matrix of type~$D$) iff $[a,c]=[b,d]=0$. Note that $\wt B_2=0$, so the algebra $\gX_{\wt B_2}(\CC)$ is a free non-commutative algebra with 2 generators (without relations). This implies that any $2\times2$ matrix is an $\big(\wt B_2,A_2\big)$-Manin matrix (a Manin matrix of type $C$).

Let us generalise the relations~\eqref{Mikjk} and \eqref{Mikjl} to the case of Manin matrices of the ty\-pes~$B$,~$C$,~$D$.

\begin{Prop}
 A matrix $M\in\gR\otimes\Hom(\CC^m,\CC^n)$ is an $(A_n,B_m)$-Manin matrix iff
\begin{align} \label{AnBm}
 \big[M^i_k,M^j_l\big]+\big[M^i_l,M^j_k\big]=\Lambda^{ij}\delta_{k+l,m+1}, \qquad i<j,\quad k\le l,
\end{align}
where $\Lambda^{ij}=\big[M^i_1,M^j_m\big]+\big[M^i_m,M^j_1\big]$.
The matrix $M$ is a $\big(\wt B_n,A_m\big)$-Manin matrix iff
\begin{align} \label{wtBnAm}
\big[M^i_k,M^j_l\big]+\big[M^i_l,M^j_k\big]=\wt\Lambda_{kl}\delta_{i+j,n+1}, \qquad i<j,\quad k\le l,
\end{align}
where $\wt\Lambda_{kl}=\big[M^1_k,M^n_l\big]+\big[M^1_l,M^n_k\big]$.
\end{Prop}

\begin{proof} We use Proposition~\ref{propAB}. Let $\psi_1,\dots,\psi_n\in\Xi_{A_n}(\CC)$ be the Grassmann variables. We need to subject the new variables $\phi_k=\sum_{i=1}^n\psi_iM^i_k\in\Xi_{A_n}(\gR)$ to the relations of the algebra~$\Xi_{B_m}(\CC)$, we derive $\phi_k\phi_l+\phi_l\phi_k=\lambda\delta_{k+l,m+1}$, where $k\le l$ and $\lambda=\phi_1\phi_m+\phi_m\phi_1$. Substitution $\phi_k=\sum_{i=1}^n\psi_iM^i_k$ and pairing with $\psi^i\psi^j=-\psi^j\psi^i\in\Xi^*_{A_n}(\CC)$ gives the formula~\eqref{AnBm}. The~rela\-tions~\eqref{wtBnAm} is obtained similarly by considering the elements $y_i=\sum_{k=1}M^i_kx^k\in\gX_{A_m}(\gR)$, where~$x^k$ are the generators of $\gX_{A_m}(\CC)=\CC[x^1,\dots,x^m]$. \end{proof}

\begin{Cor}
The algebra $\gU_{A_n,B_m}$ is generated by $M^i_k$, $i=1,\dots,n$, $k=1,\dots,m$, and $\Lambda^{ij}$, $i,j=1,\dots,n$, $i<j$, with the relations~\eqref{AnBm}.
The algebra $\gU_{\wt B_n,A_m}$ is generated by $M^i_k$, $i=1,\dots,n$, $k=1,\dots,m$, and $\wt\Lambda_{kl}$, $k,l=1,\dots,m$, $k\le l$, with the relations~\eqref{wtBnAm}.
\end{Cor}

A Manin matrix of type $B$, $C$ or $D$ does not always keep so under such operations as taking a submatrix, permutation of rows or columns, doubling of a row or column, a composition of them, but sometimes it does.

\begin{Cor} Let $I=(i_1,\dots,i_k)$ and $J=(j_1,\dots,j_l)$, where $1\le i_s\le n$ and $1\le j_t\le m$ for any $s=1,\dots,k$ and $t=1,\dots,l$. Let $M=\big(M^i_j\big)$ be an $n\times m$ matrix over $\gR$ and $M_{IJ}$ be $k\times l$ matrix with entries $(M_{IJ})_{\rm st}=M^{i_s}_{j_t}$.
\begin{itemize}
\itemsep=0pt
 \item Let $M$ be an $(A_n,B_m)$-Manin matrix and $j_s+j_t\ne m+1$ for any $s,t=1,\dots,l$, then $M_{IJ}$ is an $(A_k,A_l)$-Manin matrix.
 \item Let $M$ be an $(A_n,B_m)$-Manin matrix and $j_s+j_t=m+1$ iff $s+t=l+1$ $($this condition implies that $j_1,\dots,j_l$ are pairwise different and hence $l\le m)$, then the matrix $M_{IJ}$ is an~$(A_k,B_l)$-Manin matrix.
 \item Let $M$ be a $\big(\wt B_n,A_m\big)$-Manin matrix and $i_s+i_t\ne n+1$ for any $s,t=1,\dots,k$, then $M_{IJ}$ is an $(A_k,A_l)$-Manin matrix.
 \item Let $M$ be a $\big(\wt B_n,A_m\big)$-Manin matrix and $i_s+i_t=n+1$ iff $s+t=k+1$ $($this condition implies that $i_1,\dots,i_k$ are pairwise different and hence $k\le n)$, then $M_{IJ}$ is a $\big(\wt B_k,A_l\big)$-Manin matrix.
\end{itemize}
\end{Cor}

Finally we give the relation of the Manin matrices of types $B$, $C$ and $D$ with the Yangians~\cite{AACFR,AMR,MolevSO}. Let $\g_n=\mathfrak{so}_n$ for the $B$ and $D$ cases and $\g_n=\mathfrak{sp}_n$ for the $C$ case. Consider the $R$-matrices
\[
 R_{\mathfrak{so}_n}(u)=1-\frac{P_n}{u}+\frac{Q_n}{u-n/2+1},\qquad R_{\mathfrak{sp}_n}(u)=1-\frac{P_n}{u}+\frac{\wt Q_n}{u-n/2-1}.
\]
The Yangian $Y(\g_n)$ is an algebra generated by $t^{(r)}_{ij}$, $i,j=1,\dots,n$, $r\in\ZZ_{\ge1}$, with the commutation relations
\begin{align}
 R_{\g_n}(u-v)T^{(1)}(u)T^{(2)}(v)&=T^{(2)}(v)T^{(1)}(u)R_{\g_n}(u-v), \label{RTTB}
\end{align}
and $T'(u)T(u)=1$, where $T(u)=1+\sum_{r\ge1}\sum_{i,j=1}^n t^{(r)}_{ij}E_{i}{}^{j}\,u^{-r}$,
\begin{gather*}
T'(u)=1+\sum_{r\ge1}\sum_{i,j=1}^nt^{(r)}_{j'i'}E_{i}{}^{j}\,(u+n/2-1)^{-r} \qquad
\text{for} \quad \g_n=\mathfrak{so}_n,
\\
T'(u)=1+\sum_{r\ge1}\sum_{i,j=1}^n\varepsilon_i\varepsilon_jt^{(r)}_{j'i'}E_{i}{}^{j}\,(u+n/2+1)^{-r} \qquad\text{for}\quad \g_n=\mathfrak{sp}_n.
\end{gather*}
Note that $R_{\mathfrak{so}_n}(-1)=2(1-B_n)$ and $R_{\mathfrak{sp}_n}(1)=2\wt B_n$, hence by substituting $u=v-1$ and $v=u-1$ to~\eqref{RTTB} we obtain
\begin{gather}
 (1-B_n)T^{(1)}(v-1)T^{(2)}(v)=T^{(2)}(v)T^{(1)}(v-1)(1-B_n),\qquad \text{for}\quad \g_n=\mathfrak{so}_n, \label{BTTso}
 \\
 \wt B_nT^{(1)}(u)T^{(2)}(u-1)=T^{(2)}(u-1)T^{(1)}(u)\wt B_n, \qquad\text{for}\quad \g_n=\mathfrak{sp}_n. \label{BTTsp}
\end{gather}

By multiplying~\eqref{BTTso} by $B_n$ from the left and by exchanging tensor factors we derive that $M=T(v)e^{-\frac{\partial}{\partial v}}$ is a $B_n$-Manin matrix. Due to Propositions~\ref{PropMMBCD} and~\ref{PropMAAAB} this is a Manin matrix of type $B$ or $D$.

By multiplying~\eqref{BTTsp} by $1-\wt B_n$ from the right we see that $M=T(u)e^{-\frac{\partial}{\partial u}}$ is a $\wt B_n$-Manin matrix and hence is a Manin matrix of the type $C$.

\subsection[Minor operators for B, C, D cases and Brauer algebras]
{Minor operators for $\boldsymbol B$, $\boldsymbol C$, $\boldsymbol D$ cases and Brauer algebras}\label{secBrauer}

Remind that the pairing operators $S_{(k)}$ and $A_{(k)}$ for the idempotent $A_n$ are the symmetrizers and anti-symmetrizers of the $k$-th tensor power. We will denote them as $S^{\mathfrak{gl}_n}_{(k)}$ and $A^{\mathfrak{gl}_n}_{(k)}$ to differ them from other pairing operators. The $A$-minors for a Manin matrix of type $B$ or $D$ (or,~more generally, of an $(A_n,B_m)$-Manin matrix) is the column determinants of submatrices. The $S$-minors for a Manin matrix $M$ of type $C$ (or, more generally, of a $\big(\wt B_n,A_m\big)$-Manin matrix $M$) are the normalised row permanents $\frac1{\nu_J}\perm(M_{IJ})$ (see the formulae~\eqref{rowperm} and~\eqref{nuI}).

The $S$-minors for the Manin matrices of types $B$, $D$ and the $A$-minors for the Manin matrices of type $C$ are given by $k$-th $S$-operators for $B_n$ and by $k$-th $A$-operators for $\wt B_n$ respectively. They can be constructed by the method described in Section~\ref{secConstr}. In this case the role of the algebra~$\gU_k$ is played by the Brauer algebra~\cite{Br}. Invariant idempotents in this algebra were constructed in~\cite{I,HX,IM,IMO}.

\begin{Def}
 Let $\omega\in\CC$. The Brauer algebra $\B_k(\omega)$ is an algebra generated by the elements $\sigma_1,\dots,\sigma_{k-1}$ and $\epsilon_1,\dots,\epsilon_{k-1}$ with the relations
\begin{gather}
 \sigma_a^2=1, \qquad \epsilon_a^2=\omega\epsilon_a, \qquad \sigma_a\epsilon_a=\epsilon_a\sigma_a=\epsilon_a, \qquad a=1,\dots,k-1, \label{DefBr1} \\
 \sigma_a\sigma_b=\sigma_b\sigma_a, \qquad \epsilon_a\epsilon_b=\epsilon_b\epsilon_a, \qquad \sigma_a\epsilon_b=\epsilon_b\sigma_a, \qquad|a-b|>1, \label{DefBr2} \\
 \sigma_a\sigma_{a+1}\sigma_a=\sigma_{a+1}\sigma_a\sigma_{a+1}, \qquad
 \epsilon_a\epsilon_{a+1}\epsilon_a=\epsilon_a, \qquad
 \epsilon_{a+1}\epsilon_a\epsilon_{a+1}=\epsilon_{a+1},\qquad \notag \\
 \sigma_a\epsilon_{a+1}\epsilon_a=\sigma_{a+1}\epsilon_a, \qquad
 \epsilon_{a+1}\epsilon_a\sigma_{a+1}=\epsilon_{a+1}\sigma_a, \qquad a=1,\dots,k-2. \label{DefBr3}
\end{gather}
\end{Def}

The subalgebra of $\B_k(\omega)$ generated by $\sigma_1,\dots,\sigma_{k-1}$ is naturally identified with the group algebra $\CC[\SSS_k]$. Also, for $b\le k$ the subalgebra of $\B_k(\omega)$ generated by $\sigma_1,\dots,\sigma_{b-1},\epsilon_1,\dots,\epsilon_{b-1}$ is naturally identified with $\B_b(\omega)$.

If $\omega$ is a positive integer or an even negative integer then the Brauer algebra $\B_k(\omega)$ has the following representation on a tensor power of a finite-dimensional vector space, which extends the representation $\rho^+$ or $\rho^-$ of the symmetric group (see, e.g.,~\cite{Br,IMO,Nazarov,MolevSO}).

\begin{Prop}
The formulae
\begin{gather*}
\rho(\sigma_a)=P_n^{(a,a+1)}, \qquad \rho(\epsilon_a)=Q_n^{(a,a+1)},
\\
\wt\rho(\sigma_a)=-P_n^{(a,a+1)}, \qquad \wt\rho(\epsilon_a)=-\wt Q_n^{(a,a+1)}.
\end{gather*}
define algebra homomorphisms $\rho\colon\B_k(n)\to\End\big((\CC^n)^{\otimes k}\big)$ and $\wt\rho\colon\B_k(-n)\to\End\big((\CC^n)^{\otimes k}\big)$.
\end{Prop}

\begin{proof} The relations for the generators $\sigma_a$ are valid since $P_n$ satisfies the braid relation. The relations~\eqref{DefBr1} follow from~\eqref{Qproper}, \eqref{wtQproper}. The relations~\eqref{DefBr2} are obvious. Further, by direct calculation we obtain
\begin{gather}
 Q_n^{(12)}Q_n^{(23)}Q_n^{(12)}=Q_n^{(12)}, \qquad
 P_n^{(12)}Q_n^{(23)}Q_n^{(12)}=P_n^{(23)}Q_n^{(12)}, \label{QQPQ}
 \\
 \wt Q_n^{(12)}\wt Q_n^{(23)}\wt Q_n^{(12)}=\wt Q_n^{(12)}, \qquad
 P_n^{(12)}\wt Q_n^{(23)}\wt Q_n^{(12)}=-P_n^{(23)}\wt Q_n^{(12)}. \label{QQPQwh}
\end{gather}
They give some of the relations~\eqref{DefBr3}. The rest of~\eqref{DefBr3} is a consequence of the formulae~\eqref{QQPQ}, \eqref{QQPQwh}, \eqref{P1Def}, $Q_n^{(21)}=Q_n^{(12)}$ and $\wt Q_n^{(21)}=\wt Q_n^{(12)}$. \end{proof}

Define an augmentation $\varepsilon\colon\B_k(\omega)\to\CC$ by
\[
\varepsilon(\sigma_a)=1,\qquad \varepsilon(\epsilon_a)=0.
\]
The subalgebras $\U_k$ and $\wt\U_k$ of $\End\big((\CC^n)^{\otimes k}\big)$ generated by $B_n^{(a,a+1)}$ and $\wt B_n^{(a,a+1)}$ are contained in the images of $\rho$ and $\wt\rho$ respectively, but they do not coincide with these images, so we need to check the conditions~\eqref{Xpi}--\eqref{piPsi} of Theorem~\ref{ThUkrhosk}. They follow from Proposition~\ref{PropBnXX}. Moreover, we have $\varepsilon(u_a)=0$, $\rho(u_a)=B_n^{(a,a+1)}$ and $\wt\rho(\wt u_a)=1-\wt B_n^{(a,a+1)}$ for the elements $u_a=\frac{1-\sigma_a}2+\frac{\epsilon_a}{\omega}$ $\in\B_k(\omega)$, where $\omega=n$ and $\omega=-n$ for $\rho$ and $\wt\rho$ respectively. Hence, if we construct an element $s_{(k)}\in\B_k(\omega)$ such that
\begin{align}
\sigma_as_{(k)}=s_{(k)}\sigma_a=s_{(k)}, \qquad \epsilon_as_{(k)}=s_{(k)}\epsilon_a=0, \qquad \varepsilon(s_{(k)})=1 \label{skCond}
\end{align}
for any $a=1,\dots,k-1$ we obtain the corresponding pairing operators, namely a $k$-th $S$-operator $S_{(k)}=\rho(s_{(k)})$ for $A=B_n$ and a $k$-th $A$-operator $A_{(k)}=\wt\rho(s_{(k)})$ for $A=\wt B_n$.

Let $\tau_{ab}=\sigma_a\sigma_{a+1}\cdots\sigma_{b-1}$, $a<b$. This is the cycle $(a,a+1,\dots,b)\in\SSS_k$. Note that the transpositions have the form $\sigma_{ab}=\sigma_{ba}=\tau_{ab-1}\sigma_{b-1}\tau_{ab-1}^{-1}$. Define $\epsilon_{ab}=\epsilon_{ba}=\tau_{ab-1}\epsilon_{b-1}\tau_{ab-1}^{-1}$. We have $\sigma\sigma_{ab}\sigma^{-1} =\sigma_{\sigma(a)\sigma(b)}$ and $\sigma\epsilon_{ab}\sigma^{-1} =\epsilon_{\sigma(a)\sigma(b)}$ for any $\sigma\in\SSS_k$.
These elements can be considered graphically in the following way~\cite{Br,MolevSO}. Remind that an element $\sigma\in\SSS_k$ is presented by a diagram where each number $i$ from the top row is connected with $\sigma(i)$ from the bottom row. In particular, the diagram corresponding to the element $\sigma_{ab}$ is
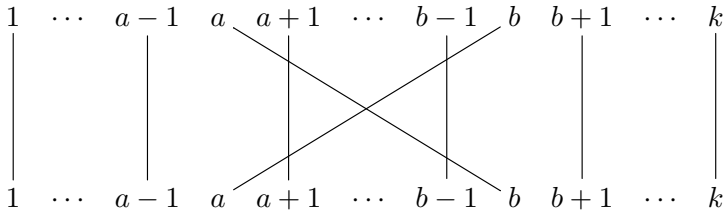
\begin{align*}
 \xymatrix@C=0.5em{ 1\ar@{-}[dd] & \cdots & a-1\ar@{-}[dd] & a\ar@{-}[ddrrrr] & a+1\ar@{-}[dd] & \cdots & b-1\ar@{-}[dd]& b\ar@{-}[ddllll] & b+1\ar@{-}[dd] & \cdots & k\ar@{-}[dd] \\
\\
 1 & \cdots & a-1 & a & a+1 & \cdots & b-1 & b & b+1 & \cdots & k
}
\end{align*}
The element $\epsilon_{ab}$ is presented by the diagram
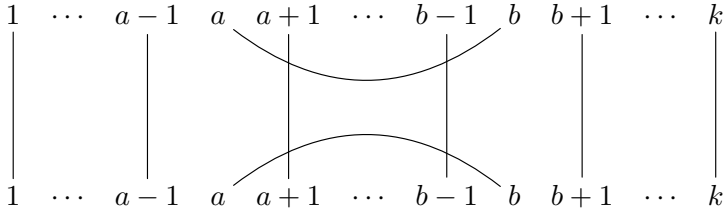
\begin{align*}
 \xymatrix@C=0.5em{ 1\ar@{-}[dd] & \cdots & a-1\ar@{-}[dd] & a\ar@/_2pc/@{-}[rrrr] & a+1\ar@{-}[dd] & \cdots & b-1\ar@{-}[dd] & b & b+1\ar@{-}[dd] & \cdots & k\ar@{-}[dd] \\
\\
 1 & \cdots & a-1 & a\ar@/^2pc/@{-}[rrrr] & a+1 & \cdots & b-1 & b & b+1 & \cdots & k
}
\end{align*}
To multiply two elements one needs to put the corresponding diagrams one over the other and substitute each arising loop by a factor $\omega$. This allows to simplify calculations in the Brauer algebra.

Consider the elements
\[
 y_b=\sum_{a=1}^{b-1}(\sigma_{ab}-\epsilon_{ab})\in\B_b(\omega).
\]
These elements (up to a constant term) were introduced in~\cite{Nazarov} as analogues of Jucys--Murphy elements for the Brauer algebra. One can check that each element $y_b$ commute with the subalgebra $\B_{b-1}(\omega)$. This implies, in particular, commutativity $[y_a,y_b]=0$ (see~\cite{Nazarov} for details). The~images of the elements $y_b\in\B_k(n)$ and $y_b\in\B_k(-n)$ under the homomorphisms $\rho$ and $\wt\rho$ are the matrices
\[
 Y_b=\rho(y_b)=\sum_{a=1}^{b-1}\big(P_n^{(ab)}-Q_n^{(ab)}\big), \qquad
 \wt Y_b=\wt\rho(y_b)=-\sum_{a=1}^{b-1}\big(P_n^{(ab)}-\wt Q_n^{(ab)}\big),
\]
respectively.

The following analogue of the symmetrizer can be written in terms of the Jucys--Murphy elements (see~\cite{I,IR,MolevSO}):
\begin{align}
s_{(k)}=\frac1{k!}\prod_{b=2}^k\frac{(y_b+1)(y_b+\omega+b-3)}{2b+\omega-4}. \label{sk}
\end{align}

\begin{Prop} The elements~\eqref{sk} satisfy the conditions~\eqref{skCond}. Hence the matrices
\begin{gather*}
 S_{(k)}=S_{(k)}^{\mathfrak{so}_n}=\rho(s_{(k)})=\frac1{k!}\prod_{b=2}^k\frac{(Y_b+1)(Y_b+n+b-3)}{2b+n-4} \qquad (2\le k), 
 \\
 A_{(k)}=A_{(k)}^{\mathfrak{sp}_n}=\wt\rho(s_{(k)})=\frac1{k!}\prod_{b=2}^k\frac{(\wt Y_b+1)(\wt Y_b-n+b-3)}{2b-n-4}\qquad (2\le k\le r+1). \label{Aksp}
\end{gather*}
are the $k$-th $S$-operator for $B_n$ and the $k$-th $A$-operator for $\wt B_n$ respectively.
\end{Prop}

\begin{proof} By substituting $\varepsilon(y_b)=b-1$ we obtain $\varepsilon(s_{(k)})=1$. Suppose by induction that $\sigma_as_{(k-1)}=s_{(k-1)}$ and $\epsilon_as_{(k-1)}=0$ for all $a=1,\dots,k-2$. Let $I_{k-1,k}$ be the left ideal of $\B_k(\omega)$ generated by the elements $u-\varepsilon(u)$, $u\in\B_{k-1}(\omega)$, then $us_{(k-1)}=0$ for any $u\in I_{k-1,k}$. Since $y_k$ commutes with elements of $\B_{k-1}(\omega)$ and $s_{(k)}$ is proportional to $(y_k+\omega+k-3)(y_k+1)s_{(k-1)}$ we have $\sigma_as_{(k)}=s_{(k)}$ and $\epsilon_as_{(k)}=0$ for all $a=1,\dots,k-2$. Further, we obtain
\begin{align} \label{epsy}
 \epsilon_{k-1}(y_k+\omega+k-3)=\epsilon_{k-1}\sigma_{k-1}-\epsilon_{k-1}^2+\epsilon_{k-1}\sum_{a=1}^{k-2}(\sigma_{ak}-\epsilon_{ak})+(\omega+k-3)\epsilon_{k-1}.
\end{align}
Note that $\epsilon_{k-1}\sigma_{k-1}=\epsilon_{k-1}$, $\epsilon_{k-1}^2=\omega\epsilon_{k-1}$, $\epsilon_{k-1}\sigma_{ak}=\sigma_{ak}\epsilon_{k-1,a}\in I_{k-1,k}$ and $\epsilon_{k-1}\epsilon_{ak}=\epsilon_{k-1}\sigma_{a,k-1}\in\epsilon_{k-1}+I_{k-1,k}$ for any $a\le k-2$, so the element~\eqref{epsy} belongs to $I_{k-1,k}$. Hence $\epsilon_{k-1} s_{(k)}$ is proportional to the element $\epsilon_{k-1}(y_k+\omega+k-3)(y_k+1)s_{(k-1)}\in I_{k-1,k}(y_k+1)s_{(k-1)}=0$. Analogously, we derive
\begin{align*}
 \sigma_{k-1}(1+y_k)&=
 \sigma_{k-1}+1-\epsilon_{k-1}+\sum_{a=1}^{k-2}(\sigma_{ak} \sigma_{k-1,a} -\sigma_{k-1}\epsilon_{ak})
 \\
 &\in 1+y_k+I_{k-1,k}\sum_{a=1}^{k-2}\B_k(\omega)\epsilon_{ak}.
\end{align*}
By using the formulae $\epsilon_{ak}=\sigma_{a,k-1}\epsilon_{k-1}\sigma_{a,k-1}$, $(\sigma_{a,k-1}-1)y_k=y_k(\sigma_{a,k-1}-1)\in I_{k-1,k}$ and the fact that~\eqref{epsy} belongs to $I_{k-1,k}$ one yields $\epsilon_{ak}(y_k+\omega+k-3)\in I_{k-1,k}$. This implies the inclusion $\sigma_{k-1}(1+y_k)(y_k+\omega+k-3)\in(1+y_k)(y_k+\omega+k-3)+I_{k-1,k}$, so that $\sigma_{k-1}s_{(k)}=s_{(k)}$. Thus~$s_{(k)}$ is left invariant with respect to $\varepsilon$. Due to Proposition~\ref{propInv} we obtain the right invariance of~$s_{(k)}$.
The rest follows from Theorem~\ref{ThUkrhosk} and Proposition~\ref{PropBnXX} (see the reasoning above). \end{proof}

The element~\eqref{sk} can be rewritten in the following form~\cite{IM,IMO}:
\[
 s_{(k)}=\frac1{k!}\prod_{1\le a<b\le k}\bigg(1+\frac{\sigma_{ab}}{b-a}-\frac{\epsilon_{ab}}{\omega/2+b-a-1}\bigg),
\]
where the product is calculated in the lexicographic order on the pairs $(a,b)$. By taking representations $\rho\colon\B_k(n)\to\End\big((\CC^n)^{\otimes k}\big)$ and $\wt\rho\colon\B_k(-n)\to\End\big((\CC^n)^{\otimes k}\big)$ we obtain the corresponding pairing operators:
\[
 S_{(k)}^{\mathfrak{so}_n}=\frac1{k!}\prod_{1\le a<b\le k}R_{\mathfrak{so}_n}(a-b), \qquad
 A_{(k)}^{\mathfrak{sp}_n}=\frac1{k!}\prod_{1\le a<b\le k}R_{\mathfrak{sp}_n}(b-a).
\]
In particular, $S_{(2)}^{\mathfrak{so}_n}=1-B_n$ and $A_{(2)}^{\mathfrak{sp}_n}=\wt B_n$.

\begin{Rem}
One can define another augmentation $\varepsilon'\colon\B_k(\omega)\to\CC$ by $\varepsilon'(\sigma_a)=-1$, $\varepsilon'(\epsilon_a)=0$. The corresponding idempotent coincides with the usual anti-symmetrizer $a_{(k)}=\frac1{k!}\prod_{\sigma\in\SSS_k}(-1)^\sigma\sigma$ since $\sigma_b a_{(k)}=-a_{(k)}$, $\epsilon_b a_{(k)}=\epsilon_b\sigma_b a_{(k)}=-\epsilon_b a_{(k)}$ for all $b=1,\dots,k-1$ and $\varepsilon'(a_{(k)})=1$. Hence the operators $A^{\mathfrak{gl}_n}_{(k)}=\rho(a_{(k)})$ and $S^{\mathfrak{gl}_n}_{(k)}=\wt\rho(a_{(k)})$ define the pairings $\Xi^*_{B_n}(\CC)\times\Xi_{B_n}(\CC)\to\CC$ and $\gX_{\wt B_n}(\CC)\times\gX_{\wt B_n}^*(\CC)\to\CC$ respectively. But these pairings are degenerate since the conditions~\eqref{Xpi}--\eqref{piPsi} are not valid for $\varepsilon'$.
\end{Rem}

Let us calculate the ranks of the pairing operators $S_{(k)}^{\mathfrak{so}_n}$ and $A_{(k)}^{\mathfrak{sp}_n}$. They are equal to their traces. According to~\cite[formulae (1.52), (1.72), (1.77)]{MolevSO} the (full) traces of these pairing opera\-tors~are
\begin{gather*}
 \tr S_{(k)}^{\mathfrak{so}_n}=\frac{n+2k-2}{n+k-2}\;{n+k-2\choose k}
=\frac{n+2k-2}{k}\;{n+k-3\choose k-1}\qquad(2\le k), \\[1ex] 
 \tr A_{(k)}^{\mathfrak{sp}_n}=(-1)^k\;\frac{-n+2k-2}{-n+k-2}\;{-n+k-2\choose k}
=\frac{n-2k+2}{k}\;{n+1\choose k-1} \qquad (2\le k\le r+1). 
\end{gather*}
In particular, by substituting $n=2r$ and $k=r+1$ to the expression for $\tr A^{\mathfrak{sp}_n}_{(k)}$ we derive $\rk A^{\mathfrak{sp}_n}_{(k)}=\tr A^{\mathfrak{sp}_n}_{(k)}=0$. Hence
\[
 A^{\mathfrak{sp}_n}_{(k)}=0\qquad \text{and}\qquad \Xi_{\wt B_{2r}}(\CC)_k=0\qquad
 \text{for all}\quad k\ge r+1.
\]
This explains the inequality for $k$ in the formula~\eqref{Aksp}: for other $k$ the operator $A_{(k)}^{\mathfrak{sp}_n}$ vanishes (it vanishes even for the last value $k=r+1$). Let us conclude.

\begin{Prop}
 The dimension of the component $\gX_{B_n}(\CC)_k$ is not zero for any $k\ge0$. The dimension of $\Xi_{B_n}(\CC)_k$ is not zero iff $0\le k\le r$. Namely, we have
\begin{gather*}
\dim\gX_{B_n}(\CC)_k=\frac{n+2k-2}{k}\;{n+k-3\choose k-1} \qquad (1\le k),
\\
\dim\Xi_{\wt B_n}(\CC)_k=\frac{n-2k+2}{k}\;{n+1\choose k-1} \qquad (1\le k\le r).
\end{gather*}
\end{Prop}

The minor operators for Manin matrices $M$ of the types $B$ and $D$ have the form
\[
 \Min_{S_{(k)}}(M)=M^{(1)}\cdots M^{(k)}S^{\mathfrak{so}_n}_{(k)}, \qquad
 \Min^{A_{(k)}}(M)=A^{\mathfrak{gl}_n}_{(k)}M^{(1)}\cdots M^{(k)},
\]
$\big(A^{\mathfrak{gl}_n}_{(k)}=0$ for $k\ge n+1\big)$.
For Manin matrices $M$ of the type $C$ we have
\[
 \Min_{S_{(k)}}(M)=M^{(1)}\cdots M^{(k)}S^{\mathfrak{gl}_n}_{(k)}, \qquad
 \Min^{A_{(k)}}(M)=A^{\mathfrak{sp}_n}_{(k)}M^{(1)}\cdots M^{(k)},
\]
\big($A^{\mathfrak{sp}_{2r}}_{(k)}=0$ for $k\ge r+1$\big).

Let us give an example. The bases of the spaces $\gX_{B_2}(\CC)_2$ and $\gX_{B_2}^*(\CC)_2$ consist of the elements $x^i_{(2)}=(x^i)^2$ and $x_i^{(2)}=(x_i)^2$, $i=1,2$, respectively. The operator $S^{\mathfrak{so}_2}_{(2)}=1-B_2$ defines the pairing $\big\la x^j_{(2)},x_i^{(2)}\big\ra=\delta_i^j$. Let $y^i=\sum_{j=1,2}M^i_jx^j$, then $y^i_{(2)}=(y^i)^2$. Since $x^1x^2=x^2x^1=0$ we have $y^i_{(2)}=\sum_{j=1,2}\big(M^i_j\big)^2x^j_{(2)}$. Thus the second $S$-minor of a $(B_2,A_2)$-Manin matrix $M=\left(\begin{smallmatrix} a & b \\ c & d \end{smallmatrix}\right)$ in the basis $\big(x^1_{(2)},x^2_{(2)}\big)$ has the form
\[
 \Min_{S_{(2)}}(M)=\begin{pmatrix} a^2 & b^2 \\ c^2 & d^2 \end{pmatrix}\!.
\]

\begin{Rem}
 The existence of the pairing operators for the idempotents $A=B_n$ and $A=\wt B_n$ can be deduced from Theorem~\ref{ThPOviaDB} (in the case, when $\CC$ is exactly the field of complex numbers). Since $A^\top=A$ we have isomorphisms $V_k\cong\overline V_k$ and $W_k\cong\overline W_k$ given by restriction of the map $(\CC^n)^{\otimes k}\to\big((\CC^n)^{\otimes k}\big)^*$, $e_{i_1\dots i_k}\mapsto e^{i_1\dots i_k}$. Under these isomorphisms the natural pairings $V_k\times\overline V_k\to\CC$ and $W_k\times\overline W_k\to\CC$ are identified with restrictions of the standard bilinear form $\la e_{i_1\dots i_k},e_{j_1\dots j_k}\ra=\delta_{i_1j_1}\cdots\delta_{i_kj_k}$ to the spaces $V_k$ and $W_k$ respectively. Since the subspaces $V_k$ and $W_k$ are given by linear equations with real coefficients the restrictions of this bilinear form are non-degenerate, so the pairing operators $S_{(k)}$ and $A_{(k)}$ exist for any $k$. In particular, the pairing $A$-operators for $B_n$ and the pairing $S$-operators for $\wt B_n$ also exist.
\end{Rem}

\section*{Conclusion}
\addcontentsline{toc}{section}{Conclusion}

It was demonstrated in the works~\cite{CF,CFR,qManin,CM,MR,RST} as well as in the current paper that the Manin matrices have a lot of applications to integrable systems, representation theory and other fields of mathematics and physics. These results show importance of non-commutative geometry developed by Manin in~\cite{Manin87,Manin88,Manin89,ManinBook91,Manin92} for many questions of mathematics and mathematical physics.

By switching the consideration from $A$-Manin matrices to more general $(A,B)$-Manin matrices we obtain a larger class of useful examples such as the $(\wh q,\wh p)$-Manin matrices and the Manin matrices of types $B$, $C$ and $D$. In particular, the theory of the $(\wh q,\wh p)$-Manin matrices gives more complete picture for the $q$-Manin matrices.

It was shown that the tensor notations and usage of idempotents gives a convenient approach to the Manin matrices. In particular, it allows to generalise the notion of minors to the general case. This general theory of minors relates Manin matrices with the representation theory of the symmetric groups and their generalisations such as Hecke and Brauer algebras. This alludes to a possible relation between the theory of Manin matrices and the Schur--Weyl duality.

The left equivalence of idempotents gives us relationship between different considerations of~Manin matrices. For example, the minors of $q$-Manin matrices can be considered by means of the $q$-antisymmetrizer arisen from a representation of the symmetric group as well as by using higher idempotents of the Hecke algebra; the left equivalence implies a simple relation between the minor operators constructed in these two different ways. Moreover, left and right equivalences can be used to investigate the corresponding quadratic algebras as it was done in~Section~\ref{secPO4p}.

It is also discovered that some Lax matrices with spectral parameter is a type of Manin matrices generalised to the infinite-dimensional case. Namely, we defined an idempotent $\wh A_n$ acting on a completed tensor power $\big(\CC^n\big[u,u^{-1}\big]\big)^{\otimes2}$ and proved that these Lax matrices are exactly Manin operators that acts on the tensor factor $\CC\big[u,u^{-1}\big]$ by a function multiplication.

\section*{Appendices}
\addcontentsline{toc}{section}{Appendices}

\appendix

\section{Reduced expression and set of inversions}
\label{appRw}

Let $R$ be a finite root system and $W$ be the corresponding reflection group in the sense of~\cite[Chapter~1]{H}. Let $R^+\subset R$ be a subset of positive roots. Denote the corresponding simple roots by $\alpha_1,\dots,\alpha_r$, where $r$ is the rank of the root system $R$. The simple reflections $s_i=s_{\alpha_i}$ generate the group $W$. The {\it length} of $w\in W$, denoted by $\ell(w)$, is the minimal $k\in\NN_0$ such that $w$ can be presented as a product of $k$ simple reflections. If $\ell=\ell(w)$, then any expression $w=s_{i_1}\cdots s_{i_l}$ is called a {\it reduced expression} for the element $w\in W$.

Note that $\ell(w^{-1})=\ell(w)$. Indeed, since $s_i^2=1$ we have $(s_{i_1}\cdots s_{i_k})^{-1}=s_{i_k}\cdots s_{i_1}$, so the element $w$ can be presented as a product of $k$ simple reflections iff the element $w^{-1}$ can be presented as a product of $k$ simple reflections.

Let $R^-=-R^+$ be the set of negative roots and let $R_w^+:=R^+\cap w^{-1}R^-$. The latter is the set of positive roots $\alpha\in R^+$ such that $w\alpha\in R^-$. As it is proved in~\cite[Section~1.7]{H} the number of such roots is equal to the length of $w$, that is $\ell(w)=|R_w^+|$. Note also that from $-wR_w^+=(-wR^+)\cap(-R^-)=R^+\cap wR^-$ we obtain\vspace{-1ex}
\begin{align}
 R^+_{w^{-1}}=-wR^+_w. \label{wRw}
\end{align}

The symmetric group $\SSS_n$ is a reflection group with the root system $R=\{e_i-e_j\mid i\ne j\}$ of rank $r=n-1$. In this case we have $R^+=\{e_i-e_j\mid i<j\}$, $R^-=\{e_i-e_j\mid i>j\}$, $\alpha_i=e_i-e_{i+1}$, $s_i=\sigma_i$. For a permutation $\sigma\in \SSS_n$ the set $R^+_\sigma$ is identified with the set of inversions of $\sigma$:\vspace{-1ex}
\begin{align}
 R^+_\sigma=\{e_i-e_j\mid i<j,\,\sigma(i)>\sigma(j)\}. \label{Rsigma}
\end{align}
In particular, the number of inversions is equal to the length: $\inv(\sigma)=\ell(\sigma)$. Note that the number of elements $\sigma\in\SSS_k$ with a fixed length $\ell(\sigma)$ can be calculated from the generating function~\eqref{invGF}.

\begin{Prop}
 Let $\ell(w)=\ell$ and $w=s_{i_1}\cdots s_{i_\ell}$ be a reduced expression. Then\vspace{-1ex}
\begin{align}
 &R^+_w=\{\alpha_{i_\ell},s_{i_\ell}\alpha_{i_{\ell-1}},\dots,s_{i_\ell}\cdots s_{i_2}\alpha_{i_1}\}=\{s_{i_\ell}\cdots s_{i_{k+1}}\alpha_{i_k}\mid k=1,\dots,\ell\}, \label{PropRw1} \\
 &R^+_{w^{-1}}=\{\alpha_{i_1},s_{i_1}\alpha_{i_2},\dots,s_{i_1}\cdots s_{i_{\ell-1}}\alpha_{i_\ell}\}=\{s_{i_1}\cdots s_{i_{k-1}}\alpha_{i_k}\mid k=1,\dots,\ell\}. \label{PropRw2}
\end{align}
\end{Prop}

\begin{proof} First we prove that all the roots $\beta_k:=s_{i_\ell}\cdots s_{i_{k+1}}\alpha_{i_k}$ belongs of $R^+_w$. Note that $s_{i_1}\cdots s_{i_k}$ is a reduced expression for any $k=1,\dots,\ell$, since otherwise $\ell(w)<\ell$. By virtue of~\cite[Section~1.6]{H} the equality $\ell(s_{i_1}\cdots s_{i_{k-1}}s_{i_k})=\ell(s_{i_1}\cdots s_{i_{k-1}})+1$ implies\vspace{-1ex}
\begin{gather*}
 s_{i_1}\cdots s_{i_{k-1}}\alpha_{i_k}\in R^+, \qquad k=1,\dots,\ell.
\end{gather*}
In the same way for the reduced expression $w^{-1}=s_{i_\ell}\cdots s_{i_1}$ we obtain\vspace{-1ex}
\begin{gather*}
 \beta_k=s_{i_\ell}\cdots s_{i_{k+1}}\alpha_{i_k}\in R^+, \qquad k=1,\dots,\ell.
\end{gather*}
Since $w\beta_k=s_{i_1}\cdots s_{i_{k-1}}s_{i_k}\alpha_{i_k}=-s_{i_1}\cdots s_{i_{k-1}}\alpha_{i_k}\in R^-$ we derive $\beta_k\in R^+_w$. Due to the formula $\ell(w)=|R^+_w|$ we only need to prove that the roots $\beta_1,\dots,\beta_\ell$ are pairwise different. We prove it by induction on $\ell$. For $\ell=0$ and $\ell=1$ there is nothing to check. Let $w'=s_{i_1}\cdots s_{i_{\ell-1}}$ and $\beta'_k=s_{i_{\ell-1}}\cdots s_{i_{k+1}}\alpha_{i_k}\in R^+_{w'}$. Due to the induction assumption the roots $\beta'_1,\dots,\beta'_{\ell-1}$ are pairwise different. Hence $\beta_k=s_{i_\ell}\beta'_k$, $k=1,\dots,\ell-1$ are also pairwise different. Now suppose that $\beta_\ell=\beta_k$ for some $k=1,\dots,\ell-1$. By acting by $s_{i_\ell}$ to the both hand sides we obtain $-\alpha_{i_\ell}=\beta'_k$, but the roots $\beta'_k\in R^+_{w'}$ are positive. This contradiction ends the proof of the formula~\eqref{PropRw1}. The formula~\eqref{PropRw2} is obtained from~\eqref{wRw} and~\eqref{PropRw1}. \end{proof}

\section{Lie algebras as quadratic algebras}\label{appLie}

Here we present a quadratic algebra closely related with a finite dimensional Lie algebra $\g$.

Let $n=\dim\g+1$ and $\big(x^1,\dots,x^{n-1}\big)$ be a basis of the Lie algebra $\g$ with the brackets $[x^i,x^j]=\sum_{k=1}^{n-1}C^{ij}_kx^k$, $C^{ij}_k\in\CC$. Consider the quadratic algebra with generators $x^1,\dots,x^{n-1},x^n$ and relations
\begin{gather}
 x^ix^j-x^jx^i=\sum_{k=1}^{n-1}C^{ij}_kx^kx^n, \qquad i,j=1,\dots,n-1, \label{xxC} \\
 x^ix^n-x^nx^i=0, \qquad i=1,\dots,n-1. \label{xxn}
\end{gather}
Let $C_\g\in\End(\CC^n\otimes\CC^n)$ be a matrix with entries
\begin{align}
 (C_\g)^{ij}_{kl}=C^{ij}_k\delta_{ln}+C^{ij}_l\delta_{kn}, \qquad i,j,k,l=1,\dots,n, \label{Cijkl}
\end{align}
where we set $C^{ij}_n=C^{in}_k=C^{nj}_k=0$. Since~\eqref{Cijkl} is antisymmetric in $i,j$ and symmetric in $k,l$ we have the formulae
\[
 C_\g^2=0, \qquad C_\g A_n=0 \qquad A_nC_\g=C_\g.
\]
They imply that the operator $A_\g=A_n-\tfrac14C_\g$ is an idempotent. The relations~\eqref{xxC} and \eqref{xxn} define the algebra $\gX_{A_\g}(\CC)$.

\begin{Rem}
By fixing a non-zero value of the central element $x^n$ we obtain the universal enveloping algebra $\mathcal U(\g)$. In other words, the algebra $\gX_{A_\g}(\CC)$ is the quantisation of the algebra $S\g=\CC[\g^*]$ with the Lie--Poisson brackets $\{x^i,x^j\}=\sum_{k=1}^{n-1}C^{ij}_kx^k$. The central element $x^n$ plays the role of the quantisation parameter.
\end{Rem}

\subsection*{Acknowledgements}
The author is grateful to A.~Chervov, V.~Rubtsov, An.~Kirillov, A.~Isaev and A.~Molev for useful discussions and advice.
The author also would like to thank anonymous referees for their attentive reading and numerous useful advice and remarks.

\pdfbookmark[1]{References}{ref}
\LastPageEnding

\end{document}